\documentclass[reqno,11pt,letterpaper]{amsart}

\usepackage{amsmath,amssymb,amsthm,mathrsfs,bbm}
\usepackage{arydshln}
\usepackage{upgreek}
\usepackage{slashed}
\usepackage{xifthen}
\usepackage{graphicx}
\usepackage{longtable}
\usepackage[inline]{enumitem}

\usepackage{xcolor}
\definecolor{winered}{rgb}{0.7,0,0}
\definecolor{lessblue}{rgb}{0,0,0.7}

\usepackage[pdftex,colorlinks=true,linkcolor=winered,citecolor=lessblue,breaklinks=true,bookmarksopen=true]{hyperref}

\makeatletter
\newcommand{\myitem}[3]{\item[#2]\def\@currentlabel{#3}\label{#1}}
\makeatother

\usepackage{titletoc}

\makeatletter

\def\@tocline#1#2#3#4#5#6#7{
\begingroup
  \par
    \parindent\z@ \leftskip#3 \relax \advance\leftskip\@tempdima\relax
                  \rightskip\@pnumwidth plus 4em \parfillskip-\@pnumwidth
    \ifcase #1 
       \vskip 0.6em \hskip 0em 
       \or
       \or \hskip 0em 
       \or \hskip 1em 
    \fi%
    #6
    \nobreak\relax{\leavevmode\leaders\hbox{\,.}\hfill}
    \hbox to\@pnumwidth {\@tocpagenum{#7}}
  \par
\endgroup
}

 \def\l@section{\@tocline{0}{0pt}{0pc}{}{}}

\renewcommand{\tocsection}[3]{%
  \indentlabel{\@ifnotempty{#2}{ 
    \ignorespaces\bfseries{#2. #3}}}
  \indentlabel{\@ifempty{#2}{\ignorespaces\bfseries{#3}}{}} 
    \vspace{1.5pt}}

\renewcommand{\tocsubsection}[3]{%
  \indentlabel{\@ifnotempty{#2}{
    \ignorespaces#2. #3}}
  \indentlabel{\@ifempty{#2}{\ignorespaces #3}{}}
    \vspace{1.5pt}}

\renewcommand{\tocsubsubsection}[3]{%
  \indentlabel{\@ifnotempty{#2}{
    \ignorespaces#2. #3}}
  \indentlabel{\@ifempty{#2}{\ignorespaces #3}{}}
    \vspace{1.5pt}}

\makeatother

\makeatletter
\def\@nomenstarted{0}
\newlength{\@nomenoldtabcolsep}

\newcommand{\nomenstart}
  {%
    \def\@nomenstarted{1}%
    \setlength{\@nomenoldtabcolsep}{\tabcolsep}%
    \setlength{\tabcolsep}{3.5pt}%
    \begin{longtable}{p{0.11\textwidth} p{0.86\textwidth}}
  }

\newcommand{\nomenitem}[2]{%
    \ifcase\@nomenstarted%
      \or 
      \or \\ 
    \fi%
    #1\,{\leavevmode\leaders\hbox{\,.}\hfill} & #2%
    \def\@nomenstarted{2}%
  }%
\newcommand{\nomenend}
  {\\%
      \end{longtable}%
      \setlength{\tabcolsep}{\@nomenoldtabcolsep}%
      \def\@nomenstarted{0}%
  }
\makeatother

\makeatletter
\newcommand{\vast}{\bBigg@{4}}
\newcommand{\Vast}{\bBigg@{5}}
\makeatother

\allowdisplaybreaks

\numberwithin{equation}{section}
\numberwithin{figure}{section}
\newtheorem{thm}{Theorem}[section]

\newtheorem{prop}[thm]{Proposition}
\newtheorem{lemma}[thm]{Lemma}
\newtheorem{cor}[thm]{Corollary}
\newtheorem*{thm*}{Theorem}
\newtheorem*{prop*}{Proposition}
\newtheorem*{cor*}{Corollary}
\newtheorem*{conj*}{Conjecture}

\theoremstyle{definition}
\newtheorem{definition}[thm]{Definition}

\theoremstyle{remark}
\newtheorem{rmk}[thm]{Remark}
\newtheorem{example}[thm]{Example}

\newcommand{\mc}{\mathcal}
\newcommand{\cA}{\mc A}

\newcommand{\cC}{\mc C}

\newcommand{\cE}{\mc E}

\newcommand{\cL}{\mc L}
\newcommand{\cM}{\mc M}

\newcommand{\cO}{\mc O}

\newcommand{\cR}{\mc R}

\newcommand{\cU}{\mc U}
\newcommand{\cV}{\mc V}

\newcommand{\cX}{\mc X}
\newcommand{\cY}{\mc Y}

\newcommand{\ms}{\mathscr}

\newcommand{\sC}{\ms C}
\newcommand{\sD}{\ms D}

\newcommand{\scri}{\ms I}

\newcommand{\sR}{\ms R}

\newcommand{\sV}{\ms V}
\newcommand{\sY}{\ms Y}

\newcommand{\C}{\mathbb{C}}
\newcommand{\N}{\mathbb{N}}
\newcommand{\R}{\mathbb{R}}
\newcommand{\one}{\mathbbm{1}}

\newcommand{\Sph}{\mathbb{S}}

\newcommand{\sfM}{\mathsf{M}}

\newcommand{\bfB}{\mathbf{B}}

\newcommand{\fX}{\mathfrak{X}}
\newcommand{\fY}{\mathfrak{Y}}

\newcommand{\sld}{\slashed{d}{}}
\newcommand{\slg}{\slashed{g}{}}

\newcommand{\slG}{\slashed{G}{}}

\newcommand{\slK}{\slashed{K}{}}
\newcommand{\slalpha}{\slashed{\alpha}{}}
\newcommand{\slGamma}{\slashed{\Gamma}{}}
\newcommand{\sldelta}{\slashed{\delta}{}}
\newcommand{\slDelta}{\slashed{\Delta}{}}
\newcommand{\slnabla}{\slashed{\nabla}{}}

\newcommand{\sltr}{\operatorname{\slashed\tr}}

\newcommand{\ringrho}{\mathring\rho}

\newcommand{\End}{\operatorname{End}}

\newcommand{\Hom}{\operatorname{Hom}}
\renewcommand{\Re}{\operatorname{Re}}
\renewcommand{\Im}{\operatorname{Im}}
\newcommand{\Id}{\operatorname{Id}}

\newcommand{\supp}{\operatorname{supp}}

\newcommand{\tr}{\operatorname{tr}}

\newcommand{\dv}{\operatorname{div}}

\newcommand{\II}{I\!I}

\newcommand{\dS}{{\mathrm{dS}}}
\newcommand{\CD}{{\mathrm{CD}}}

\newcommand{\Ups}{\Upsilon}

\newcommand{\eps}{\epsilon}
\newcommand{\ff}{\tn{ff}}

\newcommand{\Lra}{\Longrightarrow}

\newcommand{\Llra}{\Longleftrightarrow}
\newcommand{\hra}{\hookrightarrow}
\newcommand{\la}{\langle}

\newcommand{\extcup}{\operatorname{\ol\cup}}
\newcommand{\ol}{\overline}
\newcommand{\pa}{\partial}
\newcommand{\ra}{\rangle}

\newcommand{\tn}{\textnormal}

\newcommand{\ul}[1]{\underline{#1}{}}

\newcommand{\wh}{\widehat}
\newcommand{\wt}{\widetilde}
\newcommand{\xra}{\xrightarrow}

\newcommand{\bop}{{\mathrm{b}}}

\newcommand{\cp}{{\mathrm{c}}}

\newcommand{\scl}{{\mathrm{sc}}}

\newcommand{\Diff}{\mathrm{Diff}}

\newcommand{\Vf}{\mathcal V}

\newcommand{\Vb}{\Vf_\bop}
\newcommand{\Diffb}{\Diff_\bop}

\newcommand{\Tb}{{}^{\bop}T}
\newcommand{\Tbeta}{{}^{\beta}T}

\newcommand{\rcTbdual}[1][]{\ensuremath{\overline{{}^{\bop}T^*\ifthenelse{\isempty{#1}}{}{_#1}}}}

\newcommand{\Nb}{{}^{\bop}N}

\newcommand{\bdiff}{{}^{\bop}d}

\newcommand{\Tsc}{{}^{\scl}T}

\newcommand{\half}{\tfrac{1}{2}}

\newcommand{\bhm}[1][]{\ensuremath{M_{\bullet\ifthenelse{\isempty{#1}}{}{,#1}}}}

\newcommand{\loc}{{\mathrm{loc}}}
\newcommand{\CI}{\cC^\infty}
\newcommand{\CIdot}{\dot\cC^\infty}

\newcommand{\CIc}{\cC^\infty_\cp}

\newcommand{\Hloc}{H_{\loc}}

\newcommand{\Hb}{H_{\bop}}
\newcommand{\Hbloc}{H_{\bop,\loc}}
\newcommand{\Hbext}{\bar H_{\bop}}

\newcommand{\Hbsupp}{\dot H_{\bop}}

\newcommand{\Hext}{\bar H}

\newcommand{\Hsupp}{\dot H}

\newcommand{\Hbeta}{H_{\beta}}
\newcommand{\Hbetab}{H_{\beta,\bop}}
\newcommand{\Hscri}{H_{\scri}}
\newcommand{\Hscrib}{H_{\scri,\bop}}

\newcommand{\phg}{{\mathrm{phg}}}

\newcommand{\tdel}{\wt{\delta}{}}

\newcommand{\Ric}{\mathrm{Ric}}
\newcommand{\Ein}{\mathrm{Ein}}

\newcommand{\openbigpmatrix}[1]
  {%
    \def\@bigpmatrixsize{#1}%
    \addtolength{\arraycolsep}{-#1}%
    \begin{pmatrix}%
  }
\newcommand{\closebigpmatrix}
  {%
    \end{pmatrix}%
    \addtolength{\arraycolsep}{\@bigpmatrixsize}%
  }

\newcommand{\setarraystretch}{\def\arraystretch{1.25}}

\newcommand{\usref}[1]{{\upshape\ref{#1}}}

\DeclareGraphicsExtensions{.mps}

\begin{document}

\title[Stability of Minkowski space]{Stability of Minkowski space and polyhomogeneity of the metric}

\date{November 1, 2017. Final revision: August 18, 2019.}

\author{Peter Hintz}
\address{Department of Mathematics, University of California, Berkeley, CA 94720-3840, USA}
\curraddr{Department of Mathematics, Massachusetts Institute of Technology, Cambridge, Massachusetts 02139-4307, USA}
\email{phintz@mit.edu}

\author{Andr\'as Vasy}
\address{Department of Mathematics, Stanford University, Stanford, CA 94305-2125, USA}
\email{andras@math.stanford.edu}

\begin{abstract}
  We study the nonlinear stability of the $(3+1)$-dimensional Minkowski spacetime as a solution of the Einstein vacuum equation. Similarly to our previous work on the stability of cosmological black holes, we construct the solution of the nonlinear initial value problem using an iteration scheme in which we solve a linearized equation globally at each step; we use a generalized harmonic gauge and implement constraint damping to fix the geometry of null infinity. The linear analysis is largely based on energy and vector field methods originating in work by Klainerman. The weak null condition of Lindblad and Rodnianski arises naturally as a nilpotent coupling of certain metric components in a linear model operator at null infinity. Upon compactifying $\R^4$ to a manifold with corners, with boundary hypersurfaces corresponding to spacelike, null, and timelike infinity, we show, using the framework of Melrose's b-analysis, that polyhomogeneous initial data produce a polyhomogeneous spacetime metric. Finally, we relate the Bondi mass to a logarithmic term in the expansion of the metric at null infinity and prove the Bondi mass loss formula.
\end{abstract}

\maketitle

\setlength{\parskip}{0.00in}
\tableofcontents
\setlength{\parskip}{0.05in}

\section{Introduction}
\label{SI}

We prove the nonlinear stability of $(3+1)$-dimensional Minkowski space as a vacuum solution of Einstein's field equation and obtain a precise full expansion of the solution, in a mildly generalized harmonic gauge, in all asymptotic regions, i.e.\ near spacelike, null, and timelike infinity. On a conceptual level, we show how some of the methods we developed for our proofs of black hole stability in cosmological spacetimes \cite{HintzVasyKdSStability,HintzKNdSStability} apply in this more familiar setting, studied by Christodoulou--Klainerman \cite{ChristodoulouKlainermanStability}, Lindblad--Rodnianski \cite{LindbladRodnianskiGlobalExistence,LindbladRodnianskiGlobalStability}, and many others: this includes the use of an iteration scheme for the construction of the metric in which we solve a \emph{linear} equation \emph{globally} at each step, keeping track of the precise asymptotic behavior of the iterates by working on a suitable \emph{compactification} $M$ of the spacetime, and the implementation of \emph{constraint damping}.

The estimates we prove for the linear equations---which arise as linearizations of the gauge-fixed Einstein equation around metrics which lie in the precise function space in which we seek the solution---are largely based on energy estimates and a version of the vector field method \cite{KlainermanNullCondition}. The estimates are rather refined in terms of a splitting of the symmetric 2-tensor bundle (different metric components behave differently at null infinity); the vector fields we use are closely related to those in~\cite{KlainermanNullCondition,ChristodoulouKlainermanStability,LindbladRodnianskiGlobalExistence,LindbladRodnianskiGlobalStability}. In our systematic approach, both the relevant notion of regularity (matching \cite{LindbladAsymptotics}) and the determination of the precise asymptotic behavior of the solution follow readily from an inspection of the geometric and algebraic properties of the linearized gauge-fixed (or `reduced') Einstein equation; correspondingly, once $M$ and the required function spaces are defined (\S\S\ref{SCpt}--\ref{SEin}), the proof of stability itself is rather concise (\S\S\ref{SBg}--\ref{SPf}).

The weak null condition of Lindblad--Rodnianski \cite{LindbladRodnianskiWeakNull} manifests itself in our linearization approach as a nilpotent coupling of certain metric components for a linear model operator at null infinity: the logarithmic growth (relative to the typical decay rate of $r^{-1}$ of waves on $(3+1)$-dimensional Minkowski space near null infinity) of one metric component is rendered harmless due to its coupling (to leading order) only to a metric component $g_{0 0}$ which governs the `long range' behavior of outgoing light cones and which decays faster than $r^{-1}$ by a factor of $r^{-\gamma}$ for some $\gamma>0$ (see the discussions in~\S\S\ref{SssISysNull} and \ref{SsEinN}). For the reader already familiar with the weak null condition, we mention here that the better decay of $g_{0 0}$ in \cite{LindbladRodnianskiGlobalStability} (corresponding, roughly, to $g_{L L}$ in the reference) is a consequence of the harmonic gauge condition being satisfied by the nonlinear solution, while in the present paper we have decay of the $(0,0)$-component of every iterate in our iteration scheme since we arrange constraint damping, which, roughly speaking, ensures that our gauge condition is satisfied to high accuracy (in the sense of decay) even though we are only solving `nongeometric' (linear) equations. (This makes constraint damping attractive for numerical analysis, see \cite{GundlachCalabreseHinderMartinConstraintDamping,PretoriusBinaryBlackHole} and Remark~\ref{RmkICD} below.)

We proceed to state a simple version of our main theorem, before returning to an in-depth discussion of our approach, the relevant estimates, and the structure of the Einstein equation in~\S\ref{SsISys}. Recall that in Einstein's theory of general relativity, a vacuum spacetime is described by a $4$-manifold $M^\circ$ which is equipped with a Lorentzian metric $g$ with signature $(+,-,-,-)$ satisfying the \emph{Einstein vacuum equation}
\begin{equation}
\label{EqIEin}
  \Ric(g)=0.
\end{equation}
The simplest solution is the \emph{Minkowski spacetime} $(M^\circ,g)=(\R^4,\ul g)$,
\begin{equation}
\label{EqIMink}
  \ul g:=d t^2-d x^2,\quad \R^4=\R_t\times\R^3_x.
\end{equation}
The far field of an isolated gravitational system $(M^\circ,g)$ with total (ADM) mass $m$ is usually described by the \emph{Schwarzschild metric}
\begin{equation}
\label{EqISchw}
  g \approx g_m^S = \Bigl(1-\frac{2 m}{r}\Bigr)d t^2 - \Bigl(1-\frac{2 m}{r}\Bigr)^{-1}d r^2 - r^2\slg,\ \ r\gg 1,
\end{equation}
where $\slg$ denotes the round metric on $\Sph^2$; the Minkowski metric $\ul g=g_0^S$ differs from this by terms of size $\cO(m r^{-1})$. In the study of weak nonlinear gravity in vacuum (in particular, black holes are excluded), one then works with metrics $g$ which are smooth extensions of (a short range perturbation of) $g_m^S$ to all of $\R^4$. Such spacetimes are \emph{asymptotically flat}: letting $|t|+|x|\to\infty$ in $\R^4$, the metric $g$ (in a suitable gauge) approaches the flat Minkowski metric $\ul g$ in a quantitative fashion.

Suitably interpreted, the field equation~\eqref{EqIEin} has the character of a quasilinear wave equation; in particular, it predicts the existence of gravitational waves, which were recently observed experimentally \cite{LIGOBlackHoleMerger}. Correspondingly, the evolution and long time behavior of solutions of~\eqref{EqIEin} can be studied from the perspective of the \emph{initial value problem}: given a $3$-manifold $\Sigma^\circ$ and symmetric 2-tensors $\gamma,k\in\CI(\Sigma^\circ;S^2 T^*\Sigma^\circ)$, with $\gamma$ a Riemannian metric, one seeks a vacuum spacetime $(M^\circ,g)$ and an embedding $\Sigma^\circ\hra M^\circ$ such that
\begin{equation}
\label{EqIEinIVP}
  \Ric(g) = 0\ \ \tn{on}\ M^\circ,\quad
  g|_{\Sigma^\circ}=-\gamma,\ \II_g=k\ \ \tn{on}\ \Sigma^\circ,
\end{equation}
where $\II_g$ denotes the second fundamental form of $\Sigma^\circ$, and where we use the embedding $\Sigma^\circ\hra M^\circ$ to identify the tensors $\gamma,k$ on $\Sigma^\circ$ with (tangential) tensors on the image of $\Sigma^\circ$ in $M^\circ$. (The minus sign in~\eqref{EqIEinIVP} is due to our sign convention for Lorentzian metrics.) A fundamental result due to Choquet-Bruhat and Geroch \cite{ChoquetBruhatLocalEinstein,ChoquetBruhatGerochMGHD} states that necessary and sufficient conditions for the well-posedness of this problem are the \emph{constraint equations} for $\gamma$ and $k$,
\begin{equation}
\label{EqIConstraints}
  R_\gamma+(\tr_\gamma k)^2-|k|_\gamma^2 = 0,\ \ 
  \delta_\gamma k+d\tr_\gamma k=0,
\end{equation}
where $R_\gamma$ is the scalar curvature of $\gamma$, and $\delta_\gamma$ is the (negative) divergence. Concretely, if these are satisfied, there exists a maximal globally hyperbolic solution $(M^\circ,g)$ of~\eqref{EqIEinIVP} which is unique up to isometries. By the \emph{future development} of an initial data set $(\Sigma^\circ,\gamma,k)$, we mean the causal future of $\Sigma^\circ$ as a Lorentzian submanifold of $(M^\circ,g)$. Our main theorem concerns the long time behavior of solutions of~\eqref{EqIEinIVP} with initial data close to those of Minkowski space:

\begin{thm}
\label{ThmIBaby}
  Let $b_0>0$. Suppose that $(\gamma,k)$ are smooth initial data on $\R^3$ satisfying the constraint equations~\eqref{EqIConstraints} which are small in the sense that for some small $\delta>0$, a cutoff $\chi\in\CIc(\R^3)$ identically $1$ near $0$, and $\wt\gamma:=\gamma-(1-\chi)(-g_m^S)|_{\{t=0\}}$,\footnote{We use polar coordinates on $\R^3$ and define $-g^S_m|_{t=0}:=(1-\frac{2 m}{r})^{-1}d r^2+r^2\slg$.} where $|m|<\delta$, we have
  \begin{equation}
  \label{EqIBabySmall}
    \sum_{j\leq N+1} \| \la r\ra^{-1/2+b_0} (\la r\ra\nabla)^j \wt\gamma \|_{L^2} + \sum_{j\leq N} \| \la r\ra^{1/2+b_0} (\la r\ra\nabla)^j k \|_{L^2} < \delta,
  \end{equation}
  where $N$ is some large fixed integer ($N=26$ works). Assume moreover that the weighted $L^2$ norms in~\eqref{EqIBabySmall} are finite for all $j\in\N$.

  Then the future development of the data $(\R^3,\gamma,k)$ is future causally geodesically complete and decays to the flat (Minkowski) solution. More precisely, there exist a smooth manifold with corners $M$ with boundary hypersurfaces $\Sigma$, $I^0$, $\scri^+$, $I^+$, and a diffeomorphism of the interior $M^\circ$ with $\{t>0\}\subset\R^4$, as well as an embedding $\R^3\cong\Sigma^\circ$ of the Cauchy hypersurface, and a solution $g$ of the initial value problem~\eqref{EqIEinIVP} which is \ul{conormal} (see below) on $M$ and satisfies $|g-\ul g|\lesssim (1+t+|r|)^{-1+\eps}$ for all $\eps>0$. See Figure~\usref{FigIBaby}. For fixed ADM mass $m$, the solution $g$ depends continuously on $\wt\gamma$, $k$, see Remark~\usref{RmkPfCts}.

  If the normalized initial data $(\la r\ra\wt\gamma,\la r\ra^2 k)$ are in addition $\cE$-smooth, i.e.\ \ul{polyhomogeneous} at infinity with index set $\cE$ (see below), then the solution $g$ is also polyhomogeneous on $M$, with index sets given explicitly in terms of $\cE$.
\end{thm}

More precise versions will be given in Theorem~\ref{ThmIDetail} and in \S\ref{SPf}. The condition~\eqref{EqIBabySmall} allows for $\wt\gamma$ to be pointwise of size $r^{-1-b_0-\eps}$, $\eps>0$; since $b_0>0$ is arbitrary, this means that we allow for the initial data to be Schwarzschildean modulo $\cO(r^{-1-\eps})$ for any $\eps>0$.

In Theorem~\ref{ThmIBaby}, \emph{conormality} is a (local) regularity notion on a manifold with corners $\sfM$ which is equivalent to smoothness in $\sfM^\circ$, but differs from it near $\pa\sfM$: in the model case $\sfM=[0,\infty)_x^p\times\R_y^q$, and with $\alpha\in\R^p$, a function $u\in x^\alpha L^\infty_\loc(\sfM)$ is called \emph{conormal} relative to the space $x^\alpha L^\infty_\loc(\sfM)$ if
\[
  V_1\cdots V_N u \in x^\alpha L^\infty_\loc(\sfM)\quad \forall\ N\in\N,
\]
where each $V_j$ is one of the vector fields $x_k\pa_{x_k}$, $\pa_{y_l}$, $1\leq k\leq p$, $1\leq l\leq q$. (A typical example of a conormal function is $x^\beta$, where $\beta\in\R^p$, $\beta\geq\alpha$ component-wise.) We say that a distribution $u$ is \emph{conormal} if it is conormal relative to $x^\alpha L^\infty_\loc(\sfM)$ for some vector $\alpha\in\R^p$ of weights. In the context of Theorem~\ref{ThmIBaby}, the weights are specified in Theorem~\ref{ThmIDetail} and Remark~\ref{RmkIDetailDecay} below; at this point we simply content ourselves with taking them to be $0$ at each hypersurface.

Before continuing the discussion of Theorem~\ref{ThmIBaby}, we remark that the assumption that \emph{all} weighted norms in~\eqref{EqIBabySmall} are finite is \emph{only} needed to conclude the conormality of $g$. If one is only interested in controlling a \emph{finite} number of derivatives of $g$, we only need to require the finiteness of \emph{finitely} many weighted norms~\eqref{EqIBabySmall} (as can be seen by inspecting the Nash--Moser theorem we use in our nonlinear iteration).

Next, $\cE$-smoothness is a refinement of conormality: the assumption of $\cE$-smoothness, i.e.\ polyhomogeneity with index set $\cE\subset\C\times\N_0$, means, roughly speaking, that $\la r\ra\wt\gamma$ (similarly $\la r\ra^2 k$) has a full asymptotic expansion as $r\to\infty$ of the form
\begin{equation}
\label{EqIBabyPhg}
  \la r\ra\wt\gamma \sim \sum_{(z,k)\in\cE} r^{-i z}(\log r)^k\wt\gamma_{(z,k)}(\omega),\ \ \omega=x/|x|\in\Sph^2,\ \wt\gamma_{(z,k)}\in\CI(\Sph^2;S^2 T^*\R^3),
\end{equation}
with $\Im z<-b_0$, where for any fixed $C$, the number of $(z,k)\in\cE$ with $\Im z>-C$ is finite. (That is, $\la r\ra\wt\gamma$ admits a generalized Taylor expansion into powers of $r^{-1}$, except the powers may be fractional or even complex---that is, oscillatory---and logarithmic terms may occur. A typical example is that all $z$ are of the form $z=-i k$, $k\in\N$, in which case~\eqref{EqIBabyPhg} is an expansion into powers $r^{-k}$, with potential logarithmic factors.) The polyhomogeneity of $g$ on the manifold with corners $M$ means that at each of the hypersurfaces $I^0$, $\scri^+$, and $I^+$, the metric $g$ admits an expansion similar to~\eqref{EqIBabyPhg}, with $r^{-1}$ replaced by a defining function of the respective boundary hypersurface (for example $\scri^+$) such that moreover each term in the expansion (which is thus a tensor on $\scri^+$) is itself polyhomogeneous at the other boundaries (that is, at $\scri^+\cap I^0$ and $\scri^+\cap I^+$). We refer the reader to~\S\ref{SsCptF} for precise definitions, and to Examples~\ref{ExPhgSchwartz} and \ref{ExPhgSmooth} for the list of index sets for two natural classes of polyhomogeneous initial data.

Christodoulou \cite{ChristodoulouNoPeeling} showed that, generically, one can only expect the metric $g$, suitably rescaled to a non-degenerate metric on a compactification of $\R^4$, to be of class $\cC^{1,\alpha}$, $\alpha<1$, due to the presence of logarithmic terms in the expansion of certain geometric quantities at null infinity; polyhomogeneity of the metric (rather than smoothness of a conformal multiple down to $\scri^+$) is thus the best one can hope for, and this is what we establish here. (We also prove that the metric is indeed conformal to a non-degenerate metric of class $\cC^{1,\alpha}$, $\alpha<\min(b_0,1)$, down to $\scri^+$; see Remark~\ref{RmkLBSReg}).

If the initial data do not have a full polyhomogeneous expansion, but only a partial expansion (containing only finitely many terms) plus a sufficiently regular remainder decaying faster than the terms in the expansion, the solution $g$ will itself have a finite partial expansion at each boundary hypersurface, plus a faster decaying remainder; we shall not, however, record results of this nature here.

\begin{figure}[!ht]
\includegraphics{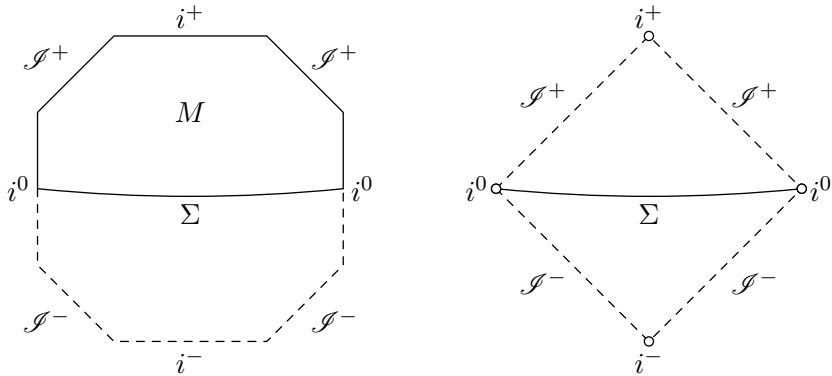}
\caption{\textit{Left:} the compact manifold $M$ (solid boundary), containing a compactification $\Sigma$ of the initial surface $\Sigma^\circ$. The boundary hypersurfaces $I^0$, $\scri^+$, and $I^+$ are called spatial infinity, (future) null infinity, and (future) timelike infinity, respectively. One can think of $M$ as the blow-up of a Penrose diagram at timelike and spatial infinity. A global compactification would extend across $\Sigma$ to the past, with additional boundary hypersurfaces $\scri^-$ (past null infinity) and $I^-$ (past timelike infinity). \textit{Right:} for comparison, the Penrose diagram of Minkowski space.}
\label{FigIBaby}
\end{figure}

Applying a suitable version of this theorem both towards the future and the past, we show that the maximal globally hyperbolic development is given by a causally geodesically complete metric $g$, with analogous regularity and polyhomogeneity statements as in Theorem~\ref{ThmIBaby}, on a suitable manifold with corners whose interior is diffeomorphic to $\R^4$ (and contains $\Sigma^\circ$), which now has the additional boundary hypersurfaces $\scri^-$ and $I^-$; see Theorem~\ref{ThmPfGlobal} and the end of~\S\ref{SPhg}.

Like many other approaches to the stability problem (see the references below), our arguments apply to the Einstein--massless scalar field system $\Ric(g)=|\nabla\phi|_g^2$, $\Box_g\phi=0$, with small initial data for the scalar field in order to obtain global stability. They also give the stability of the far end of a Schwarzschild black hole spacetime with any mass $m\in\R$, i.e.\ of the domain of dependence of the complement of a sufficiently large ball in the initial surface, without smallness assumptions on the data: in this case, we control the solution up to some finite point along the radiation face $\scri^+$. See Remark~\ref{RmkPfLargeMass}.

The compactification $M$ only depends on the ADM mass $m$ of the initial data set;\footnote{By the positive mass theorem \cite{SchoenYauPMT,WittenPMT}, we have $m\geq 0$, but we will not use this information. In fact, our analysis of the Bondi mass, summarized in Theorem~\ref{ThmIBondi} below, \emph{implies} the positive mass theorem for the restricted class of data considered in Theorem~\ref{ThmIBaby}.} for the class of initial data considered here, the mass gives the only long range contribution to the metric that significantly (namely, logarithmically) affects the bending of light rays: for the Schwarzschild metric~\eqref{EqISchw}, radially outgoing null-geodesics lie on the level sets of $t-r-2 m\log(r-2 m)$. Concretely, near $I^0\cup\scri^+$, $M$ will be the Penrose compactification of the region $\{t/r<2,\,r\gg 1\}\subset\R^4$ within the Schwarzschild spacetime, i.e.\ equipped with the metric $g_m^S$, blown up at spacelike and future timelike infinity. As in our previous work \cite{HintzVasyKdSStability,HintzKNdSStability} on Einstein's equation, we prove Theorem~\ref{ThmIBaby} using a Newton-type iteration scheme (more precisely: Nash--Moser) in which we solve a \emph{linear} equation \emph{globally} on $M$ at each step. While this approach brings many advantages (cf.\ Remark~\ref{RmkIItVsBoot}), a disadvantage of using a Nash--Moser iteration is the typically rather large number of derivatives needed compared to other approaches.

We do not quite use the wave coordinate gauge as in Lindblad--Rodnianski \cite{LindbladRodnianskiGlobalStability,LindbladRodnianskiGlobalExistence}, but rather a wave map gauge with background metric given by the Schwarzschild metric with mass $m$ near $I^0\cup\scri^+$, glued smoothly into the Minkowski metric elsewhere; this is a more natural choice than using the Minkowski metric itself as a background metric (which would give the standard wave coordinate gauge), as the solution $g$ will be a short range perturbation of $g_m^S$ there. This gauge, which can be expressed as the vanishing of a certain 1-form $\Ups(g)$, fixes the long range part of $g$ and hence the main part of the null geometry at $\scri^+$. In order to ensure the gauge condition to a sufficient degree of accuracy (i.e.\ decay) at $\scri^+$ throughout our iteration scheme, we implement \emph{constraint damping}, first introduced in the numerics literature in~\cite{GundlachCalabreseHinderMartinConstraintDamping}, and crucially used in~\cite{HintzVasyKdSStability}. This means that we use the 1-form encoding the gauge condition in a careful manner when passing from the Einstein equation~\eqref{EqIEin} to its `reduced' quasilinear hyperbolic form: we can arrange that for each iterate $g_k$ in our iteration scheme, the gauge 1-form $\Ups(g_k)$ vanishes sufficiently fast at $\scri^+$ so as to fix the long range part of $g$. In order to close the iteration scheme and control the nonlinear interactions, we need to keep precise track of the leading order behavior of the remaining metric coefficients at $\scri^+$. We discuss this in detail in \S\ref{SsIEin}.

\begin{rmk}
\label{RmkICD}
  Fixing the geometry at $\scri^+$ in this manner, the first step of our iteration scheme, i.e.\ solving the linearized gauge-fixed Einstein equation with the given (nonlinear) initial data of size $\delta$, produces a solution \emph{with the correct long range behavior} and which is $\delta^2$ close to the nonlinear solution in the precise function spaces on $M$ in which we measure the solution. (Subsequent iteration steps give much more accurate approximations since the convergence of the iteration scheme is \emph{exponential}.) This suggests that our formulation of the gauge-fixed Einstein equation could allow for improvements of the accuracy of post-Minkowskian expansions---which are iterates of a Picard-type iteration scheme as in~\cite[Equation~(1.7)]{LindbladRodnianskiGlobalStability}---used to study gravitational radiation from isolated sources \cite{BlanchetRadiation}.
\end{rmk}

The global stability of Minkowski space was established, building in particular on \cite{KlainermanNullCondition,ChristodoulouGlobalSolutionsSmallData}, in the monumental work of Christodoulou--Klainerman \cite{ChristodoulouKlainermanStability} for asymptotically Schwarzschildean data (similar to those in~\eqref{EqIBabySmall} but with $b_0\geq\half$, though requiring only $N=3$ derivatives) and precise control at null infinity, with an alternative proof using double null foliations by Klainerman--Nicol\`o \cite{KlainermanNicoloEvolution}; and more recently in~\cite{LindbladRodnianskiGlobalExistence,LindbladRodnianskiGlobalStability} using the wave coordinate gauge, for initial data as in Theorem~\ref{ThmIBaby} (but requiring only $N=10$ derivatives on the initial data). Friedrich \cite{FriedrichStability} (see~\cite{FriedrichEinsteinMaxwellYangMills} for the Einstein--Yang--Mills case) established non-linear stability, using a conformal method, for a restrictive class (shown to be nonempty in \cite{CorvinoScalar}) of initial data, but with precise information on the asymptotic structure of the spacetime. Bieri \cite{BieriZipserStability} studied the problem for a very general class of data which are merely decaying like $\la r\ra^{-1/2-\delta}$ for some $\delta>0$---thus more slowly even than the $\cO(r^{-1})$ terms of Schwarzschild---and even less regularity than Christodoulou--Klainerman; in this case, the `correct' compactification on which the metric has a simple description will have to depend on more than just the ADM mass (this is clear e.g.\ for the initial data constructed by Carlotto--Schoen \cite{CarlottoSchoenData}, which are nontrivial only in conic wedges); Bieri and Chru\'sciel \cite{BieriChruscielADMBondi,ChruscielScriPiece} construct a piece of $\scri^+$ for the data considered in \cite{BieriZipserStability} but without a smallness assumption. Further works on the stability of Minkowski space for the Einstein equations coupled to other fields, in the wake of \cite{ChristodoulouKlainermanStability,LindbladRodnianskiGlobalExistence,LindbladRodnianskiGlobalStability}, include those by Speck \cite{SpeckEinsteinMaxwell} on (a generalization of) the Einstein--Maxwell system, Taylor \cite{TaylorEinsteinVlasov}, Lindblad--Taylor \cite{LindbladTaylorVlasov}, and Fajman--Joudioux--Smulevici \cite{FajmanJoudiouxSmuleviciEinsteinVlasov} for both the massless and the massive Einstein--Vlasov system. We also mention Keir's very general quasilinear results \cite{KeirWeak} which in particular imply the global solvability for small data of the gauge-fixed Einstein equation in harmonic coordinates (but without constraint damping) even when the gauge condition is violated, albeit at the expense of losing the precise asymptotic control at null infinity. The global stability for a minimally coupled massive scalar field was proved by LeFloch--Ma \cite{LeFlochMaEinsteinMassive} and Wang \cite{WangEinsteinKleinGordon}.

The present paper contains the first proof of full conormality and polyhomogeneity of small nonlinear perturbations of Minkowski space in $3+1$ dimensions. Lindblad--Rodnianski also established high conormal regularity, see \cite[Equation~(1.14)]{LindbladRodnianskiGlobalStability}, though, in the context of the present paper, on the compactification corresponding to Minkowski rather than on $M$, and hence with a loss in the decay rates. This was improved by Lindblad~\cite{LindbladAsymptotics} who proved sharp decay for the metric at null infinity (albeit in a slightly different gauge), and uses them to establish a relationship between the ADM mass and the total amount of gravitational radiation. The decay in \cite{LindbladAsymptotics} corresponds to the leading order decay which we prove at $\scri^+$; we improve this by proving definite decay rates towards the leading order terms at $\scri^+$, and we strengthen the decay rate towards $I^+$ to $t^{-1}$; in fact, we show decay at a faster rate to an $\cO(t^{-1})$ leading order term, see the proof of Theorem~\ref{ThmLBSLoss}. (Neither improvement requires polyhomogeneous initial data.)

Previously, polyhomogeneity was established in spacetime dimensions $\geq 9$ for the Einstein vacuum and Einstein--Maxwell equations, for initial data stationary outside of a compact set, by Chru\'sciel--Wafo \cite{ChruscielWafoPhg}; this relied on earlier work by Chru\'sciel-\L{}eski \cite{ChruscielLeskiPolyhomogeneous} on the polyhomogeneity of solutions of hyperboloidal initial value problems\footnote{This means that the initial data are posed on a spacelike but asymptotically null hypersurface transversal to $\scri^+$.} for a class of semilinear equations, and Loizelet's proof \cite{LoizeletPhD,LoizeletEinsteinMaxwell} of the electrovacuum extension (using wave coordinate and Lorenz gauges) of \cite{LindbladRodnianskiGlobalExistence}; see also \cite{BeigChruscielElectroVac}. Lengard \cite{LengardThesis} studied hyperboloidal initial value problems and established the propagation of weighted Sobolev regularity for the Einstein equation, and of polyhomogeneity for nonlinear model equations. In spacetime dimensions $5$ and above, Wang~\cite{WangThesis,WangRadiation} obtained the leading term (i.e.\ the `radiation field') of $g-\ul g$ at $\scri^+$, and proved high conormal regularity. Baskin--Wang \cite{BaskinWangRad} and Baskin--S\'a Barreto \cite{BaskinSaBarretoRad} defined radiation fields for linear waves on Schwarzschild as well as for semilinear wave equations on Minkowski space. For initial data which are exactly Schwarzschildean outside a compact set and in \emph{even} spacetime dimensions $\geq 6$, a simple conformal argument, which requires very little information on the structure of the Einstein(--Maxwell) equation, stability and smoothness of $\scri^+$ were proved by Choquet-Bruhat--Chru\'sciel--Loizelet \cite{ChoquetBruhatChruscielLoizeletEinsteinMaxwell}; see also \cite{AndersonChruscielSimple} for a different approach in the vacuum case. The construction of the required initial data sets as well as questions of their smoothness and polyhomogeneity were taken up in the hyperboloidal context by Andersson--Chru\'sciel--Friedrich \cite{AnderssonChruscielFriedrich} and extended in \cite{AnderssonChruscielHypPhg,AnderssonChruscielHyp}, see also \cite{ChruscielLengardPhg}. Paetz and Chru\'sciel \cite{ChruscielPaetzCharScriI,PaetzCharScriII} studied this for \emph{characteristic} data; we refer to Corvino \cite{CorvinoScalar}, Chru\'sciel--Delay \cite{ChruscielDelayMapping}, and references therein for the case of asymptotically flat data sets.

The backbone of our proof is a \emph{systematic treatment} of the stability of Minkowski space as a problem of proving regularity and asymptotics for a quasilinear (hyperbolic) equation on a \emph{compact}, but geometrically \emph{complete} manifold with corners $M$. That is, we employ analysis based on complete vector fields on $M$ and the corresponding natural function spaces, which in this paper are \emph{b-vector fields}, i.e.\ vector fields tangent to $\pa M$, and spaces with conormal regularity or (partial) polyhomogeneous expansions; following Melrose \cite{MelroseAPS,MelroseDiffOnMwc}, this is called \emph{b-analysis} (`b' for `boundary'). The point is that once the smooth structure (the manifold $M$) and the algebra of differential operators appropriate for the problem at hand give a simple background on which to do analysis;\footnote{This is akin to how making use of the notion of a smooth manifold allows one to study PDE in an invariant, coordinate-free manner. Indeed, viewing a global PDE, a priori on a noncompact space, as a (typically degenerate) PDE on a compactification $M$ (typically a compact manifold with boundary or corners), one frees oneself from any particular local coordinate expression, and, for instance, gains the flexibility of being able to work with the local coordinate system (or, more narrowly, a set of boundary defining functions) appropriate for calculations in the region/asymptotic regime of interest. Moreover, if one defines function spaces by using only the smooth structure on $M$ (and possibly using some extra data, such as fibrations of boundary hypersurfaces), it becomes simple to verify whether estimates, done in convenient local coordinates, do give estimates of the invariantly defined function spaces.} we will give examples and details in \S\ref{SsISys}. In this context, it is often advantageous to work on a more complicated manifold $M$ if this simplifies the algebraic structure of the equation at hand. While this point of view has a long history in the study of elliptic equations, see e.g.~\cite{MazzeoMelroseHyp,MelroseAPS,SchulzePsdoSing,MazzeoEdge,GilSchulzeSeilerEdge}, its explicit use in hyperbolic problems is, to a large part, rather recent \cite{MelroseEuclideanSpectralTheory,VasyThreeBody,MelroseWunschConic,MelroseVasyWunschDiffraction,BaskinVasyWunschRadMink,BaskinVasyWunschRadMink2,HintzVasySemilinear,HintzQuasilinearDS,HintzVasyQuasilinearKdS,HintzVasyKdSStability}. We also point out that fixing the smooth structure on $M$, one gains the 

A (clean) description of polyhomogeneous expansions, in particular at the transitions between different regimes such as near $I^0\cap\scri^+$ or $\scri^+\cap I^+$, \emph{requires} working on a manifold with corners. More generally, it is often easier to define function spaces on $M^\circ$ by working uniformly up to $\pa M$, and decay rates from the perspective of $M^\circ$ can be encoded as orders of vanishing at $\pa M$ (the latter making sense since $M$ is equipped with a smooth structure).\footnote{As an example, reminiscent of the behavior of linear waves on Minkowski space near null infinity, consider the space $\cX$ of smooth functions on $[1,\infty)_r$ which for any $N\in\N$ can be written as an $N$-th degree polynomial in $1/r$, without constant term, plus a $\cO(r^{-N})$ remainder. Passing to the compactification $I$, which is diffeomorphic to a closed interval, with boundary defining functions $(r-1)/r$ for the left endpoint and $x:=1/r$ for the right endpoint (thus the point $x=0$ is a rigorous definition of `$r=\infty$'), we simply have $\cX=x\CI(I)$: smooth functions on $I$ vanishing simply at the right endpoint $x=0$.}

Working in a compactified setting furthermore makes the structures allowing for global existence clearly visible in the form of linear model operators defined \emph{at the boundary hypersurfaces}. Among the key structures for Theorem~\ref{ThmIBaby} are the symmetries of the model operator $L^0$ \emph{at} $\scri^+$, which is essentially the product of two transport ODEs, as well as constraint damping and a certain \emph{null structure}, both of which are simply a certain Jordan block structure of $L^0$, with the null structure corresponding to a nilpotent Jordan block. At $I^+$, the model operator will be closely related (via a conformal transformation) to the conformal Klein--Gordon equation on static de~Sitter space, which enables us to determine the asymptotic behavior of $g$ there via \emph{resonance expansions} from known results on the asymptotics of conformal waves on de~Sitter space.

A closely related reason for viewing a global problem (i.e.\ to be solved, at first glance, on a noncompact set) as a (degenerate) problem on a compact manifold with boundary or corners is that asymptotic data of the solution become \emph{restrictions} of the solution to boundary hypersurfaces: it was for the purpose of giving a simple and conceptually clean description of the radiation field of scalar, electromagnetic, or gravitational waves, and also of solutions of the full nonlinear Einstein equation, that Penrose introduced his compactifications and diagrams. (These restrictions may solve interesting equations by themselves, as is the case for the Bondi mass loss formula at $\scri^+$, and in the case of the scattering argument which we will use at $I^+$ to prove the vanishing of the final Bondi mass at the future boundary of $\scri^+$.) While a compactified perspective is often not strictly necessary for the description of asymptotic data and relations between them, it is usually conceptually advantageous, and brings to light the key features of a PDE problem which may be difficult to detect from the noncompact point of view, cf.\ the references above. (For example, finding the linearized version of the weak null structure of Lindblad--Rodnianski does not require any careful inspection, but simply the calculation of a partial Jordan block decomposition of a coefficient of a model operator defined at null infinity.)

We also note that the symmetries and dynamical/geometric features of (asymptotically) Minkowski metrics relevant in each of these regimes are different. Hence, we find it advantageous to adapt our descriptions of coordinates, operators, and function spaces to the various asymptotic regimes and symmetries of the problem, rather than e.g.\ working throughout with standard $(t,x)$-coordinates on $\R^4$: the latter seem to be most useful for capturing the (approximate) translation-invariance of wave equations on (asymptotically) Minkowski spacetimes---which does not play a role in the stability proof---while scaling, boosts and rotations, while of course expressible in $(t,x)$ coordinates, become very simple on $M$, simply becoming smooth vector fields on $M$ with some extra properties, such as tangency to $\pa M$.

While the manifold $M$ is compact, our analysis of the linear equations (arising from a linearization of the gauge-fixed Einstein equation) on $M$ lying at the heart of this paper is \emph{not} a short-time existence/regularity analysis near the interiors of $I^0$, resp.\ $I^+$, but rather a global in space, resp.\ global in time analysis. (Conformal methods such as~\cite{FriedrichSpaceConformal} bringing  $I^0$ to a \emph{finite} place have the drawback of imposing very restrictive regularity conditions on the initial data.) At $\scri^+$, we use a version of Friedlander's rescaling \cite{FriedlanderRadiation} of the wave equation, which does give equations with singular (conormal or polyhomogeneous) coefficients; but since $\scri^+$ is a null hypersurface, conormality or polyhomogeneity---which are notions of regularity defined with respect to (b-)vector fields, which are \emph{complete}---are essentially transported along the generators of $\scri^+$. At the past and future boundaries of $\scri^+$, i.e.\ at $I^0\cap\scri^+$ and $\scri^+\cap I^+$, the two pictures fit together in a simple and natural fashion. We discuss this in detail in \S\S\ref{SssISysLin} and \ref{SssISysPhg}.

We reiterate that our goal is to exhibit the conceptual simplicity of our approach, which we hope will allow for advances in the study of related stability problems which have a more complicated geometry on the base, i.e.\ on the level of the spacetime metric, on the fibers, i.e.\ for equations on vector bundles, or both. In particular, we are not interested in optimizing the number of derivatives needed for our arguments based on Nash--Moser iteration.

Following our general strategy, one can also prove the stability of Minkowski space in spacetime dimensions $n+1$, $n\geq 4$, for sufficiently decaying initial data, with the solution conormal (or polyhomogeneous, if the initial data are such), thus strengthening Wang's results \cite{WangRadiation}. There are a number of simplifications due to the faster decay of linear waves in $\R^{1+n}$: the compactification $M$ of $\R^{1+n}$ does not depend on the mass anymore and can be taken to be the blow-up of the Penrose diagram of Minkowski space at spacelike and future timelike infinity; we do not need to implement constraint damping as metric perturbations no longer have a long range term which would change the geometry of $\scri^+$; and we do not need to keep track of the precise behavior (such as the existence of leading terms at $\scri^+$) of the metric perturbation. We shall not discuss this further here.

\subsection{Aspects of the systematic treatment; examples}
\label{SsISys}

Consider a nonlinear partial differential equation $P(u)=0$, with $P$ encoding boundary or initial data as well, whose global behavior one wishes to understand for high regularity data which have small norm; denote by $L_u:=D_u P$ the linearized operators. In the present problem, $P$ will be the map assigning a metric to the value of the (gauge-fixed) Einstein operator on it, as well as its pair of initial data. Our strategy, with references to their implementation for the present problem, is:
\begin{enumerate}
\myitem{ItISysSmooth}{1.}{1} fix a $\CI$ structure, that is, a \emph{compact} manifold $M$, with boundary or corners, on which one expects the solution $u$ to have a simple description (regularity, asymptotic behavior)---see~\S\ref{SsCptA} for the definition of the compactification of $\R^4$ on which we will work;
\myitem{ItISysAlg}{2.}{2} choose an algebra of differential operators and a scale of function spaces on $M$, say $\cX^s,\cY^s$, encoding the amount $s\in\R$ of regularity as well as relevant asymptotic behavior, such that for $u\in\cX^\infty:=\bigcap_{s>0}\cX^s$ small in some $\cX^s$ norm, the operator $L_u$ lies in this algebra and maps $\cX^\infty\to\cY^\infty:=\bigcap_{s>0}\cY^s$---see~\S\S\ref{SsCptF} and \ref{SsEinF} for the function spaces we will use: conormal sections of certain vector bundles together with certain leading order terms at null infinity; and \S\ref{SsEinP} for the verification of the mapping property;
\myitem{ItISysMap}{3.}{3} show that for such small $u$, the operator $L_u$ has a (right) inverse
  \begin{equation}
  \label{EqISysMap}
    (L_u)^{-1}\colon\cY^\infty\to\cX^\infty
  \end{equation}
  on these function spaces---see~\S\S\ref{SBg}, \ref{SIt}, discussed below;
\myitem{ItISysSolv}{4.}{4} solve the nonlinear equation using a global iteration scheme, schematically
  \begin{equation}
  \label{EqISysSolv}
    u_0=0;\quad u_{k+1}=u_k+v_k,\ v_k=-(L_{u_k})^{-1}(P(u_k));\quad u=\lim_{k\to\infty} u_k\in\cX^\infty.
  \end{equation}
  See~\S\ref{SPf}.
\myitem{ItISysMore}{5.}{5} (Optional.) Improve on the regularity of the solution $u\in\cX^\infty$, provided the data has further structure such as polyhomogeneity or better decay properties, by using the PDE $P(u)=0$ directly, or its approximation by linearized model problems in the spirit of $0=P(u)\approx L_0 u+P(0)$ and a more precise analysis of $L_0$. See~\S\ref{SPhg}, where we prove the polyhomogeneity for asymptotically Minkowski metrics.
\end{enumerate}

We stress that steps~\ref{ItISysSmooth} and \ref{ItISysAlg} are nontrivial, as they require significant insights into the geometric and analytic properties of the PDE in question, and are thus intimately coupled to step~\ref{ItISysMap}; the function spaces in step~\ref{ItISysAlg} must be large enough in order to contain the solution $u$, but precise (i.e.\ small) enough so that the nonlinearities and linear solution operators are well-behaved on them.

Note that if one has arranged~\ref{ItISysMap}, then the iteration scheme~\eqref{EqISysSolv} formally closes, i.e.\ all iterates $u_k$ lie in $\cX^\infty$ modulo checking their required smallness in $\cX^s$. Checking the latter, thus making~\eqref{EqISysSolv} rigorous, is however easy in many cases, for example by using Nash--Moser iteration \cite{HamiltonNashMoser,SaintRaymondNashMoser}, which requires $(L_u)^{-1}$ to satisfy so-called tame estimates; these in turn are usually automatic from the proof of~\eqref{EqISysMap}, which is often ultimately built out of simple algebraic operations like multiplications and taking reciprocals of operator coefficients or symbols, and energy estimates, for all of which tame estimates follow from the classical Moser estimates. The precise bookkeeping, done e.g.\ in \cite{HintzVasyQuasilinearKdS}, can be somewhat tedious but is only of minor conceptual importance: it only affects the number of derivatives of the data which need to be controlled, i.e.\ the number $N$ in~\eqref{EqIBabySmall}; in this paper, we shall thus be generous in this regard.

As a further guiding principle, which applies in the context of our proof of Theorem~\ref{ThmIBaby}, one can often separate step~\ref{ItISysMap}, i.e.\ the analysis of the equation $L_u v=f$, into two pieces:
\begin{enumerate}
\myitem{ItISysMapReg}{3.1.}{3.1} prove infinite regularity of $v$ but without precise asymptotics---see~\S\ref{SBg}, where we accomplish this using simple energy estimates;
\myitem{ItISysMapDec}{3.2.}{3.2} improve on the asymptotic behavior of $v$ to show $v\in\cX^\infty$---see~\S\ref{SIt}, where we use integration along approximate characteristics as well as spectral theory/normal operator arguments for this purpose.
\end{enumerate}

The point is that a `background estimate' from step~\ref{ItISysMapReg} may render many terms of $L_u$ lower order, thus considerably simplifying the analysis of asymptotics and decay; see e.g.\ the discussion around~\eqref{EqISysLinTransport}.

\begin{rmk}
\label{RmkIItVsBoot}
  Let us compare this strategy to proofs using bootstrap arguments, which are commonly used for global existence problems for nonlinear evolution equations as e.g.\ in \cite{ChristodoulouKlainermanStability,LindbladRodnianskiGlobalStability,LukKerrNonlinear}. The choice of bootstrap assumptions is akin to choosing the function space $\cX^\infty$ (and thus implicitly $\cY^\infty$) in step~\ref{ItISysAlg}, while the consistency of the bootstrap assumptions, without obtaining a gain in the constants in the bootstrap, is similar to proving~\eqref{EqISysMap}. However, note that the bootstrap operates on a solution of the nonlinear equation, whereas we only consider linear equations; the gain in the bootstrap constants thus finds its analogue in the fact that one can make the iteration scheme~\eqref{EqISysSolv} rigorous, e.g.\ using Nash--Moser iteration, and keep low regularity norms of $u_k$ bounded (and $v_k$ decaying with $k$) throughout the iteration scheme. In the context in particular of Einstein's equation, a bootstrap argument has the advantage that the gauge condition is automatically satisfied as one is dealing with solutions of the nonlinear equation; thus the issue of constraint damping does not arise, whereas we do have to arrange this. In return, we gain significant flexibility in the choice of analytic tools for the global study of the linearized equations (e.g.\ methods from microlocal analysis, scattering theory), as used extensively in~\cite{HintzVasyKdSStability}; bootstrap arguments on the other hand are strongly tied to the character of $P(u)$ as a (nonlinear) \emph{hyperbolic} (or parabolic) and \emph{differential} operator, or at least to its locality in `time', and it is much less clear how to exploit global information (e.g.\ resonances).
\end{rmk}

Before discussing Einstein's equation in \S\ref{SsIEin}, we first describe this strategy for scalar nonlinear wave equations on Minkowski space. The most significant part of the work required to implement this strategy is the analysis of the linear operators called $L_u$ above; we thus begin in~\S\ref{SssISysLin} by explaining how we obtain estimates for solutions of \emph{linear wave equations} on Minkowski space in a manner that will work for linearizations of the gauge-fixed Einstein equation in~\S\ref{SBg}. In~\S\ref{SssISysNull}, we then put a few examples of nonlinear scalar equations into the abstract general framework described above, including a discussion of polyhomogeneity (step~\ref{ItISysMore} above) in~\S\ref{SssISysPhg}.

\subsubsection{Linear waves in Minkowski space}
\label{SssISysLin}

For step~\ref{ItISysSmooth}, we seek a convenient compactification $M$ of $\R^4$. The goal, from the PDE perspective, is for the asymptotic behavior of linear waves on $\R^4$ to have a simple description on $M$; closely related to this is that the asymptotic behavior of natural geometric objects such as (null)geodesics should be simple. Consider first `null infinity': a (rescaled) linear wave on $\R^4$ has a limit as $r\to\infty$ along any null-geodesic, e.g.\ the one defined by $t-r=s_0$, $\omega=\omega_0\in\Sph^2$ (using polar coordinates on $\R^3$) for $(s_0,\omega_0)\in\R\times\Sph^2$. Thus, we want to define $M$ in such a way that a sequence of points, with $r\to\infty$, along such a ray has a unique limit in $M$; that is, one boundary hypersurface of $M$ should be equal to (the closure\footnote{We also want to capture the asymptotics of the radiation field itself, leading us to consider the limits $s_0\to\pm\infty$ of such limiting points.} of) all such limiting points, with a bijection between $(s_0,\omega_0)$ and points in (the interior of) this boundary hypersurface, and such a boundary hypersurface then deserves the name $\scri^+$. (The interior of $\scri^+$ is thus $(\scri^+)^\circ\cong\R\times\Sph^2$.) The radiation field is then the restriction of the rescaled wave, extended from $\R^4$ to $M$ by continuity, to $\scri^+\subset\pa M$ (or $(\scri^+)^\circ$ in standard terminology).

For other asymptotic regimes, there are a number of choices one can make \emph{on Minkowski space}: the Penrose diagram, or the conformal embedding of Minkowski space into the Einstein universe give two (closely related) compactifications of $\R^4$ in which future timelike and spacelike geodesic rays have limit points. In order to facilitate the generalization to compactifications of \emph{asymptotically} Minkowskian spacetimes in~\S\ref{SCpt}, we choose to work with a compactification in which the closure of the set of these limiting points, called future timelike infinity $I^+$ and spacelike infinity $I^0$, are 3-dimensional (rather than 2-dimensional, as in the Penrose compactification); coordinates in their interiors are $x/t$ with $|x/t|<1$, $t^{-1}=0$ in $(I^+)^\circ$, and $(t/r,\omega)$ with $|t/r|<1$, $r^{-1}=0$ in $(I^0)^\circ$.

At future timelike infinity $I^+$, the asymptotic behavior of waves is governed, quite generally on suitable asymptotically Minkowski spacetimes, by quantum resonances \cite{BaskinVasyWunschRadMink};\footnote{See \cite[Theorem~1.1]{BaskinVasyWunschRadMink} for the rough theorem. Here, quantum resonances $\sigma_j\in\C$ are poles of the meromorphic continuation of the resolvent of an asymptotically hyperbolic Laplacian (plus a potential) arising naturally by Mellin-transforming the wave operator, or rather $L$ as in~\eqref{EqISysLinOp}, in $(t^2-r^2)^{1/2}$; linear waves then have expansions into $t^{i\sigma_j}a_j(x/t)$ for suitable distributions $a_j$, smooth in $|x/t|<1$. For present purposes, one can deduce the asymptotic behavior of linear waves equivalently by relating the linear scalar wave equation to the conformal wave equation on static de~Sitter space and the asymptotics of its solutions; see~\S\ref{SsItip}. Even so, we shall use spectral theoretic methods to accomplish the latter.} also, nonlinear interactions are much simpler to deal with than near $\scri^+$. (This is a further reason to keep $(\scri^+)^\circ$ and $(I^+)^\circ$ separate: it keeps the delicate analysis at $\scri^+$ separate on $M$ from the more straightforward analysis at $I^+$. The analysis at $I^0$ is even simpler.) We also point out that it is a specific feature of \emph{exact} Minkowski space that one can `blow down' $I^+$; that is, suitably rescaled linear waves are smooth directly on the Penrose compactification, and the blow-up of timelike infinity $i^+$ and spacelike infinity $i^0$ in the Penrose diagram, as in Figure~\ref{FigIBaby}, is not required; on more general asymptotically Minkowski spacetimes on the other hand, one needs to resolve $i^+$ and $i^0$ via real blow-up, obtaining $I^+$ and $I^0$, in order to exhibit linear waves as polyhomogeneous (read: having a simple asymptotic description) functions on the compactification.

Thus, we begin by defining $\ol{\R^4}$:
\begin{definition}
\label{DefISysLinCpt}
  The radial compactification of $\R^4$ is defined as
  \begin{equation}
  \label{EqISysLinCpt}
    \ol{\R^4}:=\R^4\sqcup([0,1)_R\times\Sph^3)/\sim,
  \end{equation}
  where $\sim$ identifies $(R,\omega)$, $R>0$, $\omega\in\Sph^3$, with the point $R^{-1}\omega\in\R^4$. The quotient carries the smooth structure in which the smooth functions are precisely those which over $\R^4$ (the interior of $\ol{\R^4}$) are smooth in the usual sense, and which over $[0,1)_R\times\Sph^3_\omega$ are smooth in $(R,\omega)$ down to $R=0$.
\end{definition}

The function $\rho:=(1+t^2+r^2)^{-1/2}\in\CI(\ol{\R^4})$ is a boundary defining function, i.e.\ $\pa\ol{\R^4}=\rho^{-1}(0)$ with $d\rho$ nondegenerate everywhere on $\pa\ol{\R^4}$. Letting $v=(t-r)/r$ away from $r=0$, all future null-geodesics tend to $S^+=\{\rho=0,\,v=0\}$, and we then define $M$ as the closure of $t\geq 0$ within the blow-up\footnote{The prototypical example of a blow-up is that of the origin in $\R^n$: we have $[\R^n;\{0\}]\cong[0,\infty)_r\times\Sph^{n-1}$, i.e.\ the origin in $\R^n$ is resolved, and $r=0$ is no longer merely a point, but a full $(n-1)$-sphere. The front face of this blow-up is $\{r=0\}\cong\Sph^{n-1}$, and the blow-down map is the map $(r,\omega)\mapsto r\omega$: it is a diffeomorphism in $r>0$, but at $r=0$ collapses an $(n-1)$-sphere to a single point (the origin). In the setting of interest for us, the blow-up $[M;X]$ of an embedded boundary submanifold $X\subset\pa M$ is, in a similar manner, the union $(M\setminus X)\sqcup S N^+X$ of the complement of $X$ and the inward pointing spherical (i.e.\ the quotient by the $\R_+$ action in the fibers of the) normal bundle of $X$ in $M$. See the local coordinate descriptions below, as well as~\cite[Chapter~5]{MelroseDiffOnMwc} for a detailed discussion of blow-ups.} $[\ol{\R^4};S^+]$ of $\ol{\R^4}$ at $S^+$ (see Figure~\ref{FigIBaby}), i.e.\ the smooth manifold obtained by declaring polar coordinates around $S^+$ to be smooth down to the origin. We refer to the front face $\scri^+$ of this blow-up as \emph{null infinity} or the \emph{radiation face}; it has a natural fibration by the fibers of the map $\scri^+\to S^+$, which we call the \emph{fibers of the radiation face/null infinity/$\scri^+$}. (The interior of a typical fiber is equal to $\R_{s_0}\times\{\omega_0\}$ for some fixed $\omega_0\in\Sph^2$.)

We can equivalently describe $M$ by giving a list of local coordinate patches and how (pieces of) $\R^4$ are glued to them. We describe two exemplary coordinate charts here: the first one is
\[
  [0,1)_{\rho_0} \times [0,1)_{\rho_I} \times \Sph^2_\omega,
\]
and we identify $(\rho_0,\rho_I,\omega)$ for $\rho_0,\rho_I>0$ with the point $(t,x)\in\R\times\R^3$ for $t=\rho_0^{-1}(\rho_I^{-1}-1)$, $x=\rho_0^{-1}\rho_I^{-1}\omega$. Thus,
\begin{equation}
\label{EqISysLin0IChart}
  \rho_0 = (r-t)^{-1},\quad
  \rho_I = (r-t)/r;
\end{equation}
then $I^0$, resp.\ $\scri^+$ is locally given by $\rho_0=0$, resp.\ $\rho_I=0$; thus, this chart describes a neighborhood of $I^0\cap\scri^+$, i.e.\ the transition from spacelike to null infinity. (For example, $\{\rho_0=0,\ \rho_I=c\}$ for some fixed $c\in(0,1)$ consists of the points `at (spacelike) infinity' of a spacelike cone in Minkowski space, while $\{\rho_0=c,\ \rho_I=0\}$ consists of the points `at (null) infinity' of a null cone.) See Figure~\ref{Fig0IChart}.

\begin{figure}[!ht]
\includegraphics{0IChart}
\caption{Illustration of the coordinate chart~\eqref{EqISysLin0IChart}. Shown are a number of level sets of $\rho_0$ (red dashed lines) and $\rho_I$ (blue dashed lines) projected onto the $(t,r)$ plane. Indicated on the top right is the $(\rho_0,\rho_I,\omega)$ coordinate system including the boundary hypersurfaces $I^0$ and $\scri^+$ which are glued onto $\R^4$.}
\label{Fig0IChart}
\end{figure}

The second coordinate chart is
\[
  [0,1)_{\tilde\rho_I} \times [0,1)_{\rho_+} \times \Sph^2_\omega,
\]
and $(\tilde\rho_I,\rho_+,\omega)$ for $\tilde\rho_I,\rho_+>0$ is identified with $(t,x)$ for $t=\rho_+^{-1}(\tilde\rho_I^{-1}+1)$, $x=\tilde\rho_I^{-1}\rho_+^{-1}\omega$; thus
\begin{equation}
\label{EqISysLinIPChart}
  \tilde\rho_I = (t-r)/r,\quad \rho_+=(t-r)^{-1}.
\end{equation}
(Now $\{\tilde\rho_I=c,\ \rho_+=0\}$ for $c\in(0,1)$ consists of the points `at (future timelike) infinity' of a timelike cone in Minkowski space.) When the coordinate system in which we work is clear, we simply write $\rho_I$ instead of $\tilde\rho_I$.

To motivate a preliminary choice of function spaces for step~\ref{ItISysAlg}, recall that the behavior of solutions of $\Box_{\ul g}u:=-u_{;\mu}{}^\mu$ near $\scri^+$ can be studied using the Friedlander rescaling
\begin{equation}
\label{EqISysLinOp}
  L:=\rho^{-3}\Box_{\ul g}\rho.
\end{equation}
This operator has smooth coefficients down to the interior $(\scri^+)^\circ$ of null infinity: it is equal to the conformal wave operator $\Box_{\rho^2\ul g}-\tfrac{1}{6}R_{\rho^2\ul g}$, and $\rho^2\ul g$ is a smooth, nondegenerate Lorentzian metric down to $(\scri^+)^\circ$: in local coordinates $\rho=r^{-1}\geq 0$, $x^1=t-r\in\R$, $\omega\in\Sph^2$ near $(\scri^+)^\circ$, we have $\rho^2\ul g=-2\,d x^1\,d\rho-\slg+\rho^2(d x^1)^2$. Thus, solutions of $L u=0$, with $\CIc(\R^3)$ initial data, are smooth up to $\scri^+$ and typically nonvanishing there. We shall refer to this reasoning as \emph{Friedlander's argument} below. (A more sophisticated version of this observation lies at the heart of Friedrich's conformal approach \cite{FriedrichConformalFE} to the study of Einstein's equation.) However, for more general initial data, and, more importantly, in many nonlinear settings (see \S\S\ref{SssISysNull} and \ref{SsIEin} below), smoothness will not be the robust notion, and we must settle for less: \emph{conormality at $\pa M$}. Namely, let $\Vb(M)$ denote the Lie algebra of \emph{b-vector fields}, i.e.\ vector fields tangent to the boundary hypersurfaces of $M$ other than the closure $\Sigma$ of the initial surface $\Sigma^\circ=\{t=0\}$, a function $u$ on $M$ is conormal iff it remains in a fixed weighted $L^2$ space on $M$ upon application of any finite number of b-vector fields. For $M$ defined above, $\Vb(M)$ is spanned over $\CI(M)$ by translations $\pa_t$ and $\pa_{x^i}$ as well as the scaling vector field $t\pa_t+x\pa_x$, boosts $t\pa_{x^i}+x^i\pa_t$, and rotation vector fields $x^i\pa_{x^j}-x^j\pa_{x^i}$.\footnote{In the coordinate chart~\eqref{EqISysLin0IChart}, $\Vb(M)$ is spanned by $\rho_0\pa_{\rho_0}=-t\pa_t-r\pa_r$, $\rho_I\pa_{\rho_I}=-r(\pa_t+\pa_r)$, and rotation vector fields. In the chart~\eqref{EqISysLinIPChart}, $\Vb(M)$ is spanned by $\rho_I\pa_{\rho_I}=-r(\pa_t+\pa_r)$, $\rho_+\pa_{\rho_+}=-t\pa_t-r\pa_r$, and rotation vector fields. It is then straightforward to check, in either of these two coordinate systems, that translations, scaling, and boosts are linear combinations, with $\CI(M)$ coefficients, of these vector fields.} (Note however that the definition of $\Vb(M)$ depends \emph{only on the smooth structure} of $M$.\footnote{The smoothness of elements of $\Vb(M)$ \emph{on the compactification $M$} in particular constrains their growth as one leaves every compact set of $\R^4$. As `counterexamples', one can check that the vector field $t^3\pa_t$, expressed in local coordinates near $\pa M$, is singular near any point of $\pa M$ (though of course it is smooth on $\R^4$!); similarly, the vector field $t\pa_t$ is singular at $\scri^+$ in the sense that it does \emph{not} extend, by continuity from $\R^4$, to be tangent to $\scri^+$ as is required from b-vector fields on $M$; it is, on the other hand, a smooth b-vector field down to $(I^0)^\circ$ and $(I^+)^\circ$.})

Let us now explain how to obtain a background estimate, step~\ref{ItISysMapReg} above, for the forcing problem $L u=f$ with trivial initial data. First, we can estimate $u$ in $H^1$ on any compact subset of $\R^4\cap\{t\geq 0\}$ by $f$ on another compact set. Then, on a neighborhood of $(I^0)^\circ$ which is diffeomorphic to $[0,1)_{\rho_0}\times(0,1)_\tau\times\Sph^2$, where
\[
  \rho_0 := r^{-1},\ \ \tau:= t/r,
\]
with $\rho_0$ a local boundary defining function of $I^0$, this problem roughly takes the form
\begin{equation}
\label{EqISysLinEq0}
  \bigl(D_\tau^2 - (\rho_0 D_{\rho_0})^2-\slDelta\bigr)u = f,
\end{equation}
where we use the standard notation
\begin{equation}
\label{EqISysLinD}
  D = \frac{1}{i}\pa,\quad i=\sqrt{-1}.
\end{equation}
In~\eqref{EqISysLinEq0}, $\slDelta=\Delta_{\slg}\geq 0$ is the Laplacian on $\Sph^2$, and $f$ has suitable decay properties making its norms in the estimates below finite. This is a wave equation on the (asymptotically) cylindrical manifold $[0,1)_{\rho_0}\times\Sph^2$. Let
\[
  U_0=\{0\leq\tau\leq c,\ \rho_0\leq 1\},\quad c\in(0,1).
\]
For any weight $a_0\in\R$, we can run an energy estimate using the vector field multiplier $\rho_0^{-2 a_0}\pa_\tau$ and obtain
\begin{equation}
\label{EqISysLinEq0Est}
  \| u \|_{\rho_0^{a_0}\Hb^1(U_0)} \lesssim \| f \|_{\rho_0^{a_0}L^2_\bop(U_0)}
\end{equation}
for $f$ supported in $U_0$; see Figure~\ref{FigISysLin0}. Here $L^2_\bop$ is the $L^2$ space with respect to the b-density $d\tau\frac{d\rho_0}{\rho_0}|d\slg|$, the weighted $L^2_\bop$ norm is defined by $\|f\|_{\rho_0^{a_0}L^2_\bop}=\|\rho_0^{-a_0}f\|_{L^2_\bop}$, and $\Hb^1$ is the space of all $u\in L^2_\bop$ such that $V u\in L^2_\bop$ for all $V\in\Vb(M)$; in $U_0$, $\Vb(M)$ is spanned (over $\CI(M)$ by $\pa_\tau$, $\rho_0\pa_{\rho_0}$, $\slnabla$, so we let
\[
  \|u\|_{\rho_0^{a_0}\Hb^1(U_0)} := \|u\|_{\rho_0^{a_0}L^2_\bop(U_0)} + \|\pa_\tau u\|_{\rho_0^{a_0}L^2_\bop(U_0)} + \|\rho_0 D_{\rho_0}u\|_{\rho_0^{a_0}L^2_\bop(U_0)} + \|\slnabla u\|_{\rho_0^{a_0}L^2_\bop(U_0)}.
\]
In order to obtain a higher regularity estimate, one can commute any number of b-vector fields through~\eqref{EqISysLinEq0}; the estimate~\eqref{EqISysLinEq0Est} only relies on the \emph{principal (wave) part} of $L$; lower order terms arising as commutators are harmless. Thus, $f\in\rho_0^{a_0}\Hb^\infty$ (weighted $L^2_\bop$-regularity with respect to any finite number of b-vector fields) implies $u\in\rho_0^{a_0}\Hb^\infty$, with estimates.

The same conclusion holds for the initial value problem for $L u=0$ with initial data which near $I^0$ are $(u|_{\tau=0},\pa_\tau u|_{\tau=0})=(u|_{t=0},r \pa_t u|_{t=0})=(u_0,u_1)$, $u_j\in\rho_0^{a_0}\Hb^\infty(\ol{\R^3})$, where $\ol{\R^3}$ is the radial compactification of $\R^3$, defined analogously to~\eqref{EqISysLinCpt}, which has boundary defining function $\rho_0=r^{-1}$. The assumption~\eqref{EqIBabySmall} on the size of initial data is a smallness condition on $\|\la r\ra\wt\gamma\|_{\rho_0^{b_0}\Hb^{N+1}}+\|\la r\ra^2 k\|_{\rho_0^{b_0}\Hb^N}$.

\begin{figure}[!ht]
\includegraphics{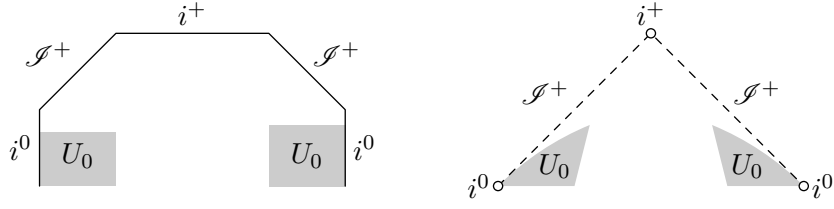}
\caption{The domain $U_0$ on which the energy estimate~\eqref{EqISysLinEq0Est} holds. \textit{Left:} as a subset of $M$. \textit{Right:} as a subset of the Penrose compactification.}
\label{FigISysLin0}
\end{figure}

Re-defining $\rho=r^{-1}$ near $S^+$, a neighborhood of $I^0\cap\scri^+$ is diffeomorphic to $[0,1)_{\rho_0}\times[0,1)_{\rho_I}\times\Sph^2$, where (as in~\eqref{EqISysLin0IChart})
\begin{equation}
\label{EqISysLinCoords0I}
  \rho_0 := -\rho/v = (r-t)^{-1},\ \ \rho_I := -v = (r-t)/r
\end{equation}
are boundary defining functions of $I^0$ and $\scri^+$, respectively. (Thus, a function bounded by $\rho_0^{a_0}\rho_I^{a_I}$ decays like $r^{-a_0}$ near $(I^0)^\circ$ and like $r^{-a_I}$ near $(\scri^+)^\circ$.) The lift of $L$ to $M$ is singular as an element of $\Diffb^2(M)$ but with a very precise structure at $\scri^+$: the equation $L u=f$ is now of the form
\begin{equation}
\label{EqISysLinEqI}
  \bigl(2\pa_{\rho_I}(\rho_0\pa_{\rho_0}-\rho_I\pa_{\rho_I}) - \slDelta\bigr)u = f
\end{equation}
modulo terms with more decay; here, ignoring weights, $\rho_I\pa_{\rho_I}\sim\pa_t+\pa_r$ and $\rho_0\pa_{\rho_0}-\rho_I\pa_{\rho_I}\sim\pa_t-\pa_r$ are the radial null vector fields. Assuming $f$ vanishes far away from $\scri^+$, we can run an energy estimate using $V=\rho_0^{-2 a_0}\rho_I^{-2 a_I}V_0$ as a multiplier, where $V_0=-c\rho_I\pa_{\rho_I}+\rho_0\pa_{\rho_0}$ is future timelike in $M\setminus\scri^+$ if we choose $c<1$; note that $V_0$ is \emph{tangent to $I^0$ and $\scri^+$} (and null at $\scri^+$); it is necessary to arrange this tangency for compatibility with our conormal function spaces, but it comes at the expense of giving control at $\scri^+$ that is weaker (but more robust, i.e.\ holds for a larger class of spacetimes) than the smoothness following from Friedlander's argument. A simple calculation, cf.\ Lemma~\ref{LemmaBgscriKCurrent}, shows that for $a_I<a_0$ and $a_I<0$,
\begin{equation}
\label{EqISysLinEqIEst}
  \|u\|_{\rho_0^{a_0}\rho_I^{a_I}L^2_\bop} + \|(\rho_0\pa_{\rho_0},\rho_I\pa_{\rho_I},\rho_I^{1/2}\slnabla)u\|_{\rho_0^{a_0}\rho_I^{a_I}L^2_\bop} \lesssim \|\rho_I f\|_{\rho_0^{a_0}\rho_I^{a_I}L^2_\bop}\ \ \tn{in}\ U_I,
\end{equation}
see Figure~\ref{FigISysLinI}, where $L^2_\bop$ is the $L^2$ space with integration measure $\frac{d\rho_0}{\rho_0}\frac{d\rho_I}{\rho_I}|d\slg|$. The assumptions on the weights are natural: since $\pa_t-\pa_r$ transports mass from $I^0$ to $\scri^+$, we certainly need $a_I\leq a_0$, while $a_I<0$ is necessary since, in view of the behavior of linear waves discussed after~\eqref{EqISysLinOp}, the estimate must apply to $u$ which are smooth and nonzero down to $\scri^+$. In~\eqref{EqISysLinEqIEst}, derivatives of $u$ along b-vector fields tangent to the fibers of the radiation face are controlled without a loss in weights, while general derivatives such as spherical ones lose a factor of $\rho_I^{1/2}$.\footnote{This is to be expected: indeed, letting $x:=\rho_I^{1/2}$, the rescaled metric $x^{-2}(\rho^2 \ul g)$ is an \emph{edge metric} \cite{MazzeoEdge}, i.e.\ a quadratic form in $\frac{d\rho_0}{\rho_0}$, $\frac{d x}{x}$, $\frac{d\theta^a}{x}$, with $\theta^a$ coordinates on $\Sph^2$, for which the natural vector fields are precisely those tangent to the fibers of $\scri^+$, that is, $\rho_0\pa_{\rho_0}$, $x\pa_x=2\rho_I\pa_{\rho_I}$, and $x\pa_{\theta^a}=\rho_I^{1/2}\pa_{\theta^a}$.} When controlling error terms later on, we thus need to separate them into terms involving fiber-tangent b-derivatives and general b-derivatives, and check that the coefficients of the latter have extra decay in $\rho_I$; see \S\ref{SsCptScri}.

\begin{figure}[!ht]
\includegraphics{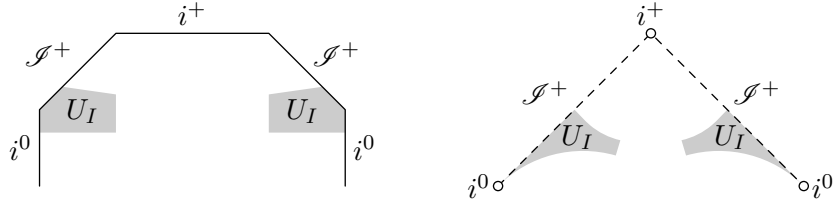}
\caption{The domain $U_I$ on which the energy estimate~\eqref{EqISysLinEqIEst} holds.}
\label{FigISysLinI}
\end{figure}

From~\eqref{EqISysLinEqI}, $L\in\rho_I^{-1}\Diffb^2(M)$ is equal to the model operator
\[
  L^0 := 2\pa_{\rho_I}(\rho_0\pa_{\rho_0}-\rho_I\pa_{\rho_I})
\]
modulo $\Diffb^2(M)$ (i.e.\ ignoring second order differential operators, such as $\slDelta$, which are sums of at most twofold products of b-vector fields). The commutation properties \emph{of this model} are what allows for higher regularity estimates:\footnote{See the discussion after~\eqref{EqISysNullLin} for an even stronger statement.} ($\rho_I$ times) equation~\eqref{EqISysLinEqI} commutes with $\rho_0\pa_{\rho_0}$ (scaling), $\rho_I\pa_{\rho_I}$ (roughly a combination of scaling and boosts), and spherical vector fields which are independent of $\rho_0$ and $\rho_I$.\footnote{We briefly sketch the argument: denoting the collection of these vector fields---which span $\Vb(M)$ locally---by $\{V_j\}$, this gives $L(V_j u)=V_j f+[L,V_j]u$ with $[L,V_j]\in\Diffb^2$ (modulo multiples of $L$ which arise for $V=\rho_I\pa_{\rho_I}$, and which we drop here), which is one order better in the sense of decay than the a priori expected membership in $\rho_I^{-1}\Diffb^2$ due to these commutation properties. Write $[L,V_j]=C_{j k}V_k$ with $C_{j k}\in\Diffb^1$ and apply the estimate~\eqref{EqISysLinEqIEst} to $V_j u$; then the additional forcing term $[L,V_j]u$ obeys the bound $\sum_k\|\rho_I C_{j k}V_k u\|_{\rho_0^{a_0}\rho_I^{a_I}L^2_\bop}\lesssim\sum_k\|V_k u\|_{\rho_0^{a_0}\rho_I^{a_I-1}\Hb^1}$, which close to $\scri^+$ is bounded by a small constant times the left hand side of~\eqref{EqISysLinEqIEst}, with $V_j u$ in place of $u$ and summed over $j$, due to the gain (of at least $\half$) in the weight in $\rho_I$.} In the end, we obtain $u\in\rho_0^{a_0}\rho_I^{a_I}\Hb^\infty$ when $f\in\rho_0^{a_0}\rho_I^{a_I-1}\Hb^\infty$.

Lastly, near $I^+$, one can use energy estimate with weight $\rho_I^{-2 a_I}\rho_+^{-2 a_+}$, $a_+<a_I$ large and negative, multiplying a timelike extension of the above $V_0$; higher regularity follows by commuting with the scaling vector field $\rho_+\pa_{\rho_+}$, where $\rho_+$ is a defining function of $I^+$, and elliptic regularity for $C(\rho_+ D_{\rho_+})^2-L$, $C>0$ large, in $I^+$ away from $\scri^+$, which is a consequence of the timelike nature of the scaling vector field $\rho_+\pa_{\rho_+}$ in $(I^+)^\circ$. See Figure~\ref{FigISysLinP}. Note that it is only at this stage that one uses the asymptotically Minkowskian nature of the metric in a neighborhood \emph{of all of} $I^+$; when dealing with a more complicated geometry, as e.g.\ in the study of perturbations of a Schwarzschild black hole, establishing this part of the background estimate (as well as the more precise asymptotics at $I^+$ discussed momentarily) becomes a major difficulty.

\begin{figure}[!ht]
\includegraphics{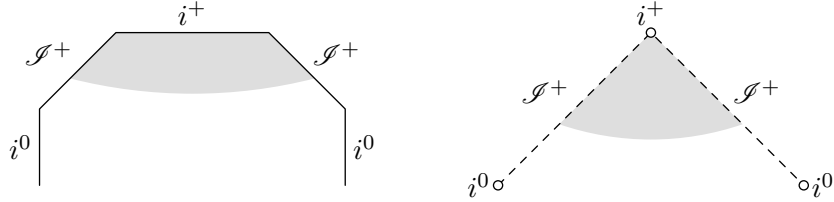}
\caption{The neighborhood (shaded) of $I^+$ on which we use a global (in $I^+$) weighted energy estimate.}
\label{FigISysLinP}
\end{figure}

Putting everything together, we find that
\begin{equation}
\label{EqISysLinBg}
  f\in\rho_0^{a_0}\rho_I^{a_I-1}\rho_+^{a_+}\Hb^\infty(M),\ f\equiv 0\ \tn{near}\ \Sigma\ \Lra\ u\in\rho_0^{a_0}\rho_I^{a_I}\rho_+^{a_+}\Hb^\infty(M),
\end{equation}
for $a_I<\min(a_0,0)$ and $a_+<a_I$.\footnote{Proving this estimate for large, negative, but nonexplicit $a_+$ is easy, while obtaining an explicit value of $a_+$ does require explicit straightforward (albeit lengthy) calculations. We accomplish this in \S\ref{SsBgExpl} by identifying $L$ with the conformal wave operator on static de~Sitter space for a suitable choice of $\rho$.}

For nonlinear applications, the information~\eqref{EqISysLinBg} on $u$ is not sufficient: the decay rate at $\scri^+$ is limited, and we do not have a good decay rate at $I^+$ either, cf.\ the discussion of $\rho_0^{a_0}\rho_I^{a_I}$ following~\eqref{EqISysLinCoords0I}. Let us thus turn to step~\ref{ItISysMapDec} and analyze $L u=f$ for $f$, vanishing near $\Sigma$, having more decay,
\begin{equation}
\label{EqISysLinY}
  f\in \cY^\infty := \rho_0^{b_0}\rho_I^{-1+b_I}\rho_+^{b_+}\Hb^\infty(M);\quad
  b_+<b_I<b_0,\ \ b_I\in(0,1).
\end{equation}
The background estimate~\eqref{EqISysLinBg} gives $u\in\rho_0^{b_0}\rho_I^{-\eps}\rho_+^{a_+}\Hb^\infty$ for all $\eps>0$. Near $I^0\cap\scri^+$ then, the conormality of $u$ allows for equation~\eqref{EqISysLinEqI} to be written as
\begin{equation}
\label{EqISysLinTransport}
  \rho_I\pa_{\rho_I}(\rho_0\pa_{\rho_0}-\rho_I\pa_{\rho_I})u = \half(\rho_I f + \rho_I\slDelta u) \in \rho_0^{b_0}\rho_I^{b_I}\Hb^\infty\ \ \tn{on}\ U_I,
\end{equation}
i.e.\ $L$ effectively becomes the composition of (linear) transport equations along the two radial null directions. See Figure~\ref{FigISysLinTrans}. Integration of $\rho_0\pa_{\rho_0}-\rho_I\pa_{\rho_I}$ is straightforward, while integrating $\rho_I\pa_{\rho_I}$, which is a regular singular ODE with indicial root $0$, implies that $u$ has a leading order term at $\scri^+$; one finds that
\[
  u = u^{(0)} + u_\bop;\ \ u^{(0)}\in\rho_0^{b_0}\Hb^\infty(\scri^+),\ u_\bop\in\rho_0^{b_0}\rho_I^{b_I}\Hb^\infty(M)\ \ \tn{near}\ I^0\cap\scri^+,
\]
which implies the existence of the radiation field.\footnote{For rapidly decaying $f$, one can plug this improved information into the right hand side of~\eqref{EqISysLinTransport}, thereby obtaining an expansion of $u$ into integer powers of $\rho_I$ and recovering the smoothness of $u$ at $\scri^+$.} The procedure to integrate along (approximate) characteristics to get sharp decay is frequently employed in the study of nonlinear waves on (asymptotically) Minkowski spaces, see e.g.\ \cite[\S2.2]{LindbladRodnianskiGlobalStability}, \cite{LindbladAsymptotics}, and their precursors~\cite{LindbladLifespan,LindbladGlobalNonlinear}.

\begin{figure}[!ht]
\includegraphics{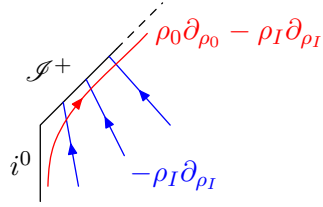}
\caption{The integral curves of the vector fields $\pa_t+\pa_r\sim-\rho_I\pa_{\rho_I}$ and $\pa_t-\pa_r\sim\rho_0\pa_{\rho_0}-\rho_I\pa_{\rho_I}$. Integration along the former gives the leading term at $\scri^+$, while integration along the latter transports weights (and polyhomogeneity) from $I^0$ to $\scri^+$.}
\label{FigISysLinTrans}
\end{figure}

At $\scri^+\cap I^+$, the same argument works, showing that $u^{(0)}$ and $u_\bop$ are bounded by $C\rho_+^{a_+}$ and $C\rho_I^{b_I}\rho_+^{a_+}$ near $I^+$ (i.e.\ by $t^{-a_+}$ as $t\to\infty$ with $r/t$ in compact subsets of $[0,1)$). Improving this weight however does \emph{not} follow from such a simple argument. Indeed, at $I^+$, the behavior of $u$ is governed by \emph{scattering theoretic phenomena}: the asymptotics are determined by scattering resonances of a model operator at $I^+$, namely the \emph{normal operator} of the b-differential operator $L$ at $I^+$, obtained by freezing its coefficients at $I^+$, see equation~\eqref{EqCptNormOp}. We thus use the arguments introduced in~\cite{VasyMicroKerrdS}, see also \cite[Theorem~2.21]{HintzVasySemilinear}, based on Mellin transform in $\rho_+$, inversion of a `spectral family' $\wh{L}(\sigma)$, which is the conjugation of the model operator (called `normal operator' in b-parlance) of $L$ at $I^+$ by the Mellin transform in $I^+$, with $\sigma$ the dual parameter to $\rho_+$, and contour shifting in the inverse Mellin transform to find the correct asymptotic behavior at $I^+$: the resonances $\sigma\in\C$, which are the poles of $\wh{L}(\sigma)^{-1}$, give rise to a term $\rho_+^{i\sigma}v$, $v$ a function on $I^+$, in the asymptotic expansion of $u$. (See \S\S\ref{SsItip} and \ref{SPhg} for details.) The resonances can easily be calculated explicitly in the present context, and they all satisfy $-\!\Im\sigma\geq 1>b_+$. The upshot is that
\begin{equation}
\label{EqISysLinX}
  f \in \cY^\infty\ \Rightarrow\ u \in \cX^\infty := \bigl\{ \chi u^{(0)} + u_\bop \colon u^{(0)}\in\rho_0^{b_0}\rho_+^{b_+}\Hb^\infty(\scri^+),\ u_\bop\in\rho_0^{b_0}\rho_I^{b_I}\rho_+^{b_+}\Hb^\infty(M) \bigr\},
\end{equation}
where $\chi$ cuts off to a neighborhood of $\scri^+$.

For later use as a simple model for constraint damping, consider a more general equation,
\begin{equation}
\label{EqISysLinDamping}
  L_\gamma u \equiv \rho^{-3}(\Box_{\ul g} - 2\gamma t^{-1}\pa_t)(\rho u) = f,
\end{equation}
for $\gamma\in\R$; near $\scri^+$ and $I^0$, this now roughly takes the form
\[
  \bigl(2\rho_I^{-1}(\rho_I\pa_{\rho_I}-\gamma)(\rho_0\pa_{\rho_0}-\rho_I\pa_{\rho_I})-\slDelta\bigr)u = f.
\]
Once the conormality of $u$ is known, integrating the first vector field on the left gives a leading term $\rho_I^\gamma$, which is decaying for $\gamma>0$. (One can show that the background estimate~\eqref{EqISysLinBg} holds for $a_I<\gamma$, but even an ineffective bound $a_I\ll 0$ would be good enough, as the transport ODE argument automatically recovers the optimal bound.)

\begin{rmk}
\label{RmkISysLinResonances}
  Note that for small $\gamma$, the normal operator of $L_\gamma$ at $I^+$ is close to the normal operator for $\gamma=0$, hence one would like to conclude that mild decay $\rho_+^{b_+}$, $b_+<1$, at $I^+$ still holds in this case. This is indeed true, but the argument has a technical twist: $L_\gamma$ does not have smooth coefficients at $\scri^+$ as a differential operator (unlike $L$ in Friedlander's argument) due to the presence of derivatives which are not tangential to $S^+$. However, we still have $L_\gamma\in\Diffb^2(\ol{\R^4})$; we thus deduce asymptotics at $I^+$ via normal operator analysis \emph{on the blown-down space} $\ol{\R^4}$, analogously to \cite{BaskinVasyWunschRadMink,BaskinVasyWunschRadMink2}. See~\S\ref{SsItip}.
\end{rmk}

\begin{rmk}
\label{RmkISysLinRadial}
  The improved decay at $\scri^+$ translates into higher b-regularity of $u$ on the blown-down space $\ol{\R^4}$, as we will show in Lemma~\ref{LemmaItipFn}; in the language of~\cite[Proposition~4.4]{BaskinVasyWunschRadMink}, this corresponds to a shift of the threshold regularity at the radial set by $\gamma$ coming from the skew-symmetric part of $L_\gamma$.
\end{rmk}

\subsubsection{Non-linearities and null structure}
\label{SssISysNull}

Equipped with this understanding of linear waves, we now discuss steps~\ref{ItISysAlg}--\ref{ItISysSolv} of the abstract strategy of~\S\ref{SsISys}. In particular, we will show how the absence of a `null structure' for a semilinear wave equation well-known to exhibit finite-time blow-up manifests itself from the global, Newton iteration scheme perspective; we will also discuss examples of equations that do satisfy a null condition, of the type arising when studying the linearization of the gauge-fixed Einstein equation.

To begin, recall that if $u$ is conormal on $M$, then its derivatives along $\pa_0:=\pa_t+\pa_r$ or size $1$ spherical derivatives $r^{-1}\slnabla$ have faster decay by one order at $\scri^+$, whereas its `bad' derivative along $\pa_1:=\pa_t-\pa_r$ does not gain decay there; indeed, modulo vector fields with more decay at $\scri^+$, we calculate near $I^0\cap\scri^+$ using~\eqref{EqISysLinCoords0I}
\[
  \pa_0 = -\half\rho_0\rho_I\,\rho_I\pa_{\rho_I},\ \ 
  \pa_1 = \rho_0(\rho_0\pa_{\rho_0}-\rho_I\pa_{\rho_I});
\]
note the extra factor of $\rho_I$ in $\pa_0$. All these derivatives gain an order of decay at $I^0$, hence the structure of nonlinearities is relevant mainly at $\scri^+$; let us thus restrict the discussion to a neighborhood of $I^0\cap\scri^+$. (Similar considerations apply to a neighborhood of $I^+\cap\scri^+$.) Consider the nonlinear equation $\Box_{\ul g}u-(\pa_1 u)^2=f$, or rather the closely related equation
\begin{equation}
\label{EqISysNullViolated}
  P(u) = L u - \rho^{-1}(\pa_1 u)^2 - f,\ \ f\in\cY^\infty\ \ \tn{small},
\end{equation}
with $L$ given by~\eqref{EqISysLinOp}; this is well-known to violate the \emph{null condition} introduced by Christodoulou~\cite{ChristodoulouGlobalSolutionsSmallData} and Klainerman~\cite{KlainermanNullCondition}. From our compactified perspective, the issue is the following. For $u\in\cX^\infty$, the linearization $L_u=L-2\rho^{-1}(\pa_1 u)\pa_1$ is, to leading order as a b-operator,
\[
  2\rho_I^{-1}(\rho_I\pa_{\rho_I}-\pa_1 u)(\rho_0\pa_{\rho_0}-\rho_I\pa_{\rho_I}),
\]
so the indicial root at $\scri^+$ is shifted from $0$ to $\pa_1 u|_{\scri^+}$. Therefore, a step $L_u v=-P(u)$ in the Newton iteration scheme~\eqref{EqISysSolv} does not give $v\in\cX^\infty$. A Picard iteration, solving $L_0 v=-P(u)$ would, due to the leading term of $\rho^{-1}(\pa_1 u)^2$ of size $\rho_I^{-1}$, cause $v$ to have a logarithmic leading term when integrating the analogue of~\eqref{EqISysLinTransport}. Neither iteration scheme closes, and this will remain true for any modification of the space $\cX^\infty$, e.g.\ if one allowed elements of $\cX^\infty$ to have leading terms involving higher powers of $\log\rho_I$. In fact, solutions of global versions of this equation blow up in finite time \cite{JohnBlowupQuasi}.

Assuming initial data to have sufficient decay, the nonlinear system $L u_1^c=0$, $L u_1-\rho^{-1}(\pa_1 u_1^c)^2=0$ on the other hand can be solved easily if we design the function space $\cX^\infty$ in step~\ref{ItISysAlg} to encode a $\rho_I^0$ leading term for $u_1^c$ at $\scri^+$, as in~\eqref{EqISysLinX}, and two leading terms, of size $\log\rho_I$ and $\rho_I^0$, for $u_1$. Extending this model slightly, let $\gamma>0$, recall $L_\gamma$ from~\eqref{EqISysLinDamping}, and consider the system for $u=(u_0,u_1^c,u_1)$,
\begin{equation}
\label{EqISysNull}
\begin{split}
  P(u) = \bigl( L_\gamma u_0,\ L u_1^c-\rho^{-1}(\pa_1 u_0)^2,\ L u_1-\rho^{-1}(\pa_1 u_1^c)^2 \bigr) = 0;
\end{split}
\end{equation}
which is a toy model for the nonlinear structure of the gauge-fixed Einstein equation with constraint damping, as we will argue in \S\ref{SsIEin}. Only working in $(\scri^+)^\circ$, i.e.\ ignoring weights at $I^0$ and $I^+$ for brevity, the above discussions suggest taking $b_I\in(0,\gamma)$ and working with the space\footnote{Here as well as in the previous example, one could of course work with much less precise function spaces since the full nonlinear system is lower triangular; for the Einstein equation on the other hand, we will need this kind of precision.}
\begin{equation}
\label{EqISysNullX}
  \cX^\infty = \{ u=(u_0,u_1^c,u_1) \colon (u_0,u_1^c-u_1^{c(0)},u_1-u_1^{(1)}\log\rho_I-u_1^{(0)}) \in \rho_I^{b_I}\Hb^\infty(M) \},
\end{equation}
where $u_1^{c(0)}$, $u_1^{(1)}$, $u_1^{(0)}\in\CI((\scri^+)^\circ)$ are the leading terms. Then
\[
  P \colon \cX^\infty \to \cY^\infty = \{ f=(f_0,f_1^c,f_1) \colon (f_0,f_1^c,f_1-\rho_I^{-1}f_1^{(0)})\in\rho_I^{-1+b_I}\Hb^\infty \},
\]
where $f_1^{(0)}\in\CI((\scri^+)^\circ)$. The linearization $L_u$ of $P$ around $u\in\cX^\infty$ then has as its model operator at $\scri^+$
\begin{equation}
\label{EqISysNullLin}
  L_u^0 = 2\rho_I^{-1}(\rho_I\pa_{\rho_I}-A_u)(\rho_0\pa_{\rho_0}-\rho_I\pa_{\rho_I}), \quad
  A_u=\begin{pmatrix}
      \gamma & 0 & 0 \\
      0 & 0 & 0 \\
      0 & \pa_1 u_1^{c(0)} & 0
    \end{pmatrix},
\end{equation}
which has a (lower triangular) \emph{Jordan block structure}, with all blocks either having positive spectrum (the upper $1\times 1$ entry) or being \emph{nilpotent} (the lower $2\times 2$ block). Thus, by integrating $\rho_I\pa_{\rho_I}-A_u$, we conclude that for $L_u v=-P(u)$, we have $v\in\cX^\infty$, thus closing the iteration scheme~\eqref{EqISysSolv}. A background estimate as well as its higher regularity version, which is the prerequisite for $L_u^0$ being of any use, can be proved as before. Error terms arising from commutation with $A_u$ have lower differential order and can thus be controlled inductively; that is, only the commutation properties \emph{of the principal part of $L_u^0$} matter for this.

\begin{rmk}
\label{RmkISysWeakNull}
  A tool for the study of the long time behavior of nonlinear wave equations on Minkowski space introduced by H\"ormander \cite{HormanderLifespan} is the \emph{asymptotic system}, see also~\cite[\S6.5]{HormanderNonlinearLectures} and \cite{LindbladRodnianskiWeakNull}: this arises by making an ansatz $u\sim \eps r^{-1} U(t-r,\eps\log r,\omega)$ for the solution and evaluating the $\eps^2$ coefficient, which gives a PDE in $1+1$ dimensions in the coordinates $t-r$ and $\ell:=\eps\log r$ which one expects to capture the behavior of the nonlinear equation near the light cone; if the classical null condition is satisfied, the PDE is linear, otherwise it it nonlinear. The weak null condition~\cite{LindbladRodnianskiWeakNull} is the assumption that solutions of the asymptotic system grow at most exponentially in $\ell$, and for the Einstein vacuum equation in harmonic gauge, solutions are \emph{polynomial} (in fact, linear) in $\ell$. The latter finds its analogue in our framework in the \emph{nilpotent} structure of the coupling matrix in~\eqref{EqISysNullLin}. (However, quasilinear equations with variable long-range perturbations, see the discussion around~\eqref{EqIEinGood}, cannot be treated directly with our methods, corresponding to the difficulty in assigning a geometric meaning to the asymptotic system in such situations.) For works which establish global existence of nonlinear equations even when the asymptotic system has merely exponentially bounded (in $\ell$) solutions, we refer to Lindblad~\cite{LindbladGlobalNonlinear,LindbladGlobalQuasilinear} and Alinhac~\cite{AlinhacBlowupAtInfty}.
\end{rmk}

\subsubsection{Polyhomogeneity}
\label{SssISysPhg}

Consider again equation~\eqref{EqISysLinEq0} near $(I^0)^\circ$, now assuming that $f$ is polyhomogeneous. For simplicity, let $f=\rho_0^{i z}f_z+\wt f$, where $f_z\in\CI(\pa\ol{\R^4})$, $z\in\C$, and $\wt f$ decays faster than the leading term, so $\wt f\in\rho_0^{b_0}\Hb^\infty$ with $b_0>-\Im z$. A useful characterization of the polyhomogeneity of $f$ is that the decay of $f$ improves upon application of the vector field $\rho_0 D_{\rho_0}-z$ in the notation~\eqref{EqISysLinD}. The solution $u$ satisfies $u\in\rho_0^{a_0}\Hb^\infty$ for any $a_0<-\Im z$; but $u':=(\rho_0 D_{\rho_0}-z)u$ solves\footnote{Commutator terms have improved decay at $\rho_0=0$ as before, hence are dropped here for clarity.}
\[
  L u' = (\rho_0 D_{\rho_0}-z)f = (\rho_0 D_{\rho_0}-z)\wt f \in \rho_0^{b_0}\Hb^\infty,
\]
so $u'\in\rho_0^{b_0}\Hb^\infty$. This is exactly the statement that $u$ has the form $u=\rho_0^{i z}u_z+\wt u$ for some $u_z\in\CI(\pa\ol{\R^4})$, $\wt u\in\rho_0^{b_0}\Hb^\infty$. If $f$ has a full polyhomogeneous expansion, an iteration of this argument shows that $u$ has one too, with the same index set.

Near the corner $I^0\cap\scri^+$ then, one can proceed iteratively as well, picking up the terms of the expansion at $\scri^+$ one by one, by analyzing the solution of the product of transport equations in equation~\eqref{EqISysLinTransport} when the right hand side has a partial polyhomogeneous expansion at $\scri^+$: the point is that $\rho_0\pa_{\rho_0}-\rho_I\pa_{\rho_I}$ transports expansions from $I^0$ to $\scri^+$, ultimately since it annihilates $\rho_0^{i z}\rho_I^{i z}$. See Lemmas~\ref{LemmaPhgODE1d}--\ref{LemmaPhgODE2}.

To obtain the expansion at $I^+$, we argue iteratively again, using the resonance expansion obtained via normal operator analysis as in the proof of \cite[Theorem~2.21]{HintzVasySemilinear}. One needs to invert the normal operator family of $L$ on spaces of functions which are polyhomogeneous at the boundary $\pa I^+$, which is easily accomplished by solving away polyhomogeneous terms formally and using the usual inverse, defined on spaces of smooth functions, to solve away the remainder; see Lemma~\ref{LemmaPhgNormalPhg}.

\subsection{Analysis of Einstein's equation}
\label{SsIEin}

For Einstein's equation, the strategy outlined in \S\ref{SsISys} needs to be supplemented by a preliminary step, the choice of the nonlinear operator $P$, which in particular means \emph{choosing a gauge}, i.e.\ a condition on the solution $g$ of $\Ric(g)=0$ which breaks the diffeomorphism invariance; by the latter we mean the fact that for any diffeomorphism $\phi$ of $M$, $\phi^*g$ also solves $\Ric(\phi^*g)=0$. Following DeTurck~\cite{DeTurckPrescribedRicci}, the presentation by Graham--Lee~\cite{GrahamLeeConformalEinstein}, and \cite{HintzVasyKdSStability}, we consider the gauge-fixed Einstein equation
\begin{equation}
\label{EqIEinOp}
  P_0(g) = \Ric(g) - \tdel^*\Ups(g) = 0,
\end{equation}
where $\tdel^*$ is a first order differential operator with the same principal symbol (which is independent of $g$) as the symmetric gradient $(\delta_g^*u)_{\mu\nu}=\half(u_{\mu;\nu}+u_{\nu;\mu})$; we comment on the choice of $\tdel^*$ below. Further, the \emph{gauge 1-form} is
\begin{equation}
\label{EqIEinUps}
  \Ups(g;g_m)_\mu := (g g_m^{-1}\delta_g G_g g_m)_\mu = g_{\mu\nu}g^{\kappa\lambda}(\Gamma(g)_{\kappa\lambda}^\nu - \Gamma(g_m)_{\kappa\lambda}^\nu),
\end{equation}
where $\delta_g$ is the adjoint of $\delta_g^*$, i.e.\ the (negative) divergence, $G_g=1-\half g\tr_g$ is the trace reversal operator, and $g_m$ is a fixed \emph{background metric}; \textit{we write $\Ups(g)\equiv\Ups(g;g_m)$ from now on.} This is a manifestly coordinate invariant generalization of the wave coordinate gauge, where one would choose $g_m=\ul g$ to be the Minkowski metric on $\R^4$ and demand that a global coordinate system $(x^\mu)\colon(M^\circ,g)\to(\R^4,\ul g)$ be a wave map. (Friedrich describes $\Ups(g)=0$ and more general gauge conditions using \emph{gauge source functions}, see in particular \cite[Equation~(3.23)]{FriedrichHyperbolicityEinstein}.)

Two fundamental properties of $P_0(g)$ are: \begin{enumerate*} \item $P_0(g)$ is a \emph{quasilinear wave equation}, hence has a well-posed initial value problem; \item by the second Bianchi identity---the fact that the Einstein tensor $\Ein(g):=G_g\Ric(g)$ is divergence-free---the equation $P_0(g)=0$ implies the wave equation \end{enumerate*}
\begin{equation}
\label{EqIEinUpsProp}
  \delta_g G_g\tdel^*\Ups(g)=0
\end{equation}
for $\Ups(g)$, which thus vanishes identically provided its Cauchy data are trivial; we call $\delta_g G_g\tdel^*$ the \emph{gauge propagation operator}. Therefore, solving~\eqref{EqIEinOp} with Cauchy data for which $\Ups(g)$ has trivial Cauchy data is equivalent to solving Einstein's equation~\eqref{EqIEin} in the gauge $\Ups(g)=0$.

Since we wish to solve the initial value problem~\eqref{EqIEinIVP}, we need to choose the Cauchy data for $g$, i.e.\ the restrictions $g_0$ and $g_1$ of $g$ and its transversal derivative to the initial surface $\Sigma^\circ$ \emph{as a Lorentzian metric on $M^\circ$} such that $\gamma$ is the pullback of $g_0$ to $\Sigma^\circ$ and $k$ is the second fundamental form of any metric with Cauchy data $(g_0,g_1)$; note that $k$ only depends on up to first derivatives of the ambient metric, hence can indeed be expressed purely in terms of $(g_0,g_1)$. These conditions do not determine $g_0,g_1$ completely, and one can arrange in addition that $\Ups(g)$ vanishes at $\Sigma^\circ$ as a 1-form on $M$. Provided then that $P_0(g)=0$, with these Cauchy data for $g$, holds near $\Sigma^\circ$, the constraint equations at $\Sigma^\circ$ can be shown to imply that also the transversal derivative of $\Ups(g)$ vanishes at $\Sigma^\circ$ (see the proof of Theorem~\ref{ThmPf}), and then the argument involving~\eqref{EqIEinUpsProp} applies.

If the initial data in Theorem~\ref{ThmIBaby} are \emph{exactly} Schwarzschildean for $r\geq R\gg 1$, the solution $g$ is equal (i.e.\ isometric) to the Schwarzschild metric in the domain of dependence of the region $r\geq R$; more generally, for initial data which are equal to those of mass $m$ Schwarzschild modulo decaying corrections, we expect all outgoing null-geodesics to be bent in approximately the same way as for the metric $g_m^S$. Thus, we should define the manifold $M$ in step~\ref{ItISysSmooth} so that $\scri^+$ is null infinity of the Schwarzschild spacetime. Now, along radial null-geodesics of $g_m^S$, the difference $t-r_*$ is constant, where
\begin{equation}
\label{EqIEinTortoise}
  r_*=r+2 m\log(r-2 m)
\end{equation}
is the \emph{tortoise coordinate} up to an additive constant, see~\cite[Equation~(6.4.20)]{WaldGR}. Correspondingly, we define the compactification ${}^m\ol{\R^4}$ near $t\sim r_*$ such that $\rho=r^{-1}$ is a boundary defining function, and ${}^m v:=(t-r_*)/r$ is smooth up to the boundary; ${}^m\!M$ is defined by blowing this up at $S^+=\{\rho=0,\,{}^m v=0\}$. (This is smoothly extended away from $t\sim r_*$ to a compactification of all of $\R^4$.) Thus, ${}^m\ol{\R^4}$ and the Minkowski compactification $\ol{\R^4}={}^0\ol{\R^4}$ are canonically identified by continuity from $\R^4$, but have slightly different smooth structures; see~\S\ref{SsCptM} and \cite[\S7]{BaskinVasyWunschRadMink2}.) The interior of the front face $\scri^+$ of the blow-up is diffeomorphic to $\R_s\times\Sph^2$, where $s:={}^m v/\rho=t-r_*$ is an affine coordinate along the fibers of the blow-up. We denote defining functions of $I^0$ (the closure of $\{\rho=0,\,{}^m v<0\}$ in ${}^m\!M$), $\scri^+$, and $I^+$ (the closure of $\{\rho=0,\,{}^m v>0\}$ in ${}^m\!M$) by $\rho_0$, $\rho_I$, and $\rho_+$, respectively.

It is then natural to fix the background metric $g_m$ to be equal to $g_m^S$ near $I^0\cup\scri^+$ and smoothly interpolate with the Minkowski metric near $r=0$ (which is nonsingular there, unlike $g_m^S$). We then work with the gauge $\Ups(g;g_m)=0$, and seek the solution of
\begin{equation}
\label{EqIEinMetric}
  P(h) := \rho^{-3} P_0(g) = 0,\ \ g = g_m + \rho h,
\end{equation}
with $h$ to be determined; the factors $\rho$ are introduced in analogy with the discussion of the scalar wave equation~\eqref{EqISysLinOp}.\footnote{Note that we use $g_m$ in two distinct roles: once as a background metric in the gauge condition, and once as a rough first guess of the solution of the initial value problem which \begin{enumerate*} \item already has the correct long range behavior at null infinity and \item is \emph{globally} close to a solution of the Einstein vacuum equation if $m$ is small. \end{enumerate*} See also Remark~\ref{RmkPfLargeMass}.} Here, $\rho$ is a global boundary defining function of ${}^m\ol{\R^4}$; one can e.g.\ take $\rho=r^{-1}$ away from the axis $r/t=0$, and $\rho=t^{-1}$ near $r/t=0$. Now, due to the quasilinear character of~\eqref{EqIEinOp}, the principal part of $L_h:=D_h P$ depends on $h$: it is given by $\half\Box_g$. Thus, one needs to ensure that throughout the iteration scheme~\eqref{EqISysSolv}, the null-geometry of $g$ is compatible with ${}^m\!M$, in the sense that the long range term of $g$ determining the bending of light rays remains unchanged. To see what this means concretely, consider a metric perturbation $h$ in~\eqref{EqIEinMetric} which is not growing too fast at $\scri^+$, say $|h|\lesssim \rho_I^{-\eps}$ for $\eps<1/2$ (that is, $|h|\lesssim r^\eps$ when $t-r_*$ remains in a bounded interval); one can then check that, modulo terms with faster decay at $\scri^+$,
\begin{equation}
\label{EqIEinLongRange}
  \Box_g = 2\rho_I^{-1}\bigl(\rho_I\pa_{\rho_I}+2\rho_0 h_{0 0}(\rho_0\pa_{\rho_0}-\rho_I\pa_{\rho_I})\bigr)(\rho_0\pa_{\rho_0}-\rho_I\pa_{\rho_I})\ \ \tn{near}\ I^0\cap\scri^+,
\end{equation}
which identifies
\begin{equation}
\label{EqIEinGood}
  h_{0 0}=h(\pa_0,\pa_0),\ \ \pa_0=\pa_t+\pa_{r_*},
\end{equation}
as the (only) long range component of $h$; see the calculation~\eqref{EqEinFinverse}.\footnote{In the case that $h_{0 0}$ vanishes at $\scri^+$, the approximate null directions $\rho_I\pa_{\rho_I}$ and $\rho_0\pa_{\rho_0}-\rho_I\pa_{\rho_I}$ have the same form as in the discussion surrounding~\eqref{EqISysLinTransport}, however, \emph{due to our choice of compactification ${}^m\!M$}, they are now the radial null directions of \emph{Schwarzschild} with mass $m$. (Integration along these more precise characteristics was key in Lindblad's proof of sharper asymptotics in~\cite{LindbladAsymptotics}.)} Indeed, the first vector field in~\eqref{EqIEinLongRange} is approximately tangent to outgoing null cones, so for $h_{0 0}\neq 0$ at $\scri^+$, outgoing null cones do \emph{not} tend to $(\scri^+)^\circ$. (Rather, if $h_{0 0}>0$, say, they are less strongly bent, like in a Schwarzschild spacetime with mass smaller than $m$.) Whether or not $h_{0 0}$ vanishes at $\scri^+$ depends on the choice of gauge. A calculation, see~\eqref{EqCoPertUpsLower}, shows that the gauge condition $\Ups(g)=0$ implies the constancy of $h_{0 0}$ along $\scri^+$; but since $h_{0 0}$ is initially $0$ due to $g_m$ already capturing the long range part of the initial data, this means that $h_{0 0}|_{\scri^+}=0$ indeed---\emph{provided that $P(h)=0$} with Cauchy data satisfying the gauge condition, as we otherwise cannot conclude the vanishing of $\Ups(g)$. We remark that $\Ups(g)=0$ implies the vanishing of further components of $h$, namely $r^{-1}h_{0 a}\equiv h(\pa_0,r^{-1}\pa_{\theta^a})$ and $r^{-2}\slg^{a b}h_{a b}$, $h_{a b}:=h(\pa_{\theta^a},\pa_{\theta^b})$, which we collectively denote by $h_0$; see~\eqref{EqEinFGood} and~\eqref{EqEinFSubbProj0}, where the notation $h_0=:\pi_0 h$ is introduced.

As we are solving approximate (linearized) equations at each step of our Newton-type iteration scheme in step~\ref{ItISysSolv}, we thus need an extra mechanism to ensure that $\Ups(g)$, $g=g_m+\rho h$, is decaying sufficiently fast at $\scri^+$ to guarantee the vanishing of $h_{0 0}$ at $\scri^+$. This is where constraint damping comes into play. Roughly speaking, if one only has an approximate solution of $P_0(g)\approx 0$, then we still get $\delta_g G_g\tdel^*\Ups(g)\approx 0$; if one chooses $\tdel^*$ carefully, solutions of this can be made to decay at $\scri^+$ sufficiently fast so as to imply the vanishing of $h_{0 0}$. We shall show that the choice
\[
  \tdel^* u = \delta_{g_m}^*u + 2\gamma\tfrac{d t}{t}\otimes_s u - \gamma(\iota_{t^{-1}\nabla t}u)g_m,\ \ 
  \gamma>0,
\]
accomplishes this.\footnote{For technical reasons related to the definition of the smooth structure on ${}^m\ol{\R^4}$, we shall modify $t$ slightly; see Definition~\ref{DefCptFtInv} and equation~\eqref{EqEinTdel}.} As a first indication, one can check that $2\delta_{g_m}G_{g_m}\tdel^*$ has a structure similar to~\eqref{EqISysLinDamping} with $\gamma>0$, for which we had showed the improved decay at $\scri^+$.

Regarding steps~\ref{ItISysAlg} and \ref{ItISysMap} of our general strategy, the correct function spaces can now be determined easily (after some tedious algebra): solving $L_0 u=0$, where $L_h=D_h P$ as usual, one finds that $u_0=\pi_0 u$, so in particular the long range component $u_{0 0}$ of $u$ decays at $\scri^+$, while the remaining components, denoted $u_0^c$, have a size $1$ leading term at $\scri^+$, just like solutions of the linear scalar wave equation. This follows from the schematic structure
\[
  \rho_I^{-1}\biggl(\rho_I\pa_{\rho_I}-\begin{pmatrix} \gamma & 0 \\ 0 & 0 \end{pmatrix}\biggr)(\rho_0\pa_{\rho_0}-\rho_I\pa_{\rho_I}) \begin{pmatrix} u_0 \\ u_0^c \end{pmatrix}
\]
of the model operator at $\scri^+$ in this case. However, for such $u$ then, solutions of $L_u u'=-P(u)$ have slightly more complicated behavior. Indeed, the model operator at $\scri^+$ has a schematic structure similar to~\eqref{EqISysNullLin}, acting on $(u'_0,(u')_{1 1}^c,u'_{1 1})$, where we separate the components of $(u')_0^c$ into two sets, one of which consists of the single component
\begin{equation}
\label{EqIEinBad}
  u'_{1 1}=u(\pa_1,\pa_1),\ \ \pa_1=\pa_t-\pa_{r_*},
\end{equation}
while $(u')_{1 1}^c$ captures the remaining components, which are $u_{0 1}$, $r^{-1} u_{1 b}$, and the part $r^{-2}(u_{a b}-\half\slg_{a b}\slg^{c d}u_{c d})$ of the spherical part of $u$ which is trace-free with respect to $\slg$. Correspondingly, we need to allow $u'_{1 1}$ to have a logarithmic leading order term, just like the component called $u_1$ in the definition of the function space~\eqref{EqISysNullX}. In the next iteration step, $L_{u'}u''=-P(u')$, no further adjustments are necessary: the structure of the model operator at $\scri^+$ is unchanged, hence the asymptotic behavior of $u''$ does not get more complicated.\footnote{The coupling matrix, called $A_u$ in~\eqref{EqISysNullLin}, is in fact slightly more complicated here, see Lemma~\ref{LemmaEinNscri}, necessitating a more careful choice of the weights of the remainder terms of elements of the spaces $\cX^\infty$ and $\cY^\infty$ at $\scri^+$, whose precise definitions we give in Definitions~\ref{DefEinF} and \ref{DefEinFY}.} We remark that due to our precise control over each iterate, encoded by membership in $\cX^\infty$, the relevant structure of the model operators and the regularity of the coefficients of the linearized equations are the same at each iteration step; in particular, the fact that equation~\eqref{EqIEinMetric} is quasilinear rather than semilinear does not cause any complications beyond the need for constraint damping.

The decoupling of the model operator at $\scri^+$ into three pieces---one for the decaying components $u_0$, one for the components $u_{1 1}^c$ which have possibly nontrivial leading terms at $\scri^+$, and one for the logarithmically growing component $u_{1 1}$---is the key structure making our proof of global stability work. The fact that the equation for the components $u_0$ decouples is not coincidental, as they are governed by the gauge condition and thus are expected to decouple to leading order in view of the second Bianchi identity as around~\eqref{EqIEinUpsProp}.\footnote{In practice, it is easier to analyze $u_0$ directly using the structure of the linearized gauge-fixed Einstein equation, rather than via an (approximate) linearized second Bianchi identity, so this is how we shall proceed.} The decoupling of $u_{1 1}$ and $u_{1 1}^c$ on the other hand is the much more subtle manifestation of the weak null condition, as discussed in Remark~\ref{RmkISysWeakNull}.

The solution $h$ of~\eqref{EqIEinMetric} is a symmetric 2-tensor in $M^\circ$; as part of step~\ref{ItISysSmooth}, we still need to specify the smooth vector bundle on $M$ which $h$ will be a section of. Consider first the Minkowski metric $\ul g$ on the radial compactification ${}^0\ol{\R^4}$. In $\R^4$, $\ul g$ is a quadratic form, with constant coefficients, in the 1-forms $d t$ and $d x^i$, which extend smoothly to the boundary as sections of the \emph{scattering cotangent bundle} $\Tsc^*\,{}^0\ol{\R^4}$ first introduced in \cite{MelroseEuclideanSpectralTheory}; in a collar neighborhood $[0,1)_\rho\times\R^3_X$ of a point in $\pa{}^0\ol{\R^4}$, the latter is by definition spanned by the 1-forms $\frac{d\rho}{\rho^2}$, $\frac{d X^i}{\rho}$, which are smooth and linearly independent sections of $\Tsc^*\,{}^0\ol{\R^4}$ \emph{down to the boundary}. For instance, near $r=0$, we can take $\rho=t^{-1}$ and $X=x/t$, in which case $\frac{d\rho}{\rho^2}=-d t$ and $\frac{d X^i}{\rho}=d x^i-X^i\,d t$. Similarly then, $g_m$ will be a smooth section of the second symmetric tensor power $S^2\,\Tsc^*\,{}^m\ol{\R^4}$. Since our nonlinear analysis takes place on the blown-up space ${}^m\!M$, we seek $h$ as a section of the pullback bundle $\beta^*S^2\,\Tsc^*\,{}^m\ol{\R^4}$, where $\beta\colon{}^m\!M\to{}^m\ol{\R^4}$ is the blow-down map. For brevity, we shall suppress the bundle from the notation here.

\begin{thm}
\label{ThmIDetail}
  Suppose the assumptions of Theorem~\usref{ThmIBaby} are satisfied, i.e.\ for some small $m\in\R$ and $b_0>0$ fixed, the normalized data $\rho_0^{-1}\wt\gamma$ and $\rho_0^{-2}k\in\rho_0^{b_0}\Hb^\infty(\ol{\R^3})$ are small in $\rho_0^{b_0}\Hb^{N+1}$ and $\rho_0^{b_0}\Hb^N$, respectively. Then there exists a solution $g$ of the initial value problem \eqref{EqIEinIVP} satisfying the gauge condition $\Ups(g)=0$, see~\eqref{EqIEinUps}, which on ${}^m\!M$ is of the form $g=g_m+\rho h$, $h\in\rho_0^{b_0}\rho_I^{-\eps}\rho_+^{-\eps}\Hb^\infty({}^m\!M)$ for all $\eps>0$; here $\rho$ is a boundary defining function of ${}^m\ol{\R^4}$, and $\rho_0$, $\rho_I$, and $\rho_+$ are defining functions of $I^0$, $\scri^+$, and $I^+$, respectively.
  
  More precisely, near $\scri^+$ and using the notation introduced after~\eqref{EqIEinGood} and \eqref{EqIEinBad}, the components $h_{0 0}$, $r^{-1}h_{0 b}$, and $r^{-2}\slg^{a b}h_{a b}$ lie in
  \begin{equation}
  \label{EqIDetailSpace}
    \rho_0^{b_0}\rho_I^{b_I}\rho_+^{-\eps}\Hb^\infty({}^m\!M)
  \end{equation}
  for all $b_I<\min(1,b_0)$ and $\eps>0$, while $h_{0 1}$, $r^{-1}h_{1 b}$, and $r^{-2}(h_{a b}-\half\slg_{a b}\slg^{c d}h_{c d})$ have size $1$ leading terms at $\scri^+$ plus a remainder in the space~\eqref{EqIDetailSpace} for all such $b_I,\eps$, and $h_{1 1}$ has a logarithmic and a size $1$ leading term at $\scri^+$ plus a remainder in the space~\eqref{EqIDetailSpace} for all such $b_I,\eps$. At $I^+$ on the other hand, $h$ has a size $1$ leading term: there exists $h^+\in\rho_I^{-\eps}\Hb^\infty(I^+)$ such that $h-h^+\in\rho_I^{-\eps}\rho_+^{b_+}\Hb^\infty({}^m\!M)$ near $I^+$ for any $b_+<\min(b_0,1)$.
\end{thm}

\begin{rmk}
\label{RmkIDetailDecay}
  Near ${}^m\scri^+$, and indeed for $r\gg 1$ and $t-r_*\leq\half r$, the membership $u\in\rho_0^{b_0}\rho_I^{b_I}\rho_+^{b_+}\Hb^\infty({}^m\!M)$ (e.g.\ $u$ being a metric coefficient of $h$, and $b_+=-\eps$ as in~\eqref{EqIDetailSpace}) is equivalent, up to arbitrarily small losses in decay (due to switching from $L^2$ to $L^\infty$ via Sobolev embedding), to
  \[
    |V_1\cdots V_N u| \lesssim r^{-b_I}(1+(r_*-t)_+)^{-b_0+b_I} (1+(t-r_*)_+)^{b_++b_I}
  \]
  for all $N\in\N_0$, where each $V_j$ is a rotation vector field in $\R^3$ or one of the vector fields $t\pa_t+r_*\pa_{r_*}$, $t\pa_{r_*}+r_*\pa_t$, $\pa_t$, $\pa_x$.
\end{rmk}

See Theorem~\ref{ThmPf} for the full statement, which in particular allows for the decay rate $b_0$ of the initial data to be larger and gives the corresponding weight at $I^0$ for the solution. The final conclusion follows from resonance considerations, as indicated before~\eqref{EqISysLinX}, and will follow from the arguments used to establish polyhomogeneity in \S\ref{SPhg}. We discuss continuous dependence on initial data in Remark~\ref{RmkPfCts}. A typical example of a polyhomogeneous expansion of $h$ arises for initial data which are smooth functions of $1/r$ in $r\gg 1$: in this case, the leading terms of the expansion of $h$ are schematically (and not showing the coefficients, which are functions on $\scri^+$)
\begin{equation}
\label{EqIDetailLogPowers}
  h_0\sim\rho_I\log^{\leq 3}\rho_I,\ \ 
  h_{1 1}^c\sim 1+\rho_I\log^{\leq 4}\rho_I,\ \ 
  h_{1 1}\sim\log^{\leq 1}\rho_I+\rho_I\log^{\leq 6}\rho_I
\end{equation}
at $\scri^+$, and $h\sim 1+\rho_+\log^{\leq 8}\rho_+$ at $I^+$; see Example~\ref{ExPhgSmooth}. Here, $\log^{\leq k}\rho_I$ stands for functions which are sums of products $|\log\rho_I|^\ell a_\ell$, $0\leq\ell\leq k$, with $a_\ell$ functions on $\scri^+$.

While a solution $g$ of $\Ric(g)=0$ in the gauge $\Ups(g)=0$ of course solves equation~\eqref{EqIEinOp} for any choice of $\tdel^*$, we argued why a careful choice is crucial to make our global iteration scheme work. Another perspective is the following: implementing constraint damping allows us to solve the gauge-fixed equation~\eqref{EqIEinOp} for \emph{any} sufficiently small Cauchy data; whether or not these data come from an initial data set satisfying the constraint equations is \emph{irrelevant}. Only at the end, once one has a solution of~\eqref{EqIEinOp}, do we use the constraint equations and the second Bianchi identity to deduce $\Ups(g)=0$.

In contrast, consider the choice $\tdel^*=\delta_g^*$ in~\eqref{EqIEinOp}; the linearization of $P_0(g)$ around the Minkowski metric $g=\ul g$ is then equal to $\half\Box_{\ul g}$, which is $\half$ times the scalar wave operator acting component-wise on the components of a symmetric 2-tensor in the frame $d x^\mu\otimes d x^\nu+d x^\nu\otimes d x^\mu$, where $x^0=t$, $x^i$, $i=1,2,3$, are the standard coordinates on $\R^{1+3}_{t,x}$. Solving $\Box_{\ul g}(\rho h)=0$ with given initial data, which would be the first step in our iteration scheme for initial data with mass $m=0$, does not imply improved behavior for any components of $h$, in particular $h_{0 0}$; this means that constraint damping fails for this choice of $\tdel^*$. Thus, the next iterate $\ul g+\rho h$ in general has a different long range behavior, and correspondingly ${}^0\!M$ is no longer the correct place for the analysis of the linearized operator in the next iteration step---even though the final solution of Einstein's equation \emph{is} well-behaved on ${}^0\!M$ for such initial data. With constraint damping on the other hand, the linearized equation always produces behavior consistent with the qualitative properties of the nonlinear solution.

\subsection{Bondi mass loss formula}
\label{SsIBondi}

The description of the asymptotic behavior of the metric $g=g_m+\rho h$ in Theorem~\ref{ThmIDetail} on the compact manifold ${}^m\!M$ and in the chosen gauge allows for a precise description of outgoing light cones close to the radiation face $\scri^+$. Work on geometric quantities at $\scri^+$ started with the seminal works of Bondi--van der Burg--Metzner \cite{BondiGravity,BondivdBMetznerGravity}, Sachs \cite{SachsGravity,SachsAsymptoticSymm}, Newman--Penrose \cite{NewmanPenroseSpin}, and Penrose \cite{PenroseAsymptotics}; the precise decay properties of the curvature tensor---in particular `peeling estimates' or their failure---were discussed in \cite{KlainermanNicoloPeeling,ChristodoulouNoPeeling}, see also \cite{DafermosChristodoulouExpose}. (For studies on conditions on initial data which ensure or prevent smoothness of the metric at $\scri^+$ in suitable coordinates, see~\cite{FriedrichConformalFE,FriedrichHyperbolicityEinstein,ChruscielMacCallumSingletonPhg,AnderssonChruscielHypPhg,ValienteKroonNonsmooth} and~\cite[\S8.2]{KlainermanNicoloEvolution}.)

As remarked before, the logarithmic bending of light cones is controlled by the ADM mass $m$, which measures mass on \emph{spacelike}, asymptotically flat, Cauchy surfaces. A more subtle notion is the \emph{Bondi mass} \cite{BondivdBMetznerGravity}, see also \cite{ChristodoulouNonlinear}, which is a function of retarded time $x^1=t-r_*$ that can be defined as follows: let $S(u)\subset\scri^+$ denote the $u$-level set of $x^1$ at null infinity; $S(u)$ is a 2-sphere, and naturally comes equipped with the round metric. If $C_u$ denotes the outgoing light cone which limits to $S(u)$ at null infinity and which asymptotically approaches the radial Schwarzschild light cone $\{x^1=u\}$, one can define a natural \emph{area radius} $\mathring r$ on $C_u$, equal to the coordinate $r$ plus lower order correction terms; the Bondi mass $M_{\rm B}(u)$ is then the limit of the \emph{Hawking mass} of the 2-sphere $\{x^1=u,\ \mathring r=R\}$ as $R\to\infty$. See~\S\ref{SL} for the precise definitions. A change $\frac{d}{d u}M_{\rm B}(u)$ of the Bondi mass reflects a flux of gravitational energy to $\scri^+$ along $C_u$. We shall calculate these quantities explicitly and show:

\begin{thm}
\label{ThmIBondi}
  Suppose we are given a metric constructed in Theorem~\usref{ThmIDetail}, and write $h_{1 1}=h_{1 1}^{(1)}\log(r)+\cO(1)$ near $\scri^+$, where $h_{1 1}^{(1)}\in\rho_0^{b_0}\rho_+^{-\eps}\Hb^\infty(\scri^+)$ is the logarithmic leading term. Then the Bondi mass is equal to
  \begin{equation}
  \label{EqIBondi}
    M_{\rm B}(u) = m + \frac{1}{4\pi}\int_{S(u)} \half h_{1 1}^{(1)}\,d\slg.
  \end{equation}
  The Bondi mass loss formula takes the form $\frac{d}{d u}M_{\rm B}(u)=-E(u)$, where
  \[
    E(u) = \frac{1}{32\pi}\int_{S(u)} |N|_{\slg}^2\,d\slg,\ \ N_{a b}=r^{-2}\pa_1 h_{a b}|_{\scri^+},
  \]
  is the outgoing energy flux. Finally, $M_{\rm B}(-\infty)=m$ and $M_{\rm B}(+\infty)=0$.
\end{thm}

We prove this for all initial data which are small and asymptotically flat in the sense of~\eqref{EqIBabySmall}. The Bondi mass was shown to be well-defined (and to satisfy a mass loss formula) for the weakly decaying initial data used in~\cite{BieriZipserStability} by Bieri--Chru\'sciel \cite{BieriChruscielADMBondi} in the geometric framework of \cite{ChristodoulouKlainermanStability}, but the question of how to define Bondi--Sachs coordinates remained open. Our result is the first to accomplish this for a large class of initial data, and to identify the Bondi mass in a (generalized) wave coordinate gauge setting. (The $\cC^{1,\min(b_0,1)-0}$ regularity of a conformally rescaled non-degenerate metric down to $\scri^+$ is a by-product of our analysis.) The key to establishing the first part of Theorem~\ref{ThmIBondi} is the construction and precise control of the aforementioned geometric quantities leading to the identification~\eqref{EqIBondi}; the mass loss formula itself is then equivalent to the vanishing of the leading term of the $(1,1)$ component of the gauge-fixed Einstein equation at $\scri^+$. The vanishing of $M_{\rm B}(u)$ as $u\to-\infty$ follows immediately from the decay properties of $h$ there. On the other hand, the proof that the total radiated energy $\int E(u)\,d u$ equals the initial mass $m$ proceeds by studying the leading order term $h|_{I^+}$ as the solution of a linear equation on $I^+$ (obtained by restricting the nonlinear gauge-fixed Einstein equation to $I^+$), with a forcing term that comes from the failure of our glued background metric $g_m$ to satisfy the Einstein equation and which is thus proportional to $m$. This equation now is closely related to the spectral family of exact hyperbolic space at the bottom of the essential spectrum;\footnote{This linear operator acts on the symmetric scattering 2-tensor bundle restricted to $I^+$; see~\cite{HadfieldSymm} for the relation with the hyperbolic Laplacian acting on its intrinsic 2-tensor bundle. The spectral parameter here is \emph{fixed}, and the definition of the scattering matrix (incoming data having logarithmic rather than algebraic growth) is specific to working at the bottom of the spectrum; this is in contrast to the description of the scattering matrix depending on the spectral parameter as e.g.\ in~\cite{GrahamZworskiScattering}.} a calculation of the scattering matrix acting on the incoming data given by $h_{1 1}^{(1)}$ and comparing the $(0,0)$ component of the outgoing data with $h_{0 0}$---which vanishes by construction!---then establishes the desired relationship.

Theorem~\ref{ThmIBondi} shows that the logarithmic term in the asymptotic expansion of $h_{1 1}$ carries physical meaning. Its vanishing forces $m=0$, which by the positive mass theorem means that the spacetime is exact Minkowski space. (The observation that $\int E(u)\,d u\geq 0$ immediately implies the nonnegativity of the ADM mass of the \emph{small} initial data under consideration here, which in this case was first proved by Choquet-Bruhat--Marsden~\cite{ChoquetBruhatMarsdenMass}.)

Further geometric properties of the vacuum metrics constructed in this paper, such as the identification of $(\scri^+)^\circ\subset M$, resp.\ $(I^+)^\circ$, as the set of endpoints of future-directed null, resp.\ timelike, geodesics,  will be discussed elsewhere.

\subsection{Outline of the paper}
\label{SsIOut}

In \S\S\ref{SCpt} and \ref{SEin}, we set the stage for the analysis (steps~\ref{ItISysSmooth} and \ref{ItISysAlg}): we give the precise definition of the compactification $M={}^m\!M$ on which we will find the solution of~\eqref{EqIEinIVP} in \S\ref{SsCptA}; the relevant function spaces are defined in \S\ref{SsCptF}, and the relationships between different compactifications are discussed in \S\ref{SsCptM}. In \S\ref{SsCptScri}, we prepare the invariant formulation of estimates such as~\eqref{EqISysLinEqIEst}; the results there are not needed until \S\ref{SBg}. In \S\ref{SsEinF}, we define the spaces $\cX^\infty$ and $\cY^\infty$ on $M$ in which we shall find the solution $h$ in Theorem~\ref{ThmIDetail}, and calculate the mapping properties and model operators of the (linearized) gauge-fixed Einstein operator in \S\S\ref{SsEinP} and \ref{SsEinN}, respectively. (The necessary algebra is moved to Appendix~\ref{SCo}.) The key structures (constraint damping, null structure) critical for our proof will be discussed there as well. We accomplish part~\ref{ItISysMapReg} of step~\ref{ItISysMap}---the proof of a high regularity background estimate with imprecise weights---by exploiting these structures in \S\ref{SBg}. The recovery of the precise asymptotic behavior in \S\ref{SIt} finishes step~\ref{ItISysMapDec}. Putting this into a Nash--Moser framework allows us to finish the proof of Theorem~\ref{ThmIDetail} in \S\ref{SPf}; the proof of polyhomogeneity, thus of the last part of Theorem~\ref{ThmIBaby}, is proved in \S\ref{SPhg}.  Finally, a finer description of the resulting asymptotically flat spacetime near null infinity, leading to the proof of Theorem~\ref{ThmIBondi}, is given in~\S\ref{SL}.

For the reader only interested in the key parts of the proof, we recommend reading \S\S\ref{SsCptA} and \ref{SsCptF} for the setup, \S\ref{SsEinF} for the form of metric perturbations we need to consider, and \S\ref{SsEinP} for an explanation of the main features of the linearized problem; taking the background estimate, Theorem~\ref{ThmBg} (which uses material from~\S\ref{SsCptScri}, and whose proof roughly follows the steps outlined in \S\ref{SssISysLin}), as a black box, the argument formally concludes in \S\ref{SIt}. (Getting the actual nonlinear solution in \S\ref{SPf} is then routine.)

\subsection*{Acknowledgments}

We are very grateful to Rafe Mazzeo for his encouragement to work on this project, and to Piotr Chru\'sciel for many comments and corrections. We would also like to thank Sergiu Klainerman, Hans Lindblad, Jonathan Luk, Richard Melrose, Michael Singer, Gunther Uhlmann, and Maciej Zworski for useful discussions, comments, and their interest in this project. A.V.\ gratefully acknowledges support by the NSF under grant numbers DMS-1361432 and DMS-1664683 as well as from a Simons Fellowship. Part of this research was conducted during the period P.H.\ served as a Clay Research Fellow; in its early stages, this project was supported by a Miller Research Fellowship. Last but not least, we would like to thank three thorough referees for many helpful comments and suggestions which significantly improved the readability and accessibility of the paper.

\section{Compactification}
\label{SCpt}

As explained in~\S\ref{SsIEin}, we shall find the metric $g$ in Theorem~\ref{ThmIDetail} as a perturbation of a background metric $g_m$ which interpolates between mass $m$ Schwarzschild in a neighborhood $\{r\gg 1,\ |t|<2 r\}$ of $I^0\cup\scri$ and the Minkowski metric elsewhere. In \S\ref{SsCptA}, we define such a metric $g_m$ as a smooth scattering metric on a suitable partial compactification ${}^m\ol{\R^4}$ of $\R^4$ to a manifold with boundary which is closely related to the radial compactifications of asymptotically Minkowski spaces used in \cite{BaskinVasyWunschRadMink,BaskinVasyWunschRadMink2}. The ideal boundaries $I^0$, $\scri^+$, and $I^+$ are then the boundary hypersurfaces of a manifold with corners obtained by blowing up ${}^m\ol{\R^4}$ at the `light cone at infinity.' The spaces of conormal and polyhomogeneous functions on this manifold are defined in \S\ref{SsCptF}.

Let us recall the notion of the \emph{scattering cotangent bundle} $\Tsc^*X$ over an $n$-dimensional manifold $X$ with boundary $\pa X$. Over the interior $X^\circ$, $\Tsc^*_{X^\circ}X:=T^*_{X^\circ}X$ is the usual cotangent bundle. Near the boundary, let
\begin{equation}
\label{EqCptCoord}
  \rho\geq 0,\quad y=(y^1,\ldots,y^{n-1})\in\R^{n-1}
\end{equation}
denote local coordinates in which $\pa X$ is given by $\rho=0$; then the 1-forms $\frac{d\rho}{\rho^2}$, $\frac{d y^j}{\rho}$ ($j=1,\ldots,n-1$) are a smooth local frame of $\Tsc^*X$, i.e.\ smooth scattering 1-forms are precisely the linear combinations $a(\rho,y)\frac{d\rho}{\rho^2}+a_j(\rho,y)\frac{d y^j}{\rho}$ with $a,a_j$ smooth. (Equivalently, we can use $d(1/\rho)$ and $d(y^j/\rho)$ as a smooth local frame.) The point is that, viewed from the perspective of $X^\circ$, such 1-forms have a very specific behavior as one approaches $\pa X$. Tensor powers and their symmetric versions $S^k\,\Tsc^*X$, $k\in\N$, are defined in the usual manner; the dual bundle is denoted $\Tsc X$ and called \emph{scattering tangent bundle}. In the case that $\pa X=Y\times Z$ and $X=[0,1)_\rho\times\pa X$ are products, so $T^*Y\subset T^*X$ is a well-defined subbundle, then the rescaling $\rho^{-1}T^*Y\subset\Tsc^*X$, spanned by covectors of the form $\rho^{-1}\eta$, $\eta\in T^*Y$, is a smooth subbundle.

To give an example, calculations similar to the ones prior to Theorem~\ref{ThmIDetail} show that the differentials of the standard coordinates on $\R^n$ extend to the radial compactification $\ol{\R^n}$ as smooth scattering 1-forms; they are in fact a basis of $\Tsc^*\ol{\R^n}$, and any metric on $\R^n$ with constant coefficients, such as the Minkowski or Euclidean metric, is a \emph{scattering metric}, i.e.\ an element of $\CI(\ol{\R^n};S^2\,\Tsc^*\ol{\R^n})$.

The \emph{b-cotangent bundle} $\Tb^*X$ is locally spanned by the 1-forms $\frac{d\rho}{\rho}$, $d y^j$ ($j=1,\ldots,n-1$); its dual is the \emph{b-tangent bundle} $\Tb X$, spanned locally by $\rho\pa_\rho$ and $\pa_{y^j}$. The space $\Vb(X)$ of b-vector fields on $X$, consisting of those vector fields $V$ on $X$ which are tangent to $\pa X$, is then canonically identified with $\CI(X;\Tb X)$. A \emph{b-metric} is a nondegenerate section of $S^2\,\Tb X$. The space $\Diffb^k(X)$ of b-differential operators of degree $k$ consists of finite sums of $k$-fold products of b-vector fields. Fixing a collar neighborhood $[0,\eps)_\rho\times\pa X$ and choosing local coordinates $y^j$ on $\pa X$ as before, the \emph{normal operator} of an operator $L\in\Diffb^k(X)$ given in the coordinates~\eqref{EqCptCoord} by $L=\sum_{j+|\alpha|\leq k} a_{j\alpha}(\rho,y)(\rho D_\rho)^j D_y^\alpha$ is defined by freezing coefficients at $\rho=0$,
\begin{equation}
\label{EqCptNormOp}
  N(L) := \sum_{j+|\alpha|\leq k} a_{j\alpha}(0,y)(\rho D_\rho)^j D_y^\alpha \in \Diffb^k([0,\infty)_\rho\times\pa X).
\end{equation}
This depends on the choice of collar neighborhood only through the choice of normal vector field $\pa_\rho|_{\pa X}$; see \cite[\S4.15]{MelroseAPS} for an invariant description. The \emph{Mellin-transformed normal operator family} $\wh{L}(\sigma)$, $\sigma\in\C$, is the conjugation of $N(L)$ by the Mellin transform in $\rho$; thus, in view of $\rho^{-i\sigma}\rho D_\rho(\rho^{i\sigma})=\sigma\rho^{i\sigma}$, one obtains $\wh{L}(\sigma)$ by formally replacing $\rho D_\rho$ by $\sigma$:
\[
  \wh{L}(\sigma) := \sum_{j+|\alpha|\leq k} a_{j\alpha}(0,y)\sigma^j D_y^\alpha.
\]
This is a holomorphic family of elements of $\Diff^k(\pa X)$. Analogous constructions can be performed for b-operators acting on vector bundles.

\subsection{Analytic structure}
\label{SsCptA}

Fix the mass $m\in\R$; for now, $m$ does not have to be small. The Schwarzschild metric, written in polar coordinates on $\R\times\R^3$, takes the form
\begin{align}
  g_m^S &= (1-\tfrac{2 m}{r})dt^2 - (1-\tfrac{2 m}{r})^{-1}dr^2 - r^2\slg \nonumber\\
\label{EqCptASchw}
        &= (1-\tfrac{2 m}{r})ds^2 + 2 ds\,dr - r^2\slg,
\end{align}
where $\slg$ denotes the round metric on $\Sph^2$, and where we let
\begin{equation}
\label{EqCptARstar}
  s := t-r_*,\ \ r_* := r + 2 m\log(r-2 m),
\end{equation}
so $dr_*=\tfrac{r}{r-2 m}dr$. Note that level sets of $s$ are radial outgoing null cones. Define
\begin{equation}
\label{EqCptACp1}
  \rho := r^{-1}, \ \ v := r^{-1}\bigl(t - r - \chi(t/r)2 m\log(r-2 m)\bigr),
\end{equation}
where $\chi(x)\equiv 1$, $x<2$, and $\chi(x)\equiv 0$, $x>3$. Let then
\begin{equation}
\label{EqCptACp1Space}
  C_1 := [0,\eps_0)_\rho \times (-\tfrac74,5)_v \times \Sph^2_\omega,
\end{equation}
where we shrink $\eps_0>0$ so that $t$ is well-defined and depends smoothly on $\rho>0$ and $v$, via the implicit function theorem applied to~\eqref{EqCptACp1}. This will provide the compactification near the future light cone (and part of spatial infinity). Near future infinity, we use standard coordinates $(t,x)\in\R\times\R^3$ on $\R^4$; define
\begin{equation}
\label{EqCptACp2}
  \rho'_+=t^{-1}, \ \ X=x/|t|,
\end{equation}
and put
\begin{equation}
\label{EqCptACp2Space}
  C_2 := [0,\eps_0)_{\rho'_+} \times \{ X\in\R^3 \colon |X|<\tfrac14 \}.
\end{equation}
For $\eps_0>0$ small enough, we can consider the interiors $C_1^\circ$, $C_2^\circ$ as smooth submanifolds of $\R^4$ using the identifications \eqref{EqCptACp1} and \eqref{EqCptACp2}. (Note in particular that the smooth structures agree with the induced smooth structure of $\R^4$.) Let us consider the transition map between $C_1^\circ$ and $C_2^\circ$ in more detail: in $C_1^\circ\cap C_2^\circ$ and for $t^{-1}$ small enough, we have $\chi(t-r)\equiv 0$ and $\tfrac{r}{t}>\tfrac17$, so the map
\begin{equation}
\label{EqCptACp12}
  (\rho'_+,X) \mapsto (\rho=\rho'_+ / |X|,\ v=|X|^{-1}-1,\ \omega=X/|X|)
\end{equation}
extends smoothly (with smooth inverse) to $\rho'_+=0$. We then let
\[
  \ol{\R^4} := \bigl( \R^4 \sqcup C_1 \sqcup C_2 \bigr) / \sim
\]
where $\sim$ identifies $C_1$ and $C_2$ with subsets of $\R^4$ as above, and the boundary points of $C_1$ and $C_2$ are identified using the map \eqref{EqCptACp12}. This is thus a smooth manifold with boundary,\footnote{Different choices of $\chi$ produce the same topological space, indeed $\cC^\alpha$ manifold ($\alpha<1$); on the other hand, the smooth structure at the boundary does depend on $\chi$, but only in the gluing region $C_1\cap C_2$. All resulting smooth structures work equally well.} though both $\ol{\R^4}$ and $\pa{\ol{\R^4}}=(\pa C_1\sqcup\pa C_2)/\sim$ are noncompact. In other words, $\ol{\R^4}$ is only a compactification of the region $v>-\tfrac74$. See Figure~\ref{FigCptAR4}.

\begin{figure}[!ht]
\includegraphics{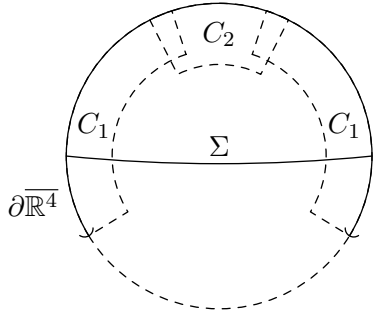}
\caption{The partial compactification $\ol{\R^4}$ of $\R^4$, constructed from $\R^4$, $C_1$, and $C_2$. Also shown is the hypersurface $\Sigma$ from~\eqref{EqCptASigma}.}
\label{FigCptAR4}
\end{figure}

The scattering cotangent bundle of $\ol{\R^4}$ near the light cone at infinity has a smooth partial trivialization $\Tsc^*_{C_1}\ol{\R^4}=\la dr\ra \oplus \la d(v/\rho)\ra \oplus \rho^{-1}T^*\Sph^2$, thus if $\psi$ is a smooth function with $\psi(v)\equiv 1$ for $v<1$ and $\psi(v)\equiv 0$ for $v>2$, then
\begin{equation}
\label{EqCptAMetric1}
  g_{m,1} := (1-\tfrac{2 m \psi(v)}{r})d(v/\rho)^2 + 2 d(v/\rho)\,dr - r^2\slg \in \CI(C_1;S^2\,\Tsc^*_{C_1}\ol{\R^4}).
\end{equation}
In $v>3$ and for $\eps_0>0$ small enough, we simply have $g_{m,1}=dt^2-dr^2-r^2\slg$, which is thus equal to
\[
  g_{m,2} := d(1/\rho'_+)^2 - d(X/\rho'_+)^2 \in \CI(C_2;S^2\,\Tsc^*_{C_2}\ol{\R^4})
\]
on the overlap $C_1\cap C_2$. Thus, we can glue $g_{m,1}$ and $g_{m,2}$ together to define a Lorentzian scattering metric $\wt g_m$ on $C_1\cup C_2$. We extend $\wt g_m$ to a global metric:

\begin{definition}
\label{DefCptAMetric}
  Fix $\phi\in\CI(\ol{\R^4})$ such that $\supp\phi\subset C_1\cup C_2$, and so that $\phi\equiv 1$ near $\pa\ol{\R^4}$. With $\tilde g_m$ as above, we then define
  \begin{equation}
  \label{EqCptAMetric}
    g_m := \phi\wt g_m + (1-\phi)(dt^2-dx^2) \in \CI(\ol{\R^4};S^2\,\Tsc^*\ol{\R^4}),
  \end{equation}
  thus gluing $\wt g_m$ to the Minkowski metric away from $C_1\cup C_2$.
\end{definition}

By construction, $g_m$ is equal to the Minkowski metric in a compact region of $\R^4$ as well as in a closed subcone of the interior of the future light cone, which we glue together with the Schwarzschild metric near spacelike and null infinity.

Next, denote the light cone at future infinity by
\begin{equation}
\label{EqCptASplus}
  S^+ := \{ \rho=0,\ v=0 \} \subset \pa\ol{\R^4}
\end{equation}
and let
\[
  M' := [\ol{\R^4};S^+]
\]
denote the blow-up of $\ol{\R^4}$ at $S^+$, see Figure~\ref{FigCptABlowup}. That is, as a set,
\[
  M'=\bigl(\ol{\R^4}\setminus S^+\bigr) \sqcup \bigl([-\pi/2,\pi/2]_\sigma \times\Sph^2_\omega\bigr),
\]
which can be endowed with the structure of a smooth (noncompact) manifold with corners by writing it as
\[
  M' = \left(\bigl(\ol{\R^4}\setminus S^+\bigr) \sqcup \bigl([0,1)_{\rho_I}\times[-\pi/2,\pi/2]_\sigma \times \Sph^2_\omega\bigr)\right) / \sim,
\]
where we identify a point in $\ol{\R^4}$ with coordinates $(\rho,v,\omega)$, $(\rho,v)\neq(0,0)$, with the point $(\rho_I=\sqrt{\rho^2+v^2},\,\sigma=\arctan(v/\rho),\,\omega)$. The map
\begin{equation}
\label{EqCptABlowdown}
  \beta \colon M' \to \ol{\R^4},
\end{equation}
equal to the identity map away from $S^+$, and given by $\beta(\rho_I,\sigma,\omega)=(\rho=\rho_I\cos\sigma,\,v=\rho_I\sin\sigma,\,\omega)$, is called the \emph{blow-down map}. Note that $\beta$ is a local diffeomorphism away from $S^+$, but is not injective at the \emph{front face}
\[
  \ff([\ol{\R^4};S^+]) := \rho_I^{-1}(0)
\]
of the blow-up. The point of doing this blow-up is that curves tending to $S^+$ but at different angles $\sigma$ have distinct limiting points on the front face. Concretely, $s=\tan(\sigma)=v/\rho=t-r_*$ is an affine parameter on the fibers $\beta^{-1}(p)$, $p\in S^+$, of the blow-down map, so $\beta^{-1}(S^+)$ is the set of all endpoints of future-directed outgoing radial null-geodesics of mass $m$ Schwarzschild, and radial null-geodesics with different $t-r_*$ are separated all the way up to $\beta^{-1}(S^+)$. It is thus natural to define:
\begin{definition}
\label{DefCptAScri}
  Null infinity $\scri^+$ is defined as the front face of the blowup of $S^+\subset\ol{\R^4}$,
  \[
    \scri^+ := \ff([\ol{\R^4};S^+]).
  \]
\end{definition}

\begin{figure}[!ht]
\includegraphics{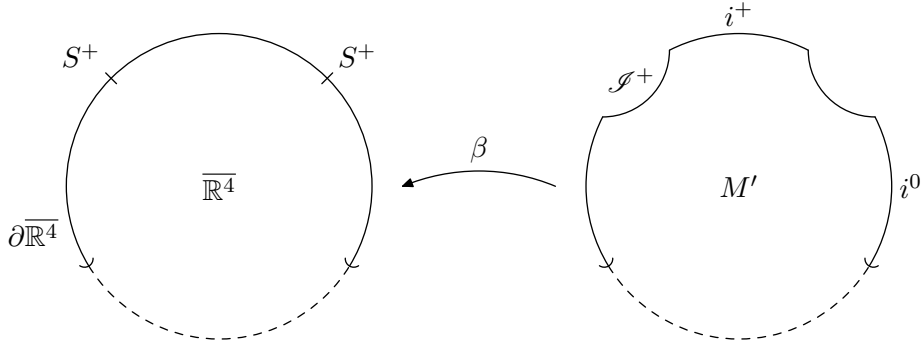}
\caption{\textit{Left:} the partial compactification $\ol{\R^4}$ and its light cone at infinity $S^+$. \textit{Right:} the blow-up $M'=[\ol{\R^4};S^+]\xra{\beta}\ol{\R^4}$, with front face $\scri^+$ (null infinity) and side faces $I^0$ (spatial infinity), $I^+$ (future timelike infinity).}
\label{FigCptABlowup}
\end{figure}

The \emph{side faces} of the blow-up are the connected components of the lift of the original boundary hypersurface $\pa\ol{\R^4}$, i.e.\ of the closure of the preimage of $\pa\ol{\R^4}\setminus S^+$ under $\beta$. In the present situation, there are two side faces:

\begin{definition}
  The \emph{future temporal face} is
  \[
    I^+=\ol{\beta^{-1}\bigl((\pa C_2\cap\pa\ol{\R^4}) \cup \{v>0\}\bigr)},
  \]
whose image $\beta(I^+)$ is a closed $3$-ball with boundary $S^+$. The \emph{spatial face} (more precisely: the part of it that we chose to include in the compactification $\ol{\R^4}$) is defined by
  \[
    I^0 := \ol{\beta^{-1}\bigl(\pa\ol{\R^4}\cap\{v<0\}\bigr)}.
  \]
\end{definition}

Using $\beta$, one can pull back natural vector bundles on $\ol{\R^4}$ to $M'$; for instance, the pullback $\beta^*g_m$, which we simply denote by $g_m$ for brevity, is an element of $\CI(M';\beta^* S^2\,\Tsc^*\ol{\R^4})$ (and constant along the fibers of $\beta$).

Let $\rho_0=r^{-1}$ for $|v+1|\leq\tfrac34$ and $r>R$, $R\gg 1$, and extend it to a smooth positive function on all of $\R^4$. Denote then by $t_\bop$ the smooth function
\begin{equation}
\label{EqCptATimeB}
  t_\bop= \rho_0(t - 2 m \chi_0(r) \log(r-2 m)),
\end{equation}
defined for $|t|/\la r\ra<\half$, where $\chi_0\equiv 0$ for $r<R$ and $\chi_0\equiv 1$ for $r>2 R$; this extends the function $v+1$ smoothly into the interior $\R^4$, and $dt_\bop$ is timelike on
\begin{equation}
\label{EqCptASigma}
  \Sigma:=t_\bop^{-1}(0).
\end{equation}
(The main point of this construction is to write the initial hypersurface $\Sigma$ in a nondegenerate way, i.e.\ as the zero set of a function whose differential does not vanish anywhere on it.) Note that the function $\rho_0$ is, in a neighborhood of $\Sigma$, a boundary defining function of $I^0$; below, we shall use different boundary defining functions adapted to our needs, but keep the same notation. See also Remark~\ref{RmkCptABdfs}.

We restrict our analysis from now on to the following smooth manifold with corners:
\begin{definition}
\label{DefCptAMfd}
  The compact manifold with corners $M$ is defined by
  \[
    M:=M'\cap\bigl(\{t_\bop\geq 0\}\cup\{t>\tfrac13 \la r\ra\}\bigr).
  \]
\end{definition}

One should think of this as (the compactification of) the causal future of $\Sigma$; and this is indeed what it is if we endow $M^\circ$ with the Minkowski metric.

We regard the boundary $\Sigma\subset M$ as `artificial,' i.e.\ incomplete, from the point of view of b-analysis; recall Figure~\ref{FigIBaby}; abusing notation slightly, we shall denote the part $I^0\cap M$ of spatial infinity contained in $M$ again by $I^0$. We denote by $\rho_0$, $\rho_I$, and $\rho_+\in\CI(M)$ defining functions of $I^0$, $\scri^+$, and $I^+$, respectively; we further let $\rho\in\CI(M)$ denote a total boundary defining function, e.g.\ $\rho=\rho_0\rho_I\rho_+$. Defining functions are well-defined up to multiplication by smooth positive functions. We shall often make concrete choices to simplify local calculations; by a \emph{local defining function} of $I^0$, say, on some open subset $U\subset M$ we then mean a function $\rho_0\in\CI(U)$ so that for any $K\Subset U$, $\rho_0|_K$ can be extended to a globally defined defining function of $I^0$. We remark that $\rho_0|_\Sigma\in\CI(\Sigma)$ is a defining function of $\pa\Sigma$ within $\Sigma$.

\begin{rmk}
\label{RmkCptABdfChar}
  The causal character (spacelike, null, timelike) of level sets of $\rho_0$, i.e.\ of $d\rho_0$, depends on the particular choice of $\rho_0$. On the other hand, the \emph{vector field} $\rho_0\pa_{\rho_0}$, defined using any local coordinate system, is well-defined as an element of $\Tb_{I^0}M$, and thus so is its causal character at $I^0$ with respect to the b-metric $\rho^2 g_m$: it is the scaling vector field at infinity, see the discussion after equation~\eqref{EqISysLinOp}, and spacelike away from the corner $I^0\cap\scri^+$. Likewise, $\rho_+\pa_{\rho_+}$ is the scaling vector field at $I^+$, which is timelike.
\end{rmk}

Let us relate $\Sigma$ to the radial compactification $\ol{\R^3}$ of Euclidean $3$-space; recall that the latter is defined using polar coordinates $(r,\omega)\in(0,\infty)\times\Sph^2$ on $\R^3$ as the closed $3$-ball
\[
  \ol{\R^3} := \bigl(\R^3 \sqcup ([0,\infty)_{\rho_0}\times\Sph^2)\bigr) / \sim,\ \ 
  \tn{where}\ (r,\omega) \sim (\rho_0,\omega),\ \rho_0=r^{-1},\ r>0.
\]
Consider the map $\iota\colon\R^3\ni x=(r,\omega)\mapsto(2 m\chi(r)\log(r-2 m),x)\in\Sigma^\circ\subset\R_t\times\R^3_x$, which is the projection along the flow of $\pa_t$. Expressed near $\pa\ol{\R^3}$, i.e.\ for small $\rho_0$, this takes the form $\iota(\rho_0,\omega)=(\rho,v,\omega)$ for $\rho=\rho_0$ and $v=-1$; thus, $\iota$ extends to a diffeomorphism
\begin{equation}
\label{EqCptASigmaR3}
  \Sigma \cong \ol{\R^3}.
\end{equation}

Whenever necessary, we shall make the mass parameter $m$ in these constructions explicit by writing
\begin{equation}
\label{EqCptAExpl}
  {}^m\ol{\R^4},\ {}^m\!M',\ {}^m\!M,\ {}^m\Sigma,\ {}^m t_\bop,\ {}^m I^0,\ {}^m\!\scri^+,\ {}^m I^+,\ {}^m\beta,\ {}^m\!\rho,\ \text{etc.}
\end{equation}
In particular, ${}^0\ol{\R^4}$ is the radial compactification of $\R^4$ with the closed subset $\{|t|^{-1}=0,\ t/r\leq-\tfrac34\}$ of the boundary removed; note here that on their respective domains of definition, $r^{-1}$ and $|t|^{-1}$ are indeed local boundary defining functions of ${}^0\ol{\R^4}$. Moreover, the metric $g_m$ for $m=0$ is equal to the Minkowski metric $\ul g$. We shall explore the relationships between ${}^m\ol{\R^4}$ etc.\ for different values of $m$ in \S\ref{SsCptM}.

\begin{rmk}
\label{RmkCptABdfs}
  For $m=0$, it is easy to write down global expressions for boundary defining functions in $t\geq\half|r|$, for instance (using notation similar to~\cite{LindbladAsymptotics})
  \begin{equation}
  \label{EqCptABdfs}
    {}^0\rho_0 = (1+q_-)^{-1}, \quad
    {}^0\rho_I = t^{-1}(1+q_-)(1+q_+),\quad
    {}^0\rho_+ = (1+q_+)^{-1}, \quad
    {}^0\rho = t^{-1};
  \end{equation}
  here $q_+=\phi_+(t-\la r\ra)$ and $q_-=\phi_+(\la r\ra-t)$, where $\phi_+(x)$ is a smooth function, $\phi_+(x)=x$ for $x\geq 1$, and $\phi_+(x)=0$, $x\leq 0$. One can write down similar expressions for general $m$ by using $r_*$ instead of $r$ near $\scri^+\cup I^0$, and inserting suitable partitions of unity to obtain expressions which are globally smooth. While expressions such as~\eqref{EqCptABdfs} offer a quick way to relate bounds by $({}^0\rho_0)^{a_0}({}^0\rho_I)^{a_I}({}^0\rho_+)^{a_+}$ into bounds in terms of standard coordinates on $\R^4$, they are of course cumbersome to work with if one used them as parts of local coordinate systems on ${}^m\!M$. Furthermore, since we fixed a smooth structure of ${}^m\!M$, boundary defining functions on ${}^m\!M$ are well-defined up to multiplication by smooth, positive functions with smooth, positive reciprocals; therefore, decay rates, such as $a_0,a_I,a_+$ above, with respect to one particular set of choices of boundary defining functions of ${}^m\!M$ are the same as for any other set of choices on the same manifold ${}^m\!M$. The advantage of defining ${}^m\!M$ is then that one can work with any convenient choices of (local) boundary defining functions for any particular local coordinate calculation or estimate for a PDE on ${}^m\!M$, and the decay rates in such an estimate, when expressed in terms of one's chosen defining functions, make invariant sense.
\end{rmk}

Working on ${}^m\ol{\R^4}$, the following coordinates are convenient for performing calculations near the light cone at infinity $S^+$:
\begin{definition}
\label{DefCptADoubleNull}
  We define the coordinates $q=x^0$ and $s=x^1$ as follows:
  \[
    q:=x^0:=t+r_*,\ \ s:=x^1:=t-r_*.
  \]
\end{definition}
Their level sets are null hypersurfaces for the mass $m$ Schwarzschild metric. Using $d q=d s+2 d r_*$ and \eqref{EqCptARstar},
\begin{equation}
\label{EqCptASpl}
  \Tsc^*\ol{\R^4} = \la dq \ra \oplus \la ds \ra \oplus r\,T^*\Sph^2
\end{equation}
therefore defines a smooth partial trivialization near $S^+$; recall that $\rho=r^{-1}$ there. Similarly,
\[
  \pa_0 \equiv \pa_{x^0} = \pa_q = \half(\pa_t+\pa_{r_*}),\ \  \pa_1 \equiv \pa_{x^1} = \pa_s = \half(\pa_t-\pa_{r_*})
\]
are smooth scattering vector fields on $\ol{\R^4}$, and together with $r^{-1} T\Sph^2$, they give a smooth partial trivialization of $\Tsc\ol{\R^4}$ near $S^+$.\footnote{On the other hand, $t^{-1}$ is \emph{not} smooth on ${}^m\ol{\R^4}$ for $m\neq 0$; see Lemma~\ref{LemmaCptFtInv} below.} Letting $x^a$, $a=2,3$, denote local coordinates on $\Sph^2$, we will denote spherical indices by early alphabet Latin letters $a,b,c,d,e$, and general indices ranging from $0$ to $3$ by Greek letters. The components of a section $\omega$ of $\Tsc^*\ol{\R^4}$ in the splitting \eqref{EqCptASpl} are denoted with barred indices:
\begin{equation}
\label{EqCptASplBar}
  \omega_{\bar 0}:=\omega(\pa_0),\ \ 
  \omega_{\bar 1}:=\omega(\pa_1),\ \ 
  \omega_{\bar a}:=\omega(\rho\pa_a)=r^{-1}\omega(\pa_a).
\end{equation}
Thus, the components of a tensor with respect to this splitting have size comparable to the components in the coordinate basis of $T^*\R^4$. The splitting \eqref{EqCptASpl} induces the splitting
\begin{equation}
\label{EqCptASpl2}
\begin{split}
  S^2\,\Tsc^*\ol{\R^4} &= \la dq^2 \ra \oplus \la 2 dq\,ds \ra \oplus (2 dq\otimes_s r\,T^*\Sph^2) \\
    &\qquad \oplus \la ds^2 \ra \oplus (2 ds\otimes_s r\,T^*\Sph^2) \oplus r^2\,S^2 T^*\Sph^2,
\end{split}
\end{equation}
as well as the dual splittings of the dual bundles $\Tsc\ol{\R^4}$ and $S^2\,\Tsc\ol{\R^4}$. We will occasionally use the further splitting
\begin{equation}
\label{EqCptASplS2}
  S^2 T^*\Sph^2=\la\slg\ra \oplus \la\slg\ra^\perp.
\end{equation}

For calculations of geometric quantities associated with the metric, the bundle splittings induced by the coordinates $q,s,x^2,x^3$, i.e.
\begin{equation}
\label{EqCptASplProd}
\begin{split}
  T^*\R^4 &= \la dq\ra \oplus \la ds\ra \oplus T^*\Sph^2, \\
  S^2 T^*\R^4 &= \la dq^2 \ra \oplus \la 2 dq\,ds \ra \oplus (2 dq\otimes_s T^*\Sph^2) \oplus \la ds^2 \ra \oplus (2 ds\otimes_s T^*\Sph^2) \oplus S^2 T^*\Sph^2,
\end{split}
\end{equation}
are more convenient. Components are denoted without bars, that is, for a 1-form $\omega$ and for $\mu=0,1$, we have $\omega_\mu:=\omega(\pa_\mu)=\omega_{\bar\mu}$, while we let $\omega_a:=\omega(\pa_a)=r\omega_{\bar a}$. In short, we have
\begin{equation}
\label{EqCptASplSphWeight}
  \omega_{\bar\mu}=r^{-s(\mu)}\omega_\mu,\quad
  s(\mu_1,\ldots,\mu_N) := \#\{\lambda\colon\mu_\lambda\in\{2,3\} \},
\end{equation}
likewise for tensors of higher rank.

On the resolved space $M$, the null derivatives $\pa_0,\pa_1$ can be computed as follows: near $I^0\cap\scri^+$, we can take
\begin{equation}
\label{EqCptASplCoords}
  \rho_0=-\rho/v=(r_*-t)^{-1},\ \ \rho_I=-v=(r_*-t)/r,\ \ \rho=\rho_0\rho_I=r^{-1};
\end{equation}
then
\begin{equation}
\label{EqCptASplNullExpl}
\begin{split}
  \pa_0 &= -\half \rho_0\rho_I(1-2 m\rho)\rho_I\pa_{\rho_I}, \\
  \pa_1 &= \rho_0\bigl(\rho_0\pa_{\rho_0} - (1-\half\rho_I(1-2 m\rho))\rho_I\pa_{\rho_I}\bigr),
\end{split}
\end{equation}
and dually
\begin{equation}
\label{EqCptASplNullDual}
  \rho\,d q =-\tfrac{2}{1-2 m\rho}\bigl(\tfrac{d\rho_0}{\rho_0}+\tfrac{d\rho_I}{\rho_I}\bigr) + \rho_I\tfrac{d\rho_0}{\rho_0}, \ \ 
  \rho\,d s = \rho_I\tfrac{d\rho_0}{\rho_0}.
\end{equation}
A similar calculation near $I^+\cap\scri^+$ yields
\begin{equation}
\label{EqCptASplNull}
  \pa_0 = f_0\rho_0\rho_I\rho_+\cdot\rho_I\pa_{\rho_I},\ \ \pa_1 \in \rho_0\rho_+\Vb(M),
\end{equation}
for some $f_0\in\CI(M)$, $f_0>0$, depending on the choices of boundary defining functions.

\subsection{Function spaces}
\label{SsCptF}

We first recall the notion of b-Sobolev spaces on $\R^{n,d}_+:=[0,\infty)^d_x\times\R^{n-d}_y$: first, we set $\Hb^0(\R^{n,d}_+)\equiv L^2_\bop(\R^{n,d}_+):=L^2(\R^{n,d}_+;|\frac{d x^1}{x^1}\ldots\frac{d x^d}{x^d}d y|)$; for $k\in\N$ then, $\Hb^k(\R^{n,d}_+)$ consists of all $u\in L^2_\bop$ such that $V_1\ldots V_j u\in L^2_\bop$ for all $0\leq j\leq k$, where each $V_\ell$ is equal to either $x^p\pa_{x^p}$ or $\pa_{y^q}$ for some $p=1,\ldots,d$, $q=1,\ldots,n-d$. For general $s\in\R$, one defines $\Hb^s(\R^{n,d}_+)$ by interpolation and duality. One can define b-Sobolev spaces on compact manifolds with corners by localization and using local coordinate charts; we give an invariant description momentarily. Note that the logarithmic change of coordinates $t^j:=-\log x^j$, $j=1,\ldots,d$, induces an isometric isomorphism $\Hb^s(\R^{n,d}_+)\cong H^s(\R^n)$ with the standard Sobolev space on $\R^n$.

Now on $M'$, fix any smooth b-density, i.e.\ in local coordinates as above a smooth positive multiple of $|\frac{d x^1}{x^1}\ldots\frac{d x^d}{x^d}d y|$, then the space $L^2_\bop(M')$ with respect to this density is well-defined; the space $L^2_\bop(M)$ of restrictions of elements $u\in L^2_\bop(M')$ to $M$ is similarly well-defined, and since $M$ is compact, any two choices of b-densities on $M'$ yield equivalent norms on $L^2_\bop(M)$. More generally, if $b_0,b_I,b_+\in\R$ are weights, we define the weighted $L^2$ space
\[
  \rho_0^{b_0}\rho_I^{b_I}\rho_+^{b_+}\Hb^0(M) \equiv \rho_0^{b_0}\rho_I^{b_I}\rho_+^{b_+}L^2_\bop(M) := \bigl\{ u \colon \rho_0^{-b_0}\rho_I^{-b_I}\rho_+^{-b_+}u\in L^2_\bop(M) \bigr\}.
\]
The b-Sobolev spaces of order $k=0,1,2,\ldots$ are defined using a finite collection of vector fields $\sV\subset\Vb(M')$ such that at each point $p\in M$, the collection $\sV_p$ spans $\Tb_p M$, namely
\[
  \Hb^k(M) := \{ u\in L^2_\bop(M) \colon V_1\dots V_j u \in L^2_\bop(M),\ 0\leq j\leq k,\ V_\ell\in\sV \};
\]
the norm on this space is the sum of the $L^2_\bop(M)$-norms of $u$ and its up to $k$-fold derivatives along elements of $\sV$. One defines $\rho_0^{b_0}\rho_I^{b_I}\rho_+^{b_+}\Hb^k(M)$ and its norm correspondingly. Note that the vector fields in $\sV$ are required to be tangent to $I^0$, $\scri^+$, and $I^+$, but \emph{not} to $\Sigma$; thus, we measure standard Sobolev regularity near $\Sigma$, and b- (conormal) regularity at $I^0$, $\scri^+$, and $I^+$. (Thus, our space $\Hb^k(M)$ would be denoted $\Hbext^k(M)$ in the notation of \cite[Appendix~B]{HormanderAnalysisPDE3}.) Due to the compactness of $M$, any two choices of collections $\sV$ and boundary defining functions $\rho_0,\rho_I,\rho_+$ give rise to the same b-Sobolev space, up to equivalence of norms. (For instance, any other defining function $\rho_0'$ of $I^0$ is related to $\rho_0$ by $\rho_0'=a\rho_0$ where $0<a\in\CI(M)$ (and thus by compactness of $M$, $C^{-1}\leq a\leq C$ for some $C>1$); the equality of the weighted spaces defined using $\rho_0$ or $\rho_0'$ is then a consequence of the fact that multiplication by $a^{b_0}$, or in fact by any smooth nonzero function on $M$ with smooth reciprocal, is an isomorphism on $\Hb^k(M)$.) The space $\Hb^\infty(M)=\bigcap_{k\geq 1}\Hb^k(M)$ and its weighted analogues have natural Fr\'echet space structures; we refer to their elements as \emph{conormal} functions. We shall also use function spaces with infinitely decaying weights, so for instance
\begin{equation}
\label{EqCptFInftyW}
  \rho_I^\infty\Hb^k(M) := \bigcap_{b_I\in\R}\rho_I^{b_I}\Hb^k(M),
\end{equation}
as well as spaces of the form
\[
  \rho_I^{b_I-0}\Hb^k(M) := \bigcap_{\eps>0}\rho_I^{b_I-\eps}\Hb^k(M),
\]
similarly for spaces with more weights.

Weighted b-Sobolev spaces of sections of vector bundles on $M$ are defined using local trivializations. We will in particular use the space
\begin{equation}
\label{EqCptFX}
  \Hb^{k;b_0,b_I,b_+}(E)\equiv\Hb^{k;b_0,b_I,b_+}(M;E) := \rho_0^{b_0}\rho_I^{b_I}\rho_+^{b_+}\Hb^k(M;E),
\end{equation}
with $E$ denoting the trivial bundle $\ul\C:=M\times\C\to M$, or $E=\beta^*\Tsc^*\ol{\R^4}$, or $E=\beta^*S^2\,\Tsc^*\ol{\R^4}$. When the bundle $E$ is clear from the context, we will simply write $\Hb^{k;b_0,b_I,b_+}$. When estimating error terms, we will often use the inclusion
\[
  \CI(\ol{\R^4}) \subset \CI(M) \subset \Hb^{\infty;-0,-0,-0} := \bigcap_{\eps>0}\Hb^{\infty;-\eps,-\eps,-\eps}.
\]

For the last part of Theorem~\ref{ThmIBaby}, we need to define the notion of \emph{polyhomogeneity} (or \emph{$\cE$-smoothness}) and discuss its basic properties; see \cite[\S2A]{MazzeoEdge} and \cite[\S4.15]{MelroseDiffOnMwc} for detailed accounts and proofs. An \emph{index set} is a discrete subset $\cE\subset\C\times\N_0$ such that
\begin{subequations}
\begin{gather}
\label{EqCptFPhg1}
  (z,j)\in\cE\ \Lra\ (z,j')\in\cE\ \ \forall\,j'\leq j; \\
\label{EqCptFPhg2}
  (z_\ell,j_\ell)\in\cE,\ |z_\ell|+j_\ell\to\infty\ \Lra\ \Im z_\ell\to-\infty; \\
\label{EqCptFPhg3}
  (z,j)\in\cE\ \Lra\ (z-i,j)\in\cE.
\end{gather}
\end{subequations}
We shall write
\begin{equation}
\label{EqCptFPhgBound}
  \Im\cE < c\ :\Longleftrightarrow\ \Im z<c\ \ \forall\,(z,k)\in\cE,
\end{equation}
likewise for the nonstrict inequality sign. Note that by condition~\eqref{EqCptFPhg2}, every index set $\cE$ has an upper bound $\Im\cE<C$ for some $C$; more precisely, if $\cE$ is an index set and $C'\in\R$, then there are only finitely many points $(z,k)\in\cE$ with $\Im z>C'$.

Let now $X$ denote a compact manifold with boundary $\pa X$, and let $\rho\in\CI(X)$ be a boundary defining function. The choice of a collar neighborhood $[0,1)_\rho\times\pa X$ makes the vector field $\rho D_\rho=\frac{1}{i}\rho\pa_\rho$ well-defined, and any two choices of collars give the same vector field $\rho D_\rho$ modulo elements of $\rho\Vb(X)$. Let $\cE$ be an index set. The space $\cA_\phg^\cE(X)$ then consists of all $u\in\rho^{-\infty}\Hb^\infty(X)=\bigcup_{N\in\R}\rho^N\Hb^\infty(X)$ for which
\begin{equation}
\label{EqCptFTesting}
  \prod_{\genfrac{}{}{0pt}{}{(z,j)\in\cE}{\Im z\geq -N}} (\rho D_\rho-z)u\in\rho^N\Hb^\infty(X)\ \ \tn{for all}\ N\in\R;
\end{equation}
equivalently, there exist $a_{(z,j)}\in\CI(X)$, $(z,j)\in\cE$, such that
\begin{equation}
\label{EqCptFPhg}
  u - \sum_{\genfrac{}{}{0pt}{}{(z,j)\in\cE}{\Im z\geq -N}} \rho^{i z}(\log\rho)^j a_{(z,j)} \in\rho^N\Hb^\infty(X).
\end{equation}
(Condition~\eqref{EqCptFPhg3} ensures that this is independent of the choice of $\rho D_\rho$.) In particular, $u\in\rho^{-\Im\cE-0}\Hb^\infty(X)$. When no confusion can arise, we write
\begin{equation}
\label{EqCptFPhgShorthand}
  (a,k) := \{(a-i n,j)\colon n\in\N_0,\,0\leq j\leq k\},\ \ 
  a:=(a,0).
\end{equation}
For example, $\cA_\phg^{-i a}(X)=\rho^a\CI(X)$. We also recall the notion of the \emph{extended union} of two index sets $\cE_1$, $\cE_2$, defined by
\[
  \cE_1 \extcup \cE_2 = \cE_1 \cup \cE_2 \cup \{ (z,k) \colon \exists\,(z,j_\ell)\in \cE_\ell,\ k\leq j_1+j_2+1 \},
\]
so e.g.\ $0\extcup 0=(0,1)$, as well as their \emph{sum}
\[
  \cE_1+\cE_2 := \{ (z,j) \colon \exists\,(z_\ell,j_\ell)\in\cE_\ell,\ z=z_1+z_2,\,j=j_1+j_2 \};
\]
thus $\cA_\phg^{\cE_1}(X)\cdot\cA_\phg^{\cE_2}(X)\subset\cA_\phg^{\cE_1+\cE_2}(X)$. For $j\in\N$ and an index set $\cE$, we define
\[
  j\cE_1:=\cE_1+\cdots+\cE_1,
\]
with $j$ summands.

If $X$ is a manifold with corners with embedded boundary hypersurfaces $H_1,\ldots,H_k$ to each of which is associated an index set $\cE_i$, we define $\cA_\phg^{\cE_1,\ldots,\cE_k}(X)$ as the space of all $u\in\rho^{-\infty}\Hb^\infty(X)$, with $\rho\in\CI(X)$ a total boundary defining function, such that for each $1\leq i\leq k$, there exist weights $b_j\in\R$, $j\neq i$, such that, with $\rho_i\in\CI(X)$ denoting a defining function of $H_i$,\footnote{As before, the vector fields $\rho_i D_{\rho_i}$, defined using a collar neighborhood of $H_i$, are in fact well-defined modulo $\rho_i\Vb(X)$, which is all that matters in this definition.}
\[
  \prod_{\genfrac{}{}{0pt}{}{(z,j)\in\cE_i}{\Im z\geq-N}} (\rho_i D_{\rho_i}-z)u \in\rho_i^N\prod_{j\neq i}\rho_j^{b_j} \Hb^\infty(X)\ \ \tn{near}\ H_i.
\]
This is equivalent to $u$ admitting an asymptotic expansion at each $H_i$ as in~\eqref{EqCptFPhg}, with each $a_{(z,j)}$ polyhomogeneous with index set $\cE_j$ at each nonempty boundary hypersurface $H_j\cap H_i$ of $H_i$.

We shall also need spaces encoding polyhomogeneous behavior at one hypersurface but not others; for brevity, we only discuss this in the case of two boundary hypersurfaces $H_1,H_2$: for an index set $\cE$ and $\alpha\in\R$, $\cA^{\cE,\alpha}_{\phg,\bop}$ consists of all $u\in\rho^{-\infty}\Hb^\infty$ such that
\[
  \prod_{\genfrac{}{}{0pt}{}{(z,j)\in\cE_i}{\Im z\geq-N}} (\rho_1 D_{\rho_1}-z)u\in\rho_1^N\rho_2^\alpha\Hb^\infty\ \ \tn{near}\ H_1,\ \ \tn{for all}\ N\in\R;
\]
this is equivalent to $u$ having an expansion at $H_1$ with terms $a_{(z,j)}\in\rho_2^\alpha\Hb^\infty(H_2)$.

We briefly discuss nonlinear properties of b-Sobolev and polyhomogeneous spaces; for brevity, we work on an $n$-dimensional compact manifold $X$ with boundary $\pa X$, and leave the statements of the obvious generalizations to the setting of manifolds with corners to the reader. Thus, if $s>n/2$, then $\Hb^s(X)$ is a Banach algebra, and more generally $u_1\cdot u_2\in\rho^{a_1+a_2}\Hb^s(X)$ if $u_j\in\rho^{a_j}\Hb^s(X)$, $j=1,2$. Regarding the interaction with polyhomogeneous spaces, if $\cE$ is an index set, then $\cA_\phg^\cE(X)\cdot\rho^a\Hb^s(X)\subset\rho^{a-e}\Hb^s(X)$ for all $a,s\in\R$ when $e>\Im\cE$; in the case that $\cE=(a_0,0)\cup\cE'$ with $\Im\cE'<\Im a_0$, we may take $e=\Im a_0$. One can also take inverses, to the effect that $u/(1-v)\in\Hb^s(X)$ provided $u,v\in\Hb^s(X)$, $s>n/2$, and $v\leq C<1$, which follows readily from the corresponding results on $\R^n$, see e.g.\ \cite[\S13.10]{TaylorPDE}, by a logarithmic change of coordinates.

For comparisons with the Minkowski metric, we study the regularity properties of $t^{-1}$ on ${}^m\ol{\R^4}$. Define the index set
\begin{equation}
\label{EqCptIndexSetLog}
  \cE_{\rm log}:=\{(-i k,j)\colon k\in\N_0,\ 0\leq j\leq k\},\ \ 
  \cE'_{\rm log} := \cE_{\rm log}\setminus\{(0,0)\}.
\end{equation}

\begin{lemma}
\label{LemmaCptFtInv}
  Letting $U=\ol{\{t>\tfrac23 r\}}\subset{}^m\ol{\R^4}$, we have
  \begin{equation}
  \label{EqCptFtInvPhg}
    t^{-1} \in \rho\cdot\cA_\phg^{\cE_{\rm log}}(U) \subset \rho\,\CI(U)+\rho^{2-0}\Hb^\infty(U)\subset\rho^{1-0}\Hb^\infty(U),
  \end{equation}
  and $t^{-1}/\rho\in\CI(U\cap\pa\ol{\R^4})$ is everywhere nonzero.
\end{lemma}

\begin{definition}
\label{DefCptFtInv}
  We define $\rho_t\in\CI({}^m\ol{\R^4})$ to be any boundary defining function satisfying $\rho_t/\rho=t^{-1}/\rho$ at $U\cap\pa{}^m\ol{\R^4}$.
\end{definition}

By Lemma~\ref{LemmaCptFtInv}, this fixes $\rho_t$ in $U$ modulo $\rho^2\CI({}^m\ol{\R^4})$; away from $U$, $\rho_t$ is merely well-defined modulo $\rho\,\CI({}^m\ol{\R^4})$.

\begin{proof}[Proof of Lemma~\usref{LemmaCptFtInv}]
  Using the notation of~\S\ref{SsCptA}, we have $t^{-1}\in\CI(C_2)$. Thus, it suffices to work in $C_1\cap\{v>-\half\}$, where we can take $\rho=r^{-1}$; we then need to prove $f:=\rho t\in\cA_\phg^{\cE_{\rm log}}$ and $f|_{\pa\ol{\R^4}}\neq 0$ there, which implies the claim about $t^{-1}/\rho=1/f$ as $\cA_\phg^{\cE_{\rm log}}$ is closed under multiplication. Note that $f\in\CI(\R^4)$, and $f>\tfrac14$. Let $\tilde\chi(x)=\chi(x^{-1})\in\CI((0,\infty);[0,1])$ in the notation \eqref{EqCptACp1}, so $\tilde\chi(x)\equiv 0$, $x<\tfrac13$, and $\tilde\chi(x)\equiv 1$, $x>\half$, then
  \begin{equation}
  \label{EqCptFtInvRecursion}
    f = 1 + v - 2 m\rho \tilde\chi(f) \bigl(\log\rho-\log(1-2 m\rho)\bigr).
  \end{equation}
  Note that near $\rho=0$, $f=\rho^{-1}t^{-1}$ is the \emph{unique} positive function satisfying this equation: indeed, if $f'$ is another such function, then $|f-f'|\lesssim(\rho\log\rho)|f-f'|$. At $\rho=0$, we have $f=1+v$. Thus, let $k\geq 2$ be an integer, and consider the map
  \[
    T \colon \tilde f \mapsto - 2 m\rho(\log\rho-\log(1-2 m\rho))\tilde\chi(1+v+\tilde f)
  \]
  on $\rho^{1-\delta}\Hb^k([0,\eps_k)_\rho\times(-1/2,5)_v)$, where $\delta\in(0,1)$ is fixed. Now
  \[
    \|T(\tilde f)-T(\tilde f')\|_{\rho^{1-\delta}\Hb^k}\leq C_k\|\rho\log\rho-\rho\log(1-2 m\rho)\|_{\Hb^k} \|\tilde\chi\|_{\cC^k}\|\tilde f-\tilde f'\|_{\rho^{1-\delta}\Hb^k};
  \]
  choosing $\eps_k>0$ sufficiently small, the first norm on the right can be made arbitrarily small. By the contraction mapping principle, this gives $f-1-v\in\rho^{1-\delta}\Hb^\infty$ since $k$ was arbitrary. We can now improve the remainder term by plugging this into \eqref{EqCptFtInvRecursion}, which gives
  \[
    f - \bigl(1 + v - 2 m\tilde\chi(1+v)\bigl(\rho\log\rho-\rho\log(1-2 m\rho)\bigr)\bigr) \in \rho^{2-\delta}\Hb^\infty,
  \]
  so $f\in\cA_\phg^{\cE_{\rm log}}+\rho^{2-\delta}\Hb^\infty$. Using that $\chi\circ(\cdot)$ maps $\cA_\phg^{\cE_{\rm log}}$ into itself, as follows from the testing definition~\eqref{EqCptFTesting}, the desired conclusion follows from an iterative argument.
\end{proof}

\subsection{Relationships between different compactifications}
\label{SsCptM}

The only difference between the compactifications ${}^m\ol{\R^4}$ for different values of $m$ is the manner in which a smooth collar neighborhood of $\pa{}^m\ol{\R^4}$ is glued together with $\R^4$. Since this difference is small due to the logarithmic correction in~\eqref{EqCptACp1} being only of size $r^{-1}\log r$, different compactifications are closely related; see also \cite[\S7]{BaskinVasyWunschRadMink2}. Indeed:

\begin{lemma}
\label{LemmaCptAComp}
  The identity map $\R^4\to\R^4$ induces a homeomorphism $\phi\colon{}^m\ol{\R^4}\to{}^0\ol{\R^4}$, which in fact is a polyhomogeneous diffeomorphism with index set $\cE_{\rm log}$; that is, in smooth local coordinate systems near $\pa{}^m\ol{\R^4}$ and $\pa{}^0\ol{\R^4}$, the components of both $\phi$ and $\phi^{-1}$ are real-valued functions on $[0,\infty)\times\R^3$ of class $\cA_\phg^{\cE_{\rm log}}$. Moreover, $\phi$ induces a smooth diffeomorphism $\pa{}^m\ol{\R^4}\cong\pa{}^0\ol{\R^4}$, which restricts to ${}^m\beta({}^m I^+)\cong{}^0\beta({}^0 I^+)$, and also induces a smooth diffeomorphism ${}^m I^+\cong{}^0 I^+$.
\end{lemma}
\begin{proof}
  We have $\cA_\phg^{\cE_{\rm log}}\subset\CI+\rho^{1-0}\Hb^\infty\subset\cC^0$, so it suffices to prove the polyhomogeneity statement. Defining the smooth coordinates $\rho$ and $v$ as in~\eqref{EqCptACp1}, and the corresponding smooth coordinates ${}^0\rho=r^{-1}$ and ${}^0 v=r^{-1}(t-r)$ on ${}^0\ol{\R^4}$, we then observe that ${}^0\rho=\rho$, while in the notation of equation~\eqref{EqCptFtInvRecursion}, we established that $1+{}^0 v=f\in\cA_\phg^{\cE_{\rm log}}$ on ${}^m\ol{\R^4}$, giving the desired conclusion for $\phi$. For $\phi^{-1}$, we write $v={}^0 v-r^{-1}\chi(t/r)2 m\log(r-2 m)$ and note that $t/r\in\CI({}^0\ol{\R^4})$. For the last claim, we observe that
  \begin{equation}
  \label{EqCptACompBdy}
    v={}^0 v\ \ \tn{at}\ \pa{}^m\ol{\R^4}
  \end{equation}
  under the identification with $\pa{}^0\ol{\R^4}$ given by $\phi$. This also shows that the sets ${}^m\beta({}^m I^+)=\{v\geq 0\}$ and ${}^0\beta({}^0 I^+)=\{{}^0 v\geq 0\}$ are diffeomorphic. On ${}^m\!M$, resp.\ ${}^0\!M$ then, $v$, resp.\ ${}^0 v$, are local defining functions of the boundaries $\pa{}^m I^+$, resp.\ $\pa{}^0 I^+$, hence by~\eqref{EqCptACompBdy}, the identification ${}^m I^+\cong{}^0 I^+$ in the interior of ${}^m I^+$ indeed extends smoothly to its boundary.
\end{proof}

In a similar vein, the scattering (co)tangent bundles can be naturally identified over the boundary:
\begin{lemma}
\label{LemmaCptACompBundle}
  The identity map $T^*\R^4\to T^*\R^4$ extends by continuity to a continuous bundle map $\Tsc^*\,{}^m\ol{\R^4}\to\Tsc^*\,{}^0\ol{\R^4}$ which restricts to a smooth bundle isomorphism over the boundary.
\end{lemma}
\begin{proof}
  Since away from $r=0$, $\la d(r^{-1})\ra$ and $r\,T^*\Sph^2$ are smooth subbundles of $\Tsc^*\,{}^m\ol{\R^4}$ for any $m$, it suffices to show that $d(t^{-1})$, which is a smooth section of $\Tsc^*\,{}^0\ol{\R^4}$, extends by continuity from $\R^4$ to $\pa{}^m\ol{\R^4}$ and restricts to a smooth section of $\Tsc^*_{\pa{}^m\ol{\R^4}}{}^m\ol{\R^4}$. By Lemma~\ref{LemmaCptFtInv}, we have $t=\rho^{-1}f$, $f\in\cA_\phg^{\cE_{\rm log}}$, so $d t=f\,d(\rho^{-1})+\rho^{-1}d f$; but $f|_{\pa{}^m\ol{\R^4}}$ is smooth indeed, while in a local product neighborhood $[0,1)_\rho\times\R_X^3$ of a point in $\pa{}^m\ol{\R^4}$, $\rho^{-1}d f=(\rho\pa_\rho f)\frac{d\rho}{\rho^2}+(\pa_X f)\frac{d X}{\rho}$ restricts to the smooth scattering 1-form $(\pa_X f)\frac{d X}{\rho}$ on $\pa{}^m\ol{\R^4}$.
\end{proof}

Let us discuss this on the level of function spaces. The map $\phi$ in Lemma~\ref{LemmaCptAComp} induces $\CI({}^m\ol{\R^4})\subset\cA_\phg^{\cE_{\rm log}}({}^0\ol{\R^4})$ and vice versa. Moreover, it induces an isomorphism
\begin{equation}
\label{EqCptACompHb}
  ({}^m\rho)^\alpha \Hbloc^s({}^m\ol{\R^4})\cong({}^0\rho)^\alpha\Hbloc^s({}^0\ol{\R^4}),\ \ s,\alpha\in\R,
\end{equation}
as follows from $\phi\in\cA_\phg^{\cE_{\rm log}}$. The corresponding statement is not quite true on the blown-up spaces ${}^m\!M$, the failure happening at ${}^m\!\scri^+$; there, let us use
\[
  {}^m\rho={}^0\rho=r^{-1}; \quad
  {}^m v={}^0 v-2 m({}^0\rho)\log(({}^0\rho)^{-1}-2 m),\ 
  {}^0 v=r^{-1}(t-r).
\]
Now, the b-tangent bundle on ${}^0\!M$ is spanned near ${}^0\!\scri^+$ by spherical derivatives,
\[
  {}^0\rho\pa_{{}^0\rho} \in {}^m\rho\pa_{{}^m\rho} + \cA_\phg^{\cE'_{\rm log}}\cdot\pa_{{}^m v},\ \ 
  {}^0 v\pa_{{}^0 v} \in \bigl({}^m v+\cA_\phg^{\cE'_{\rm log}}\bigr)\pa_{{}^m v},
\]
and ${}^0\rho\pa_{{}^0 v}={}^m\rho\pa_{{}^m v}$; due to the logarithmic loss at $\scri^+$, we thus only have
\[
  ({}^m\rho_0)^{b_0}({}^m\rho_I)^{b_I}({}^m\rho_+)^{b_+}\Hbloc^s({}^m\!M)
  \subset ({}^0\rho_0)^{b_0}({}^0\rho_I)^{b_I-\eps}({}^0\rho_+)^{b_+}\Hbloc^s({}^0\!M)
\]
for all $\eps>0$, but the inclusion fails for $\eps=0$. That is, conormal function spaces are the same on ${}^m\!M$ and ${}^0\!M$ up to an arbitrarily small loss in the weight at $\scri^+$.

Polyhomogeneous spaces on ${}^m\ol{\R^4}$ for different values of $m$ are related in a simple manner: if $\cE\subset\C\times\N_0$ is an index set and $\cE_{\rm log}$ is given by~\eqref{EqCptIndexSetLog}, then $\phi$ induces inclusions
\begin{equation}
\label{EqCptACompPhg}
  \cA_\phg^\cE({}^m\ol{\R^4}) \hra \cA_\phg^{\cE+\cE_{\rm log}}({}^0\ol{\R^4}),\ \ 
  \cA_\phg^\cE({}^0\ol{\R^4}) \hra \cA_\phg^{\cE+\cE_{\rm log}}({}^m\ol{\R^4});
\end{equation}
this is only nontrivial where the two compactifications differ, i.e.\ away from $r=0$, i.e.\ where we can use $r^{-1}$ as a boundary function for both ${}^0\ol{\R^4}$ and ${}^m\ol{\R^4}$. Considering a single term $r^{-i z}(\log r)^k f({}^m v,\omega)$, with $\omega\in\Sph^2$ and $f$ smooth, in the expansion of an element of $\cA_\phg^\cE({}^m\ol{\R^4})$, the first inclusion in~\eqref{EqCptACompPhg} follows from $f\circ\phi\in\cA_\phg^{\cE_{\rm log}}({}^0\ol{\R^4})$, which in turn can be seen by Taylor expanding $f({}^0 v-2 m({}^0\rho)\log(({}^0\rho)^{-1}-2 m),\omega)$ in the first argument around ${}^0 v$. The proof of the second inclusion is similar. See \cite[Proposition~7.8]{BaskinVasyWunschRadMink2} for an alternative argument.

Polyhomogeneity on different spaces ${}^m\!M$ on the other hand is much less well-behaved: for instance, a function $u\in\CI({}^m\!M)$ compactly supported near a point in $({}^m\!\scri^+)^\circ$, $m>0$, so $u\in\cA_\phg^{\emptyset,0,\emptyset}({}^m\!M)$, is not polyhomogeneous on ${}^0\!M$: it vanishes near $({}^0\!\scri^+)^\circ$ and $({}^0 I^+)^\circ$, but is nontrivial at the corner ${}^0\!\scri^+\cap {}^0 I^+$.

\subsection{Bundles and connections near null infinity}
\label{SsCptScri}

In the energy estimate~\eqref{EqISysLinEqIEst} for the toy problem~\eqref{EqISysLinEqI}, derivatives of $u$ along vector fields tangent to the fibers of $\beta\colon\scri^+\to S^+$ are better controlled than general b-derivatives. In this section, we introduce analytic structures on the blow-up $M$ of $\ol{\R^4}$ capturing this in an invariant manner.

\begin{definition}
\label{DefCptScriModule}
  For vector bundles $E_j\to\ol{\R^4}$, $j=1,2$, let
  \[
    \cM_{\beta^*E_1,\beta^*E_2} \subset \Diffb^1(M;\beta^*E_1,\beta^*E_2)
  \]
  denote the $\CI(M)$-module of all first order b-differential operators $A$ which satisfy the following condition near $\scri^+$: if $E_j\cong\cU\times\C^{k_j}$, $j=1,2$, is a local trivialization of $E_j$, with $\cU\subset\ol{\R^4}$ a neighborhood of $S^+$, see \eqref{EqCptASplus}, and we pull these trivializations back to $\beta^*E_j\cong\beta^{-1}(\cU)\times\C^{k_j}$, then $A=V+f$, where $V$ is a $k_2\times k_1$ matrix of vector fields $V_{i j}\in\Vb(M)$ which are tangent to the fibers of $\beta$, and $f\in\CI(M)^{k_2\times k_1}$. Let moreover
  \[
    {}^0\cM_{\beta^*E_1,\beta^*E_2}\subset\cM_{\beta^*E_1,\beta^*E_2}
  \]
  denote the submodule for which $f|_{\scri^+}=0$

  For a single vector bundle $E\to\ol{\R^4}$, we write ${}^{(0)}\cM_{\beta^*E}:={}^{(0)}\cM_{\beta^*E,\beta^*E}$. Whenever the bundle $E$ is clear from the context, we shall simply write ${}^{(0)}\cM:={}^{(0)}\cM_{\beta^*E}$. For $k\in\N$, we write $\cM^k\subset\Diffb^k$ for sums of $k$-fold products of elements of $\cM$.
\end{definition}

It is easy to check that the definition of $\cM_{\beta^*E_1,\beta^*E_2}$ is independent of the choice of local trivializations; for ${}^0\cM$, this is true as well, since vector fields tangent to the fibers of $\beta$ annihilate the matrices for changes of frames of $E_1$ and $E_2$ which lift to be \emph{constant} along the fibers of $\beta$. We make some elementary observations:

\begin{lemma}
\label{LemmaCptScriModule}
  We have:
  \begin{enumerate}
  \item\label{ItCptScriModuleIncl} $\rho_I\Diffb^1(M;\beta^*E)\subset{}^0\cM_{\beta^*E}\subset\cM_{\beta^*E}$;
  \item\label{ItCptScriModuleComm} if $A,B\in\cM_{\beta^*E}$, and $A$ has a scalar principal symbol, then $[A,B]\in\cM_{\beta^*E}$. Strengthening the assumption to $A,B\in{}^0\cM_{\beta^*E}$, we have $[A,B]\in{}^0\cM_{\beta^*E}$;
  \item\label{ItCptScriModuleBdl} there is a well-defined map
     \[
       \cM_{\ul\C}\ni A\mapsto A\otimes\Id\in{}^0\cM_{\beta^*E}/\rho_I\,\CI(M;\End(\beta^*E)).
     \]
  \end{enumerate}
\end{lemma}
\begin{proof}
  \eqref{ItCptScriModuleIncl} and \eqref{ItCptScriModuleComm} are clear from the definition. The map in \eqref{ItCptScriModuleBdl} is given in a local trivialization $E\cong\cU\times\C^k$ of $E$ near $S^+$ as $A\cdot\Id_{k\times k}\in\Diffb^1(M)^{k\times k}$; the transition function between two different trivializations is given by $C\in\CI(\cU;\C^{k\times k})$, which pulls back to $M$ to be constant along the fibers of $\beta$; but then $C^{-1}(A\cdot\Id_{k\times k})C-(A\cdot\Id_{k\times k})=C^{-1}A(C)\in\CI(M;\C^{k\times k})$, with $A$ acting component-wise, vanishes on $\scri^+$ by definition of $\cM_{\ul\C}$.
\end{proof}

In local coordinates $[0,\eps_0)_{\rho_0}\times[0,\eps_0)_{\rho_I}\times\R_{x^2 x^3}^2$ near $I^0\cap\scri^+$ as in~\eqref{EqISysLinCoords0I}, with $\R^2$ a local coordinate patch on $\Sph^2$, elements of $\cM_{\ul\C}$ are linear combinations of $\rho_0\pa_{\rho_0}$, $\rho_I\pa_{\rho_I}$, and $\rho_I\pa_{x^a}$, $a=2,3$, plus smooth functions. We thus see that ${}^{(0)}\cM_{\ul\C}$ is generated over $\CI(M)$ by $(\rho_I)\CI(M)$ and lifts of elements $V\in\Vb(\ol{\R^4})$ which vanish at $S^+$ as incomplete vector fields, i.e.\ $V|_{S^+}=0\in T_{S^+}\ol{\R^4}$. (This should be compared to the larger space $\Vb(M)$, which is generated by lifts of elements $V\in\Vb(\ol{\R^4})$ which are merely \emph{tangent} to $S_+$.) Note that by \eqref{EqCptASplNull}, we have
\begin{equation}
\label{EqCptScriNull}
  \rho^{-1}\pa_0,\ \rho_0^{-1}\rho_+^{-1}\pa_1 \in {}^0\cM_{\ul\C};
\end{equation}
for a fixed choice of $\rho$, the operators $\rho^{-1}\pa_0$ and $\pa_1$ acting on sections of any bundle $\beta^*E$ are therefore well-defined, modulo $\rho_I\,\CI$ and $\rho_0\rho_I\rho_+\CI$ valued in $\End(\beta^*E)$, respectively.

The modules defined above are closely related to a natural subbundle of $\Tb_{\scri^+}M$:

\begin{definition}
\label{DefCptScriSubb}
  Denote by
  \[
    \Tbeta_{\scri^+}M\subset\Tb_{\scri^+}M
  \]
  the rank $2$ subbundle generated by all $V\in\Tb_{\scri^+}M$ which are tangent to the fibers of $\beta$, see \eqref{EqCptABlowdown}, and let $\Tbeta M$ be any smooth rank $2$ extension of $\Tbeta_{\scri^+}M$ to a neighborhood of $\scri^+$. Let then
  \[
    (\Tbeta M)^\perp := \{ \alpha\in\Tb^* M \colon \alpha(V)=0\ \tn{for all}\ V\in\Tbeta M \} \subset \Tb^*M
  \]
  denote the annihilator of $\Tbeta M$ in $\Tb^*M$.
\end{definition}

Near $I^0\cap\scri^+$, we can for instance take $\Tbeta M\subset\Tb M$ to be the subbundle whose fibers are spanned by $\rho_I\pa_{\rho_I}$ and $\rho_0\pa_{\rho_0}$.

\begin{rmk}
  Another equivalent characterization of $\cM$ is that the principal symbols of its elements vanish on $(\Tbeta_{\scri^+}M)^\perp$. We also note that for $p\in\scri^+$, there is a natural isomorphism
  \begin{equation}
  \label{EqCptScriSubbIso}
    (\Tbeta M)^\perp_p \cong T^*_{\beta(p)}S^+.
  \end{equation}
  Indeed, given $V\in\Tb_p M$, note that $\beta_*V\in\Tb_{S^+}\ol{\R^4}$ is tangent to $S^+$, hence has a well-defined image in $T_p S^+$; and $V\in\Tbeta_p M$ is precisely the condition that this image be $0$. Thus, the isomorphism \eqref{EqCptScriSubbIso} is obtained by mapping $\eta\in T^*_{\beta(p)}S^+$ to $\Tb_p M\ni V\mapsto\eta(\beta_*V)$.
\end{rmk}

Using this subbundle, we have
\[
  \cM_{\ul\C} = \CI(M;\Tbeta M+\rho_I\Tb M) + \CI(M) \subset \Diffb^1(M),
\]
where we write
\begin{equation}
\label{EqCptScriSections}
  \CI(M;\Tbeta M+\rho_I\Tb M) := \CI(M;\Tbeta M) + \rho_I\,\CI(M;\Tb M).
\end{equation}
Note here that the sum of the first two spaces on the right is globally well-defined on $M$ even though we only defined $\Tbeta M$ in a neighborhood of $\scri^+$: this is due to $\Tbeta M\subset\Tb M$. The general modules $\cM_{\beta^*E_1,\beta^*E_2}$ have a completely analogous description obtained by tensoring the bundles with $\Hom(\beta^*E_1,\beta^*E_2)$.

We next prove some lemmas allowing us to phrase energy estimates for bundle-valued waves invariantly.
\begin{lemma}
\label{LemmaCptScriConn}
  Let $E\to\ol{\R^4}$ be a vector bundle, and let $d^E\in\Diff^1(\ol{\R^4};E,T^*\ol{\R^4}\otimes E)$ be a connection. Then $d^E$ induces a b-connection, i.e.\ a differential operator
  \begin{equation}
  \label{EqCptScriConnDiffb}
    d^E \in \Diffb^1(M;\beta^*E,\Tb^*M \otimes \beta^*E),
  \end{equation}
  on $\beta^*E\to M$. If $\tilde d^E$ is another connection on $E$, then, with notation analogous to \eqref{EqCptScriSections},
  \begin{equation}
  \label{EqCptScriConnDiff}
    d^E-\tilde d^E \in \CI\bigl(M;((\Tbeta M)^\perp+\rho_I\Tb^*M)\otimes\End(\beta^*E)\bigr).
  \end{equation}
\end{lemma}
\begin{proof}
  Fix a local frame $e^i$ of $E$, then for $u_i\in\CIdot(M)\subset\CIdot(\ol{\R^4})$, we have
  \[
    d^E(u_i e^i) = d u_i \otimes e^i + u_i\,d^E e^i.
  \]
  Now the map $u_i\mapsto d u_i$ extends to $M$ as the map $u_i\mapsto \bdiff u_i$, with $\bdiff\in\Diffb^1(M;\ul\C,\Tb^*M)$; and $f^i:=d^E e^i\in\CI(\ol{\R^4};T^*\ol{\R^4}\otimes E)$ canonically induces $\beta^*f^i\in\CI(M;\Tb^*M\otimes\beta^*E)$ by $\beta^*f^i(V)=f^i(\beta_*V)$, $V\in\Tb M$. Therefore, the expression $d^E(u_i\cdot\beta^*e^i)=\bdiff u_i\otimes \beta^*e^i + u_i\cdot\beta^*f^i$ proves \eqref{EqCptScriConnDiffb}.
  
  Letting $\tilde f^i:=\tilde d^E e^i$, we have $(d^E-\tilde d^E)(u_i\cdot\beta^*e^i)=u_i\cdot(\beta^*f^i-\beta^*\tilde f^i)$. But $\Tbeta_{\scri^+}M\subset\ker\beta_*$, so the bundle map $d^E-\tilde d^E$ annihilates $\Tbeta M$ at $\scri^+$, giving \eqref{EqCptScriConnDiff}.
\end{proof}

\begin{lemma}
\label{LemmaCptScriConnDKD}
  In the notation of Lemma~\usref{LemmaCptScriConn}, suppose $E$ is equipped with a fiber metric $\la\cdot,\cdot\ra_E$, and let
  \begin{equation}
  \label{EqCptScriConnDKDK}
    K \in \CI\bigl(M;(S^2\,\Tbeta M+\rho_I\,S^2\,\Tb M)\otimes\End(\beta^*E)\bigr).
  \end{equation}
  Moreover, let $B\in\CI(M;\Hom(\Tb M,\Tb^*M))$ denote a fiber metric on $\Tb M$. Then, acting on sections of $\beta^*E$, we have
  \begin{equation}
  \label{EqCptScriConnDKDKDiff}
    (d^E)^* B K d^E - (\tilde d^E)^*B K \tilde d^E \in \rho_I\Diffb^1(M;\beta^*E),
  \end{equation}
  where we take adjoints with respect to the fiber metrics on $\Tb M$ and $E$, and any fixed b-density on $M$. Moreover, if $(d^E)^\dag$ denotes the adjoint with respect to another fiber metric on $E$, then $(d^E)^\dag B K d^E-(d^E)^*B K d^E\in\rho_I\Diffb^1(M;\beta^*E)$.
\end{lemma}

Note that for $K$ as in \eqref{EqCptScriConnDKDK} with both the $S^2\,\Tbeta M$ and the $S^2\,\Tb M$ summands positive definite, and adding weights, the pairing $\la (d^E)^*B K d^E u,u\ra$ provides the control on fiber-tangential derivatives of $u$ as in the toy model~\eqref{EqISysLinEqIEst}, but is weaker by $\rho_I^{1/2}$ for general b-derivatives; we will take care of this in Definition~\ref{DefBgHscri}. The space in \eqref{EqCptScriConnDKDKDiff} will be weak enough to be treated as an error term (similar to the $\Diffb$ spaces arising as error terms in Lemma~\ref{LemmaEinNscri} below).

\begin{proof}[Proof of Lemma~\usref{LemmaCptScriConnDKD}]
  We write the left hand side of \eqref{EqCptScriConnDKDKDiff} as
  \[
    (d^E)^* B K (d^E-\tilde d^E) + (d^E-\tilde d^E)^* B K \tilde d^E,
  \]
  with one summand being the adjoint of the other. Now, $(d^E)^*B\in\Diffb^1(M;\Tb M\otimes\beta^*E,\allowbreak\beta^*E)$, while Lemma~\ref{LemmaCptScriConn} implies
  \[
    K(d^E-\tilde d^E) \in \rho_I\,\CI(M;\Tb M\otimes\End(\beta^*E)).
  \]
  This proves \eqref{EqCptScriConnDKDKDiff}. (Alternatively, one can analyze the second summand directly, using that over $p\in M$, $\la(d^E-\tilde d^E)^*(B(V)\otimes e),e'\ra_E = \la e, (d^E-\tilde d^E)(V\otimes e')\ra_E$ for $V\in\Tb_p M$, $e,e'\in E_{\beta(p)}$.) For the second part, note that the two adjoints are related via $(d^E)^\dag=C^{-1}(d^E)^*C$ for some $C\in\CI(\ol{\R^4};\End(E))$, hence $\tilde d^E:=(d^E)^{\dag *}=d^E+C^*[d^E,(C^{-1})^*]$ is a connection on $E$, and therefore
  \[
    ((d^E)^\dag-(d^E)^*)B K d^E=(\tilde d^E-d^E)^*B K d^E \in \rho_I\Diffb^1(M;\beta^*E)
  \]
  by what we already proved.
\end{proof}

\begin{lemma}
\label{LemmaCptScriWWstar}
  Equip $E\to\ol{\R^4}$ with a fiber metric and fix a b-density on $\ol{\R^4}$. Then for principally scalar $W\in{}^0\cM_{\beta^*E}$, with principal symbol equal to that of the real vector field $W_1\in\Vb(M)$, we have $W+W^*\in-\dv W_1+\rho_I\,\CI(M;\End(\beta^*E))$.
\end{lemma}
\begin{proof}
  In a local trivialization on $E$, we have $W=W_1\otimes 1+W_0$, $W_0\in\rho_I\,\CI(M;\End(\beta^*E))$, while the fiber inner product $k$ on $E$ is related to the standard Euclidean fiber inner product $\ul k$ in the trivialization by $k(e,e')=\ul k(\wt C e,\wt C e')$ for some $\wt C$ smooth on $\ol{\R^4}$, hence fiber constant on $M$. Denoting adjoints with respect to $\ul k$ by $\dag$, and letting $C:=\wt C^*\wt C$, we thus have
  \begin{align*}
    W+W^*&=\bigl(W_1\otimes 1+C^{-1}(W_1^\dag\otimes 1)C\bigr) + (W_0+W_0^*) \\
      &\in -(\dv W_1)\otimes 1 + C^{-1}[W_1^\dag\otimes 1,C] + \rho_I\,\CI,
  \end{align*}
  with the second term also lying in $\rho_I\,\CI$ since $C$ is fiber-constant.
\end{proof}

\section{Gauge-fixed Einstein equation}
\label{SEin}

As motivated in \S\ref{SsIEin}, we work in the wave map gauge with respect to the background metric $g_m$ constructed in~\S\ref{SsCptA}, since we expect the solution $g$ of the initial value problem \eqref{EqIEinIVP} for the Einstein vacuum equation with initial data asymptotic to mass $m$ Schwarzschild to be well-behaved on the space ${}^m\!M$. The gauge condition reads
\begin{equation}
\label{EqEinUps}
  \Ups(g;g_m)_\mu := (g g_m^{-1}\delta_g G_g g_m)_\mu = g_{\mu\nu}g^{\kappa\lambda}(\Gamma(g)_{\kappa\lambda}^\nu - \Gamma(g_m)_{\kappa\lambda}^\nu) = 0,
\end{equation}
where we recall the notation $G_g=1-\half g\tr_g$, and $(\delta_g u)_\mu=-u_{\mu\nu;}{}^\nu$. For brevity, we shall write
\[
  \Ups(g) \equiv \Ups(g;g_m),
\]
when the background metric $g_m$ is clear from the context. A simple calculation shows that if $h\in\Hb^{\infty;-\eps,-\eps,-\eps}({}^m\!M)$, $\eps>0$ small, is a metric perturbation, and $g=g_m+\rho h$, then the gauge condition $\Ups(g;g_m)=0$ implies that the $\pa_1$-derivatives of the \emph{good components} $h_{0 0}$, $h_{0\bar b}$, and $\sltr h:=\slg^{a b}h_{\bar a\bar b}$ decay towards $\scri^+$. (See equation~\eqref{EqCoPertUpsLower} for this calculation for $h$ with special structure.) A key ingredient of our iteration scheme is therefore constraint damping, which ensures that the gauge condition, or, more directly, the improved decay of the good components at $\scri^+$, is satisfied to leading order for each iterate $h$. We implement constraint damping by considering the gauge-fixed Einstein operator
\begin{equation}
\label{EqEin}
  P(h) := \rho^{-3}P_0(g_m + \rho h),\ \ \ P_0(g) := \Ric(g) - \tdel^*\Ups(g;g_m),
\end{equation}
where on 1-forms $u$
\begin{equation}
\label{EqEinTdel}
  \tdel^*u = \delta_{g_m}^*u - 2\gamma \tfrac{d\rho_t}{\rho_t}\otimes_s u + \gamma (\iota_{\rho_t^{-1}\nabla^{g_m}\rho_t}u)g_m
\end{equation}
is a modification of the symmetric gradient $\delta_{g_m}^*$ by a $0$-th order term; here $\rho_t$ is fixed according to Definition~\ref{DefCptFtInv}. We discuss the effect of this modification in~\S\ref{SsEinN}, see in particular~\eqref{EqEinNscriPi0}. \emph{From now on, the mass parameter $m$ will be fixed and dropped from the notation whenever convenient.}

\subsection{Form of metric perturbations}
\label{SsEinF}

One can easily establish the existence of a solution of \eqref{EqIEinIVP} near $I^0\setminus(I^0\cap\scri^+)$ for normalized initial data (see Theorem~\ref{ThmIDetail}) which lie merely in $\rho_0^{1/2+0}\Hb^\infty$; this is due to nonlinear interactions being weak at $I^0$, which in turn can ultimately be traced back to the null derivatives \eqref{EqCptASplNull} coming with extra factors of $\rho_0$.\footnote{This is related to the solvability of semilinear equations with initial data or forcing terms which are mildly growing at spatial infinity, see \cite[Theorem~5.14]{HintzVasySemilinear}, where one can take the weight $l<-1/2$ in certain circumstances. This is also the level of decay for which Bieri \cite{BieriZipserStability} establishes the \emph{global} stability of Minkowski space.} However, we will use (and prove) the existence of leading terms of the perturbation $h$ of $g=g_m+\rho h$ at $\scri^+$; as discussed around \eqref{EqISysLinEqI}, this requires the initial data to be \emph{decaying} to mass $m$ Schwarzschild data. At $I^+$ however, weak control, i.e.\ $h\in\rho_+^{-1/2+0}\Hb^\infty$ away from $\scri^+$, suffices due to the nonlinear interactions being as weak there as they are at $I^0$. (The decay of our initial data does imply the existence of a leading term at $I^+$, see \S\ref{SPhg}.) Motivated by this and the discussion of constraint damping above, and recalling the notation \eqref{EqCptFX} and the bundle splittings \eqref{EqCptASpl} and \eqref{EqCptASpl2}, we will seek the solution $h$ of $P(h)=0$ in the function space $\cX^{k;b_0,b_I,b'_I,b_+}$:

\begin{definition}
\label{DefEinF}
  Let $k\in\N_0\cup\{\infty\}$, and fix weights\footnote{The imposed upper bound of $\half$ for $b_I$ and $b'_I$ simplifies the arithmetic in \S\ref{SBg} but is otherwise artificial; the natural bound is $b_I<b'_I<\min(1,b_0)$, with the upper bound $1$ arising from the expected presence of lower order terms in expansion of the metric at $\scri^+$ as well as from the requirement that the function space be independent of the choice of collar neighborhood of $\scri^+$.}
  \[
    -1 < b_+ < 0 < b_I < b'_I < \min(\half,b_0);
  \]
  let further $\chi\in\CI(M)$ be identically $1$ near $\scri^+$, with support in a small neighborhood of $\scri^+$ where the bundle splitting~\eqref{EqCptASpl} is defined; different choices of $\chi$ will produce the same function space, as we shall discuss below. The space $\cX^{k;b_0,b_I,b'_I,b_+}$ consists of all $h\in\Hb^{k;b_0,-1,b_+}(M;\beta^*S^2\,\Tsc^*\ol{\R^4})$ such that
  \begin{align}
  \label{EqEinFGood}
     &\chi h_{0 0},\ \chi\sltr h\in \Hb^{k;b_0,b'_I,b_+}(\ul\C),\ \chi h_{0\bar b}\in\Hb^{k;b_0,b'_I,b_+}(\beta^*(r\,T^*\Sph^2)), \\
  \label{EqEinF11}
     &\chi h_{1 1} = \chi h_{1 1}^{(1)}\log\rho_I + \chi h_{1 1}^{(0)} + h_{1 1,\bop}, \\
  \label{EqEinFRest}
     &\chi(h_{0 1},h_{1\bar b},h_{\bar a\bar b}) = \chi(h_{0 1}^{(0)},h_{1\bar b}^{(0)},h_{\bar a\bar b}^{(0)}) + (h_{0 1,\bop},h_{1\bar b,\bop},h_{\bar a\bar b,\bop}),
  \end{align}
  where the leading and remainder terms are
  \begin{gather*}
    h_{1 1}^{(\ell)},\ h_{0 1}^{(0)},\ h_{1\bar b}^{(0)},\ h_{\bar a\bar b}^{(0)} \in \rho_0^{b_0}\rho_+^{b_+}\Hb^k(\scri^+), \\
    h_{0 1,\bop},\ h_{1 1,\bop},\ h_{1\bar b,\bop},\ h_{\bar a\bar b,\bop} \in\Hb^{k;b_0,b_I,b_+},
  \end{gather*}
  the latter supported on $\supp\chi$ and valued in the bundles $\ul\C$ (for $\ell=0,1$), $\ul\C$, $\beta^*(r\,T^*\Sph^2)$, and $\beta^*(r^2\,S^2 T^*\Sph^2)$, respectively; we describe the topology on $\cX^{k;b_0,b_I,b'_I,b_+}$ below. Here, we use a collar neighborhood to extend functions from $\scri^+$ to a neighborhood of $\scri^+$ in $M$, and to extend the relevant bundles from $\scri^+$ to smooth subbundles of $\beta^*S^2\,\Tsc^*\ol{\R^4}$ near $\scri^+$; all choices of collar neighborhoods and extensions give the same function space. We shall suppress the parameters $b_0,b_I,b'_I,b_+$ from the notation when they are clear from the context, so
  \[
    \cX^k := \cX^{k;b_0,b_I,b'_I,b_+}.
  \]
\end{definition}

\begin{rmk}
  The partial expansions amount to a statement of partial polyhomogeneity: for example, the condition on $h_{0 1}$ in~\eqref{EqEinFRest} for $k=\infty$ can be phrased as $h_{0 1}\in\cA_{\bop,\phg,\bop}^{b_0,0,b_+}+\Hb^{\infty;b_0,b_I,b_+}$, and similarly for $k<\infty$ if one replaces the first summand by a function space capturing the finite regularity of the leading term at $\scri^+$. In view of the existence of at most logarithmically growing leading terms of $h\in\cX^k$ at $\scri^+$, we automatically have $h\in\Hb^{k;b_0,-0,b_+}$.
\end{rmk}

Thus, $h\in\cX^k$ decays at $I^0$, while \eqref{EqEinFGood} encodes the vanishing of the good components at $\scri^+$; \eqref{EqEinF11} and \eqref{EqEinFRest} assert the existence of leading terms of the remaining components, in the case of $h_{1 1}$ allowing for a logarithmic term;\footnote{The slightly faster decay $b_I'$ of the good components as compared to the decay $b_I$ of the remainder terms of the other components is needed to handle the logarithmically large size of the coefficients coupling good components into the others, encoded in the $(4,1)$ entries of $A_h$ and $B_h$ in Lemma~\ref{LemmaEinNscri}; see the discussion following~\eqref{EqEinNscriPi11}.} at $I^+$ finally, $h$ is allowed to have mild growth. The existence of leading terms of $h\in\cX^{k;b_0,b_I,b'_I,b_+}$ at $\scri^+$ implies in particular that
\begin{equation}
\label{EqEinFLeading}
\begin{gathered}
  \rho_I\pa_{\rho_I}h_{\bar\mu\bar\nu}\in\Hb^{k-1;b_0,b_I,b_+},\ \ (\bar\mu,\bar\nu)=(0,1),(1,\bar b),(\bar a,\bar b), \\
  \rho_I\pa_{\rho_I}h_{1 1}\in h_{1 1}^{(1)}+\Hb^{k-1;b_0,b_I,b_+},\ (\rho_I\pa_{\rho_I})^2 h_{1 1}\in\Hb^{k-2;b_0,b_I,b_+},
\end{gathered}
\end{equation}
which we will frequently use without further explanation.

For $h\in\cX^{\infty;b_0,b_I,b'_I,b_+}$, we describe $P(h)$ using a closely related function space:
\begin{definition}
\label{DefEinFY}
  For $k\in\N_0\cup\{\infty\}$ and weights $b_0,b_I,b'_I,b_+$ as above, the function space $\cY^{k;b_0,b_I,b'_I,b_+}$ consists of all $f\in\Hb^{k;b_0,-2,b_+}(M;\beta^*S^2\,\Tsc^*\ol{\R^4})$ so that near $\scri^+$,
  \begin{equation}
  \label{EqEinFGoodY}
  \begin{gathered}
    f_{0 0},\ f_{0\bar b},\ \sltr f \in \Hb^{k;b_0,-1+b'_I,b_+}, \ \ 
    f_{0 1},\ f_{1\bar b},\ f_{\bar a\bar b} \in \Hb^{k;b_0,-1+b_I,b_+}, \\
    f_{1 1} = f_{1 1}^{(0)}\rho_I^{-1} + f_{1 1,\bop},\ \ f_{1 1}^{(0)}\in\rho_0^{b_0}\rho_+^{b_+}\Hb^k(\scri^+),\ f_{1 1,\bop}\in\Hb^{k;b_0,-1+b_I,b_+}.
  \end{gathered}
  \end{equation}
\end{definition}
The shift by $-1$ in the decay order at $\scri^+$ is due to the linearized gauge-fixed Einstein equation, or even the linear scalar wave equation, being $\rho_I^{-1}$ times a b-differential operator at $\scri^+$, cf.\ \eqref{EqISysLinEqI}. A calculation will show that for $h$ as above, the gauge-fixed Einstein operator $P(h)$ defined in~\eqref{EqEin} satisfies $P(h)\in\cY^{\infty;b_0,b_I,b'_I,b_+}$, see Lemma~\ref{LemmaEinP} for a more precise statement. Note here that $P(h)$ is well-defined (i.e.\ $g_m+\rho h$ is a nondegenerate symmetric 2-tensor, making $P(h)$ computable) in a neighborhood of $\pa M$ due to the decay (in $L^\infty$) of $g=g_m+\rho h$ to $g_m$. In order for $P(h)$ to be defined globally, we need to assume $\rho h$ to be small in $L^\infty$.

Fixing a smooth cutoff $\chi$ as in Definition~\ref{DefEinF}, we can define a norm on $\cY^{k;b_0,b'_I,b_I,b_+}$ using the notation of Definition~\ref{DefEinFY} by setting
\begin{align*}
  \|f\|_{\cY^{k;b_0,b'_I,b_I,b_+}} &:= \|(\chi f_{0 0},\ \chi f_{0\bar b},\ \chi\sltr f)\|_{\Hb^{k;b_0,-1+b'_I,b_+}} + \| (\chi f_{0 1},\ \chi f_{1\bar b},\ \chi f_{\bar a\bar b}) \|_{\Hb^{k;b_0,-1+b_I,b_+}} \\
    &\qquad + \|\chi f_{1 1}^{(0)}\|_{\rho_0^{b_0}\rho_+^{b_+}\Hb^k(\scri^+)} + \|\chi(f_{1 1}-f_{1 1}^{(0)})\|_{\Hb^{k;b_0,-1+b_I,b_+}} + \|f\|_{\Hb^{k;b_0,-2,b_+}},
\end{align*}
where the choice of $\rho_I$-weight in the remainder term is arbitrary (as long as it is fixed and less than $-1$). Equipped with this norm, $\cY^{k;b_0,b_I,b'_I,b_+}$ is a Banach space. A completely analogous definition gives a norm $\|\cdot\|_{\cX^{k;b_0,b_I,b'_I,b_+}}$. The spaces $\cX^{\infty;b_0,b_I,b'_I,b_+}$ and $\cY^{\infty;b_0,b_I,b'_I,b_+}$, equipped with the projective limit topologies, are Fr\'echet spaces.

In particular, using the Sobolev embedding $\Hb^3(M)\hra L^\infty(M)$ (which uses that $3>\dim(M)/2$), we have an embedding $\cX^3\hra\rho_0^{b_0}\rho_I^{-1}\rho_+^{b_+}L^\infty$; thus, $P(h)$ is well-defined globally on $M$ provided $h$ is small in $\cX^3$.

It will occasionally be useful to write
\begin{equation}
\label{EqEinFDecomp}
  \cX^k = \cX^k_\phg \oplus \cX^k_\bop, \ \ 
  \cY^k = \cY^k_\phg \oplus \cY^k_\bop,
\end{equation}
where $\cY^k_\phg=\{ \chi f_{1 1}^{(0)} \colon f_{1 1}^{(0)}\in \rho_0^{b_0}\rho_+^{b_+}\Hb^k(\scri^+) \}$ encodes the leading term of elements of $\cY^k$, while $\cY^k_\bop=\{f\in\cY^k\colon f_{1 1}^{(0)}=0\}$ captures the remainder terms (i.e.\ with vanishing leading terms at $\scri^+$); the spaces $\cX^k_\phg$ and $\cX^k_\bop$ are defined analogously.

In order to exhibit the `null structure,' or upper triangular block structure, of the linearized gauge-fixed Einstein operator $D_h P$ for $h\in\cX$ at $\scri^+$ in a compact fashion, we introduce subbundles of the symmetric 2-tensor bundle. We use the following notation: given a nowhere vanishing section $e$ of a complex vector bundle $E\to U$ over base manifold $U$, we denote by $\la e\ra$ the line subbundle of $E$ whose fiber of $p\in U$ is given by $\{\lambda e(x)\colon\lambda\in\C\}$.

\begin{definition}
\label{DefEinFSubb}
  Define the subbundles
  \[
    K_{1 1}^c := \la 2\,ds\,dq\ra \oplus (2\,ds\otimes r\,T^*\Sph^2) \oplus \la r^2\slg \ra^\perp,\ \ 
    K_0^c := K_{1 1}^c \oplus \la ds^2\ra,
  \]
  of $S^2\,\Tsc^*\ol{\R^4}|_{S^+}$, which we extend in a smooth but otherwise arbitrary fashion to a neighborhood of $S^+$ as rank $5$, resp.\ $6$, subbundles of $S^2\,\Tsc^*\ol{\R^4}$, still denoted by $K_{1 1}^c$ and $K_0^c$. Furthermore, define near $S^+$ the subbundles
  \begin{equation}
  \label{EqEinFK0}
    K_0=\la dq^2\ra\oplus\la 2 dq\otimes r\,T^*\Sph^2\ra\oplus\la r^2\slg\ra,\ \ 
    K_{1 1}=\la ds^2\ra.
  \end{equation}
\end{definition}

The only property of $K_0$ and $K_{1 1}$ which we will need is
\[
  K_0^c \oplus K_0 = S^2\,\Tsc^*\ol{\R^4},\ \ 
  K_{1 1}^c \oplus K_{1 1} = K_0^c.
\]
Denote by
\begin{equation}
\label{EqEinFSubbProj0}
\begin{gathered}
  \pi_0\colon S^2\,\Tsc^*\ol{\R^4}\to S^2\,\Tsc^*\ol{\R^4}/K_0^c \cong K_0, \\
  \tilde\pi_{1 1}\colon K_0^c \to K_0^c/K_{1 1}^c\cong K_{1 1}
\end{gathered}
\end{equation}
the projections onto the quotient bundles,
\[
  \pi_0^c := 1-\pi_0 \colon S^2\,\Tsc^*\ol{\R^4}\to K_0^c,
\]
and
\begin{equation}
\label{EqEinFSubbProj11}
  \pi_{1 1}:=\tilde\pi_{1 1}\pi_0^c \colon S^2\,\Tsc^*\ol{\R^4}\to K_{1 1},\ \ 
  \pi_{1 1}^c := (1-\tilde\pi_{1 1})\pi_0^c \colon S^2\,\Tsc^*\ol{\R^4} \to K_{1 1}^c.
\end{equation}
Writing
\begin{equation}
\label{EqEinFbetaS2}
  \beta^*S^2 \equiv \beta^*S^2\,\Tsc^*\ol{\R^4}
\end{equation}
from now on, the improved decay \eqref{EqEinFGood} of the good components of $h\in\cX^{k;b_0,b_I,b'_I,b_+}$ can then be expressed, using local coordinates $(\theta^2,\theta^3)$ on $\Sph^2$, as
\[
  \pi_0 h = h_{0 0}\,d q^2 + h_{0 a}\,d q\,d\theta^a + (\sltr h)\slg_{a b}\,d\theta^a\,d\theta^b\in\Hb^{k;b_0,b'_I,b_+}(\beta^*K_0),
\]
similarly for \eqref{EqEinFGoodY}. The refinement $K_{1 1}^c\subset K_0^c$,
\[
  \pi_{1 1}^c h = 2 h_{0 1}\,d s\,d q + 2 h_{0 b}\,d s\,d\theta^b + (h_{a b}-(\half\sltr h)\slg_{a b})\,d\theta^a\,d\theta^b
\]
will be used to encode part of the `null structure' of the linearized gauge-fixed Einstein equation at $\scri^+$, as discussed in \S\ref{SIt}; the component
\[
  \pi_{1 1} h = h_{1 1}\,d s^2
\]
will capture the logarithmically growing (relative to $r^{-1}$) component at $\scri^+$.

Consider now a fixed $h\in\cX^\infty$ which is small in $\cX^3$ so that $g:=g_m+\rho h$ is a Lorentzian metric on $\R^4$. Working near $\scri^+$, we recall $g_m=(1-\tfrac{2 m}{r})dq\,ds-r^2\slg$ and the barred index notation~\eqref{EqCptASplBar}, so with $\rho=r^{-1}$, the coefficients of $g$ in the product splitting \eqref{EqCptASplProd} are
\begin{equation}
\label{EqEinFMetric}
  \setarraystretch
  \begin{array}{lll}
    g_{0 0}=r^{-1}h_{0 0}, & g_{0 1}=\half+r^{-1}(h_{0 1}-m), & g_{0 b}=h_{0\bar b}, \\
    g_{1 1}=r^{-1}h_{1 1}, & g_{1 b}=h_{1\bar b}, & g_{a b}=-r^2\slg_{a b}+r h_{\bar a\bar b};
  \end{array}
\end{equation}
the coefficients $g^{\mu\nu}$ of the inverse metric $g^{-1}=g_m^{-1}-r^{-1}g_m^{-1}h g_m^{-1}+r^{-2}g_m^{-1}h g_m^{-1}h g_m^{-1}+\Hb^{\infty;3+3 b_0,3-0,3+3 b_+}$ are
\begin{equation}
\label{EqEinFinverse}
  \setarraystretch
  \begin{array}{ll}
    g^{0 0} \in -4 r^{-1}h_{1 1} + \Hb^{\infty;2+b_0,2-0,2+2 b_+}, & g^{0 1} \in 2 + 4 r^{-1}(m-h_{0 1}) + \Hb^{\infty;2-0,2-0,2+2 b_+}, \\
    g^{0 b} \in 2 r^{-2}h_1{}^{\bar b} + \Hb^{\infty;3+b_0,3-0,3+2 b_+}, & g^{1 1} \in -4 r^{-1}h_{0 0} + \Hb^{\infty;2+b_0,2+b'_I,2+2 b_+}, \\
    g^{1 b} \in 2 r^{-2}h_0{}^{\bar b} + \Hb^{\infty;3+b_0,3+b'_I,3+2 b_+}, & g^{a b} \in -r^{-2}\slg^{a b} - r^{-3}h^{\bar a\bar b} + \Hb^{\infty;4+2 b_0,4-0,4+2 b_+},
  \end{array}
\end{equation}
where we raise spherical indices using the round metric $\slg$, i.e.\ $h_0{}^{\bar a}=\slg^{a b}h_{0\bar b}$ etc. Thus,
\begin{equation}
\label{EqEinFMetricRough}
  g_{\bar\mu\bar\nu},\ g^{\bar\mu\bar\nu}\in\CI+\Hb^{\infty;1+b_0,1-0,1+b_+};\ \ 
  g^{\bar a\bar b}+\slg^{a b},\ g^{0 1}-2 \in \rho\,\CI+\Hb^{\infty;1+b_0,1-0,1+b_+}.
\end{equation}
The calculation of the connection coefficients, components of Riemann and Ricci curvature, and other geometric quantities associated with the metric $g$ is then straightforward; the results of these calculations are given in Appendix~\ref{SCo}.

\subsection{Mapping properties of the gauge-fixed Einstein operator}
\label{SsEinP}

Let $h\in\cX^\infty=\cX^{\infty;b_0,b_I,b'_I,b_+}$. In order to compute the leading terms of the gauge-fixed Einstein operator $P(h)=\rho^{-3} P_0(g)$, $g=g_m+\rho h$, see \eqref{EqEin}, we first use the definition \eqref{EqEinTdel} of $2(\tdel^*-\delta_{g_m^S}^*)$ (given explicitly by~\eqref{EqCoSchwTdeldiff} in the case $m=0$) and the observation, from \eqref{EqCoPertUpsLower}, that $\Ups(g)\in\Hb^{\infty;2+b_0,1+b_I',2+b_+}$ (note that the explicit terms given in~\eqref{EqCoPertUpsLower} lie in this space in view of \eqref{EqCptASplNull} and the decay of the coefficients of $h$ in Definition~\ref{DefEinF}), to deduce that
\begin{equation}
\label{EqEinPTdelDelDiffUps}
  2(\tdel^*-\delta_{g_m}^*)\Ups(g) \in \Hb^{\infty;3+b_0,2+b_I',3+b_+}.
\end{equation}
The decay rate at $I^+$ holds \emph{globally} there---not only near $I^+\cap\scri^+$ where $g_m=g_m^S$. To see this, it suffices to show that $\Ups(g)\in\rho_+^{2+b_+}\Hb^\infty$ near $(I^+)^\circ$ (since $\tdel^*-\delta_{g_m}^*\in\rho_+\Diffb^1$, cf.\ \eqref{EqEinTdel}, then maps it into the stated space).\footnote{Recall that on ${}^0\!M$, we can take $t^{-1}$ as a local defining function of $(I^+)^\circ$; on ${}^m\!M$, this needs to be modified by a term of size $t^{-2}\log t$ due to the different smooth structure.} But this follows from the fact that there $g$ differs from the smooth scattering metric $g_m$ by an element of $\rho_+^{1+b_+}\Hb^\infty$ (with values in $S^2\,\Tsc^*\ol{\R^4}$). Concretely, choosing local coordinates $y^1,y^2,y^3$ in $\pa\ol{\R^4}$, near any point $p\in (I^+)^\circ$, we can introduce coordinates $z^0:=\rho_+^{-1}$, $z^a=\rho_+^{-1}y^a$ ($a=1,2,3$), in a neighborhood of $p$ intersected with $\rho>0$, and $\{\pa_{z^\mu}\colon \mu=0,\ldots,3\}$ is a frame of $\Tsc\ol{\R^4}$ there; but then, using $\pa_{z^\mu}\in\rho_+\Vb(\ol{\R^4})$, one sees that $\Gamma(g_m+\rho h)_{\kappa\lambda}^\nu-\Gamma(g_m)_{\kappa\lambda}^\nu$ is a sum of terms of the form
\[
    ((g_m+\rho h)^{\mu\nu}-(g_m)^{\mu\nu})\pa_{z^\kappa}(g_m)_{\lambda\sigma}\in \rho_+^{1+b_+}\Hb^\infty\cdot\rho_+\CI(\ol{\R^4})\subset\rho_+^{2+b_+}\Hb^\infty\ \ (\tn{near}\ p),
\]
and $(g_m+\rho h)^{\mu\nu}\pa_{z^\kappa}(\rho h_{\lambda\sigma})$, which likewise lies in $\rho_+^{2+b_+}\Hb^\infty$ near $p$. (The Christoffel symbols themselves satisfy $\Gamma(g_m)_{\kappa\lambda}^\nu\in\rho^+\CI(\ol{\R^4})$, $\Gamma(g_m+\rho h)_{\kappa\lambda}^\nu\in\rho_+\CI(\ol{\R^4})+\rho_+^{2+b_+}\Hb^\infty$.)

We can now prove:

\begin{lemma}
\label{LemmaEinP}
  For any $h\in\cX^\infty$, the tensor $P(h)$ is well-defined near $\pa M$ (in the sense explained in the paragraph after Definition~\usref{DefEinFY}), and we have $\chi P(h)\in\cY^\infty$ for any $\chi\in\CI(M)$ with support sufficiently close (depending on $h$) to $\pa M$. We have $P(h)\in\cY^\infty$ provided $\|h\|_{\cX^3}$ is small. More precisely, we have $P(h)_{\bar a\bar b}\in\Hb^{\infty;b_0,-1+b'_I,b_+}$ and
  \begin{equation}
  \label{EqEinP11}
    P(h)_{1 1} \in -2\rho^{-2}\pa_1\pa_0 h_{1 1} - \tfrac14 \rho^{-1}\pa_1 h^{\bar d\bar e}\pa_1 h_{\bar d\bar e} + \Hb^{\infty;b_0,-1+b_I,b_+}
  \end{equation}
  when $\rho=r^{-1}$ near $\scri^+$.
\end{lemma}
\begin{proof}
  We use the calculations (near $I^0\cup\scri^+$) of $\delta_{g_m}^*\Ups(g)$ in \eqref{EqCoPertDelUps} and of $\Ric(g)$ in \eqref{EqCoPertRic}; in view of the calculation~\eqref{EqEinPTdelDelDiffUps}, it suffices to prove that $\rho^{-3}(\Ric(g)-\delta_{g_m}^*\Ups(g))\in\cY^\infty$ near $\pa M$. In a neighborhood of $I^0\cup\scri^+$, this follows by subtracting~\eqref{EqCoPertDelUps} from \eqref{EqCoPertRic} and dividing by $\rho^3$ (thus shifting the three orders down by $3$); the expression~\eqref{EqEinP11} is a particular result of this subtraction.

  It remains to justify the decay rate globally at $I^+$, which is a slight extension of the calculations justifying~\eqref{EqEinPTdelDelDiffUps} above. We use local coordinates near $p\in(I^+)^\circ$ as above: firstly, the membership of $\delta_{g_m}^*\Ups(g)$ follows directly from the above arguments. Secondly, the difference of curvature components $R(g_m+\rho h)^\mu{}_{\nu\kappa\lambda}-R(g_m)^\mu{}_{\nu\kappa\lambda}$ is a sum of terms of the schematic forms $\pa_\mu(\Gamma(g_m+\rho h)^\kappa_{\nu\lambda}-\Gamma(g_m)^\kappa_{\nu\lambda})$ and $(\Gamma(g_m+\rho h)_{\mu\nu}^\kappa-\Gamma(g_m)_{\mu\nu}^\kappa)\Gamma(g_m+\rho h)_{\kappa\lambda}^\nu$, both of which lie in $\rho_+^{3+b_+}\Hb^\infty$ by the calculations above. But by construction, see equations~\eqref{EqCptAMetric1}--\eqref{EqCptAMetric}, $g_m$ differs from a \emph{flat} metric by a smooth symmetric scattering 2-tensor of class $\rho_+\CI(\ol{\R^4})$, which implies that $R(g_m)^\mu{}_{\nu\kappa\lambda}\in\rho_+^3\CI(\ol{\R^4})$ near $p$. Therefore, the Riemann curvature tensor satisfies
  \begin{equation}
  \label{EqEinPRiem}
    R(g_m+\rho h) \in \rho_+^{3+b_+}\Hb^\infty
  \end{equation}
  as a section of $\Tsc\ol{\R^4}\otimes(\Tsc^*\ol{\R^4})^{\otimes 3}$ near $(I^+)^\circ$, which a fortiori gives $\Ric(g)\in\rho_+^{3+b_+}\Hb^\infty$, as desired. (The vanishing of $P(h)$ modulo the \emph{faster} decaying space $\rho_0^{b_0}\Hb^\infty$ near $(I^0)^\circ$ requires more structure of $g_m$, namely the Ricci flatness of the background metric $g_m$.)
\end{proof}

Note that one component of $P(h)$ has a nontrivial leading term at $\scri^+$; in order for this to not create logarithmically growing terms in components (other than the $(1,1)$ component) of the next iterate of our Newton-type iteration scheme (which would cause the iteration scheme to not close), one needs to exploit the special structure of the operator $D_h P$. See also the discussion around~\eqref{EqISysNull}.

\subsection{Leading order structure of the linearized gauge-fixed Einstein operator}
\label{SsEinN}

For $h\in\cX^{\infty;b_0,b_I,b'_I,b_+}$ small, write
\begin{equation}
\label{EqEinNLin}
  L_h:=D_h P,
\end{equation}
and let $g=g_m+\rho h$. We shall now calculate the structure of $L_h$ `at infinity,' that is, its leading order terms at $I^0$, $\scri^+$, and $I^+$: at $\scri^+$, we will find that the equation $L_h u=f$ can be partially decoupled to leading order; this is the key structure for proving global existence for the nonlinear problem later. Recall from \cite{GrahamLeeConformalEinstein} that 
\begin{equation}
\label{EqEinNOp0}
\begin{split}
  D_g\Ric &= \half\Box_g - \delta_g^*\delta_g G_g + \sR_g, \\
    &\qquad \sR_g(u)_{\mu\nu}=(R_g)^\kappa{}_{\mu\nu\lambda}u_\kappa{}^\lambda + \half(\Ric(g)_\mu{}^\lambda u_{\lambda\nu}+\Ric(g)_\nu{}^\lambda u_{\lambda\mu}), \\
  D_g\Ups(g)u &= -\delta_g G_g u - \sC_g(u)+\sY_g(u),
\end{split}
\end{equation}
where (our notation differs from the one used in \cite{GrahamLeeConformalEinstein} by various signs)
\[
  \sC_g(u)_\kappa = g_{\kappa\lambda}C^\lambda_{\mu\nu}u^{\mu\nu},\ C^\lambda_{\mu\nu}=\Gamma(g)^\lambda_{\mu\nu}-\Gamma(g_m)^\lambda_{\mu\nu}; \qquad
  \sY_g(u)_\kappa = \Ups(g)^\lambda u_{\kappa\lambda}.
\]
Here, index raising and lowering as well as covariant derivatives are defined using the metric $g$, and $(\Box_g u)_{\mu\nu}=-u_{\mu\nu;\kappa}{}^\kappa$. Thus, recalling the definition \eqref{EqEinTdel} of $\tdel^*$, we have
\begin{equation}
\label{EqEinNOp}
  L_h = \rho^{-3}\bigl(\half\Box_g + (\tdel^*-\delta_g^*)\delta_g G_g + \tdel^*(\sC_g-\sY_g) + \sR_g\bigr)\rho,\qquad g=g_m+\rho h,
\end{equation}
which has principal symbol
\begin{equation}
\label{EqEinNSymbol}
  \sigma_2(L_h) = \half G_\bop := \half(g_\bop)^{-1},\ \ g_\bop:=\rho^2 g,
\end{equation}
where $G\in\CI(T^*\R^4)$ is the dual metric function $G(\zeta)=|\zeta|_G^2$. As a first step towards understanding the nature of $L_h$ as a b-differential operator on $M$, we prove:
\begin{lemma}
\label{LemmaEinNL0}
  We have $L_0\in\rho_I^{-1}\Diffb^2(M;\beta^*S^2)$ (see~\eqref{EqEinFbetaS2}).
\end{lemma}
\begin{proof}
  Since $g_m$ is a smooth scattering metric, we see, using local coordinates $z^\mu$ and the membership $\pa_{z^\mu}\in\rho\Vb(\ol{\R^4})$ as in the discussion preceding Lemma~\ref{LemmaEinP} to compute Christoffel symbols, that
  \[
    \sR_{g_m}\in\rho^2\,\CI(\ol{\R^4};\End(S^2\,\Tsc^*\ol{\R^4})),\ \ 
    \delta_{g_m}\in\rho\,\Diffb^1(\ol{\R^4};S^2\,\Tsc^*\ol{\R^4},\Tsc^*\ol{\R^4}),
  \]
  and $\Box_{g_m}\in\rho^2\,\Diffb^2(\ol{\R^4};S^2\,\Tsc^*\ol{\R^4})$. This gives $L_0\in\Diffb^2(\ol{\R^4};S^2\,\Tsc^*\ol{\R^4})$, and thus the desired conclusion away from $\scri^+$. Near $\scri^+$, any element of $\Diffb^1(\ol{\R^4})$ lifts to an element of $\rho_I^{-1}\Diffb^1(M)$; moreover, for $V_1,V_2\in\Vb(\ol{\R^4})$, the product $V_1 V_2$ lifts to an element of $\rho_I^{-1}\Diffb^2(M)$ provided at least one of the $V_j$ is tangent to $S^+$. Thus, expressing $\Box_{g_m}$ in the null frame $\pa_0,\pa_1,\pa_a$ ($a=2,3$), we merely need to check that the coefficient of $\pa_1^2$ vanishes at $S^+$; but this coefficient is $g_m^{1 1}\equiv 0$.
\end{proof}

As suggested by the toy estimate~\eqref{EqISysLinEqIEst} and explained in \S\ref{SsCptScri}, we need to describe lower order terms of $L_h$ near $\scri^+$ in two stages, one involving the module $\cM$ from Definition~\ref{DefCptScriModule}, the other being general b-differential operators but with extra decay at $\rho_I=0$. For illustration and for later use, we calculate the leading terms, i.e.\ the `normal operator,' of the scalar wave operator:

\begin{lemma}
\label{LemmaEinNScalarBox}
  The scalar wave operator $\Box_{g_\bop}$ (see~\eqref{EqEinNSymbol}) satisfies
  \begin{equation}
  \label{EqEinNScalarBox}
    \Box_{g_\bop} \in -4\rho^{-2}\pa_0\pa_1 +\Hb^{\infty;1+b_0,-1+b'_I,1+b_+}\cM^2_{\ul\C} + (\CI+\Hb^{\infty;1+b_0,-0,1+b_+})\Diffb^2(M).
  \end{equation}
\end{lemma}

For the linearized gauge-fixed Einstein operator $L_h$, the analogous result is:

\begin{lemma}
\label{LemmaEinNscri}
  For $h\in\cX^\infty$ small in $\cX^3$, we have
  \[
    L_h = L_h^0 + \wt L_h
  \]
  where, using the notation~\eqref{EqEinFbetaS2} and fixing $\rho=r^{-1}$ near $\scri^+$,
  \begin{equation}
  \label{EqEinNscri}
  \begin{split}
    L_h^0 & = -\rho^{-1}\bigl((2\rho^{-1}\pa_0+A_h)\pa_1 - B_h\bigr), \\
    \wt L_h &\in \Hb^{\infty;1+b_0,-1+b'_I,1+b_+}\cM_{\beta^*S^2}^2 + (\CI+\Hb^{\infty;1+b_0,-0,1+b_+})\Diffb^2(M;\beta^*S^2);
  \end{split}
  \end{equation}
  here $\rho^{-1}\pa_0$ and $\pa_1$ are defined using equation~\eqref{EqCptScriNull} and Lemma~\usref{LemmaCptScriModule}\eqref{ItCptScriModuleBdl}. In the refinement of the bundle splitting \eqref{EqCptASpl2} by \eqref{EqCptASplS2}, $A_h$ and $B_h$ are given by
  \[
    A_h=
     \openbigpmatrix{2pt}
       2\gamma & 0 & 0 & 0 & 0 & 0 & 0 \\
       -2\pa_1 h_{0 1} & 0 & 0 & 0 & 0 & 0 & 0 \\
       0 & 0 & \gamma & 0 & 0 & 0 & 0 \\
       0 & 0 & -2\pa_1 h_1{}^{\bar a} & 0 & 0 & \gamma+2\pa_1 h_{0 1} & \half\pa_1 h^{\bar a\bar b} \\
       -2\pa_1 h_{1\bar b} & 0 & \gamma & 0 & 0 & 0 & 0 \\
       2\gamma & 0 & 0 & 0 & 0 & \gamma & 0 \\
       -2\pa_1 h_{\bar a\bar b} & 0 & 0 & 0 & 0 & 0 & 0
     \closebigpmatrix
  \]
  and
  \[
    B_h=
    \begin{pmatrix}
      0 & 0 & 0 & 0 & 0 & 0 & 0 \\
      2\pa_1\pa_1 h_{0 1} & 0 & 0 & 0 & 0 & 0 & 0 \\
      0 & 0 & 0 & 0 & 0 & 0 & 0 \\
      2\pa_1\pa_1 h_{1 1} & 0 & 0 & 0 & 0 & 0 & 0 \\
      2\pa_1\pa_1 h_{1\bar b} & 0 & 0 & 0 & 0 & 0 & 0 \\
      0 & 0 & 0 & 0 & 0 & 0 & 0 \\
      2\pa_1\pa_1 h_{\bar a\bar b} & 0 & 0 & 0 & 0 & 0 & 0
    \end{pmatrix}.
  \]
\end{lemma}

The proofs of these lemmas only involve simple calculations and careful bookkeeping; they are given in Appendix~\ref{SAppEinN}. We thus see that at $\scri^+$, $L_h$ effectively becomes a differential operator in the null coordinates $x^0=q$ and $x^1=s$ only, as spherical derivatives have decaying coefficients; this is to be expected since $r^{-1}V$, $V\in\cV(\Sph^2)\subset\Vb(M)$, is the naturally appearing (scattering) derivative just like $\pa_0$ and $\pa_1$. We point out that a number of terms of $L_h$ which are not of leading order at $\scri^+$ do contribute to the normal operators at $I^0$ and $I^+$; this includes in particular the spherical Laplacian, which is crucial for proving an energy estimate.

For the analysis of the linearized operator $L_h$, the structure of the leading term $L_h^0$ will be key for obtaining the rough background estimate, Theorem~\ref{ThmBg}, as well as the precise asymptotic behavior at $\scri^+$, as encoded in the space $\cX^\infty$. To describe this structure concisely, recall the projection $\pi_0$ defined in \eqref{EqEinFSubbProj0} projecting a metric perturbation onto the bundle $K_0$ encoding the components which we expect to be decaying from the gauge condition; and the projection $\pi_{1 1}$ defined in \eqref{EqEinFSubbProj11} onto the bundle $K_{1 1}$ encoding the $(1,1)$ component, which we allow to include a logarithmic term. Thus, in the splitting used in Lemma~\ref{LemmaEinNscri}, $\pi_0$ picks out components $1,3,6$, $\pi_{1 1}$ picks out component $4$, and $\pi_{1 1}^c$ picks out components $2,5,7$. Suppose now $h'$ satisfies the asymptotic equation $L_h^0 h'=0$. Since $\pi_0 A_h|_{K_0^c}=0$ and $\pi_0 B_h|_{K_0^c}=0$, the components $\pi_0 h'$, which we hope to be decaying, satisfy a \emph{decoupled} equation
\begin{subequations}
\begin{equation}
\label{EqEinNscriPi0}
  (2\rho^{-1}\pa_0+A_\CD)\pa_1(\pi_0 h') = 0, \quad
  A_\CD
  :=\begin{pmatrix}
     2\gamma & 0 & 0 \\
     0 & \gamma & 0 \\
     2\gamma & 0 & \gamma
   \end{pmatrix},
\end{equation}
where $A_\CD\in\CI(M;\End(K_0))$ is the endomorphism induced by $\pi_0 A_h$ on $\beta^*S^2/K_0^c\cong K_0$. (Thus, this matrix is the expression for $A_{h,0}$ in the splitting of $K_0\cong\beta^*S^2/K_0^c$ induced by the splittings~\eqref{EqCptASpl2}--\eqref{EqCptASplS2} via the projection $\pi_0$.) Note that by equation~\eqref{EqCptASplNull}, $\rho^{-1}\pa_0$ is proportional to the dilation vector field $-\rho_I\pa_{\rho_I}$ (which is the asymptotic generator of dilations on outgoing light cones), hence equation~\eqref{EqEinNscriPi0} is, schematically, $(\rho_I\pa_{\rho_I}-A_\CD)(\pi_0 h')=0$. Choosing $\gamma>0$, the spectrum of $A_\CD$ is positive, which will allow us to prove that $\pi_0 h'$ decays at $\scri^+$, similarly to the discussion of the model equation~\eqref{EqISysLinDamping}; we will make this precise in \S\S\ref{SsBgscri} and \ref{SsIti0Scri}.

Next, using that $\pi_{1 1}^c A_h|_{K_{1 1}}=0$ and $\pi_{1 1}^c B_h|_{K_{1 1}}=0$, i.e.\ the logarithmic component $h_{1 1}$ does not couple into the other nondecaying components, we can obtain an equation for the nonlogarithmic components $\pi_{1 1}^c h'$ which only couples to \eqref{EqEinNscriPi0}, namely
\begin{align}
\label{EqEinNscriPi11c}
  &2\rho^{-1}\pa_0\pa_1(\pi_{1 1}^c h') = (-A_{h,1 1}^c\pa_1 + B_{h,1 1}^c)(\pi_0 h'), \\
  &\quad
  A_{h,1 1}^c=
    \openbigpmatrix{2pt}
      -2\pa_1 h_{0 1} & 0 & 0 \\
      -2\pa_1 h_{1\bar b} & \gamma & 0 \\
      -2\pa_1 h_{\bar a\bar b} & 0 & 0
    \closebigpmatrix,\ 
  B_{h,1 1}^c=
    \begin{pmatrix}
      2\pa_1\pa_1 h_{0 1} & 0 & 0 \\
      2\pa_1\pa_1 h_{1\bar b} & 0 & 0 \\
      2\pa_1\pa_1 h_{\bar a\bar b} & 0 & 0
    \end{pmatrix};
  \nonumber
\end{align}
the precise form of $A_{h,1 1}^c,B_{h,1 1}^c$, mapping sections of $K_0$ to sections of $K_{1 1}^c$, is irrelevant: only their boundedness matters (even mild growth towards $\scri^+$ would be acceptable). The operator on the left hand side of~\eqref{EqEinNscriPi11c} has the same structure as the model operator in~\eqref{EqISysLinTransport}; the fact that the forcing term in \eqref{EqEinNscriPi11c} is \emph{decaying} will thus allow us to prove that $\pi_{1 1}^c h'$ is bounded at $\scri^+$, consistent with what the function space $\cX^\infty$ encodes.

Lastly, $\pi_{1 1} h'$ couples to all previous quantities,
\begin{align}
\label{EqEinNscriPi11}
  &2\rho^{-1}\pa_0\pa_1(\pi_{1 1} h') = (-A_{h,1 1}\pa_1 + B_{h,1 1})\begin{pmatrix}\pi_0 h' \\ \pi_{1 1}^c h'\end{pmatrix}, \\
  &\qquad A_{h,1 1}=\begin{pmatrix} 0 & -2\pa_1 h_1{}^{\bar a} & \gamma+2\pa_1 h_{0 1} & 0 & 0 & \half\pa_1 h^{\bar a\bar b} \end{pmatrix}, \nonumber\\
  &\qquad B_{h,1 1}=\begin{pmatrix} 2\pa_1\pa_1 h_{1 1} & 0 & 0 & 0 & 0 & 0 \end{pmatrix}.\nonumber
\end{align}
\end{subequations}
The logarithmic growth of the first component of $B_{h,1 1}$ is more than balanced by the fast decay of the $(0,0)$-component of $h'$ that it acts on.

\begin{rmk}
  The fact that the logarithmic growth of $h_{1 1}$ is rendered harmless due to its coupling only to the faster decaying $\pi_0 h'$ is the manifestation of the weak null condition \cite{LindbladRodnianskiWeakNull} in our framework. Here, the faster decay of $\pi_0 h'$ is accomplished by means of constraint damping, whereas in \cite{LindbladRodnianskiGlobalExistence,LindbladRodnianskiGlobalStability} the faster decay of $\pi_0$ applied to the difference of the nonlinear solution and the background (Minkowski) metric follows from the gauge condition which the nonlinear solution verifies, cf.\ \cite[Corollary~9.7]{LindbladRodnianskiGlobalStability}.
\end{rmk}

More subtly, the $\rho_I^{b'_I}$ decay of $h'_{0 0}$ is required at this point to allow for an estimate of the remainder of $h_{1 1}$ with weight $\rho_I^{b_I}$ ($\gg \rho_I^{b'_I}\log\rho_I$). The last component of $A_{h,1 1}$, acting on the trace-free spherical part of $h'$, in general has a nonzero leading term at $\scri^+$;\footnote{The discussion of Theorem~\ref{ThmIBondi} shows that for nontrivial data, this leading term \emph{must} be nontrivial somewhere on $\scri^+$.} hence, solving the equation~\eqref{EqEinNscriPi11}, schematically $\rho_I\pa_{\rho_I}(\pa_1\pi_{1 1} h')\approx\pa_1 h^{\bar a\bar b}(\pa_1 h')_{\bar a\bar b}$, requires $\pi_{1 1} h'$ to have a $\log\rho_I$ term.

At the other boundaries $I^0$ and $I^+$, we only need crude information about $L_h$ for the purpose of obtaining an energy estimate in \S\ref{SBg}:
\begin{lemma}
\label{LemmaEinNi0p}
  We have $L_h-L_0\in\Hb^{\infty;1+b_0,-1-0,1+b_+}\Diffb^2(M;\beta^*S^2)$.
\end{lemma}
\begin{proof}
  Near $(I^+)^\circ$, the stated $\rho_+^{1+b_+}$ decay is a consequence of the calculation of differences of Christoffel symbols and curvature components as in the proof of Lemma~\ref{LemmaEinP}. Near $\scri^+$, we revisit the proof of Lemma~\ref{LemmaEinNscri}: in the notation of equation~\eqref{EqEinNscri}, the expressions for $A_h$ and $B_h$ give $L_h^0-L_0^0\in\Hb^{\infty;1+b_0,-0,1+b_+}\Diffb^1$. Regarding the second remainder term in $\wt L_h$, we note that the leading order terms, captured by the $\Diffb^2$ summand with $\CI$ coefficients, come from terms of the metric and the Christoffel symbols which do not involve $h$; thus, these are equal to the corresponding terms of $L_0$.
\end{proof}

In order to obtain optimal decay results at $I^+$ in \S\ref{SsItip}, we shall need the \emph{precise} form of the normal operator of $L_h$, which by Lemma~\ref{LemmaEinNi0p} is the same as that of $L_0$. Now, $g_m$ is itself merely a perturbation of the Minkowski metric, pulled back by a diffeomorphism, see \eqref{EqCptAMetric1}. It is convenient for the normal operator analysis at $I^+$ in \S\S\ref{SsItip} and \ref{SPhg} to relate this to the usual presentation of the Minkowski metric $\ul g=dt^2-dx^2$ on $\R^4$ in $U=\{t>\tfrac23 r\}$:
\begin{lemma}
\label{LemmaEinNgmMink}
  The metric $\ul g$ lies in $\cA_\phg^{\cE_{\rm log}}(U;S^2\,\Tsc^*\,{}^m\ol{\R^4})$ for the index set $\cE_{\rm log}$ defined in \eqref{EqCptIndexSetLog}, and $\ul g-g_m\in\cA_\phg^{\cE'_{\rm log}}(U;S^2\,\Tsc^*\,{}^m\ol{\R^4})\subset\rho^{1-0}\Hb^\infty(U;S^2\,\Tsc^*\,{}^m\ol{\R^4})$.
\end{lemma}

The failure of smoothness (for $m\neq 0$) of $\ul g$ is due to the logarithmic correction, see \eqref{EqCptACp1}, in the definition of the compactification ${}^m\ol{\R^4}$. On the \emph{radial} compactification ${}^0\ol{\R^4}$ on the other hand, $\ul g$ \emph{is} a smooth scattering metric.

\begin{proof}[Proof of Lemma~\usref{LemmaEinNgmMink}]
  In the region $C_2$ defined in \eqref{EqCptACp2Space}, $g_m=\ul g$ is smooth, see the discussion after equation~\eqref{EqCptAMetric1}. In the region $C_1$, see equation~\eqref{EqCptACp1Space}, the spatial part $dr^2+r^2\slg$ is a smooth symmetric scattering 2-tensor on ${}^m\ol{\R^4}$. In the region $t\geq\tfrac23 r$ and for large $r$, the claim follows from Lemma~\ref{LemmaCptFtInv} in that region.
\end{proof}

Define
\begin{equation}
\label{EqEinNipulL}
  \ul L := \half\Box_{\ul g} + (\ul\tdel^*-\delta_{\ul g}^*)\delta_{\ul g}G_{\ul g}, \ \ (\ul\tdel^*-\delta_{\ul g}^*)u := 2\gamma t^{-1}\,dt\otimes_s u - \gamma t^{-1}(\iota_{\nabla^{\ul g}t}u)\ul g,
\end{equation}
cf.\ the definition~\eqref{EqEinTdel}, which is the linearization $\Ric(g)-\ul\tdel^*\ul\Ups(g)$ around $g=\ul g$, where $\ul\Ups(g)$ is defined like $\Ups(g)$ in \eqref{EqEinUps} with $\ul g$ in place of $g_m$. Using Lemma~\ref{LemmaEinNgmMink}, one finds $\ul L\in \cA_\phg^{\cE_{\rm log}}\cdot\Diffb^2(U;S^2\,\Tsc^*\ol{\R^4})$. Furthermore,
\begin{equation}
\label{EqEinNipulLL0}
  \ul L-L_0\in\cA_\phg^{\cE'_{\rm log}}(U)\cdot\Diffb^2(U;S^2\,\Tsc^*\ol{\R^4});
\end{equation}
but $\pa_v\in\rho_I^{-1}\Vb(M)$, while derivatives along b-vector fields tangent to $S^+$ lift to elements of $\Vb(M)$; thus,
\begin{equation}
\label{EqEinNipL0}
  \ul L-L_0\in\rho_I^{-1-0}\rho_+^{1-0}\Hb^\infty\cdot\Diffb^2\ \ (\tn{near}\ I^+\subset M).
\end{equation}

\section{Global background estimate}
\label{SBg}

We prove a global energy estimate for solutions of the linearized equation $L_h u=f$ with $h\in\cX^\infty$, and show that $u$ lies in a weighted conormal space provided $f$ does; recall here the definition~\eqref{EqEinNLin} of $L_h$. The weak asymptotics of $u$ at the boundaries $I^0$, $\scri^+$, and $I^+$ can be improved subsequently using normal operator arguments in \S\ref{SIt}. At $\scri^+$, the estimate loses a weight of $\rho_I^{1/2}$ for general b-derivatives, as we will explain in detail in \S\ref{SsBgscri}. We capture this using the function space $\Hscri^1$:

\begin{definition}
\label{DefBgHscri}
  Let $E\to\ol{\R^4}$ be a smooth vector bundle. With $\cM_{\beta^*E}$ defined in~\S\ref{SsCptScri}, let
  \begin{align*}
    \Hbeta^1(M;\beta^*E) &:= \{ u\in L^2_\bop(M;\beta^*E) \colon \cM_{\beta^*E}u\subset L^2_\bop(M;\beta^*E) \}, \\
    \Hscri^1(M;\beta^*E) &:= \{ u\in \Hbeta^1(M;\beta^*E) \colon \rho_I^{1/2}\Diffb^1(M;\beta^*E)u\subset L^2_\bop(M;\beta^*E) \}.
  \end{align*}
  For $k\in\N_0$ and $\bullet=\beta,\scri$, define
  \[
    H_{\bullet,\bop}^{1,k}(M;\beta^*E) := \{ u\in L^2_\bop(M;\beta^*E) \colon \Diffb^k(M;\beta^*E)u\subset H_{\bullet}^1(M;\beta^*E) \}.
  \]
    If $\{A_j\}\subset\cM_{\beta^*E}$ is a finite set spanning $\cM_{\beta^*E}$ over $\CI(M)$, we define norms on these spaces by
    \begin{align*}
      \| u \|_{\Hbetab^{1,k}(M;\beta^*E)} &:= \|u\|_{\Hb^k(M;\beta^*E)} + \sum_j \|A_j u\|_{\Hb^k(M;\beta^*E)}, \\
      \| u \|_{\Hscrib^{1,k}(M;\beta^*E)} &:= \|u\|_{\Hbetab^{1,k}(M;\beta^*E)} + \|\rho_I^{1/2}u\|_{\Hb^{k+1}(M;\beta^*E)}.
    \end{align*}
\end{definition}

Note that for $u\in\Hbeta^1$, we automatically have $\rho_I\Diffb^1(M)u\subset L^2_\bop$ by Lemma~\ref{LemmaCptScriModule}\eqref{ItCptScriModuleIncl}, so the subspace $\Hscri^1\subset\Hbeta^1$ encodes a $\rho_I^{1/2}$ improvement over this. Away from $\scri^+$, the spaces $\Hbetab^{1,k}$ and $\Hscrib^{1,k}$ are the same as $\Hb^{k+1}$.

Fix a vector field
\begin{equation}
\label{EqBgPartialNu}
  \pa_\nu\in\Vb(\ol{\R^4})
\end{equation}
transversal to the Cauchy surface $\Sigma$; we extend the action of $\pa_\nu$ to sections $u$ of a vector bundle $E$ using an arbitrary fixed b-connection $d^E$ on $E$, see~\eqref{EqCptScriConnDiffb}, by setting $\pa_\nu u:=(d^E u)(\pa_\nu)$.

\begin{thm}
\label{ThmBg}
  Fix weights $b_0,b'_I,b_I,b_+$ as in Definition~\usref{DefEinF}, let $\gamma>b'_I$ in the definition~\eqref{EqEinTdel} of $\tdel^*$, and fix $a_0,a_I,a'_I\in\R$ satisfying
  \[
    a_I<a'_I<a_0, \quad
    a_I<0, \quad
    a'_I<a_I+b'_I.
  \]
  Then there exists $a_+\in\R$ such that the following holds for all $h\in\cX^{\infty;b_0,b_I,b'_I,b_+}$ which are small in $\cX^3$: for $k\in\N$, $u_j\in\rho_0^{a_0}\Hb^{k-j}(\Sigma)$, $j=0,1$, and $f\in\Hb^{k-1;a_0,a_I-1,a_+}(M;\beta^*S^2)$ with $\pi_0 f\in\Hb^{k-1;a_0,a'_I-1,a_+}(M;\beta^*S^2)$, the linear wave equation
  \begin{equation}
  \label{EqBgEquation}
    L_h u = f, \ \ (u,\pa_\nu u)|_\Sigma = (u_0,u_1),
  \end{equation}
  has a unique global solution $u$ satisfying
  \begin{equation}
  \label{EqBgEstimate}
  \begin{split}
    &\| u \|_{\rho_0^{a_0}\rho_I^{a_I}\rho_+^{a_+}\Hscrib^{1,k-1}(M;\beta^*S^2)} + \| \pi_0 u \|_{\rho_0^{a_0}\rho_I^{a'_I}\rho_+^{a_+}\Hscrib^{1,k-1}(M;\beta^*S^2)} \\
    &\qquad\leq C\Bigl(\| u_0 \|_{\rho_0^{a_0}\Hb^k} + \| u_1 \|_{\rho_0^{a_0}\Hb^{k-1}} + \| f \|_{\Hb^{k-1;a_0,a_I-1,a_+}} + \| \pi_0 f \|_{\Hb^{k-1;a_0,a'_I-1,a_+}}\Bigr).
  \end{split}
  \end{equation}
  In particular, if the assumptions on $u_j$ and $f$ hold for all $k$, then
  \begin{equation}
  \label{EqBgEstDecay}
    u\in\Hb^{\infty;a_0,a_I,a_+},\ \ 
    \pi_0 u\in\Hb^{\infty;a_0,a'_I,a_+}.
  \end{equation}
\end{thm}

We refer the reader to Remark~\ref{RmkIDetailDecay} for a translation of the memberships~\eqref{EqBgEstDecay} to pointwise decay estimates. (For obtaining pointwise decay for any \emph{fixed} number of derivatives of $u$, the estimate of~\eqref{EqBgEstimate} for sufficiently large $k$ is of course sufficient.)

For completeness, we prove a version of such a background estimate with an \emph{explicit} weight $a_+$ in~\S\ref{SsBgExpl}. As we will see in \S\ref{SsItip}, this allows us to give an explicit bound on the number of derivatives needed to close the nonlinear iteration in~\S\ref{SPf}. A nonexplicit value of $a_+$ as in Theorem~\ref{ThmBg} is sufficient to prove Theorem~\ref{ThmIBaby} if one is content with a nonexplicit value for $N$.\footnote{One could obtain an explicit value for $N$ even from a nonexplicit weight $a_+$ if one improved the argument in~\S\ref{SPf}, which proves precise decay rates at $I^+$, to not lose regularity. We expect that this can be accomplished by microlocal propagation estimates along $\scri^+$ and radial point estimates at $\scri^+\cap I^+$, though we do not pursue this here.} We will prove Theorem~\ref{ThmBg} by means of energy estimates, as outlined in \S\ref{SssISysLin}. Microlocal techniques on $\ol{\R^4}$ on the other hand, as employed in \cite{BaskinVasyWunschRadMink}, would work well away from the light cone at infinity $S^+$, but since the coefficients of $L_h$ are singular at $S^+$, it is a delicate question how `microlocal' the behavior of $L_h$ is at $S^+$, i.e.\ whether or not and what strengths of singularities could `jump' from one part of the b-cotangent bundle to another at $S^+$; since we do not need precise microlocal control of $L_h$ for present purposes, we do not study this further.

Since $d t$ is globally timelike for $g=g_m+\rho h$ provided $\rho h$ is small in $\rho\cX^3\subset L^\infty$, existence and uniqueness of a solution $u\in\Hloc^k(M\cap\R^4;S^2 T^*\R^4)$ are immediate, together with an estimate for any compact set $K\Subset M\cap\R^4$,
\begin{equation}
\label{EqBgCompact}
  \| u \|_{H^k(K)} \leq C_K ( \| u_0 \|_{\rho_0^{a_0}\Hb^k} + \| u_1 \|_{\rho_0^{a_0}\Hb^{k-1}} + \| f \|_{\Hb^{k-1;a_0,a_I,a_+}} ),
\end{equation}
where one could equally well replace the norms on the right by standard Sobolev norms on sufficiently large compact subsets of $M\cap\R^4$ depending on $K$, due to the domain of dependence properties of solutions of \eqref{EqBgEquation}.

Using Lemma~\ref{LemmaEinNi0p}, it is straightforward to prove \eqref{EqBgEstimate} near any compact subset of $(I^0)^\circ$, where $\Hscrib^{1,k-1}$ is the same as $\Hb^k$. Let us define $\rho_0,\rho_I,\rho$ near $I^0$ as in equation~\eqref{EqCptASplCoords}. Fix $\eps>0$, and define for $\delta,\eta>0$ small
\[
  U := \{ \rho_I>\eps,\ \rho_0-\eta\rho_I<\delta \} \subset M,
\]
which for $\eps$ small is a neighborhood of any fixed compact subset of $M\cap (I^0)^\circ$. (Since $\rho_I$ is bounded from above, $U$ can be made to lie in any fixed neighborhood $\{\rho_0<\delta_0\}$ of $I^0$ provided $\delta$ and $\eta$ are sufficiently small.) In view of~\eqref{EqEinFinverse}, we have $G\in 4\pa_0\pa_1-r^{-2}\slG+\rho_0^{1-0}\Hb^\infty(U;S^2\,\Tsc^*\ol{\R^4})$, hence the calculation~\eqref{EqCptASplNullExpl} gives
\begin{equation}
\label{EqBgApproxMetric}
  G_\bop=G_{0,\bop} + \rho_0^{1-0}\Hb^\infty(U;S^2\,\Tb^*\ol{\R^4}),\ \ 
  G_{0,\bop} := 2\pa_{\rho_I}(\rho_I\pa_{\rho_I}-\rho_0\pa_{\rho_0})-\slG.
\end{equation}
Thus, $d\rho_I$ and $d(\rho_0-\eta\rho_I)$ are timelike in $U$ once we fix $\delta,\eta>0$ to be sufficiently small, and thus $U$ is bounded by $\Sigma\cap U$ and two spacelike hypersurfaces, $U^\pa_1=\{\rho_I=\eps\}$ and $U^\pa_2=\{\rho_0-\eta\rho_I=\delta\}$ (as well as by $U\cap\pa M$ at infinity), see Figure~\ref{FigBgi0}.

\begin{figure}[!ht]
\includegraphics{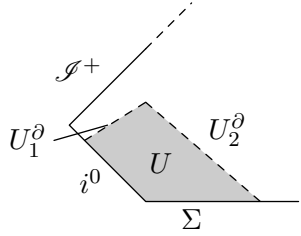}
\caption{The domain $U$ with its spacelike boundaries $U_1^\pa$, and $U_2^\pa$. We draw $I^0$ at a $45$ degree angle as the level sets of the chosen boundary defining function $\rho_0$ are approximately null (namely, $|d\rho_0|_{G_{0,\bop}}^2=0$). The level sets of $\rho_I$ are spacelike in $\rho_I>0$, but not uniformly so as $\rho_I\to 0$.}
\label{FigBgi0}
\end{figure}

\begin{prop}
\label{PropBgi0}
  Under the assumptions of Theorem~\usref{ThmBg}, we have
  \begin{equation}
  \label{EqBgi0Est}
    \| u \|_{\rho_0^{a_0}\Hb^k(U)} \leq C\bigl(\|u_0\|_{\rho_0^{a_0}\Hb^k(\Sigma\cap U)}+\|u_1\|_{\rho_0^{a_0}\Hb^{k-1}(\Sigma\cap U)} + \|f\|_{\rho_0^{a_0}\Hb^{k-1}(U)}\bigr).
  \end{equation}
\end{prop}
\begin{proof}
  We give a positive commutator proof of this standard estimate, highlighting the connection to the more often encountered fashion in which energy estimates are phrased \cite{DafermosRodnianskiLectureNotes}. Let us work in a trivialization $\Tb^*\ol{\R^4}\cong\ol{\R^4}\times\R^4$, and fix the fiber inner product to be the Euclidean metric in this trivialization. For proving the case $k=1$ of the lemma, we set $L:=L_h$; it will be convenient however for showing higher regularity to allow $L\in\Diffb^2+\rho_0^{1-0}\Hb^\infty\Diffb^2$ to be any principally scalar operator with $\sigma_{\bop,2}(L)=\half G_\bop$, acting on $\C^N$-valued functions for some $N\in\N$; we equip $\C^N$ with the standard Hermitian inner product. (One may also phrase the proof invariantly, i.e.\ not using global bundle trivializations, as we shall do in \S\S\ref{SsBgscri} and \ref{SsBgip} for conceptual clarity.)
  
  We will use a positive commutator argument: let $V=-\nabla\rho_I\in\Vb(\ol{\R^4})$, with $\nabla$ defined with respect to $g_\bop$; this is future timelike. For $\digamma>0$ chosen later, let $w=\rho_0^{-a_0}e^{\digamma\rho_I}$, and let $\one_U$ denote the characteristic function of $U$. Put $W=\one_U w^2 V$. Write $L=L_2+L_1$, where $L_2=\half\Box_{g_\bop}\otimes 1_{10\times 10}$, $L_1\in(\CI+\rho_0^{1-0}\Hb^\infty)\Diffb^1$. We then calculate the commutator
  \begin{equation}
  \label{EqBgi0Pairing}
    2 \Re \la \one_U w f, \one_U w V u\ra = 2 \Re \la L u, W u\ra = \la A u, u\ra + 2\Re \la \one_U wL_1 u, \one_U w V u \ra
  \end{equation}
  using the $L^2_\bop$ inner product, where $A=[L_2,W]+(W+W^*)L_2$. A simple calculation gives $\sigma_{\bop,2}(A)(\xi)=K_W(\xi,\xi)$, where
  \begin{equation}
  \label{EqBgi0KCurrent}
    K_W := -\half(\cL_W G_\bop + (\dv_{g_\bop}W)G_\bop).
  \end{equation}
  (The $K$-current is often given in its covariant form $\half(\cL_W g_\bop-(\dv_{g_\bop}W)g_\bop)$.) Therefore, $A=d^*K_W d$, since the principal symbols of both sides agree, hence the difference is a scalar\footnote{That is, it is a scalar operator tensored with the identity operator on $\C^N$.} first order b-differential operator which has real coefficients and is symmetric---thus is in fact of order zero, and since it annihilates constant vectors in $\C^N$, the difference vanishes. Differentiation of the exponential weight in $W$ upon evaluating $K_W$ will produce the main positive term into which all other terms can be absorbed. Indeed, the identity $\cL_{f V}G_\bop=f\cL_V G_\bop - 2\nabla f\otimes_s V$ for $V\in\Vb$ and $f\in\CI$ gives
  \begin{equation}
  \label{EqBgi0KCurrentRule}
    K_{f V} = T(\nabla f,V) + f K_V,
  \end{equation}
  where
  \[
    T(X,Y)=X\otimes_s Y-\half g_\bop(X,Y)G_\bop
  \]
  denotes the (abstract) energy-momentum tensor. (The energy-momentum tensor of a scalar wave $u$, say, is given by $T(X,Y)(d u,d u)$.) Therefore, $K_W=w^2(2\digamma\one_U K_0+\one_U K_1+K_2)$, where
  \[
    K_0 = T(\nabla\rho_I,V), \ \ 
    K_1 = -2 a_0 T\bigl(\tfrac{\nabla\rho_0}{\rho_0},V\bigr), \ \ 
    K_2 = T(\nabla\one_U,V).
  \]
  Since $\nabla\rho_I$ is past timelike, the main term $K_0$ is negative definite; $K_2$ has support in $\pa U\setminus\pa M$, so $\nabla\one_U$ being past timelike at $U^\pa_1$ and $U^\pa_2$, $K_2$ has the same sign as $K_0$ there. Lastly, $K_1$ has no definite sign, but can be absorbed into $K_0$ by choosing $\digamma>0$ large: indeed, $|T(\tfrac{\nabla\rho_0}{\rho_0},V)(\xi,\xi)| \leq -C T(\nabla\rho_I,V)$ for some constant $C$ depending only on $K$, since $g_\bop$ is a b-metric. Thus, \eqref{EqBgi0Pairing} gives the estimate
  \begin{equation}
  \label{EqBgi0Pairing2}
  \begin{split}
    \la \one_U w (-2\digamma K_0-K_1)d u, \one_U d u\ra &\leq 2(\|\one_U w V u\|^2+\|\one_U w L_1 u\|^2) \\
      &\quad\qquad + \|\one_U w f\|^2 + C\| \one_U w(d u_0, u_1) \|^2.
  \end{split}
  \end{equation}
  In order to control $u$ itself, consider the `commutator'
  \begin{equation}
  \label{EqBgi0UItself}
    2\Re\la \one_U w u,\one_U w V u\ra = 2\Re\la u, W u\ra = \la -\one_U w(\dv V)u, \one_U w u\ra - \la V(\one_U w^2)u,u\ra,
  \end{equation}
  where $V(\one_U w^2)=2\digamma\one_U w^2(V\rho_I)-2 a_0\one_U w^2\tfrac{V\rho_0}{\rho_0}+w^2 V(\one_U)$. In the first, main, term, $V\rho_I=-|d\rho_I|_{g_\bop}^2\leq -c_0<0$ has a strictly negative upper bound on $U$; the third term gives $\delta$-distributions at $\pa U$ with the same sign as this main term at $U^\pa_1$ and $U^\pa_2$ since $V$ is outward pointing there. Choosing $\digamma$ large to absorb the contribution of the second term, we get
  \[
    c_0\digamma\|\one_U w u\|^2 \leq \eta\digamma\|\one_U w u\|^2 + C_\eta\digamma^{-1}\| \one_U w V u\|^2 + C\| \one_U w u_0 \|^2,
  \]
  so fixing $\eta=c_0/2$, this gives $\|\one_U w u\|^2 \leq C\digamma^{-2}\|\one_U w V u\|^2+C_\digamma\|\one_U w u_0\|^2$. Adding $C'$ times this to \eqref{EqBgi0Pairing2} yields
  \begin{align*}
    \la \one_U w&(-2\digamma K_0-K_1)d u,\one_U d u\ra + C' \|\one_U w u\|^2 \\
      &\leq (2+C C'\digamma^{-2})\|\one_U w V u\|^2 + 2\|\one_U w L_1 u\|^2 \\
      &\qquad\qquad + C_\digamma\bigl(\|\one_U w f\|^2+(C+C')\|\one_U w(u_0,d u_0,u_1)\|^2\bigr).
  \end{align*}
  Fixing $C'$ sufficiently large and then $\digamma>0$ large, we can absorb the two first terms on the right into the first term on the left hand side, using that $-\digamma K_0>-2\digamma K_0-K_1$ for large $\digamma$. This gives \eqref{EqBgi0Est} for $k=1$.

  We now proceed by induction, assuming \eqref{EqBgi0Est} holds for some value of $k$ for all operators $L$ of the form considered above. If $L u=f$, let $X\in(\Diffb^1(\ol{\R^4}))^N$ denote an $N$-tuple of b-differential operators which generate $\Diffb^1(\ol{\R^4})$ over $\CI(\ol{\R^4})$; writing $[L,X]=L'\cdot X$ for $L'$ an $N$-tuple of operators in $(\CI+\rho_0^{1-0}\Hb^\infty)\Diffb^1$, we then have $(L-L')(X u)=X f$. Applying \eqref{EqBgi0Est} to this equation, we obtain the estimate \eqref{EqBgi0Est} for $L u=f$ itself with $k$ replaced by $k+1$.
\end{proof}

Given the structure of the operator $L_h$ on the manifold with corners $M$ as described in \S\ref{SsEinN}, it is natural to proceed proving the estimate \eqref{EqBgEstimate} in steps: in \S\ref{SsBgscri}, we propagate the control given by Proposition~\ref{PropBgi0} uniformly up to a neighborhood of the past corner $I^0\cap\scri^+$ of null infinity and thus into $(\scri^+)^\circ$. In \S\ref{SsBgip}, we prove the energy estimate uniformly up to $I^+$; the last estimate cannot be localized near the corner $\scri^+\cap I^+$ since typically limits of future-directed null-geodesic tending to $\scri^+\cap I^+$ pass through points in $I^+$ far from $\scri^+$.

\subsection{Estimate up to null infinity}
\label{SsBgscri}

We work near the past corner $I^0\cap\scri^+$ of the radiation field; recall the definition of the boundary defining functions $\rho_0$ and $\rho_I$ of $I^0$ and $\scri^+$ from~\eqref{EqCptASplCoords}, and let $\rho=r^{-1}$. At $\scri^+$, we need to describe $G_\bop$ more precisely than was needed near $(I^0)^\circ$; we make extensive use of the structures defined in \S\ref{SsCptScri}. Equations~\eqref{EqEinFinverse} and \eqref{EqCptASplNullExpl} give
\begin{equation}
\label{EqBgscriGb}
  G_\bop = G_{0,\bop} + G_{1,\bop} + \wt G_\bop,\ \ 
  G_{1,\bop}:=\rho^{-2}g_m^{-1}-G_{0,\bop}\in\CI(M;S^2\,\Tbeta M),
\end{equation}
with $G_{0,\bop}=2\pa_{\rho_I}(\rho_I\pa_{\rho_I}-\rho_0\pa_{\rho_0})-\slG\in \rho_I^{-1}\CI(M;S^2\,\Tbeta M+\rho_I\,S^2\,\Tb M)$ as before, and
\[
  \wt G_\bop \in \rho_0^{1+b_0}\rho_I^{-1+b'_I}\Hb^\infty(M;S^2\,\Tbeta M+\rho_I\,S^2\,\Tb M).
\]
Dually, equation~\eqref{EqCptASplNullDual} gives
\begin{equation}
\label{EqBgscrig0Space}
  g_\bop \in (\CI+\rho_0^{1+b_0}\rho_I^{b'_I}\Hb^\infty)(M;S^2(\Tbeta M)^\perp + \rho_I\,S^2\,\Tb^*M)
\end{equation}
where the smooth term is $\rho^2 g_m=-2\rho_I\tfrac{d\rho_0}{\rho_0}\bigl(\tfrac{d\rho_0}{\rho_0}+\tfrac{d\rho_I}{\rho_I}\bigr) - \slg + \rho_I^2\,\CI(M;S^2\,\Tb^*M)$.

Fix $\beta\in(0,b'_I)$. For small $\eps>0$, we define the domain
\begin{equation}
\label{EqBgscriDom}
  U_\eps := \{ \rho_I < \eps,\ \rho_0-\rho_I^\beta<1 \} \subset M,\ \ 
  U_\eps^0 := U_\eps \cap \{ \half\eps < \rho_I < \eps \},
\end{equation}
see Figure~\ref{FigBgscri}. Thus, $U_\eps$ is bounded by $I^0$, $\scri^+$, $\{\rho_I=\eps\}$, and $U^\pa_\eps=\{\rho_0-\rho_I^\beta=1,\ \rho_I<\eps\}$. At $U^\pa_\eps$, we use \eqref{EqBgApproxMetric} and \eqref{EqBgscriGb} to compute
\begin{equation}
\label{EqBgscriFinalDefFn}
  |d(\rho_0-\rho_I^\beta)|_{G_\bop}^2 \in 2\beta\rho_I^{-1+\beta}(\rho_0+\beta\rho_I^\beta) + \rho_I^{2\beta}\CI + \rho_0^{1-0}\rho_I^{-1+b'_I}\Hb^\infty,
\end{equation}
hence $U_\eps^\pa$ is timelike for small enough $\eps$. As in the proof of Proposition~\ref{PropBgi0}, the main term is the $K$-current of a timelike vector field with suitable weights:

\begin{lemma}
\label{LemmaBgscriKCurrent}
  Fix $c_V\in\R$, let $W:=\rho_0^{-2 a_0}\rho_I^{-2 a_I} V$, and $V:=-(1+c_V)\rho_I\pa_{\rho_I}+\rho_0\pa_{\rho_0}$, then
  \begin{equation}
  \label{EqBgscriKCurrent}
  \begin{split}
    K_W &\in \rho_0^{-2 a_0}\rho_I^{-2 a_I-1}\Bigl(2 a_I(\rho_0\pa_{\rho_0}-\rho_I\pa_{\rho_I})^2 - 2 c_V(a_0-a_I)(\rho_I\pa_{\rho_I})^2 \\
      &\hspace{14em} -\half\bigl(1+2(a_0-a_I)+c_V(1-2 a_I)\bigr)\rho_I\slG\Bigr) \\
      &\quad + \rho_0^{-2 a_0}\rho_I^{-2 a_I}(\CI+\rho_0^{1+b_0}\rho_I^{-1+b'_I}\Hb^\infty)(M;S^2\,\Tbeta M + \rho_I\,S^2\,\Tb M).
  \end{split}
  \end{equation}
  Furthermore,
  \begin{equation}
  \label{EqBgscriKCurrentDiv}
  \begin{split}
    \dv_{g_\bop}W &\in -2\rho_0^{-2 a_0}\rho_I^{-2 a_I}\bigl(1+2(a_0-a_I)+c_V(1-2 a_I) \bigr) \\
      &\qquad + \rho_0^{-2 a_0}\rho_I^{-2 a_I+1}(\CI+\rho_0^{1+b_0}\rho_I^{-1+b'_I}\Hb^\infty)
   \end{split}
   \end{equation}
\end{lemma}

Here, $\rho_I^{-1}|V|_{g_\bop}^2\in 2 c_V+\rho_I\,\CI+\rho_0^{1+b_0}\rho_I^{b'_I}\Hb^\infty$, so $V$ is timelike for $c_V>0$. This calculation also shows that the level sets of $\rho_I$ are spacelike in $U_\eps$. The term $\rho_I K_W(d u,d u)$ will provide control of $u$ in $\rho_0^{a_0}\rho_I^{a_I}\Hscri^1$ (modulo control of $|u|^2$ itself, which we obtain by integration), similarly to~\eqref{EqISysLinEqIEst}.

\begin{rmk}
\label{RmkBgscriKCurrent0}
  For easier comparison with energy estimates expressed in standard coordinates on $\R^4$, consider the special case $m=0$, so $\rho_0=(r-t)^{-1}$ and $\rho_I=(r-t)/r$; then $\rho_0\pa_{\rho_0}=-(r\pa_r+t\pa_t)$ (scaling) and $\rho_I\pa_{\rho_I}=-r(\pa_t+\pa_r)$ (weighted outgoing derivative). Thus, the multiplier vector field $W$ in $t<r$, $r>0$, equals
  \[
    W = r^{2 a_I+1}(r-t)^{2(a_0-a_I)}\bigl(c_V\pa_r + (c_V+\tfrac{r-t}{r})\pa_t\bigr).
  \]
\end{rmk}

\begin{proof}[Proof of Lemma~\usref{LemmaBgscriKCurrent}]
  Recall that $K_W=\half(\pi-\half(\tr_{g_\bop}\pi)G_\bop)$, $\pi:=-\cL_W G_\bop$. Since $V\in\cM_{\ul\C}$, Lemma~\ref{LemmaCptScriModule}\eqref{ItCptScriModuleComm} shows that $\wt\pi:=-\cL_W \wt G_\bop$, expressed using vector field commutators, lies in the remainder space in \eqref{EqBgscriKCurrent}; using \eqref{EqBgscrig0Space}, this implies $\tr_{g_\bop}\wt\pi \in \rho_0^{-2 a_0+1+b_0}\rho_I^{-2 a_I+b'_I}\Hb^\infty$, so $(\tr_{g_\bop}\wt\pi)G_\bop$ also lies in the remainder space. Similarly, $G_{1,\bop}$ contributes a (weighted) smooth remainder term to $K_W$. Lastly, for $\pi_0=-\cL_W G_{\bop,0}$, the term $\half(\pi_0-\half(\tr_{g_\bop}\pi_0)G_\bop)$ contributes the main term, i.e.\ the first line of \eqref{EqBgscriKCurrent} after a short calculation, as well as two more error terms, one from $\wt G_\bop$, the other coming from the nonsmooth remainder term in \eqref{EqBgscrig0Space}. The calculation~\eqref{EqBgscriKCurrentDiv} drops out as a by-product of this, and can also be recovered by $\dv_{g_\bop}W=-\tr_{g_\bop}K_W$.
\end{proof}

In order to get the sharp weights\footnote{As explained before in the context of the weight at $I^+$, this is not necessary, but easy to accomplish here without lengthy calculations.} for the decaying components $\pi_0 u$ of $u$ at $\scri^+$ in Theorem~\ref{ThmBg}, we need to exploit the sign of the leading subprincipal part of $L_h$ at $\scri^+$, given by the term involving $\rho^{-1}A_h\pa_1$ in Lemma~\ref{LemmaEinNscri}, in the decoupled equation for $\pi_0 u$, see \eqref{EqEinNscriPi0} for the model. We thus prove:
\begin{lemma}
\label{LemmaBgscriImprovedK}
  Define $W=\rho_0^{-2 a_0}\rho_I^{-2 a'_I}(\rho_0\pa_{\rho_0}-(1+c_V)\rho_I\pa_{\rho_I})$ similarly to previous lemma. Let $\gamma\in\R$, and fix $a_0,a'_I\in\R$ such that $a'_I<\min(\gamma,a_0)$. Then for small $c_V>0$, there exists a constant $C>0$ such that
  \begin{equation}
  \label{EqBgscriImprovedK}
    K_W - 2\gamma W\otimes_s \rho^{-1}\pa_1 \leq -C \rho_0^{-2 a_0}\rho_I^{-2 a'_I-1}\bigl((\rho_I\pa_{\rho_I})^2+(\rho_0\pa_{\rho_0})^2+\rho_I\slG\bigr),
  \end{equation}
  in the sense of quadratic forms, in $U_\eps$, $\eps>0$ small.
\end{lemma}
\begin{proof}
  Using the expression~\eqref{EqCptASplNullExpl} for $\rho_0^{-1}\rho_I^{-1}\pa_1$, we have
  \begin{align*}
    &\rho_0^{2 a_0}\rho_I^{2 a'_I+1}W\otimes_s\rho^{-1}\pa_1 \\
    &\qquad\in(\rho_0\pa_{\rho_0}-\rho_I\pa_{\rho_I})^2-c_V\rho_I\pa_{\rho_I}\otimes_s(\rho_0\pa_{\rho_0}-\rho_I\pa_{\rho_I}) + \rho_I\,\CI(M;\Tbeta M)
  \end{align*}
  We can then calculate the leading term of $\rho_0^{2 a_0}\rho_I^{2 a'_I+1}$ times the left hand side of \eqref{EqBgscriImprovedK} by completing the square:
  \begin{align*}
    &-2(\gamma-a'_I)\Bigl(\rho_0\pa_{\rho_0}-\rho_I\pa_{\rho_I}-\frac{\gamma c_V}{2(\gamma-a'_I)}\rho_I\pa_{\rho_I}\Bigr)^2 - c_V\Bigl(a_0-a'_I-\frac{\gamma^2 c_V}{2(\gamma-a'_I)}\Bigr)(\rho_I\pa_{\rho_I})^2 \\
    &\qquad - \half \bigl(1+2(a_0-a'_I)+c_V(1-2 a'_I)\bigr)\rho_I\slG.
  \end{align*}
  The first term is the negative of a square, and so is the second term if we choose $c_V>0$ sufficiently small; reducing $c_V$ further if necessary, the coefficient of the last term is negative as well, finishing the proof.
\end{proof}

\begin{rmk}
\label{RmkBgscriDivSign}
  For the value of $c_V$ determined in the proof, we have $\dv_{g_\bop}W\leq -C\rho_0^{-2 a_0}\rho_I^{-2 a'_I}$ near $\scri^+$ by inspection of the expression~\eqref{EqBgscriKCurrentDiv}.
\end{rmk}

Suppose now $u$ solves $L_h u=f$ with initial data $(u_0,u_1)$ as in \eqref{EqBgEquation}. Note that the estimates~\eqref{EqBgCompact} and \eqref{EqBgi0Est} provide control of $u$ on $U_\eps^0$ for any choice of $\eps>0$; thus, it suffices to prove an estimate in $U_\eps$ for any arbitrary but fixed $\eps>0$. Let $\chi\in\CI(\R)$ be a cutoff, $\chi(\rho_I)\equiv 1$ for $\rho_I<\eps/4$ and $\chi(\rho_I)\equiv 0$ for $\rho_I>\eps/2$, and put $\wt u:=\chi u$, then $\wt u$ solves the forward problem
\begin{equation}
\label{EqBgscriCutoffEqn}
  L_h\wt u = \wt f := \chi f + [L_h,\chi]u
\end{equation}
in $U_\eps$, with $\|\wt f\|_{\rho_0^{a_0}\rho_I^{a_I-1}\Hb^{k-1}(U_\eps)}+\|\pi_0\wt f\|_{\rho_0^{a_0}\rho_I^{a'_I-1}\Hb^{k-1}(U_\eps)}$ controlled by the corresponding norm of $f$ plus the right hand sides of \eqref{EqBgCompact} and \eqref{EqBgi0Est}. (Use Lemma~\ref{LemmaEinNi0p} to compute the rough form of the commutator term.) Note that $\wt u=\chi u$ is the \emph{unique} solution of $L_h\wt u=\wt f$ vanishing in $\rho_I>\half\eps$. See Figure~\ref{FigBgscri}.

\begin{figure}[!ht]
\includegraphics{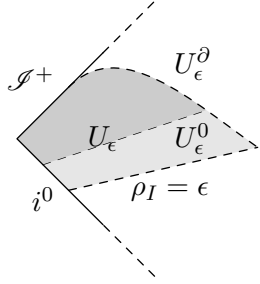}
\caption{The domain $U_\eps$ and its subdomain $U_\eps^0$ where we have a priori control of $u$, allowing us to cut off and study equation~\eqref{EqBgscriCutoffEqn} instead.}
\label{FigBgscri}
\end{figure}

Thus, the estimate~\eqref{EqBgEstimate} of $u$ in $U_\eps$ is a consequence of the following result (dropping the tilde on $\wt u$ and $\wt f$):

\begin{prop}
\label{PropBgscri}
 For weights $b_0,b'_I,b_I,a_0,a'_I,a_I$, and for $h\in\cX^\infty$, small in $\cX^3$, as in Theorem~\usref{ThmBg}, and for $k\in\N$, let $f\in\rho_0^{a_0}\rho_I^{a_I-1}\Hb^{k-1}(U_\eps)$, $\pi_0 f\in\rho_0^{a_0}\rho_I^{a'_I-1}\Hb^{k-1}(U_\eps)$; suppose $f$ vanishes in $\rho_I>\half\eps$. Let $u$ denote the unique forward solution of $L_h u=f$. Then
  \begin{equation}
  \label{EqBgscri}
  \begin{split}
    &\|u\|_{\rho_0^{a_0}\rho_I^{a_I}\Hscrib^{1,k-1}(U_\eps)} + \|\pi_0 u\|_{\rho_0^{a_0}\rho_I^{a'_I}\Hscrib^{1,k-1}(U_\eps)} \\
    &\qquad\qquad \leq C\Bigl(\|f\|_{\rho_0^{a_0}\rho_I^{a_I-1}\Hb^{k-1}(U_\eps)} + \|\pi_0 f\|_{\rho_0^{a_0}\rho_I^{a'_I-1}\Hb^{k-1}(U_\eps)}\Bigr).
  \end{split}
  \end{equation}
\end{prop}
\begin{proof}
  The idea is to exploit the decoupling of the leading terms of $L_h$ at $\scri^+$ given by equations~\eqref{EqEinNscriPi0}--\eqref{EqEinNscriPi11}: this allows us to prove an energy estimate (for the case $k=1$)
  \begin{equation}
  \label{EqBgscri1}
    \|\pi_0 u\|_{\rho_0^{a_0}\rho_I^{a'_I}\Hscri^1} \leq C\bigl( \|\pi_0 f\|_{\rho_0^{a_0}\rho_I^{a'_I-1}L^2_\bop} + \|\pi_0^c u\|_{\rho_0^{a_0}\rho_I^{a_I-\delta}\Hscri^1}\bigr),
  \end{equation}
  where $\delta>0$ fixed such that
  \begin{equation}
  \label{EqBgscriDelta}
    a'_I-b'_I<a_I-\delta,\quad
    a_I<a'_I-\delta.
  \end{equation}
  The estimate~\eqref{EqBgscri1} contains $\pi_0^c u$ as an error term, but with a \emph{weaker weight} due to the decay of the coefficients of the error term $\wt L_h$---which is dropped in \eqref{EqEinNscriPi0}. On the other hand, $\pi_0 u$ couples into $\pi_0^c u$ via at most logarithmic terms, hence we can prove
  \begin{equation}
  \label{EqBgscri2}
    \|\pi_0^c u\|_{\rho_0^{a_0}\rho_I^{a_I}\Hscri^1} \leq C\bigl( \|\pi_0^c f\|_{\rho_0^{a_0}\rho_I^{a_I-1}L^2_\bop} + \|\pi_0 u\|_{\rho_0^{a_0}\rho_I^{a_I+\delta}\Hscri^1}\bigr)
  \end{equation}
  Close to $\scri^+$, the last term in the estimate~\eqref{EqBgscri1}, resp.\ \eqref{EqBgscri2}, is controlled by a \emph{small} constant times the left hand side of~\eqref{EqBgscri2}, resp.\ \eqref{EqBgscri1}, hence summing the two estimates yields the full estimate~\eqref{EqBgscri}. The proof of~\eqref{EqBgscri2} and its higher regularity version will itself consist of two steps, corresponding to the weak null structure expressed by the decoupling of~\eqref{EqEinNscriPi11c} and \eqref{EqEinNscriPi11}.

  All energy estimates will use the vector field
  \[
    V_1=-(1+c_V)\rho_I\pa_{\rho_I}+\rho_0\pa_{\rho_0}
  \]
  from Lemma~\ref{LemmaBgscriKCurrent}, with $c_V>0$ chosen according to Lemma~\ref{LemmaBgscriImprovedK}. Denote $u_0:=\pi_0 u$, $u_{1 1}:=\pi_{1 1}u$, $u_{1 1}^c:=\pi_{1 1}^c u$, and $u_0^c:=\pi_0^c u=u_{1 1}+u_{1 1}^c$. We expand $L_h u=f$ as
  \begin{subequations}
  \begin{align}
  \label{EqBgscriEq0}
    \pi_0 L_h\pi_0 u_0 &= \pi_0 f - \pi_0 L_h\pi_0^c u_0^c , \\
  \label{EqBgscriEq11c}
    \pi_{1 1}^c L_h\pi_{1 1}^c u_{1 1}^c &= \pi_{1 1}^c f - \pi_{1 1}^c L_h\pi_0 u_0 - \pi_{1 1}^c L_h\pi_{1 1} u_{1 1} , \\
  \label{EqBgscriEq11}
    \pi_{1 1}L_h\pi_{1 1} u_{1 1} &= \pi_{1 1}f - \pi_{1 1}L_h\pi_0 u_0 - \pi_{1 1}L_h\pi_{1 1}^c u_{1 1}^c .
  \end{align}
  \end{subequations}
  Here, we regard $\beta^*K_0\to M$ as a vector bundle in its own right, and $u_0$ as a section of $\beta^*K_0$: the inclusion $K_0\hra S^2\,\Tsc^*\ol{\R^4}$ and the structures on the latter bundle induced by $g$ or $g_m$ play no role; likewise for $K_{1 1}$ and $K_{1 1}^c$.
  
  Starting the proof of the estimate~\eqref{EqBgscri1} using equation~\eqref{EqBgscriEq0}, let us abbreviate $L:=\pi_0 L_h\pi_0$. By Lemma~\ref{LemmaEinNscri} and recalling the definition of $A_\CD$ from equation~\eqref{EqEinNscriPi0}, we have
  \begin{equation}
  \label{EqBgscriLpi0}
    L = L^0 + \wt L,\ \ L^0 = -2\rho^{-2}\pa_0\pa_1 + L^0_1,\ L^0_1=-\rho^{-1}A_\CD\pa_1,
  \end{equation}
  with $\wt L$ lying in the same space as $\wt L_h$ in \eqref{EqEinNscri} with $\beta^*S^2$ replaced by $\beta^*K_0$. Here, $L^0_1$ denotes a fixed representative in $\rho_I^{-1}\cdot{}^0\cM_{\beta^*K_0}$, defined by fixing a representative of $\rho_0^{-1}\pa_1\in {}^0\cM_{\beta^*K_0}$, see equation~\eqref{EqCptScriNull}, in the image space of Lemma~\ref{LemmaCptScriModule}\eqref{ItCptScriModuleBdl}. Let $w=\rho_0^{-a_0}\rho_I^{-a'_I}$; let further $\one_{U_\eps}$ denote the characteristic function of $U_\eps$. Fix $V\in{}^0\cM_{\beta^*K_0}$, with scalar principal symbol equal to that of $V_1$. Let
  \[
    W:=\one_{U_\eps}W^\circ,\ \ W^\circ:=w^2 V.
  \]
  Fix a positive definite fiber inner product $B\colon\Tb M\to\Tb^*M$ on $\Tb M$, a connection $d\in\Diff^1(\ol{\R^4};K_0,T^*\ol{\R^4}\otimes K_0)$ on $K_0$, and a positive definite fiber metric $k_0$ on $K_0$ with respect to which $A_\CD=A_\CD^*$; note here that $A_\CD$ is constant on the fibers of $\scri^+$, hence indeed descends to an endomorphism of $K_0|_{S^+}$. Let $\la\cdot,\cdot\ra$ denote the $L^2$ inner product with respect to $k_0$ and the density $|d g_\bop|\sim|\frac{d\rho_0}{\rho_0}d\rho_I\,d\slg|$; defining the b-density $d\mu_\bop:=\rho_I^{-1}|d g_\bop|\sim|\frac{d\rho_0}{\rho_0}\frac{d\rho_I}{\rho_I}d\slg|$ to define $L^2_\bop(M)$, we then have
  \begin{equation}
  \label{EqBgscriL2Relation}
    \la u,v\ra = \la \rho_I u,v\ra_{L^2_\bop}.
  \end{equation}
  We shall evaluate
  \begin{equation}
  \label{EqBgscriPairing}
  \begin{split}
    &2\Re\la w L u_0, \one_{U_\eps}w V u_0\ra = \la \cC u_0, u_0\ra, \\
    &\qquad \cC := L^*W+W^*L = [L,W]+(W+W^*)L+(L^*-L)W.
  \end{split}
  \end{equation}
  Let $K_W$ denote the current associated with the scalar principal part of $W$, see~\eqref{EqBgi0KCurrent}, now understood as taking values in the bundle $S^2\,\Tb M\otimes\End(\beta^*K_0)$, acting on $\beta^*K_0$ by scalar multiplication. While $K_W$ provides positivity of $\cC$ near $\scri^+$ for suitable weights by Lemma~\ref{LemmaBgscriKCurrent}---in particular, this would require $a'_I<0$---we will show around~\eqref{EqBgscriC11Main} below how to obtain a better result by exploiting the sign of $A_\CD$ entering through $(L^*-L)W$.

  In the proof of Proposition~\ref{PropBgi0}, where we worked in a global trivialization, all terms of $W$ and $L$ other than the top order ones could be treated as error terms; we show that the same is true here by patching together estimates obtained from calculations in local coordinates and trivializations. Thus, let $\{\cU_j\}$ be a covering of a neighborhood of $S^+$ containing $U_\eps$ by open sets on which $K_0$ is trivial, and let $\{\chi_j\}$, $\chi_j\in\CIc(\cU_j)$, denote a subordinate partition of unity; let $\wt\chi_j\in\CIc(\cU_j)$, $\wt\chi_j\equiv 1$ on $\supp\chi_j$. Fix trivializations $(K_0)|_{\cU_j}\cong\cU_j\times\C^4$ and the induced trivializations of $\beta^*K_0$. Write
  \[
    L=L_{j,2} + L_{j,1},\ \ 
    W=W_{j,1} + W_{j,0},
  \]
  where $L_{j,2}:=\half\Box_{g_\bop}$ acts component-wise as the scalar wave operator and $L_{j,1}$ is a first order operator, while $W_{j,1}:=\one_{U_\eps}w^2 V_1$ acts component-wise, and $W_{j,0}\in\one_{U_\eps}w^2\rho_I\,\CI(\cU_j,\Tbeta M)$, with the extra factor of $\rho_I$ due to the choice of $V$. On $(K_0)|_{\cU_j}$, let moreover $d_j$ denote the standard connection, given component-wise as the exterior derivative on functions, and let $k_j$ denote the standard Hermitian fiber metric; we denote adjoints with respect to $k_j$ by $\dag$. Now,
  \begin{equation}
  \label{EqBgscriPairingLoc}
    \la\cC u_0,u_0\ra=\sum\la\cC_j u_0,\chi_j u_0\ra,
  \end{equation}
  where
  \[
    \cC_j = \sum_{k,\ell} \cC_{j,k\ell},\ \ 
    \cC_{j,k\ell} := L_{j,k}^*W_{j,\ell} + W_{j,\ell}^*L_{j,k}.
  \]
  The usual calculation in the scalar case, see the discussion around~\eqref{EqBgi0Pairing}, gives
  \[
    \ul\cC_{j,2 1}:=L_{j,2}^\dag W_{j,1}+W_{j,1}^\dag L_{j,2} = d_j^\dag B K_W d_j,
  \]
  so
  \begin{align}
  \label{EqBgscriC21}
    \la \cC_{j,2 1}u_0,\chi_j u_0\ra &= \la d^* B K_W d u_0, \chi_j u_0\ra + \la (\cC_{j,2 1}-\ul\cC_{j,2 1})u_0,\chi_j u_0\ra \\
      &\quad + \la(d_j^\dag-d_j^*)B K_W d_j u_0,\chi_j u_0\ra + \la(d_j^* B K_W d_j-d^* B K_W d)u_0,\chi_j u_0\ra.\nonumber
  \end{align}
  Summing the first term over $j$ yields
  \begin{equation}
  \label{EqBgscriC21Main}
    \int_{U_\eps} \rho_I K_{W^\circ}(d u_0,d u_0)\,d\mu_\bop + \int T(\rho_I\nabla\one_{U_\eps},W^\circ)(d u_0,d u_0)\,d\mu_\bop
  \end{equation}
  upon application of the formula~\eqref{EqBgi0KCurrentRule}. The first summand---after adding the term~\eqref{EqBgscriC11Main} below---is negative definite, controlling derivatives of $u_0$ as in \eqref{EqBgscri1}; the second term gives a contribution of the same sign: we have
  \[
    T(\rho_I\nabla\one_{U_\eps},W^\circ) = \delta_{U_\eps^\pa}\otimes w^2 T^\pa,
  \]
  with $T^\pa\leq 0$ since $-\nabla\one_{U_\eps}$ and $W^\circ$ are future causal. The remaining terms in~\eqref{EqBgscriC21} are error terms: the second term is equal to
  \[
    \la W_{j,1}u_0,(L_{j,2}-L_{j,2}^{\dag*})\chi_j u_0\ra + \la L_{j,2} u_0, (W_{j,1}-W_{j,1}^{\dag*})\chi_j u_0\ra.
  \]
  Now, $k_0$ and $k_j$ are related by $k_j(\cdot,\cdot)=k_0(\wt Q_j\cdot,\wt Q_j\cdot)$, with $\wt Q_j\in\CI(\cU_j;\End(K_0))$ invertible, and then $A^\dag=Q_j^{-1}A^*Q_j$ for $Q_j:=\wt Q_j^*\wt Q_j$ when $A$ is an operator acting on sections of $K_0$. Thus, $W_{j,1}-W_{j,1}^{\dag *}=[W_{j,1},Q_j^*](Q_j^{-1})^*$. On $M$, the constancy of $Q_j$, and hence of $Q_j^*$, along the fibers of $\beta$ and $V_1\in{}^0\cM$ give the extra vanishing factor $\rho_I$ in
  \[
    W_{j,1}-W_{j,1}^{\dag *} = \one_{U_\eps}\rho_I w^2 q_{j,1},\ \ 
    q_{j,1}\in \CI(\beta^{-1}(\cU_j);\End(\beta^*K_0)),
  \]
  with $q_{j,1}$ only depending on $Q_j$. Similarly, $L_{j,2}-L_{j,2}^{\dag *}=[L_{j,2},Q_j^*](Q_j^{-1})^*$; using Lemma~\ref{LemmaEinNScalarBox} and $[\pa_1,Q_j^*]\in\rho\,\CI$, we find (replacing the weight $-0$ there by $-1/2+b'_I$ for definiteness)
  \begin{equation}
  \label{EqBgscriLj2Adj}
    L_{j,2}-L_{j,2}^{\dag *} \in \rho_0^{1+b_0}\rho_I^{-1+b'_I}\Hb^\infty(M)\cM_{\beta^*K_0} + (\CI+\rho_0^{1+b_0}\rho_I^{-1/2+b'_I}\Hb^\infty)\Diffb^1(\cU_j;\beta^*K_0).
  \end{equation}
  Writing $L_{j,2}u_0=L u_0-L_{j,1}u_0$ and using the relationship~\eqref{EqBgscriL2Relation}, we thus get
  \begin{equation}
  \label{EqBgscriC21DiffEst}
  \begin{split}
    |\la(\cC_{j,2 1}-\ul\cC_{j,2 1})u_0,\chi_j u_0\ra| &\leq C\|\wt\chi_j w V_1 u_0\|_{L^2_\bop} \bigl(\|\wt\chi_j\rho_I^{b'_I}w u_0\|_{\Hbeta^1} + \|\wt\chi_j\rho_I^{1/2+b'_I}w u_0\|_{\Hb^1}\bigr) \\
      &\quad + C\bigl(\|\wt\chi_j\rho_I w L u_0\|_{L^2_\bop} + \|\wt\chi_j\rho_I w L_{j,1}u_0\|_{L^2_\bop}\bigr) \| \wt\chi_j \rho_I w u_0\|_{L^2_\bop},
  \end{split}
  \end{equation}
  where the norms are taken on $U_\eps$. Note that in all terms on the right, at least one factor comes with an extra decaying power of $\rho_I$ relative to $w u_0$, hence is \emph{small} compared to $w u_0$ if we localize to $U_\eps$ for small $\eps>0$, i.e.\ to a small neighborhood of $\scri^+$. Next, we combine Lemmas~\ref{LemmaCptScriConn} and~\ref{LemmaBgscriKCurrent} in the same fashion as in the proof of Lemma~\ref{LemmaCptScriConnDKD} to estimate the last two terms of \eqref{EqBgscriC21} by
  \begin{equation}
  \label{EqBgscriC21DiffEst2}
  \begin{split}
    &C\bigl( \| \wt\chi_j \rho_I w u_0 \|_{\Hb^1} \| \chi_j w u_0 \|_{L^2_\bop} \\
      &\qquad + (\| \wt\chi_j\rho_I w T^\pa(d u_0,d u_0)^{1/2} \|_{L^2_\bop(U_\eps^\pa)}+\|\wt\chi_j\rho_I w u_0\|_{L^2_\bop(U_\eps^\pa)}) \| \chi_j w u_0\|_{L^2_\bop(U_\eps^\pa)} \bigr);
  \end{split}
  \end{equation}
  where the second term in the inner parenthesis comes from the pointwise estimate $T^\pa(d_j u_0,\allowbreak d_j u_0)^{1/2}\leq C(T^\pa(d u_0,d u_0)^{1/2}+|u_0|)$.
  
  The next interesting term in \eqref{EqBgscriPairingLoc} is $\cC_{j,1 1}+\cC_{j,1 0}$, specifically the term coming from the `constraint damping part' $L^0_1$ defined in~\eqref{EqBgscriLpi0}. In a local trivialization, $L^0_1=-\rho^{-1}A_\CD\pa_1+L^0_{1,j}$, $L^0_{1,j}\in\CI(\cU_j)$ (using the discussion around~\eqref{EqCptScriNull} for this membership), so we have the pointwise equality
  \begin{align*}
    2\Re k_0(W u_0,L_1^0\chi_j u_0) &= -2\Re k_0(W_{j,1}u_0,\rho^{-1}A_\CD\pa_1 \chi_j u_0) \\
      &\qquad + 2\Re k_0(W_{j,0}u_0,L_1^0\chi_j u_0) + 2\Re k_0(W_{j,1} u_0,L^0_{1,j}\chi_j u_0);
  \end{align*}
  letting
  \begin{align*}
    K' &:= -2 w^2 (V_1\otimes_s \rho^{-1}\pa_1) \otimes A_\CD \\
       &\in \rho_0^{-2 a_0}\rho_I^{-2 a'_I-1}\CI\bigl(U_\eps; (S^2\,\Tbeta M+\rho_I\,S^2\,\Tb M)\otimes\End(\beta^*K_0)\bigr),
  \end{align*}
  the first term integrates to $\int \rho_I K'(d_j u_0,d_j\chi_j u_0)\,d\mu_\bop$, which equals
  \begin{equation}
  \label{EqBgscriC11Main}
    \int\rho_I K'(d u_0,d\chi_j u_0)\,d\mu_\bop
  \end{equation}
  plus error terms of the same kind as in the second line of~\eqref{EqBgscriC21}. The extra factor of $\rho_I$ in $W_{j,0}$ and $L^0_{1,j}$ (as compared to $W_{j,1}$ and $L^0_1$) allows the remaining two terms to be estimated in a fashion similar to~\eqref{EqBgscriC21DiffEst}. The remaining contributions to $\cC_{j,1 1}+\cC_{j,1 0}$ are error terms coming from $\wt L$ in~\eqref{EqBgscriLpi0} and can be estimated as in \eqref{EqBgscriC21DiffEst}.

  Lastly, the terms of \eqref{EqBgscriPairingLoc} involving $\cC_{j,2 0}$ can be rewritten and estimated as follows:
  \begin{align*}
    \bigl|2\Re\la (L&-L_{j,1})u_0, W_{j,0}\chi_j u_0\ra + \la W_{j,0}u_0,[L_{j,2},\chi_j]u_0\ra\bigr| \\
    & \leq 2\bigl(\|\rho_I w L u_0\|_{L^2_\bop}+\|\wt\chi_j\rho_I w L_{j, 1}u_0\|_{L^2_\bop}\bigr) \|\chi_j\rho_I w u_0 \|_{L^2_\bop} \\
    &\qquad + \|\wt\chi_j\rho_I w u_0\|_{L^2_\bop}\bigl(\|\wt\chi_j\rho_I^{b'_I}w u_0\|_{\Hbeta^1}+\|\wt\chi_j\rho_I^{1/2+b'_I}w u_0\|_{\Hb^1}\bigr);
  \end{align*}
  the norms are taken on $U_\eps$, and we use that $[L_{j,2},\chi_j]$ lies in the same space as \eqref{EqBgscriLj2Adj}. We note that by Lemma~\ref{LemmaEinNscri}, the terms involving $L_{j,1}$ here and in \eqref{EqBgscriC21DiffEst} can be estimated by
  \[
    \|\wt\chi_j\rho_I w L_{j,1}u_0\|_{L^2_\bop} \leq C\bigl(\| \bar\chi_j\rho_I^{b'_I}w u_0\|_{\Hbeta^1}+\| \bar\chi_j\rho_I^{1/2+b'_I} w u_0\|_{\Hb^1}\bigr),
  \]
  where $\bar\chi_j\in\CIc(\cU_j)$ is identically $1$ on $\supp\wt\chi_j$.

  This finishes the evaluation of \eqref{EqBgscriPairing}; we now turn to the estimate of $w u_0$ itself by $w V u$. As in the proof of Proposition~\ref{PropBgi0}, this follows from integration along $V$. Concretely, we consider a `commutator' as in \eqref{EqBgi0UItself}, that is,
  \begin{equation}
  \label{EqBgscriCommUItself}
    2\Re\la \one_{U_\eps}w V u_0, \rho_I^{-1}w u_0\ra = -\la \rho_I^{-1}\dv_{g_\bop}(\one_{U_\eps}w^2 V_1)u_0, u_0\ra + E,
  \end{equation}
  where $|E|\leq C\|w u_0\|_{L^2_\bop}\|\rho_I w u_0\|_{L^2_\bop}$ by Lemma~\ref{LemmaCptScriWWstar}. Using the negativity of the divergence near $\scri^+$ due to Lemma~\ref{LemmaBgscriKCurrent} and Remark~\ref{RmkBgscriDivSign}, and that $V_1$ is outward pointing at $U_\eps^\pa$, so $V_1(\one_{U_\eps})$ is a \emph{negative} $\delta$-distribution at $U_\eps^\pa$, we get
  \begin{equation}
  \label{EqBgscriUEst}
    \|w u_0\|_{L^2_\bop(U_\eps)} + \|w u_0\|_{L^2_\bop(U_\eps^\pa)} \leq C\|w V u_0\|_{L^2_\bop(U_\eps)};
  \end{equation}
  recall here that $u_0$ vanishes in $\rho_I>\tfrac{\eps}{2}$, hence there is no a priori control term on the right. Subtracting this estimate from~\eqref{EqBgscriPairing} (the latter having main terms which are \emph{negative} definite in $d u_0$), the main terms are the left hand side of \eqref{EqBgscriUEst} and $\int_{U_\eps} \rho_I(K'+K_{W^\circ})(d u_0,d u_0)\,d\mu_\bop$ from~\eqref{EqBgscriC21Main} and \eqref{EqBgscriC11Main}. By Lemma~\ref{LemmaBgscriImprovedK}, they control $\| w u_0 \|_{\Hscri^1(U_\eps)}$: the error terms in $U_\eps$ can be absorbed into this, while those at $U_\eps^\pa$ in~\eqref{EqBgscriC21DiffEst2} can be absorbed into the second terms of~\eqref{EqBgscriC21Main} and \eqref{EqBgscriUEst}, due to the extra decaying weights on at least one of the factors in each of those error terms as discussed after~\eqref{EqBgscriC21DiffEst}. Thus, we have proved
  \begin{equation}
  \label{EqBgscri1almost}
    \| u_0 \|_{\rho_0^{a_0}\rho_I^{a'_I}\Hscri^1} \leq C\bigl(\|\pi_0 f\|_{\rho_0^{a_0}\rho_I^{a'_I-1}L^2_\bop} + \| \pi_0 L_h\pi_0^c u_0^c \|_{\rho_0^{a_0}\rho_I^{a'_I-1}L^2_\bop}\bigr),
  \end{equation}
  valid for $a'_I<\min(a_0,\gamma)$. Since $L_h$ is principally scalar, $\pi_0 L_h\pi_0^c$ is a \emph{first} order operator, and by Lemma~\ref{LemmaEinNscri}, we have
  \begin{equation}
  \label{EqBgscriPi0LPi0c}
    \pi_0 L_h\pi_0^c \in \rho_0^{1-0}\rho_I^{-1+b'_I}\cM_{\beta^*K_0} + (\CI+\rho_0^{1-0}\rho_I^{-0}\Hb^\infty)\Diffb^1(M;\beta^*K_0);
  \end{equation}
  since $a'_I<a_I+b'_I<a_I+\half$, the second term in \eqref{EqBgscri1almost} is bounded by $\|u_0^c\|_{\rho_0^{a_0}\rho_I^{a_I-\delta}\Hscri^1}$ for sufficiently small $\delta>0$ (by the assumptions on the weights in Theorem~\ref{ThmBg}), which establishes the estimate~\eqref{EqBgscri1}.

  The proof of the estimate~\eqref{EqBgscri2} proceeds along completely analogous lines, using the weight $w=\rho_0^{-a_0}\rho_I^{-a_I}$ and positive commutator estimates for the equations~\eqref{EqBgscriEq11c} and \eqref{EqBgscriEq11}. The main difference is that $\pi_{1 1}L_h\pi_{1 1}$ and $\pi_{1 1}^c L_h\pi_{1 1}^c$ have no leading order subprincipal terms like $\pi_0 L_h\pi_0$ does, hence we need $a_I<\min(a_0,0)$ for $K_{w^2 V}$ to have a sign---this is the case $a'_I=a_I$, $\gamma=0$ in the notation of Lemma~\ref{LemmaBgscriImprovedK}. In order to estimate the coupling terms on the right hand side of~\eqref{EqBgscriEq11c}, we use Lemma~\ref{LemmaEinNscri}, so
  \begin{align}
  \label{EqBgscri11cLh0Space}
    \pi_{1 1}^c L_h\pi_0 &\in (\rho_I^{-1}\CI + \rho_0^{1-0}\rho_I^{-1-0}\Hb^\infty)\cM + (\CI+\rho_0^{1-0}\rho_I^{-0}\Hb^\infty)\Diffb^1, \\
    \pi_{1 1}^c L_h\pi_{1 1} &\in \rho_0^{1-0}\rho_I^{-1+b'_I}\cM + (\CI+\rho_0^{1-0}\rho_I^{-0})\Diffb^1, \nonumber
  \end{align}
  which gives
  \begin{equation}
  \label{EqBgscri11c}
    \| u_{1 1}^c \|_{\rho_0^{a_0}\rho_I^{a_I}\Hscri^1} \leq C\bigl(\|\pi_{1 1}^c f\|_{\rho_0^{a_0}\rho_I^{a_I-1}L^2_\bop} + \| u_0 \|_{\rho_0^{a_0}\rho_I^{a_I+\delta}\Hscri^1} + \| u_{1 1} \|_{\rho_0^{a_0}\rho_I^{a_I-\delta}\Hscri^1}\bigr);
  \end{equation}
  for our choice~\eqref{EqBgscriDelta} of $\delta$, the second term is bounded by a \emph{small} constant times the left hand side of \eqref{EqBgscri1}. For analyzing the equation~\eqref{EqBgscriEq11} for $u_{1 1}$, we observe that $\pi_{1 1}L_h\pi_0$ lies in the space \eqref{EqBgscri11cLh0Space}, while
  \[
    \pi_{1 1}L_h\pi_{1 1}^c \in \bigl(\rho_0^{1+b_0}\rho_I^{-1}\Hb^\infty(\scri^+\cap U_\eps)+\rho_0^{1-0}\rho_I^{-1+b_I}\Hb^\infty\bigr)\cM + (\CI+\rho_0^{1-0}\rho_I^{-0}\Hb^\infty)\Diffb^1,
  \]
  where we exploit that $h^{\bar a\bar b}$ has a leading term at $\scri^+$. Thus,
  \begin{equation}
  \label{EqBgscri11}
    \| u_{1 1} \|_{\rho_0^{a_0}\rho_I^{a_I}\Hscri^1} \leq C'\bigl(\|\pi_{1 1} f\|_{\rho_0^{a_0}\rho_I^{a_I-1}L^2_\bop} + \| u_0 \|_{\rho_0^{a_0}\rho_I^{a_I-\delta}\Hscri^1} + \| u_{1 1}^c \|_{\rho_0^{a_0}\rho_I^{a_I}\Hscri^1}\bigr).
  \end{equation}
  In order to obtain the estimate~\eqref{EqBgscri2}, we add \eqref{EqBgscri11c} and a small multiple, $\eta$, of \eqref{EqBgscri11}, so that $\eta C'<1$ and $u_{1 1}^c$ can be absorbed into the left hand side of \eqref{EqBgscri11c}; note that the $u_{1 1}$ term in \eqref{EqBgscri11c} is arbitrarily small compared to the left hand side of \eqref{EqBgscri11} when we localize sufficiently closely to $\scri^+$. As explained at the beginning of the proof, this establishes the desired estimate~\eqref{EqBgscri} for $k=1$.

  To prove~\eqref{EqBgscri} for $k\geq 2$, we proceed by induction on the level of the hierarchy~\eqref{EqBgscriEq0}--\eqref{EqBgscriEq11} and the corresponding estimates~\eqref{EqBgscri1}, \eqref{EqBgscri11c}, and \eqref{EqBgscri11}. The key structures for obtaining higher regularity are the symmetries of the normal operators of $\pi_0 L_h\pi_0$ etc.\ at $\scri^+$. Namely, $-2\rho^{-2}\pa_0\pa_1\in\pa_{\rho_I}(\rho_0\pa_{\rho_0}-\rho_I\pa_{\rho_I})+\Diffb^2$ commutes (modulo $\Diffb^2$) with $\rho_0\pa_{\rho_0}$, while for the vector field $\rho_I\pa_{\rho_I}$ generating dilations along approximate (namely, Schwarzschildean) light cones, we have
  \[
    [-2\rho^{-2}\pa_0\pa_1,\rho_I\pa_{\rho_I}] \in -2\rho^{-2}\pa_0\pa_1 + \Diffb^2.
  \]
  Commutation with spherical vector fields is more subtle: we need to define rotation `vector fields' somewhat carefully. We only define these on $\beta^*K_0$, the definition for the other bundles being analogous. Using the product splitting $\R_q\times\R_s\times\Sph^2$ of $\R^4$ near $S^+$, denote by $\{\Omega_{1,i}\colon i=1,2,3\}\subset\cV(\Sph^2)\hra\Vb(M)$ a spanning set of the space of vector fields on $\Sph^2$, e.g.\ rotation vector fields, though the concrete choice or their (finite) number are irrelevant; we can then define elements $\Omega_i\in\Diffb^1(M;\beta^*K_0)$ with scalar principal symbols equal to those of $\Omega_{1,i}$ such that
  \begin{equation}
  \label{EqBgscriRotComm}
    [\rho^{-1}\pa_0,\Omega_i],\ [\rho_0^{-1}\pa_1,\Omega_i] \in \rho_I\Diffb^1(M;\beta^*K_0),
  \end{equation}
  where $\rho^{-1}\pa_0,\rho_0^{-1}\pa_1$ denote elements in ${}^0\cM_{\beta^*K_0}$. (Note that the $\rho_I\,\CI$ indeterminacy of $\rho^{-1}\pa_0,\rho_0^{-1}\pa_1$ does not affect~\eqref{EqBgscriRotComm}.) Here, it is crucial that we \emph{fix} $\rho_0$ and $\rho$ to be given by \eqref{EqCptASplCoords} and thus rotationally invariant: $\Omega_{i,1}\rho_0=0$, so $[\Omega_i,\rho_0]\in\rho_I\,\CI$; we also have $[\Omega_i,\rho_I]\in\rho_I\,\CI$ independently of choices. Regarding \eqref{EqBgscriRotComm} then, we automatically have membership in $\Diffb^1$ by principal symbol considerations; to get the additional vanishing at $\rho_I$ is then exactly the statement that the normal operators of $\rho^{-1}\pa_0$, resp.\ $\rho_0^{-1}\pa_1$, and $\Omega_i$ commute. For $\rho^{-1}\pa_0$, whose normal operator is $-\half\rho_I\pa_{\rho_I}$, this is automatic, while for $\rho_0^{-1}\pa_1$, we merely need to arrange $[\rho_0\pa_{\rho_0},\Omega_i]=0$ at $\scri^+$, which holds if we define $\Omega_i$ in the decomposition~\eqref{EqEinFK0} by $\Omega_{1,i}\oplus\slnabla_{\Omega_{1,i}}\oplus\Omega_{1,i}$. We therefore obtain
  \[
    [-2\rho^{-2}\pa_0\pa_1, \Omega_i ],\ [ L^0, \Omega_i ] \in \Diffb^2,
  \]
  with $L^0$ given in \eqref{EqBgscriLpi0}, which improves over the a priori membership in $\rho_I^{-1}\Diffb^2$. Let us now assume that for the solution of equation~\eqref{EqBgscriEq0}, we have already established the estimate
  \begin{equation}
  \label{EqBgscriInductiveHyp}
    \| u_0 \|_{\rho_0^{a_0}\rho_I^{a'_I}\Hscrib^{1,k-1}} \leq C\bigl(\|\pi_0 f\|_{\rho_0^{a_0}\rho_I^{a'_I-1}\Hb^{k-1}} + \| \pi_0^c u\|_{\rho_0^{a_0}\rho_I^{a_I-\delta}\Hscrib^{1,k-1}}\bigr).
  \end{equation}
  We use $\{G_j\}:=\{\rho_0\pa_{\rho_0}, \rho_I\pa_{\rho_I},\ \Omega_1,\ \Omega_2,\ \Omega_3,\ 1\}$, which spans $\Diffb^1(M;\beta^*K_0)$ over $\CI(M)$, as a set of commutators. Writing $L=\pi_0 L_h\pi_0$, we then have
  \begin{equation}
  \label{EqBgscriInductiveStep}
    L G_j u_0 = f_j+[L,G_j]u_0,\ \ f_j := G_j\pi_0 f - G_j\pi_0 L_h\pi_0^c u_0^c.
  \end{equation}
  We estimate the first term by
  \[
    \|f_j\|_{\rho_0^{a_0}\rho_I^{a'_I-1}\Hb^{k-1}} \leq C\bigl(\|\pi_0 f\|_{\rho_0^{a_0}\rho_I^{a'_I-1}\Hb^k} + \|\pi_0^c u\|_{\rho_0^{a_0}\rho_I^{a_I-\delta}\Hscrib^{1,k}}\bigr).
  \]
  For the second, delicate, term, we use the above discussion to see that
  \begin{equation}
  \label{EqBgscriComm}
    [L,G_j] \in c_j L + \rho_0^{1-0}\rho_I^{-1+b'_I}\cM\circ\Diffb^1 + (\CI+\rho_0^{1-0}\rho_I^{-0})\Diffb^2
  \end{equation}
  with $c_j=1$ if $G_j=\rho_I\pa_{\rho_I}$, and $c_j=0$ otherwise. Thus, $[L,G_j]=c_j L + C_j^\ell G_\ell$ with $C_j^\ell\in\rho_0^{1-0}\rho_I^{-1+b'_I}\cM + (\CI+\rho_0^{1-0}\rho_I^{-0})\Diffb^1$, and therefore
  \begin{equation}
  \label{EqBgscriInductiveComm}
    \|[L,G_j]u_0\|_{\rho_0^{a_0}\rho_I^{a'_I-1}\Hb^{k-1}} \leq c_j\| L u_0 \|_{\rho_0^{a_0}\rho_I^{a'_I-1}\Hb^{k-1}} + C \sum_\ell \| G_\ell u_0 \|_{\rho_0^{a_0}\rho_I^{a_I-\delta-\delta'}\Hscrib^{1,k-1}}
  \end{equation}
  for $\delta'>0$ small; recall that our choice~\eqref{EqBgscriDelta} of $\delta$ leaves some extra room. Now, applying~\eqref{EqBgscriInductiveHyp} to $G_j u_0$ in equation~\eqref{EqBgscriInductiveStep} and summing over $j$, we can absorb the term~\eqref{EqBgscriInductiveComm} into the left hand side of the estimate due to the weaker weight. This establishes \eqref{EqBgscriInductiveHyp} for $k$ replaced by $k+1$. The higher regularity analogues of the estimates~\eqref{EqBgscri11c} and \eqref{EqBgscri11} are proved in the same manner; as before, this then yields the estimate~\eqref{EqBgscri} for all $k$.
\end{proof}

This proposition remains valid near \emph{any} compact subset of $\scri^+\setminus I^+$: the proof only required localization near $\scri^+$. At this point, we therefore have quantitative control of the solution of the initial value problem for $L_h u=f$ in any compact subset of $M\setminus I^+$.

\subsection{Estimate near timelike infinity}
\label{SsBgip}

Near the corner $I^+\cap\scri^+$, fix the local defining functions
\begin{equation}
\label{EqBgipNullDefFn}
  \rho_I := v = (t-r_*)/r,\ \ \ringrho_+:=(t-r_*)^{-1}
\end{equation}
of $\scri^+$ and $I^+$, and let $\rho:=\rho_I\ringrho_+=r^{-1}$; these only differ from the expressions for the defining functions $\rho_I$ and $\rho_0$ used in \S\ref{SsBgscri} by a sign. We thus have $G_\bop=\rho^{-2}G=G_{0,\bop}+G_{1,\bop}+\wt G_\bop$ for
\begin{equation}
\label{EqBgipDualMetric}
  G_{0,\bop} = -2\pa_{\rho_I}(\rho_I\pa_{\rho_I}-\ringrho_+\pa_{\ringrho_+})-\slG \in \rho_I^{-1}\CI(M;S^2\,\Tbeta M+\rho_I\,S^2\,\Tb M)
\end{equation}
and $G_{1,\bop}\in\CI(M;S^2\,\Tbeta M)$, $\wt G_\bop\in\rho_I^{-1+b'_I}\rho_+^{1+b_+}\Hb^\infty(M;S^2\,\Tbeta M+\rho_I\,S^2\,\Tb M)$, while
\[
  g_\bop \in (\CI+\rho_I^{b'_I}\rho_+^{1+b_+}\Hb^\infty)(M;S^2(\Tbeta M)^\perp + \rho_I\,S^2\,\Tb^*M)
\]
with smooth term given by $\rho^2 g_m=2\rho_I\tfrac{d\ringrho_+}{\ringrho_+}(\tfrac{d\ringrho_+}{\ringrho_+}+\tfrac{d\rho_I}{\rho_I})+\rho_I^2\,\CI(M;S^2\,\Tb^*M)$. In order to be able to work near \emph{all of} $I^+$, we first prove:

\begin{lemma}
\label{LemmaBgipDefFn}
  There exists a defining function $\rho_+\in\CI(M)$ of $I^+$ such that $d\rho_+/\rho_+$ is past timelike near $I^+$ for the dual b-metric $\rho^{-2}g_m^{-1}$. Moreover, if $C>0$ is fixed, then for any $h\in\cX^\infty$ with $\|h\|_{\cX^3}<C$ and for any $\eps>0$, there exists $\delta>0$ such that $d\rho_+/\rho_+$ is past timelike with $|d\rho_+/\rho_+|_{G_\bop}^2>0$ in $\{\rho_I\geq\eps,\ \rho_+\leq\delta\}$ for the dual b-metric $G_\bop=\rho^{-2}g^{-1}$, $g=g_m+\rho h$.
\end{lemma}
\begin{proof}
  For the second claim, note that in $\rho_I\geq\eps>0$, we have $G_\bop-\rho^{-2}g_m^{-1}\in\rho_+^{1+b_+}L^\infty$ with norm controlled by $\|h\|_{\cX^3}$, so
  \begin{equation}
  \label{EqBgipDefFnEst}
    |d\rho_+/\rho_+|_{G_\bop}^2 \in |d\rho_+/\rho_+|^2_{\rho^{-2}g_m^{-1}} + \rho_+^{1+b_+}\Hb^\infty
  \end{equation}
  is indeed positive near $\rho_+=0$. To prove the first claim, we compute on Minkowski space $|f_0^{-1}d f_0|^2\equiv 1$, $f_0=t/(t^2-r^2)$ in $t>r$, computed with respect to the dual metric of $t^{-2}(dt^2-dr^2)$.\footnote{See also the related calculations and geometric explanations around equation~\eqref{EqBgExplDSIso}.} Similarly, in $r/t>\tfrac14$, and $t>r_*$ large, we have $|f_*^{-1}d f_*|_{\rho^{-2}g_m}^2>0$ for $f_*=t/(t^2-r_*^2)$: this is a simple calculation where $g_m=g_m^S$ is the Schwarzschild metric, and follows in general by an estimate similar to~\eqref{EqBgipDefFnEst} since $g_m$ differs from $g_m^S$ by a scattering metric of class $\rho^{1-0}\Hb^\infty$ in $r/t<\tfrac34$. Moreover, $f_*$ is (apart from minor smoothness issues, which we address momentarily) a defining function of $I^+$ near $\scri^+$. But $f_0-f_*\in\rho^{2-0}\Hb^\infty$ for $r/t\in(\tfrac14,\tfrac34)$, hence $f':=\chi f_*+(1-\chi)f_0$ has $|(f')^{-1}d f'|_{\rho^{-2}g_m}^2>0$ near $I^+$, where $\chi=\chi(r/t)$ is smooth and identically $0$, resp.\ $1$, in $r/t<\tfrac14$, resp.\ $r/t>\tfrac34$. Fixing any defining function $\rho'_+$ of $I^+$, Lemma~\ref{LemmaCptFtInv} implies $f'\in\rho'_+\,\CI(M)+(\rho'_+)^{2-0}\Hb^\infty(M)$ (with the nonsmooth summand supported away from $\scri^+$ by construction), so we may take $\rho_+\in\CI(M)$ to be any defining function of $I^+$ such that $f'-\rho_+\in(\rho'_+)^{2-0}\Hb^\infty$.
\end{proof}

\emph{For the remainder of this section, $\rho_+$ will denote this particular defining function.} Near $I^+\cap\scri^+$, we need to modify $\rho_+$ in the spirit of~\eqref{EqBgscriFinalDefFn} in order to get a \emph{timelike} (but not quite smooth) boundary defining function. Thus, fix $\beta\in(0,b'_I)$ and some small $\eta>0$, and let $p_\beta\in\rho_I^\beta\,\CI(U)$ be a nonnegative function in a neighborhood $U$ of $I^+$ such that $p_\beta\equiv\eta^\beta$ in $\rho_I\geq 2\eta$, $p_\beta(\rho_I)=\rho_I^\beta$ in $\rho_I\leq\half\eta$, and $0\leq p'_\beta\leq\beta\rho_I^{\beta-1}$; let then
\begin{equation}
\label{EqBgipDefFnMod}
  \wt\rho_+ := \rho_+(1+p_\beta) \in \rho_+(1+\rho_I^\beta)\,\CI(M).
\end{equation}
It is easy to see that $\wt\rho_+^{a_+}\Hb^k(M)=\rho_+^{a_+}\Hb^k(M)$, likewise for weighted $\Hscri$ and $\Hscrib$ spaces.

\begin{lemma}
  Fix $C>0$. Then there exist $\eta,\delta>0$ such that for all $h\in\cX^\infty$ with $\|h\|_{\cX^3}<C$, we have $|d\wt\rho_+/\wt\rho_+|_{G_\bop}^2>0$ in $\rho_+\leq\delta$.
\end{lemma}
\begin{proof}
  We compute the $G_\bop$-norms
  \begin{equation}
  \label{EqBgipDefFnModNorm}
    \Bigl|\frac{d\wt\rho_+}{\wt\rho_+}\Bigr|^2 = \Bigl|\frac{d\rho_+}{\rho_+}\Bigr|^2 + \frac{\rho_I p_\beta'}{1+p_\beta}\Bigl(2\Big\la\frac{d\rho_+}{\rho_+},\frac{d\rho_I}{\rho_I}\Big\ra + \frac{\rho_Ip_\beta'}{1+p_\beta}\Bigl|\frac{d\rho_I}{\rho_I}\Bigr|^2\Bigr).
  \end{equation}
  In $\rho_I<2\eta$ and thus near $\scri^+$, we first note that $\rho_+=f\ringrho_+$ with $f>0$ smooth; since $d f/f$ thus vanishes at $\scri^+\cap I^+$ as a b-1-form, we have
  \[
    2\Big\la\frac{d\rho_+}{\rho_+},\frac{d\rho_I}{\rho_I}\Big\ra \in (2+\rho_I\,\CI+\rho_+\,\CI)\rho_I^{-1} + \rho_I^{-1+b'_I}\rho_+^{1+b_+}\Hb^\infty,
  \]
  thus the second summand of~\eqref{EqBgipDefFnModNorm} is $\gtrsim\rho_I^{-1+\beta}$ in $\rho_I\leq\half\eta$ and $\rho_+$ small. The first and third terms on the other hand are dominated by this, as they are bounded by $\rho_I^{-1+b'_I}$ and $\rho_I^{-1+2\beta}$, respectively. In $\half\eta<\rho_I<2\eta$ and $\rho_+$ small, the parenthesis in~\eqref{EqBgipDefFnModNorm} is positive, the second summand being bounded by $\rho_I^{-1+\beta}$; the prefactor being positive due to $p_\beta'\geq 0$, the claimed positivity thus follows from Lemma~\ref{LemmaBgipDefFn}.
\end{proof}

We also note that $\rho_+\pa_{\rho_+}$, which is well-defined as a b-vector field at $I^+$ and equals the scaling vector field in $(I^+)^\circ$, is past timelike in $(I^+)^\circ$. Let
\[
  U = \{ \wt\rho_+<\delta \} \subset M
\]
denote the neighborhood of $I^+\subset M$ on which we will formulate our energy estimate. Near $\scri^+$, we need to exploit the weak null structure as in \S\ref{SsBgscri}; thus, let
\begin{equation}
\label{EqBgipChi}
  \chi \in \CIc([0,\infty)_{\rho_I}),\ \ 
  \chi\equiv 1\ \ \tn{near}\ \rho_I=0,
\end{equation}
denote a smooth function on $U$ localizing in a neighborhood of $\scri^+$ where the projections $\pi_0$ etc.\ are defined, see the discussion around Definition~\ref{DefEinFSubb}.

\begin{prop}
\label{PropBgip}
  For weights $b'_I,b_I,b_+,a'_I,a_I$ as in Theorem~\usref{ThmBg}, there exists $a_+\in\R$ such that for all $h\in\cX^\infty$ which are small in $\cX^3$, the following holds: Let $f\in\rho_I^{a_I-1}\rho_+^{a_+}\Hb^{k-1}(U)$, $\chi\pi_0 f\in\rho_I^{a'_I-1}\rho_+^{a_+}\Hb^{k-1}(U)$, and suppose $f$ vanishes in $\wt\rho_+>\half\delta$. Let $u$ denote the unique forward solution of $L_h u=f$. Then
  \begin{equation}
  \label{EqBgip}
  \begin{split}
    &\|u\|_{\rho_I^{a_I}\rho_+^{a_+}\Hscrib^{1,k-1}(U)} + \|\chi\pi_0 u\|_{\rho_I^{a'_I}\rho_+^{a_+}\Hscrib^{1,k-1}(U)} \\
    &\qquad\qquad \leq C\Bigl(\|f\|_{\rho_I^{a_I-1}\rho_+^{a_+}\Hb^{k-1}(U)} + \|\chi\pi_0 f\|_{\rho_I^{a'_I-1}\rho_+^{a_+}\Hb^{k-1}(U)}\Bigr).
  \end{split}
  \end{equation}
\end{prop}
\begin{proof}
  We first consider $k=1$. Near $\pa I^+$, we will make use of the vector field $V_0'=(1-c_V)\rho_I\pa_{\rho_I}-\ringrho_+\pa_{\ringrho_+}$, $c_V>0$ small, analogously to Lemma~\ref{LemmaBgscriKCurrent}; away from $\pa I^+$, the vector field $V_0'':=-\nabla\rho_+/\rho_+$ is future timelike. Fix $a^0_+\leq-\half$ and consider the vector field $V_I:=\rho_I^{-2 a_I}\ringrho_+^{-2 a^0_+}V_0'$, then
  \begin{align*}
    K_{V_I} &\in \rho_I^{-2 a_I-1}\ringrho_+^{-2 a^0_+}\Bigl(2 c_V(a^0_+-a_I)(\rho_I\pa_{\rho_I})^2 + 2 a_I(\rho_I\pa_{\rho_I}-\ringrho_+\pa_{\ringrho_+})^2 \\
      &\hspace{14em} + \bigl(\half(1-c_V)+a^0_+-a_I + c_V a_I\bigr)\rho_I\slG \Bigr) \\
      &\quad + \rho_I^{-2 a_I}\rho_+^{-2 a^0_+}(\CI+\rho_I^{-1+b'_I}\rho_+^{1+b_+})(M;S^2\,\Tbeta M+\rho_I\,S^2\,\Tb M)
  \end{align*}
  is $\lesssim -\rho_I^{-2 a_I-1}\rho_+^{-2 a^0_+}$ as a quadratic form, and $\dv_{g_\bop} V_I\lesssim -\rho_I^{-2 a_I}\rho_+^{-2 a^0_+}$. Analogously to Lemma~\ref{LemmaBgscriImprovedK}, if $V_I'=\rho_I^{-2 a'_I}\ringrho_+^{-2 a_+^0}V_0'$, then $K_{V_I'}-2\gamma V_I'\otimes_s\rho^{-1}\pa_1$ is negative definite near $\pa I^+$ for $c_V>0$ sufficiently small.

  To explain the idea for obtaining a \emph{global} (near $I^+$) negative commutator, consider the timelike vector field $W_0:=\chi V_I+(1-\chi)\rho_+^{-2 a_+^0}V''_0$, and let $W=\wt\rho_+^{-2 a_+^1}W_0$; then formula~\eqref{EqBgi0KCurrentRule} gives
  \begin{equation}
  \label{EqBgi0Weight}
    K_W = \wt\rho_+^{-2 a_+^1}K_{W_0} + 2 a_+^1\wt\rho_+^{-2 a_+^1} T(W_0,-\tfrac{\nabla\wt\rho_+}{\wt\rho_+}).
  \end{equation}
  Letting
  \[
    a_+:=a_+^0+a_+^1,
  \]
  the first term gives control in $\rho_I^{a_I}\rho_+^{a_+}\Hscri^1$ near $\scri^+$ in a positive commutator argument. On the other hand, its size is bounded by a fixed constant times $\rho_+^{-2 a_+}$ in $\rho_I\geq\eps>0$; but there, $T(W_0,-\tfrac{d\wt\rho_+}{\wt\rho_+})\gtrsim \rho_+^{-2 a_+^0}$ in the sense of quadratic forms on $\Tb^*M$ since $W_0$ and $-d\wt\rho_+/\wt\rho_+$ are both future timelike. Therefore, choosing $a_+^1$ large and negative, we obtain
  \[
    K_W \leq -C\rho_I^{-2 a_I-1}\rho_+^{-2 a_+}K_W',
  \]
  where $K_W'$ is positive definite on $\Tb^*M$ in $\rho_I\geq\eps>0$, while near $\scri^+$, we have $K_W'=K_1+\rho_I K_2$, with $K_1$, resp.\ $K_2$, positive definite on $\Tb^*M$, resp.\ $(\Tbeta M)^\perp$. This gives global (near $I^+$) control in $\rho_I^{a_I}\rho_+^{a_+}\Hscri^1$.
  
  We now apply this discussion to the situation at hand. For brevity, let us use the same symbol to denote a b-vector field in ${}^0\cM$ and an arbitrary but fixed representative in ${}^0\cM_{\beta^*E}$ according to Lemma~\ref{LemmaCptScriModule}\eqref{ItCptScriModuleBdl}, similarly for b-vector fields with weights (such as $V_I$ and $V'_I$); the bundle $E\to\ol{\R^4}$ will be clear from the context. For $a^1_+\in\R$ chosen later, consider then the operator $W$ acting on sections of $\beta^*S^2$,
  \begin{equation}
  \label{EqBgipCommutant}
    W:=\wt\rho_+^{-2 a_+^1}W_0,\ \ W_0:=\chi\bigl(\pi_0 V'_I\pi_0 + \pi_{1 1}^c V_I\pi_{1 1}^c + \eta \pi_{1 1} V_I\pi_{1 1}\bigr) + (1-\chi)\rho_+^{-2 a_+^0}V''_0,
  \end{equation}
  where $\eta>0$ will be taken small, as in the discussion after~\eqref{EqBgscri11}. (Since $u$ vanishes in $\rho_+>\half\delta$, we do not need to include a cutoff term here.) `Integrating' along $W$ via a commutator calculation for $2\Re\la W u,\rho_I^{-1} u\ra$ as in~\eqref{EqBgscriCommUItself} gives control on $u$ in the function space appearing in~\eqref{EqBgi0Est} in terms of $W u$. The evaluation of the commutator $2\Re\la L_h u, W u\ra=\la\cC u,u\ra$, $\cC=[L_h,W]+(W+W^*)L_h+(L_h^*-L_h)W$, then combines the three separate calculations for the equations~\eqref{EqBgscriEq0}--\eqref{EqBgscriEq11} into one: near $\scri^+$, one writes $L_h$ in block form according to the bundle decomposition $\beta^*S^2=K_0\oplus K_{1 1}^c\oplus K_{1 1}$, with the diagonal elements $\pi_0 L_h\pi_0$ etc.\ giving rise to the main terms of the commutator, while the off-diagonal terms can be estimated using Cauchy--Schwarz and absorbed into the main terms due to the weak null structure, as explained in detail in the proof of Proposition~\ref{PropBgscri}. Away from $\scri^+$, all error terms can be absorbed in the main term, corresponding to the second term in~\eqref{EqBgi0Weight} upon choosing $a_+^1<0$ negative enough. This proves the proposition for $k=1$.

  Suppose now we have proved~\eqref{EqBgip} for some $k\geq 1$. First, the b-operator $L_h$ \emph{automatically} commutes with $\rho_+\pa_{\rho_+}$ to leading order at $I^+$; concretely, Lemma~\ref{LemmaEinNscri} gives
  \[
    [L_h,\rho_+\pa_{\rho_+}] \in \rho_I^{-1+b'_I}\rho_+^{1+b_+}\cM_{\beta^*S^2}^2 + (\rho_+\,\CI+\rho_I^{-0}\rho_+^{1+b_+}\Hb^\infty)\Diffb^2.
  \]
  Here, by an abuse of notation, $\rho_+\pa_{\rho_+}\in {}^0\cM_{\beta^*S^2}$ is defined by first extending the vector field $\rho_+\pa_{\rho_+}\in\CI(I^+,\Tb_{I^+}M)$ to an element of ${}^0\cM_{\ul\C}$, and then taking a representative of the image space in Lemma~\ref{LemmaCptScriModule}\eqref{ItCptScriModuleBdl}; for this particular vector field, such a representative is in fact well-defined modulo $\rho_I\rho_+\,\CI(M;\End(\beta^*S^2))$, the extra vanishing at $\rho_+$ being due to the special (b-normal) nature of $\rho_+\pa_{\rho_+}$.

  Therefore, commuting $\rho_+\pa_{\rho_+}$ through the equation $L_h u=f$, we have the estimate
  \begin{equation}
  \label{EqBgipHigher1}
  \begin{split}
    &\|\rho_+\pa_{\rho_+}u\|_{\rho_I^{a_I}\rho_+^{a_+}\Hscrib^{1,k-1}} + \|\chi\pi_0\rho_+\pa_{\rho_+}u\|_{\rho_I^{a'_I}\rho_+^{a_+}\Hscrib^{1,k-1}} \\
      &\qquad \leq C\Bigl( \|f\|_{\rho_I^{a_I-1}\rho_+^{a_+}\Hb^k} + \|\chi\pi_0 f\|_{\rho_I^{a'_I-1}\rho_+^{a_+}\Hb^k} + \|u\|_{\rho_I^{a_I-\delta}\rho_+^{a_+-(1+b_+)}\Hscrib^{1,k}}\Bigr)
  \end{split}
  \end{equation}
  by the inductive hypothesis, where we used $a_I-\delta>a'_I-b'_I$ for $\delta>0$ small to bound the forcing term $[L_h,\rho_+\pa_{\rho_+}]u$ by the third term on the right; see the related discussion around~\eqref{EqBgscriPi0LPi0c}.
  
  Second, the timelike character of $\rho_+\pa_{\rho_+}$ at $(I^+)^\circ$ for $\eps>0$ implies that $C(\rho_+ D_{\rho_+})^2-L_h$ is \emph{elliptic} in $\rho_I\geq\eps$ for large $C$ (depending on $\eps$); therefore, letting $\chi_j\in\CIc(U\setminus\scri^+)$, $j=1,2$, denote cutoffs with $\chi_1\equiv 1$ on $\supp(1-\chi)$ and $\chi_2\equiv 1$ on $\supp\chi_1$, we have an elliptic estimate away from $\scri^+$,
  \begin{equation}
  \label{EqBgipEll}
    \| \chi_1 u \|_{\rho_+^{a_+}\Hb^{k+1}} \leq C\bigl( \| \chi_2\rho_+\pa_{\rho_+}u \|_{\rho_+^{a_+}\Hb^k} + \| \chi_2 u \|_{\rho_+^{a_+}\Hb^k} + \| \chi_2 f \|_{\rho_+^{a_+}\Hb^{k-1}}\bigr),
  \end{equation}
  for $u$ supported in $\rho_+\leq\half\delta$. Near $\scri^+$ on the other hand, we have the symmetries of null infinity at our disposal, encoded by the operators $\rho_I\pa_{\rho_I}$ and the spherical derivatives $\Omega_j$, see the discussion around~\eqref{EqBgscriRotComm}. Let $\wt\chi\in\CI(U)$ be identically $1$ on $\supp\chi$, and supported close to $\scri^+$. Defining the set of (cut-off) commutators $\{\chi G_j\}:=\{\chi\rho_I\pa_{\rho_I},\ \chi\Omega_1,\ \chi\Omega_2,\ \chi\Omega_3\}$ which together with $\rho_+\pa_{\rho_+}$ spans $\Vb(M)$ near $\scri^+$, and recalling the commutation relations~\eqref{EqBgscriComm}, we find
  \begin{equation}
  \label{EqBgipHigher2}
  \begin{split}
    &\|\chi G_j u\|_{\rho_I^{a_I}\rho_+^{a_+}\Hscrib^{1,k-1}} + \|\chi\pi_0 G_j u\|_{\rho_I^{a'_I}\rho_+^{a_+}\Hscrib^{1,k-1}} \\
    &\qquad \leq C\Bigl(\|f\|_{\rho_I^{a_I-1}\rho_+^{a_+}\Hb^k} + \|\chi\pi_0 f\|_{\rho_I^{a'_I-1}\rho_+^{a_+}\Hb^k} \\
    &\qquad\qquad\qquad + \sum_\ell\|\wt\chi G_\ell u\|_{\rho_I^{a_I-\delta}\rho_+^{a_+}\Hscrib^{1,k-1}} + \|\wt\chi\rho_+\pa_{\rho_+}u\|_{\rho_I^{a_I-\delta}\rho_+^{a_+}\Hscrib^{1,k-1}} \Bigr).
  \end{split}
  \end{equation}
  But for any $\eta>0$, we have the estimate
  \[
    \|\wt\chi G_\ell u\|_{\rho_I^{a_I-\delta}\rho_+^{a_+}\Hscrib^{1,k-1}} \leq \eta\|\chi G_\ell u\|_{\rho_I^{a_I}\rho_+^{a_+}\Hscrib^{1,k-1}} + C_\eta\|\chi_1 u\|_{\rho_+^{a_+}\Hb^{k+1}},
  \]
  and the second term can in turn be estimated using~\eqref{EqBgipEll}. Summing the estimate~\eqref{EqBgipHigher2} over $j$ and fixing $\eta>0$ sufficiently small, we can thus absorb the terms involving $\wt\chi G_\ell u$ into the left hand side, getting control by the norm of $f$, plus a control term $C\|\rho_+\pa_{\rho_+}u\|_{\rho_+^{a_+}\Hb^k}$. Adding to this estimate $2 C$ times~\eqref{EqBgipHigher1}, this control term can be absorbed in the left hand side of~\eqref{EqBgipHigher1}. This gives control of $u$ as in the left hand side of~\eqref{EqBgip} with $k$ replaced by $k+1$, but with an extra term on the right coming from the last term in~\eqref{EqBgipHigher1}; however, this term has a weaker weight at $I^+$, $\rho_+^{a_+-(1-b_+)}\gg\rho_+^{a_+}$, hence can be absorbed. This gives~\eqref{EqBgip} for $k$ replaced by $k+1$.
\end{proof}

Combining the estimate~\eqref{EqBgCompact} in compact subsets of $M^\circ$ with Proposition~\ref{PropBgi0} near $(I^0)^\circ$, Proposition~\ref{PropBgscri} near $\scri^+\setminus(\scri^+\cap I^+)$, and Proposition~\ref{PropBgip} near $I^+$ proves Theorem~\ref{ThmBg}.

\subsection{Explicit weights for the background estimate}
\label{SsBgExpl}

We sketch the calculations needed to obtain explicit values for the weights in the background estimate. More precisely, we prove the following slight modification of Theorem~\ref{ThmBg}:

\begin{thm}
\label{ThmBgExpl}
  Let $a_+=-\tfrac32$. There exists an $\eps>0$ such that for $a_I<\bar a_I<a'_I<\min(0,a_0)$ with $|a_I|$, $|a'_I|$, $|\bar a_I|$, $b_I$, $b'_I<\gamma<\eps$ subject to the conditions in Definition~\ref{DefEinF}, as well for $h\in\cX^{\infty;b_0,b_I,b'_I,b_+}$ with $\|h\|_{\cX^3}<\eps$, the unique global solution of the linear wave equation
  \[
    L_h u = f, \ \ (u,\pa_\nu u)|_\Sigma = (u_0,u_1)
  \]
  satisfies the estimate
  \begin{equation}
  \label{EqBgExpl}
  \begin{split}
    &\|u\|_{\rho_0^{a_0}\rho_I^{a_I}\rho_+^{a_+}\Hscrib^{1,k-1}} + \|\pi_{1 1}^c u\|_{\rho_0^{a_0}\rho_I^{\bar a_I}\rho_+^{a_+}\Hscrib^{1,k-1}} + \|\pi_0 u\|_{\rho_0^{a_0}\rho_I^{a'_I}\rho_+^{a_+}\Hscrib^{1,k-1}} \\
    &\qquad \leq C\Bigl(\|u_0\|_{\rho_0^{a_0}\Hb^k}+\|u_1\|_{\rho_0^{a_0}\Hb^{k-1}} \\
    &\qquad\qquad\qquad + \|f\|_{\Hb^{k-1;a_0,a_I-1,a_+}} + \|\pi_{1 1}^c f\|_{\Hb^{k-1;a_0,\bar a_I-1,a_+}} + \|\pi_0 f\|_{\Hb^{k-1;a_0,a'_I-1,a_+}}\Bigr).
  \end{split}
  \end{equation}
\end{thm}
\begin{proof}
  The usage of an intermediate weight $\bar a_I\in(a_I,a'_I)$ allows for a small but useful modification of the argument following~\eqref{EqBgscri11}: namely, in the notation of that proof, we are presently estimating $u_{1 1}$ with weight $\rho_I^{a_I}$, while the term $u_{1 1}^c$ coupling into the equation for $u_{1 1}$ via $\pi_{1 1}L_h\pi_{1 1}^c$ is estimated with weight $\rho_I^{\bar a_I}\ll\rho_I^{a_I}$, hence automatically comes with a small prefactor if we work in a sufficiently small neighborhood of $\scri^+$. Correspondingly, in the proof of Proposition~\ref{PropBgip}, we would replace the third inner summand in~\eqref{EqBgipCommutant} by $\pi_{1 1}\bar V_I\pi_{1 1}$, with $\bar V_I=\rho_I^{-2\bar a_I}\ringrho_+^{-2 a_+^0}V'_0$ in order to obtain~\eqref{EqBgExpl} (with $a_+\ll 0$ not explicit at this point yet).

  The only part of the proof of Theorem~\ref{ThmBg} in which we did not get explicit control on the weights is the energy estimate near $I^+$. In order to obtain the explicit weights there, we note that for $\gamma=0$, $h=0$, and Schwarzschild mass $m=0$, we simply have $2 L_h=\Box_{\ul g}$, the wave operator of the Minkowski metric $\ul g=dt^2-dx^2$, which acts component-wise on $S^2 T^*\R^4$ in the trivialization given by coordinate differentials. Recalling from~\eqref{EqCptAExpl} that ${}^0\!M$ denotes the manifold with corners constructed in \S\ref{SsCptA} for $m=0$, we shall prove that the solution of the \emph{scalar wave equation} $t^3\Box_{\ul g}t^{-1}u=f$, with $f\in\rho_I^{a_I-1}\rho_+^{a_+}L^2_\bop$ supported in $\rho_+<1$, satisfies the estimate
  \begin{equation}
  \label{EqBgExplMink}
    \|u\|_{\rho_I^{a_I}\rho_+^{a_+}\Hscri^1} \lesssim \|f\|_{\rho_I^{a_I-1}\rho_+^{a_+}L^2_\bop}
  \end{equation}
  for $a_+=-\tfrac32$ and $a_I<0$ small, using a vector field multiplier argument; here, $\rho_I={}^0\rho_I$ and $\rho_+={}^0\rho_+$. But then, if the weights $a_I,a'_I,\bar a_I$ etc.\ are very close to one another, the nonscalar commutant used in~\eqref{EqBgipCommutant}, modified as above, is very close to being principally scalar away from $\scri^+$; correspondingly, a slight modification of our arguments below for the Minkowski case~\eqref{EqBgExplMink} yield the estimate~\eqref{EqBgExpl} for $k=1$. Higher b-regularity follows as in the proof of Proposition~\ref{PropBgip}.

  In order to prove the estimate~\eqref{EqBgExplMink}, we introduce explicit coordinates near the temporal face $I^+\subset M$ within the blow-up of compactified Minkowski space. First of all, the calculations in~\ref{SsCoMink} imply
  \begin{equation}
  \label{EqBgExplDS}
    t^3\Box_{\ul g}t^{-1}=\Box_{g_\dS}-2,
  \end{equation}
  where
  \begin{equation}
  \label{EqBgExplDStx}
    g_\dS=t^{-2}(dt^2-dx^2)
  \end{equation}
  is the de~Sitter metric; notice though that we are interested in $t\gg 1$. Thus, consider the isometry
  \begin{equation}
  \label{EqBgExplDSIso}
    (t,x)\mapsto (\hat\tau,\hat x)=\frac{1}{t^2-r^2}(t,x)\in[0,\infty)_{\hat\tau}\times\R^3_{\hat x}
  \end{equation}
  of $g_\dS$, defined in $t>r=|x|$: it maps $I^+$ to $(0,0)$ and $\scri^+$ to $\{\hat\tau=|\hat x|\}$, see Figure~\ref{FigBgExplConformal}. (The map~\eqref{EqBgExplDSIso} is the change of coordinates between the upper half space models of de~Sitter space associated with $q$ on the one hand and its antipodal point on the future conformal boundary of de~Sitter space on the other hand; see \cite[\S6.1]{HintzZworskiHypObs} for the relevant formulas.)

  \begin{figure}[!ht]
  \includegraphics{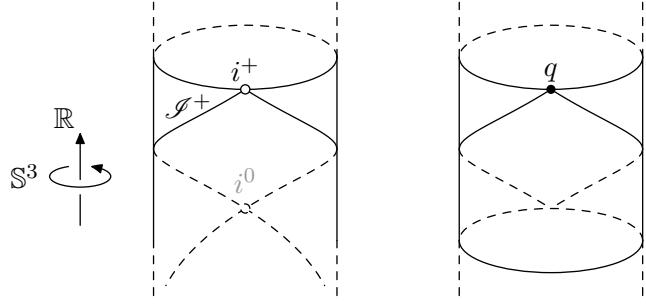}
  \caption{\textit{Left:} part of the conformal embedding of Minkowski space into the Einstein universe $(E,dt^2-g_{\Sph^3})$, $E=\R\times\Sph^3$. \textit{Right:} conformal embedding of de~Sitter space into $E$, and the backward light cone of a point $q$ on its conformal boundary, whose interior is the domain of the upper half space model~\eqref{EqBgExplDStx} of de~Sitter space, which near $q$ is equal to the static model of de~Sitter space near its future timelike infinity, $q$. The coordinates $(\hat\tau,\hat x)$ are regular near $q=(\hat\tau=0,\hat x=0)$.}
  \label{FigBgExplConformal}
  \end{figure}

  Define the blow-up $M':=\bigl[[0,\infty)_{\hat\tau}\times\R^3_{\hat x},\{(0,0)\}\bigr]$ at the image of $I^+$. Then the lift of $\{\hat\tau\leq |\hat x|\}$ to $M'$ is canonically identified with a neighborhood of $I^+\subset M$. Concretely,
  \[
    (\rho_+,Z):=(\hat\tau,\hat x/\hat\tau)=\bigl(t/(t^2-r^2),x/t\bigr) \in [0,\infty)\times\R^3
  \]
  gives coordinates on $M'$, in which $U:=[0,1)_{\rho_+}\times\{|Z|\leq 1\}$ is identified with a collar neighborhood of $I^+\subset M$ so that
  \begin{equation}
  \label{EqBgExplStaticdS}
    g_\dS=\hat\tau^{-2}(d\hat\tau^2-d\hat x^2) = (1-|Z|^2)\frac{d\rho_+^2}{\rho_+^2} - 2 Z\,dZ\frac{d\rho_+}{\rho_+} - dZ^2.
  \end{equation}
  Furthermore, $\rho_I:=1-|Z|^2=1-r^2/t^2$ is a defining function of $\scri^+$ in $U$. Let us write $R:=|Z|$. Instead of the vector field $V_\loc=(1-c_V)\rho_I\pa_{\rho_I}-\rho_+\pa_{\rho_+}$, which is defined locally near $\scri^+$ and was used in the proof of Proposition~\ref{PropBgip}, we use the global vector field
  \[
    V_0 = -(1+R^2)\rho_+\pa_{\rho_+} - (1-c_V)(1-R^2)R\pa_R
  \]
  which is equal to $V_\loc$ near $\scri^+$, up to an overall scalar and modulo $\rho_I\Vb+\rho_+\Vb$; moreover, $V_0$ is timelike in $U\setminus\scri^+$ for small $c_V\geq 0$. Considering the commutant/vector field multiplier $W:=\rho_I^{-2 a_I}\rho_+^{-2 a_+}V_0$ with $a_+=-\tfrac32$ and $a_I<0$ small, the expression for the $K$-current $K_W$ is somewhat lengthy, so we merely list its main features in $0\leq R\leq 1$, writing $\rho_I^{2 a_I+1}\rho_+^{2 a_+}K_W=:K_1+\slK\slG$, with $K_1$ a section of $S^2\la \rho_+\pa_{\rho_+},\pa_R\ra$ (considered a $2\times 2$ matrix in this frame) and $\slK$ a scalar:
  \begin{itemize}
  \item $\tr K_1|_{c_V=0}=-2(1-R^4-a_I R^2(4+R^2))<0$, which persists for small $c_V>0$;
  \item $\det K_1|_{c_V=0}=-4 a_I(1+a_I)R^2(1-R^2)\geq 0$ and
    \[
      (\pa_{c_V}\det K_1)|_{c_V=0} = -16 a_I^2(R^2-\tfrac{1}{1+4 a_I})(R^2+\tfrac{3}{3+4 a_I}) > 0,
    \]
    so $\det K_1>0$ for small $c_V>0$;
  \item $\slK|_{c_V=0}=-2(1+a_I R^2)<0$, which persists for small $c_V>0$;
  \item $\rho_I^{2 a_I}\rho_+^{2 a_+}\dv_{g_\dS}W|_{c_V=0}=6-(2-4 a_I)R^2>0$.
  \end{itemize}
  Thus, fixing $c_V>0$ to be small, the main term arising in the evaluation of the commutator $-2\Re\la(\Box_{g_\dS}-2)u,W u\ra$ is $\int_U -K_W(d u,d u)+4(\dv_{g_\dS}W)|u|^2\,\frac{d\rho_+}{\rho_+}\,dZ$, which thus gives the desired control on $u$ in $\Hscri^1$, except $|u|^2$ itself is only controlled in $\rho_I^{a_I}\rho_+^{a_+-1/2}L^2_\bop$ due to the weaker weight of $\dv_{g_\dS}W$ at $\scri^+$; control in $\rho_I^{a_I}\rho_+^{a_+}L^2_\bop$ is obtained by integrating $\rho_+\pa_{\rho_+}u\in\rho_I^{a_I}\rho_+^{a_+}L^2_\bop$ from $\rho_+=1$. This yields~\eqref{EqBgExplMink}.
\end{proof}

\section{Newton iteration}
\label{SIt}

Fix $b_0$, $b_I$, $b'_I$, $b_+$ and $\gamma$ as in Theorem~\ref{ThmBg}. Recall that we want to solve the symmetric 2-tensor-valued wave equation
\[
  P(h)=0,\ \ 
  (h,\pa_\nu h)|_\Sigma = (h_0, h_1)
\]
for initial data $(h_0,h_1)$, $h_j\in\rho_0^{b_0}\Hb^\infty(\Sigma)$, small in a suitable high regularity norm, and we hope to find a solution $h\in\cX^{\infty;b_0,b_I,b'_I,b_+}$. Following the strategy, outlined in \S\ref{SI}, of solving a linearized equation at each step of an iteration scheme, we consider, formally, the iteration scheme with initialization
\[
  L_0 h^{(0)}=0,\ \ 
  (h^{(0)},\pa_\nu h^{(0)})|_\Sigma = (h_0, h_1),
\]
and iterative step $h^{(N+1)}=h^{(N)}+u^{(N+1)}$, where
\[
  L_{h^{(N)}}u^{(N+1)} = -P(h^{(N)}),\ \ 
  (h^{(N+1)},\pa_\nu h^{(N+1)})|_\Sigma = 0.
\]
Assume that $h^{(N)}\in\cX^\infty$ has small $\cX^3$ norm. In order for this iteration scheme to close, we need to show that $h^{(N+1)}\in\cX^\infty$. Since $P(h^{(N)})\in\cY^\infty$ by Lemma~\ref{LemmaEinP}, this means that we need to prove:

\begin{thm}
\label{ThmIt}
  For weights as above, there exists $\eps>0$ such that for $h\in\cX^{\infty;b_0,b_I,b'_I,b_+}$ with $\|h\|_{\cX^3}<\eps$, the following holds: if $f\in\cY^{\infty;b_0,b_I,b'_I,b_+}$ and $h_0,h_1\in\rho_0^{b_0}\Hb^\infty(\Sigma)$, then the solution of the initial value problem
  \[
    L_h u = f,\ \ 
    (u,\pa_\nu u)|_\Sigma = (u_0,u_1),
  \]
  satisfies $u\in\cX^{\infty;b_0,b_I,b'_I,b_+}$.
\end{thm}

\begin{rmk}
\label{RmkItDecay}
  We recall that membership, of a scalar function $u$ for simplicity, in $\rho_0^{b_0}\Hb^\infty(\ol{\R^3})$ is equivalent (up to an arbitrarily loss in decay) to pointwise estimates $|V_1\cdots V_N u|\lesssim\la r\ra^{-b_0}$ where the $V_i$ are translation, rotation, or scaling vector fields on $\R^3$. The membership $h\in\cX^{\infty;b_0,b_I,b'_I,b_+}$ means pointwise decay of various components of $h$ towards leading order terms at $\scri^+$ or to zero; see Definition~\ref{DefEinF} and Remark~\ref{RmkIDetailDecay}.
\end{rmk}

According to Theorem~\ref{ThmBg}, we have the background estimate
\begin{equation}
\label{EqItBg}
  u\in\Hb^{\infty;b_0,-0,a_+}(M;\beta^*S^2),\ \ 
  \pi_0 u\in\Hb^{\infty;b_0,b'_I-0,a_+}(M;\beta^*S^2),
\end{equation}
for suitable $a_+$. We shall improve this to $u\in\cX^{\infty;b_0,b_I,b'_I,b_+}$ using normal operator analysis in several steps, which were outlined around~\eqref{EqISysLinTransport}: using the leading order form~\eqref{EqEinNscri} of $L_h$, or rather its decoupled versions~\eqref{EqEinNscriPi0}--\eqref{EqEinNscriPi11}, we obtain the precise behavior of $u$ near $\scri^+\setminus(\scri^+\cap I^+)$ in \S\ref{SsIti0Scri} by simple ODE analysis; the correct weight at $I^+$ but losing some precision at $\scri^+$ near its future boundary in \S\ref{SsItip} by normal operator analysis and a contour shifting argument; and finally the precise behavior near $\scri^+$, uniformly up to $\scri^+\cap I^+$, again by ODE analysis in \S\ref{SsItipScri}.

For later use, we record the mapping properties of $P$ and its linearization on the polyhomogeneous and conormal parts of $\cX^\infty$---recall \eqref{EqEinFDecomp}.
\begin{lemma}
\label{LemmaIti0ScriAct}
  Let $h\in\cX^{\infty;b_0,b_I,b'_I,b_+}$, with $\|h\|_{\cX^3}$ small; write $h=h_\phg+h_\bop$, $h_\phg\in\cX^\infty_\phg$, $h_\bop\in\cX^\infty_\bop$. Then:
  \begin{enumerate*}
  \item\label{ItIti0ScriActP} $P(h_\phg)\in\cY^\infty$,
  \item\label{ItIti0ScriActLh0phg} $L_h^0\colon\cX^\infty_\phg\to\cY^\infty$,
  \item\label{ItIti0ScriActLh0b} $L_h^0$, $\wt L_h\colon\cX^\infty_\bop\to\cY^\infty_\bop$,
  \item\label{ItIti0ScriActwtLh} $\wt L_h\colon\cX^\infty_\phg\to\cY^\infty_\bop$.
  \end{enumerate*}
\end{lemma}

The point is that the behavior~\eqref{ItIti0ScriActLh0phg}--\eqref{ItIti0ScriActLh0b} of the leading term $L_h^0$ and simple information~\eqref{ItIti0ScriActP} on the nonlinear operator automatically imply precise mapping properties~\eqref{ItIti0ScriActwtLh} of the error term $\wt L_h$ which are not encoded in~\eqref{EqEinNscri}.

\begin{proof}[Proof of Lemma~\usref{LemmaIti0ScriAct}]
  Part~\eqref{ItIti0ScriActP} follows from Lemma~\ref{LemmaEinP}. One obtains~\eqref{ItIti0ScriActLh0phg} by inspection of~\eqref{EqEinNscri}; note that $L_h^0$ is only well-defined modulo terms in $(\CI+\rho_0^{1+b_0}\rho_I^{-0}\rho_+^{1+b_+}\Hb^\infty)\Diffb^1$ which always map $\cX^\infty_\phg\to\cY^\infty$. Likewise, the first part of~\eqref{ItIti0ScriActLh0b} follows from~\eqref{EqEinNscri}; the fact that the `good components' (encoded by the bundle $K_0$) have a better weight $b'_I$ than the weight $b_I$ of the remaining components (in $K_0^c$) is again due to the structure of $L_h^0$ discussed after Lemma~\ref{LemmaEinNscri}. The second part of~\eqref{ItIti0ScriActLh0b} is clear, since this concerns the remainder operator $\wt L_h$, whose coefficients are \emph{decaying} relative to $\rho_I^{-1}\Diffb^2$, acting on $\cX^\infty_\bop$, which consists of tensors \emph{decaying} at $\scri^+$.
  
  Finally, to prove~\eqref{ItIti0ScriActwtLh}, we take $u_\phg\in\cX^\infty_\phg$ and write $\wt L_h u_\phg=D_h P(u_\phg)-L_h^0 u_\phg$. The second term lies in $\cY^\infty$ by~\eqref{ItIti0ScriActLh0phg}, while the first term equals
  \begin{align*}
    &\frac{d}{d s}P(h_\phg+s u_\phg+h_\bop)\big|_{s=0} \\
    &=\frac{d}{d s}\biggl(P(h_\phg+s u_\phg) + \int_0^1 L^0_{h_\phg+s u_\phg+t h_\bop}(h_\bop) + \wt L_{h_\phg+s u_\phg+t h_\bop}(h_\bop)\,dt \biggr)\biggr|_{s=0};
  \end{align*}
  but each of the three terms in parentheses depends smoothly on $s$ as an element of $\cY^\infty$ by~\eqref{ItIti0ScriActP}, \eqref{ItIti0ScriActLh0phg}, and \eqref{ItIti0ScriActLh0b}, respectively.
\end{proof}

\subsection{Asymptotics near \texorpdfstring{$I^0\cap\scri^+$}{the past boundary of the radiation face}}
\label{SsIti0Scri}

With conormal regularity of $u$ at our disposal, all but the leading order terms of $L_h$ can be regarded as error terms at $\scri^+$: from~\eqref{EqItBg} and Lemma~\ref{LemmaEinNscri}, we get
\[
  L_h^0 u \in \cY^{\infty;b_0,b_I,b'_I,b_+} + \Hb^{\infty;b_0,-1+b'_I-0,a_+}.
\]
Let us now work in a neighborhood $U\subset M$ of $I^0\cap\scri^+$ and drop the weight at $I^+$ from the notation. To improve the asymptotics of $u_{1 1}^c:=\pi_{1 1}^c u$, we use part~\eqref{EqEinNscriPi11c} of the constraint damping/weak null structure hierarchy as well as $b'_I>b_I$: this gives
\[
  2\rho^{-2}\pa_0\pa_1 u_{1 1}^c \in \rho_0^{b_0}\rho_I^{b_I-1}\Hb^\infty.
\]
Using the local defining functions $\rho_0$ and $\rho_I$ from~\eqref{EqCptASplCoords} and multiplying by $\rho_I$, this becomes
\begin{equation}
\label{EqIti0Scri11c}
  \rho_I\pa_{\rho_I}(\rho_0\pa_{\rho_0}-\rho_I\pa_{\rho_I})u_{1 1}^c \in \rho_0^{b_0}\rho_I^{b_I}\Hb^\infty.
\end{equation}
We can integrate the second vector field from $\rho_I\geq\eps$, where $u_{1 1}^c\in\rho_0^{b_0}\Hb^\infty$, obtaining $\rho_I\pa_{\rho_I}u_{1 1}^c\in\rho_0^{b_0}\rho_I^{b_I}\Hb^\infty$; this uses $b_I<b_0$ (see Lemma~\ref{LemmaPhgODE2} for details). Integrating out $\rho_I\pa_{\rho_I}$ (see Lemma~\ref{LemmaPhgODE}) shows that $u_{1 1}^c$ is the sum of a leading term in $\rho_0^{b_0}\Hb^\infty(\scri^+\cap U)$ and a remainder in $\rho_0^{b_0}\rho_I^{b_I}\Hb^\infty(U)$. This then couples into the equation for $u_{1 1}=\pi_{1 1}u$, corresponding to part~\eqref{EqEinNscriPi11} of the hierarchy:
\begin{equation}
\label{EqIti0Scri11}
  \rho_I\pa_{\rho_I}(\rho_0\pa_{\rho_0}-\rho_I\pa_{\rho_I})u_{1 1} \in \rho_I\pi_{1 1}f - \half(\pa_1 h^{\bar a\bar b})\pa_1(u_{1 1}^c)_{\bar a\bar b} + \rho_0^{b_0}\rho_I^{b_I}\Hb^\infty.
\end{equation}
The first two summands lie in $\rho_0^{b_0}\Hb^\infty(\scri^+\cap U)+\rho_0^{b_0}\rho_I^{b_I}\Hb^\infty$; integrating this along $\rho_I\pa_{\rho_I}$ generates the logarithmic leading term of $u_{1 1}$. Thus, $u_{1 1}=u_{1 1}^{(1)}\log\rho_I+u_{1 1}^{(0)}+u_{1 1,\bop}$ with $u_{1 1}^{(j)}\in\rho_0^{b_0}\Hb^\infty(\scri^+\cap U)$ and $u_{1 1,\bop}\in\rho_0^{b_0}\rho_I^{b_I}\Hb^\infty$, as desired.

It remains to improve $u_0=\pi_0 u$. Write $u=u_\phg+u_\bop$, where $u_\phg\in\cX_\phg^\infty$ and $u_\bop\in\rho_0^{b_0}\rho_I^{b_I}\Hb^\infty$ according to what we have already established; note that the space $\cX_\phg^\infty$ is \emph{independent} of the choice of $b_I,b'_I\in(0,1)$. Then
\[
  \pi_0 L_h^0\pi_0 u_0 = \pi_0 f - \pi_0\wt L_h(\pi_0 u_0) - \pi_0\wt L_h(\pi_0^c u_\bop) - \pi_0\wt L_h(\pi_0^c u_\phg) \in \rho_0^{b_0}\rho_I^{b'_I-1}\Hb^\infty:
\]
for the first summand, this follows from $f\in\cY^{\infty;b_0,b_I,b'_I,b_+}$, for the second summand from $u_0\in\rho_0^{b_0}\rho_I^{b'_I-0}\Hb^\infty$ and the decay of the coefficients of $\wt L_h$, similarly for the third summand; and for the fourth summand, we use Lemma~\ref{LemmaIti0ScriAct}\eqref{ItIti0ScriActwtLh}. Using the notation of part~\eqref{EqEinNscriPi0} of the hierarchy, this means
\[
  (\rho_I\pa_{\rho_I}-A_\CD)(\rho_0\pa_{\rho_0}-\rho_I\pa_{\rho_I})u_0 \in \rho_0^{b_0}\rho_I^{b'_I}\Hb^\infty.
\]
Since we are taking $\gamma>b'_I$, all eigenvalues of $A_\CD$ are $>b'_I$, so integration of $\rho_I\pa_{\rho_I}-A_\CD$ and then of $\rho_0\pa_{\rho_0}-\rho_I\pa_{\rho_I}$ (using $b'_I<b_0$) gives $u_0\in\rho_0^{b_0}\rho_I^{b'_I}\Hb^\infty$. We have thus shown that $u\in\cX^{\infty;b_0,b_I,b'_I,b_+}$ near $I^0\cap\scri^+$; in fact, this holds away from $I^+$.

\subsection{Asymptotics at the temporal face}
\label{SsItip}

We work near $I^+$ now and drop the weight at $I^0$ from the notation. Recall from~\eqref{EqEinNipulL} the gauge-damped operator $\ul L$ on \emph{Minkowski} space; by Lemma~\ref{LemmaEinNi0p} and~\eqref{EqEinNipL0}, we have
\begin{equation}
\label{EqItipLhMinusL}
  L_h-\ul L\in \rho_I^{-1-0}\rho_+^{1+b_+}\Hb^\infty(M)\cdot\Diffb^2(M;\beta^*S^2).
\end{equation}
We shall deduce the asymptotic behavior of $u$ at $I^+$ from a study of the operator $\ul L$ (and its resonances) on a partial radial compactification $N$ of $\R^4$---\emph{without blowing up the latter at the light cone at future infinity}. Before making this precise, we study $\ul L$ in detail \emph{as a b-operator on $N$}. Let
\[
  \tau=t^{-1},\ \ X=x/t;
\]
these are smooth coordinates on the radial compactification
\[
  N:=[0,\infty)_\tau\times\R^3_X
\]
of $\R^4$ in $t>0$, see Figure~\ref{FigItipRadComp}. We have $d_X=t d_x$, $t\delta_e=\delta_X$, $t\delta_e^*=\delta_X^*$, and $t\pa_t=-\tau\pa_\tau-X\pa_X$. Thus, if we trivialize $S^2\,\Tsc^*\,{}^0\ol{\R^4}$ using coordinate differentials, the explicit expression of $\ul L$ given in \S\ref{SsCoMink} shows that $\ul L$ is a dilation-invariant element of $\Diffb^2(N;\ul\C^{10})$, i.e.\ $\ul L=N(\ul L)$, recalling the definition~\eqref{EqCptNormOp} of the normal operator.

Note that $L_h$ (and even $L_0$) has singular coefficients at $\pa I^+\subset{}^m\!M$ due to the gauge/con\-straint damping term: the singular terms come from $-\rho^{-1}A_h\pa_1$ in Lemma~\ref{LemmaEinNscri}. Likewise, $\ul L$, on the blow-up of $N$ at the light cone $\{\tau=0,\,|X|=1\}$ at infinity, has coefficients with $\rho_I^{-1}$ singularities, which would complicate the normal operator analysis at the temporal face ${}^0i^+$, the lift of
\[
  B:=\{\tau=0,\,|X|\leq 1\},
\]
On the other hand, $\ul L$ \emph{does} have smooth coefficients on the un-blown-up space $N$, and we recall its well-understood b- and normal operator analysis at $\pa N$ momentarily. The discussion of the relation between the blown-up and the un-blown-up picture starts with Lemma~\ref{LemmaItipFn} below.

\begin{figure}[!ht]
\includegraphics{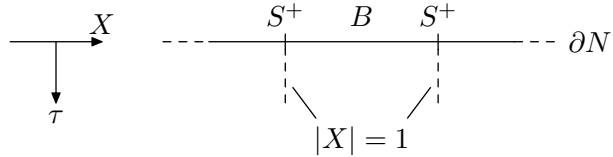}
\caption{Illustration of the compactification $N$ near its boundary at infinity $\pa N=\{\tau=0\}$. Shown are future timelike infinity $B={}^0\beta({}^0 I^+)$, its boundary $\pa B=S^+$, and, for illustration, the light cone $|x|=t$ (dashed).}
\label{FigItipRadComp}
\end{figure}

Conjugating $\ul L$ by the Mellin-transform in $\tau$, thus formally replacing $\tau\pa_\tau$ by $i\sigma$, gives the Mellin-transformed normal operator family $\wh{\ul L}(\sigma)\in\Diff^2(\pa N;\ul\C^{10})$, depending holomorphically on $\sigma\in\C$; the principal symbol of $\wh{\ul L}$ is independent of $\sigma$.

We already control $u$ in Theorem~\ref{ThmIt} away from $I^+\subset M$, so only need to study $u$ (and how $\ul L$ relates to it) near ${}^m I^+$, whose image under the blow-down map ${}^m\beta$ on ${}^mM$ is identified with $B$, see Lemma~\ref{LemmaCptAComp}. For $s\in\R$, we then define the function space $\Hsupp^s(B;\ul\C^{10})$ as the space of all $v\in\Hloc^s(\pa N;\ul\C^{10})$ which are supported in $B$. (We are using the notation of \cite[Appendix~B]{HormanderAnalysisPDE3}.) Let
\[
  \fX^s := \{ u\in\Hsupp^s(B;\ul\C^{10}) \colon \wh{\ul L}(0)u \in \Hsupp^{s-1}(B;\ul\C^{10})\},\ \ 
  \fY^s := \Hsupp^s(B;\ul\C^{10}).
\]
Semiclassical Sobolev spaces are defined by $\Hsupp_h^s=\Hsupp^s$ with $h$-dependent norm $\|u\|_{\Hsupp_h^s}=\|\la h D\ra^s u\|_{L^2}$ on $\pa N\cong\R_X^3$. Let further $\ul\cM\subset\Diff^1(\pa N;\ul\C^{10})$ denote the $\CI(\pa N)$-module of first order operators with principal symbol vanishing on $N^*\pa B$, and fix a finite set $\{A_j\}\subset\ul\cM$ of generators.\footnote{Near $\pa B$, and omitting the bundle $\ul\C^{10}$, one can take as generators the vector fields $(|X|-1)X\pa_X$, $X^j\pa_{X^i}-X^i\pa_{X^j}$.} For $k\in\N_0$, we then define
\[
  \Hsupp^{s,k}(B;\ul\C^{10}) = \{ u\in\Hsupp^s \colon A_{j_1}\cdots A_{j_\ell}u\in\Hsupp^s,\,0\leq\ell\leq k \}
\]
and the semiclassical analogue $\Hsupp_h^{s,k}=\Hsupp^{s,k}$ with norm
\[
  \|u\|_{\Hsupp_h^{s,k}}^2 = \|u\|_{\Hsupp_h^s}^2 + \sum_{0\leq\ell\leq k} \|(h A_{j_1})\cdots(h A_{j_\ell})u\|_{\Hsupp_h^s}^2.
\]
\begin{lemma}
\label{LemmaItipFred}
  Let $C>0$, and fix $s<\half-C$. Then $\wh{\ul L}(\sigma)\colon\fX^s\to\fY^{s-1}$ is an analytic family of Fredholm operators in $\{\sigma\in\C\colon\Im\sigma>-C\}$, with meromorphic inverse satisfying
  \[
    \| \wh{\ul L}(\sigma)^{-1}f\|_{\Hsupp_{\la\sigma\ra^{-1}}^{s,k}} \leq C'_k\la\sigma\ra^{-1}\|f\|_{\Hsupp_{\la\sigma\ra^{-1}}^{s-1,k}},\ \ 
    |\Im\sigma|\leq C,\ |\Re\sigma| \gg 1,
  \]
  for any $k\in\N_0$.
\end{lemma}
\begin{proof}
  For $k=0$, this is almost the same statement as proved in \cite[\S5]{VasyMicroKerrdS}, see also \cite{BaskinVasyWunschRadMink} and the summary of the presently relevant results in \cite[\S6]{BaskinVasyWunschRadMink2}; adding higher module regularity, i.e.\ $k\geq 1$, follows by a standard argument, commuting (compositions of) a well-chosen spanning set of $\ul\cM$ through the equation $\wh{\ul L}(\sigma)u=f$; see \cite[Proof of Proposition~4.4]{BaskinVasyWunschRadMink} and the discussion prior to~\cite[Theorem~5.4]{HintzVasySemilinear} for details in the closely related b-setting (i.e.\ prior to conjugation by the Mellin transform). We shall thus be brief.
  
  The only two differences between the references and the present situation are:
  \begin{enumerate*}
  \item\label{ItItipFredVB} $\hat{\ul L}(\sigma)$ is an operator acting on a vector bundle;
  \item\label{ItItipFredSupp} we are working with \emph{supported} function spaces in $B$, i.e.\ future timelike infinity, rather than globally on the boundary of the radial compactification of Minkowski space. 
  \end{enumerate*}
  Since $\hat{\ul L}(\sigma)$ is principally scalar, \eqref{ItItipFredVB} only affects the threshold regularity at the radial set $N^*\pa B$. For $\gamma=0$, $\ul L$ is simply a conjugation of $\half$ times the scalar wave operator, acting diagonally on $\ul\C^{10}$, and in this case the threshold regularity is given as $s<\half+\Im\sigma$ in \cite[\S6]{BaskinVasyWunschRadMink2}, which is implied by our assumption $s<\half-C$. For small $\gamma>0$ (depending on the choice of $s$), this assumption is still sufficient. A straightforward calculation (which we omit) shows that the eigenvalues of $\sigma_1(t^3(\ul\tdel^*-\delta_{\ul g}^*)\delta_{\ul g}G_{\ul g}t^{-1})|_{N^*\pa B}$ are $\geq 0$, hence the threshold regularity is $s<\half+\Im\sigma$ for \emph{any} $\gamma\geq 0$. (This is closely related to the fact that the components of the solution of $\ul L u=f\in\CIc(\R^4)$ do not grow at $\scri^+$; see Lemma~\ref{LemmaItipFn} below for the relation between growth/decay on $M$ and regularity on $\ol{\R^4}$.)

  In order to deal with \eqref{ItItipFredSupp}, it is convenient to first study $\wh{\ul L}(\sigma)$ acting on supported distributions on a larger ball $B_d:=\{|X|\leq 1+d\}$. The only slightly delicate part of the argument establishing the Fredholm property of $\wh{\ul L}(\sigma)$ acting between $\Hsupp^s(B_2;\ul\C^{10})$-type spaces is the adjoint estimate: we need to show that $\wh{\ul L}(\sigma)^*$ satisfies an estimate
  \begin{equation}
  \label{EqItipFredAdjoint}
    \|u\|_{\Hext^{1-s}(B_2^\circ)}\lesssim \|\wh{\ul L}(\sigma)^*u\|_{\Hext^{-s}(B_2^\circ)}+\|u\|_{\Hext^{s_0}(B_2^\circ)}
  \end{equation}
  for some $s_0<1-s$; here $\Hext^s(B_2^\circ)$ denotes extendible distributions, i.e.\ restrictions of $\Hloc^s$ sections on $\pa N$ to $B_2^\circ$. This estimate however is straightforward to obtain by combining elliptic, real principal type, and radial point estimates in $B_1$, as in the references, with energy estimates for $\wh{\ul L}(\sigma)^*$ which is a \emph{wave} operator (on the principal symbol level) in $B_2\setminus B_{1/2}$, see e.g.\ \cite[\S3.2]{ZworskiRevisitVasy} where our $\wh{\ul L}(\sigma)^*$ is denoted $P$. High energy estimates for $\wh{\ul L}(\sigma)$ on $\Hsupp^s(B_2)$-type spaces follow by similar arguments (using \cite[Proposition~3.8]{VasyMicroKerrdS} for the energy estimate).

  Suppose now $\wh{\ul L}(\sigma)u=f\in\Hsupp^{s-1}(B)$ with $u\in\Hsupp^s(B_2)$. Then energy estimates in $B_2\setminus B$ imply $\supp u\subset B$. This and the Fredholm property of $\wh{\ul L}$ on $B_2$ yield the desired Fredholm property of $\wh{\ul L}\colon\fX^s\to\fY^{s-1}$ (specifically, the finite codimensionality of the range). Similarly, the high energy estimates on $B_2$ imply those on $B$, finishing the proof.
\end{proof}

\begin{lemma}
\label{LemmaItipResQual}
  For small $\gamma\geq 0$, all resonances $\sigma\in\C$ of $\ul L$ satisfy $\Im\sigma<0$.
\end{lemma}

\begin{rmk}
\label{RmkItipDivisor}
  One can in fact compute the divisor of $\ul L$, i.e.\ the set of $(z,k)\in\C\times\N_0$ such that $\wh{\ul L}(\sigma)^{-1}$ has a pole of order $\geq k+1$ at $\sigma=z$, quite explicitly for any $\gamma$: it is contained in $-i \extcup -2 i \extcup -i(1+\gamma) \extcup -i(1+2\gamma)$, using the shorthand notation~\eqref{EqCptFPhgShorthand}.
\end{rmk}

\begin{proof}[Proof of Lemma~\usref{LemmaItipResQual}]
  For $\gamma=0$, and in the trivialization of $S^2 T^*\R^4$ by coordinate differentials, $\ul L$ acts, up to conjugation and rescaling, component-wise as the scalar wave operator on Minkowski space, for which the divisor is known to be $-i$, see \cite[\S10.1]{BaskinVasyWunschRadMink}. For small $\gamma$, $\ul L$ is a small perturbation of this, and the lemma follows. (See also \cite[\S2.7]{VasyMicroKerrdS}.)
\end{proof}

Since by equation~\eqref{EqEinNipulLL0}, $L_0-\ul L\in\rho^{1-0}\Hb^\infty\Diffb^2({}^m\ol{\R^4})$, the normal operators \emph{as b-differential operators on ${}^m\ol{\R^4}$} are the same, $N(L_0)=N(\ul L)$, hence the above results hold for $N(L_0)$ in place of $\ul L$.

We next relate the relevant function spaces on ${}^m\!M$, ${}^m\ol{\R^4}$. We only need to consider supported distributions near ${}^mi^+\subset{}^m\!M$. We drop $m$ from the notation. If $\rho_+\in\CI(M)$ denotes a defining function of $I^+$ such that $\rho_+>2$ at $I^0$, let
\[
  U := \{ \rho_+ \leq 1 \} \subset M.
\]

Let $\cM_\bop\subset\Diffb^1(\ol{\R^4})$ be the $\CI(\ol{\R^4})$-module of b-differential operators with b-principal symbol vanishing on $\Nb^*S^+$,\footnote{The b-conormal bundle $\Nb^*S^+\subset\Tb^*_{S^+}\ol{\R^4}$ is the annihilator of the space of b-vector fields tangent to $S^+$. In the coordinates~\eqref{EqCptACp1Space}, $\cM_\bop$ is spanned by $\rho\pa_\rho$, $\rho\pa_v$, $v\pa_v$, and spherical vector fields.} and define $\Hbloc^{s,k}(\ol{\R^4})$ to consist of all $u\in\Hbloc^s(\ol{\R^4})$ for which $A_1\cdots A_\ell u\in\Hbloc^s(\ol{\R^4})$ for all $0\leq\ell\leq k$, $A_j\in\cM_\bop$. Supported distributions on a compact set $V\subset\ol{\R^4}$ are denoted $\Hbsupp^{s,k}(V)$.

\begin{lemma}
\label{LemmaItipFn}
  For $a_+\in\R$, $d>-\half$, and $k\in\N_0$, the map $\beta|_{U\setminus\pa M}\colon U\setminus\pa M\xra{\cong}\beta(U)\setminus\pa\ol{\R^4}$ induces a continuous inclusion
  \begin{equation}
  \label{EqItipFnBlowDown}
    \rho_I^{a_++d-1/2}\rho_+^{a_+} \Hbsupp^{k+d}(U) \hra \rho^{a_+}\Hbsupp^{d,k}(\beta(U)),
  \end{equation}
  and conversely
  \begin{equation}
  \label{EqItipFnBlowDown2}
    \rho^{a_+}\Hbsupp^{d,k}(\beta(U)) \hra \rho_I^{a_++d-1/2}\rho_+^{a_+}\Hbsupp^k(U).
  \end{equation}
\end{lemma}

Thus, given the condition on supports, b-regularity near $S^+$ is, apart from losses in module regularity, the same as decay at $\scri^+$. See Figure~\ref{FigItipNbh}. A version of the inclusion~\eqref{EqItipFnBlowDown2} is (implicitly) a key ingredient of~\cite{BaskinVasyWunschRadMink2}, see in particular \S9.2 there.

\begin{figure}[!ht]
\includegraphics{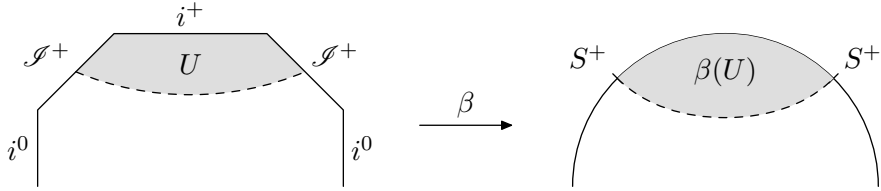}
\caption{The neighborhood $U$ of $I^+\subset M$ as well as its image in $\ol{\R^4}$ under the blow-down map $\beta$.}
\label{FigItipNbh}
\end{figure}

\begin{proof}[Proof of Lemma~\usref{LemmaItipFn}]
  First consider~\eqref{EqItipFnBlowDown}. Dividing by $\rho^{a_+}=\rho_I^{a_+}\rho_+^{a_+}$, it suffices to prove this for $a_+=0$. Furthermore, elements of $\cM_\bop$ lift to b-differential operators on $M$; in fact, $\Diffb^1(M)$ is generated, over $\CI(M)$, by the lift of $\cM_\bop$ to $M$. Therefore, it suffices to consider the case $k=0$ and prove
  \begin{equation}
  \label{EqItipFnBlowDownRed}
    \rho_I^{d-1/2} \Hbsupp^d(U) \hra \Hbsupp^d(\beta(U)),\ \ d>-\half.
  \end{equation}
  For $d=0$, this is a consequence of the fact that $\rho_I$ times a b-density on $M$ pushes forward to a b-density on $\ol{\R^4}$, cf.\ \eqref{EqBgscriL2Relation}. Next, note that $\Vb(\ol{\R^4})$ lifts to $\rho_I^{-1}\Vb(M)$ and thus maps $\rho_I^\alpha\Hbloc^s(M)\to\rho_I^{\alpha-1}\Hbloc^{s-1}(M)$; the Leibniz rule thus reduces the case $d\in\N$ to the already established case $d=0$. For general $d\geq 0$, \eqref{EqItipFnBlowDownRed} follows by interpolation; we discuss $d\in(-\half,0]$ below.

  For~\eqref{EqItipFnBlowDown2}, we again only need to consider $a_+=0$, $k=0$, and prove
  \begin{equation}
  \label{EqItipFnBlownDown2Red}
    \Hbsupp^d(\beta(U)) \hra \rho_I^{d-1/2}L^2_\bop(U) \cong \rho_I^d L^2_\bop(\beta(U)).
  \end{equation}
  For $d=0$, this is clear; for $d=1$, integrating the 1-dimensional Hardy inequality, $\|x^{-1}u\|_{L^2(\R_+)}\lesssim\|u'\|_{L^2(\R_+)}$, $u\in\CIc(\R_+)$, in fact gives $\Hbsupp^1(\beta(U))\hra x L^2_\bop(\beta(U))$, where $x$ is a defining function for $\beta(\pa U)$ \emph{within $\ol{\R^4}$}. In particular, $\beta^*x\in\CI(M)$ vanishes at $\scri^+$ and is hence a bounded multiple of $\rho_I$, from which~\eqref{EqItipFnBlownDown2Red} follows. For general $d\in\N$, we use the following generalization of the Hardy inequality: for $u\in\CIc(\R_+)$,
  \begin{align*}
    \|x^{-d}u\|_{L^2} &= \biggl\|\int_0^1\int_0^{s_2}\cdots\int_0^{s_d} u^{(d)}(t x)\,dt\,dt_d\cdots dt_2\,dx\biggr\|_{L^2} \\
      &\leq \int_0^1\int_0^{s_2}\cdots\int_0^{s_d} \|u^{(d)}(t\cdot)\|_{L^2}\,dt\,dt_d\cdots dt_2\,dx \\
      &= \frac{2^{2 d}d!}{(2 d)!} \|u^{(d)}\|_{L^2}.
  \end{align*}
  For real $d\geq 0$, \eqref{EqItipFnBlownDown2Red} again follows by interpolation.

  For $d\in(-\half,0]$, we dualize~\eqref{EqItipFnBlowDownRed} with respect to $L^2_\bop(\beta(U))$ and thus need to show $\Hbext^e(\beta(U))\hra\rho_I^{e-1/2}\Hbext^e(U)$, $e=-d\in[0,1/2)$. But this follows from~\eqref{EqItipFnBlowDown2}, as in this regularity range, supported and extendible Sobolev spaces are naturally isomorphic \cite[\S4.5]{TaylorPDE}. Similarly, \eqref{EqItipFnBlownDown2Red} for $d\in(-\half,0]$ follows from~\eqref{EqItipFnBlowDownRed} for $d\in[0,\half)$ by dualization.
\end{proof}

Returning to the proof of Theorem~\ref{ThmIt}, we have already proved $(1-\chi)u\in\cX^\infty$ where $\chi=\chi(\rho_+)$ is identically $1$ for $\rho_+\leq\half$ and vanishes for $\rho_+\geq 1$. Consider $\chi u\in\rho_I^{-0}\rho_+^{a_+}\Hbsupp^\infty(U)$, $a_+<b_+$, which satisfies
\[
  L_h \chi u=f_1:=\chi f+[L_h,\chi]u\in \rho_I^{-1-0}\rho_+^{b_+}\Hbsupp^\infty(U),
\]
where we use that $[L_h,\chi]u$ is supported away from $I^+$. Let
\[
  a'_+=\min(a_+ +1+b_+,b_+)<0,
\]
and fix $d\in(-\half,-\half-a'_+)$, then $L_h-N(L_0)\in \rho_I^{-1-0}\rho_+^{1+b_+}\Hb^\infty(M)\cdot\Diffb^2(M)$ (see Lemma~\ref{LemmaEinNi0p}) and Lemma~\ref{LemmaItipFn} yield
\begin{equation}
\label{EqItipEqn}
  N(L_0)\chi u =: f_2 \in \rho_I^{-1-0}\rho_+^{a'_+}\Hbsupp^\infty(U) \hra \rho^{a'_+}\Hbsupp^{d,\infty}(\beta(U)).
\end{equation}
Shrinking $U$ if necessary, we may assume that $t>1+r_*$ in $U$. It then suffices to use dilation-invariant operators on ${}^m\ol{\R^4}$ to measure module regularity at ${}^mS^+$. Indeed, for $m=0$ and thus $r_*=r$ (the discussion for general $m$ being similar), recall that with $R=|X|$, $\omega=X/|X|$, we can take $\tau\pa_\tau$, $(1-R)\pa_R$, $\pa_\omega$, and $\tau\pa_R$ as generators of $\cM_\bop$; but $\tau\pa_R=c(1-R)\pa_R$ with $c=\tau/(1-R)\in[0,1]$ bounded. Write~\eqref{EqItipEqn} using the Mellin transform in $\tau$ as
\[
  \chi u = \frac{1}{2\pi}\int_{\Im\sigma=-\alpha} \tau^{i\sigma}\wh{\ul L}(\sigma)^{-1}\wh{f_2}(\sigma)\,d\sigma,
\]
initially for $\alpha=-a_+$; then $\wh{f_2}(\sigma)$ is holomorphic in $\Im\sigma>-a'_+$ with values in $\Hsupp^{d,\infty}(B;\ul\C^{10})$, and in fact extends by continuity to
\begin{equation}
\label{EqItipF2Mellin}
  \wh{f_2}(\sigma) \in L^2\Bigl(\{\Im\sigma=-a'_+\}; \la\sigma\ra^{-d-N}\Hsupp_{\la\sigma\ra^{-1}}^{d,N}(B;\ul\C^{10})\Bigr) \quad (\forall\,N).
\end{equation}
By Lemmas~\ref{LemmaItipFred} and~\ref{LemmaItipResQual}, $\wh{\ul L}(\sigma)^{-1}\wh{f_2}(\sigma)$ is thus holomorphic in $\Im\sigma>-a'_+$ as well, with values in $\Hsupp^{d+1,\infty}$, extending by continuity to the space in~\eqref{EqItipF2Mellin} with $d$ replaced by $d+1$; therefore $\chi u\in \rho^{a'_+}\Hbsupp^{d+1,\infty}(\beta(U))$, so $\chi u\in\rho_I^{-0}\rho_+^{a'_+}\Hbsupp^\infty(U)$ by Lemma~\ref{LemmaItipFn}, as we may choose $d$ arbitrarily close to $-\half-a'_+$. This improves the weight of $u$ at $I^+$ by $a'_+-a_+$; iterating the argument gives $\chi u\in\rho_I^{-0}\rho_+^{b_+}\Hbsupp^\infty(U)$.

\subsection{Asymptotics near \texorpdfstring{$\scri^+\cap I^+$}{the future boundary of the radiation face}}
\label{SsItipScri}

It remains to show that the precise asymptotics at $\scri^+$ which we established away from $I^+$ in \S\ref{SsIti0Scri} extend all the way up to $I^+$, with the weight $\rho_+^{b_+}$ at $I^+$. This is completely parallel to the arguments in \S\ref{SsIti0Scri}: working near $I^+$, we now have $L^0_h u\in\cY^{\infty;b_0,b_I,b'_I,b_+}+\Hb^{\infty;b_0,-1+b'_I-0,b_+}$, so with coordinates $\rho_I,\rho_+$ as in~\eqref{EqBgipNullDefFn} (dropping the superscript `$\circ$'),
\[
  \rho_I\pa_{\rho_I}(\rho_+\pa_{\rho_+}-\rho_I\pa_{\rho_I})u_{1 1}^c\in\rho_I^{b_I}\rho_+^{b_+}\Hb^\infty;
\]
now, in $\rho_+>0$ (and away from $I^0$), $u_{1 1}^c$ has a leading term at $\scri^+$, plus a remainder in $\rho_I^{b_I}\Hb^\infty$, while in $\rho_I>0$, $u_{1 1}^c=\pi_{1 1}^c u$ lies in $\rho_+^{b_+}\Hb^\infty$. Using Lemma~\ref{LemmaPhgODE} to integrate the above equation for $u_{1 1}^c$, we conclude that $u_{1 1}^c$ is the sum of a leading term in $\rho_0^{b_0}\rho_+^{b_+}\Hb^\infty(\scri^+)$ and a remainder in $\rho_0^{b_0}\rho_I^{b_I}\rho_+^{b_+}\Hb^\infty$, as desired. Similarly, we obtain the desired asymptotic behavior, uniformly up to $I^+$, of $u_{1 1}$ and then of $u_0$. Therefore, $u\in\cX^{\infty;b_0,b_I,b'_I,b_+}$, completing the proof of Theorem~\ref{ThmIt}.

\section{Proof of global stability}
\label{SPf}

We now make Theorem~\ref{ThmIt} quantitative by keeping track of the number of derivatives used and proving \emph{tame estimates}, the crucial ingredient in Nash--Moser iteration. Fix the mass $m$; for weights $b_0,b_I,b'_I,b_+$ as in Definitions~\ref{DefEinF} and \ref{DefEinFY}, let
\[
  B^k := \cX^{k;b_0,b_I,b'_I,b_+}; \quad
  \bfB^k := \cY^{k;b_0,b_I,b'_I,b_+} \oplus D^{k;b_0},\ \ D^{k;b_0}:=\rho_0^{b_0}\Hb^{k+1}(\Sigma) \oplus \rho_0^{b_0}\Hb^k(\Sigma).
\]
Let us write $|\cdot|_s$, resp.\ $\|\cdot\|_s$, for the norm on $B^s$, resp.\ $\bfB^s$. Put
\[
  B^\infty = \bigcap_{k\in\N} B^k,\quad
  \bfB^\infty = \bigcap_{k\in\N} \bfB^k.
\]
We recall Saint-Raymond's version \cite{SaintRaymondNashMoser} of the Nash--Moser inverse function theorem:

\begin{thm}[See \cite{SaintRaymondNashMoser}]
\label{ThmPfNM}
  Let $\phi\colon B^\infty\to\bfB^\infty$ be a $\cC^2$ map, and assume that there exist $d\in\N$, $\eps>0$, and constants $C_1,C_2,(C_s)_{s\geq d}$ such that for any $h,u,v\in B^\infty$ with $|h|_{3 d}<\eps$,
  \begin{subequations}
  \begin{align}
  \label{EqPfNM1}
    \|\phi(h)\|_s &\leq C_s(1+|h|_{s+d})\ \ \forall\,s\geq d, \\
  \label{EqPfNM2}
    \|\phi'(h)u\|_{2 d} &\leq C_1|u|_{3 d}, \\
  \label{EqPfNM3}
    \|\phi''(h)(u,v)\|_{2 d} &\leq C_2|u|_{3 d}|v|_{3 d}.
  \end{align}
  \end{subequations}
  Moreover, assume that for such $h$, there exists an operator $\psi(h)\colon\bfB^\infty\to B^\infty$ satisfying $\phi'(h)\psi(h)f=f$ and the tame estimate
  \begin{equation}
  \label{EqPfNMTame}
    |\psi(h)f|_s \leq C_s(\|f\|_{s+d}+|h|_{s+d}\|f\|_{2 d}),\ \ \forall\,s\geq d,\ f\in\bfB^\infty.
  \end{equation}
  Then if $\|\phi(0)\|_{2 d}<c$, where $c>0$ is a constant depending on $\eps$ and $C_s$ for $s\leq D$, where $D=16 d^2+43 d+24$, there exists $h\in B^\infty$, $|h|_{3 d}<\eps$, such that $\phi(h)=0$.
\end{thm}

This uses a family of smoothing operators $(S_\theta)_{\theta>1}\colon B^\infty\to B^\infty$ satisfying the estimates
\begin{equation}
\label{EqPfNMSmoothing}
  |S_\theta v|_s\leq C_{s,t}\theta^{s-t}|v|_t,\ \ s\geq t; \qquad |v-S_\theta v|_s\leq C_{s,t}\theta^{s-t}|v|_t,\ \ s\leq t.
\end{equation}
Acting on standard Sobolev spaces $H^s(\R^n)$, the existence of such a family is proved in \cite[Appendix]{SaintRaymondNashMoser}, and the extension to weighted b-Sobolev spaces on manifolds with corners is straightforward: the arguments on manifolds with boundary given in~\cite[\S11.2]{HintzVasyKdSStability} generalize directly to the corner setting. For the spaces $B^s=\cX^s$ at hand then, one writes $h\in B^\infty$ as $\chi_1 h+(1-\chi_1)h$, with $\chi_j\in\CI(M)$, $j=0,1,2$, identically $1$ in a small neighborhood of $\scri^+$, and $\chi_{j+1}\equiv 1$ on $\supp\chi_j$. We smooth out $(1-\chi_1)h\in\rho_0^{b_0}\rho_I^\infty\rho_+^{b_+}\Hb^\infty(M)$ (see~\eqref{EqCptFInftyW} for the notation $\rho_I^\infty$) as usual and cut the result off using $(1-\chi_0)$; since we are working away from $\scri^+$, the weight of $\rho_I$ plays no role here. (The proof of \cite[Lemma~5.9]{HintzVasyQuasilinearKdS} shows that cutting off the smoothing of $(1-\chi_1)h$ away from its support does not affect the estimates~\eqref{EqPfNMSmoothing}.) Near $\scri^+$ on the other hand, we have $\chi_1 h=(\chi_1 h_\alpha)$, where we denote by $h_\alpha$ the components of $h$ in the bundle splitting~\eqref{EqCptASpl2}. The decaying components~\eqref{EqEinFGood} as well as the remainder terms $h_{\alpha,\bop}$ in \eqref{EqEinF11}--\eqref{EqEinFRest} can then be smoothed out and cut off using $\chi_2$. To smooth out the leading terms, fix a collar neighborhood of $\scri^+$; considering for example $\chi_1 h_{0 1}=\chi_0 h_{0 1}^{(0)} + \chi_1 h_{0 1,\bop}$, see~\eqref{EqEinFRest}, we smooth out $h_{0 1}^{(0)}$ in the weighted b-Sobolev space $\rho_0^{b_0}\rho_+^{b_+}\Hb^\infty(\scri^+)$, extend the result to the collar neighborhood, and cut off using $\chi_0$; similarly for the other components of $h$.

Given initial data $(h_0,h_1)\in D^\infty$, we want to apply Theorem~\ref{ThmPfNM} to the map
\begin{equation}
\label{EqPfMap}
  \phi(h)=\bigl(P(h),\,(h,\pa_\nu h)|_\Sigma-(h_0,h_1)\bigr),
\end{equation}
with $P$ given in~\eqref{EqEin}. Note that the smallness of $\phi(0)$ in particular requires $P(0)=\rho^{-3}\Ric(g_m)$ to be small. Now, $P(0)$ is nonzero only in the region where we interpolate between the mass $m$ Schwarzschild metric and the Minkowski metric (both of which are Ricci-flat!), i.e.\ on $\supp d\psi\cup\supp d\phi$ in the notation of~\eqref{EqCptAMetric1}--\eqref{EqCptAMetric}; thus in fact $P(0)\in\cA_\phg^{\emptyset,\emptyset,0}$. It is then easy to see that $\|P(0)\|_{\cY^k}\leq C_k m$ for all $k\in\N$, which is the reason why we need to assume the ADM mass $m$ to be \emph{small} to get global solvability.

For $h\in\cX^\infty$ with $|h|_3$ small, the tensor
\[
  g=g_m+\rho h
\]
is Lorentzian (by Sobolev embedding) and hence $\phi(h)$ is defined; since $P$ is a second order (nonlinear) differential operator with coefficients which are polynomials in $g^{-1}$ and up to $2$ derivatives of $g$, and since $h\mapsto(h,\pa_\nu h)|_\Sigma$ is continuous as a map $\cX^k\to D^{k-3/2}$ for $k\geq 2$, the estimate~\eqref{EqPfNM1} follows for $d=3$. The estimate~\eqref{EqPfNM2} also holds for $d=3$ and $|h|_{3 d}<\eps$ small, since the first component of $\phi'(h)u$, namely $L_h u$, is a second order linear differential operator acting on $u$, with coefficients involving at most $2$ derivatives of $h$; similarly for~\eqref{EqPfNM3}.

The existence of the right inverse $\psi(u)\colon\bfB^\infty\to B^\infty$ is the content of Theorem~\ref{ThmIt}; we merely need to determine a value for $d$ such that the tame estimate~\eqref{EqPfNMTame} holds. (As stressed in the introduction, the mere \emph{existence} of such a $d$ is clear, since the estimates on $\psi(u)$ are obtained using energy methods, integration along approximate characteristics, and inversion of a linear, smooth coefficient, model operator in \S\ref{SBg}, \S\S\ref{SsIti0Scri} and \ref{SsItipScri}, and \S\ref{SsItip}, respectively.) Consider the first term on the right in~\eqref{EqPfNMTame}: we need to quantify the loss of derivatives of the solution $v$ of $L_h u=f$, $(u,\pa_\nu u)|_\Sigma=(u_0,u_1)$, relative to the regularity $k\geq 0$ of $(f,(u_0,u_1))\in\bfB^k$.

Now, dropping the $\Hscri^1$ regularity part of Theorem~\ref{ThmBg}, we obtain $u\in\rho_0^{b_0}\rho_I^{a_I}\rho_+^{a_+}\Hb^k$, $\pi_0 u\in\rho_0^{b_0}\rho_I^{a'_I}\rho_+^{a_+}\Hb^k$. The arguments near $I^0\cap\scri^+$ in~\S\ref{SsIti0Scri} first express $u_{1 1}^c$ as the solution of a transport equation~\eqref{EqIti0Scri11c}, with the right hand side involving up to two derivatives of $u$; since integration of this equation does not regain full b-derivatives, the leading terms (and the remainder term) of $u_{1 1}^c$ lie in $\Hb^{k-2}$, with the correct weight $b_0$ at $I^0$ (and $b_I$ at $\scri^+$); next, this couples into the transport equation~\eqref{EqIti0Scri11} for $u_{1 1}$, again with up to $2$ derivatives of $u$, so integrating this yields leading and remainder terms of $u_{1 1}$ in $\Hb^{k-4}$; and similarly then $u_0\in\rho_0^{b_0}\rho_I^{b'_I}\Hb^{k-6}$ near $I^0\cap\scri^+$.

On the other hand, improving the b-weight at $I^+$ by $1+b_+$, which we may take to be arbitrarily close to $1$ by taking $b_+<0$ close to $0$, uses the rewriting~\eqref{EqItipEqn}, which due to the second order nature of $L_h-N(L_0)$ involves an error term (subsumed into $f_2$ there) with $2$ derivatives on $u$. Passing to the blow-down using Lemma~\ref{LemmaItipFn} loses at most $1$ module derivative; inverting $N(L_0)$ gains $1$ b-derivative (which is used to recover the $\rho_I^{-0}$ bound at $\scri^+$), but no module derivatives, so passing back to the blow-up, we have lost at most $3$ b-derivatives. Thus, improving the weight at $I^+$ from $a_+$ to $b_+\approx 0$ loses at most $d_+:=1+3\lceil a_+\rceil$ derivatives relative to $\Hb^k$.

These two pieces of information are combined near $\scri^+\cap I^+$ in \S\ref{SsItipScri}, where we lose at most $6$ derivatives, just as in the discussion near $I^0\cap\scri^+$, relative to the less regular of the two spaces $\Hb^{k-6}$ and $\Hb^{k-d_+}$ from above; we thus take $d=6+\max(6,d_+)$. If we use the explicit background estimate, Theorem~\ref{ThmBgExpl}, so $a_+=\tfrac32$, this gives $d_+=7$ and therefore
\[
  d=13.
\]
For this value of $d$, one may then verify the tame estimate~\eqref{EqPfNMTame} by going through the proofs of Theorems~\ref{ThmBg} and \ref{ThmIt} and proving tame estimates by exploiting Moser estimates; this is analogous to the manner in which the microlocal estimates for smooth coefficient operators in \cite[\S2]{VasyMicroKerrdS}, \cite[\S2.1]{HintzVasySemilinear} were extended to estimates for rough coefficient operators in \cite[\S\S3--6]{HintzQuasilinearDS}, which were subsequently sharpened to tame estimates in \cite[\S\S3--4]{HintzVasyQuasilinearKdS}. In the present setting, obtaining tame estimates is much simpler than in the references, as the estimates in~\S\S\ref{SBg}--\ref{SIt} are based on standard energy estimates, so one can appeal directly to the Moser estimates; or, in view of the fact that our energy estimates can be proved using positive commutators (and are indeed phrased this way here), which also underlie the tame estimates in these references, the arguments given there (using vector fields instead of microlocal commutants) apply here as well. We omit the details, but we do point out that it is key that the proofs as stated only use pointwise control of up to $1$ derivative of $h$ (via causality considerations and deformation tensors, see e.g.\ the calculation~\eqref{EqBgscriFinalDefFn} and Lemma~\ref{LemmaBgscriImprovedK}) in order to obtain the main positive terms in the commutator arguments; thus, control of $|h|_4$ suffices in this sense, that is, the constant in~\eqref{EqBgEstimate} for $k=1$ only depends on $|h|_4$. The proofs of higher b-regularity use commutation arguments, which \emph{do not} affect the principal part of $L_h$, as well as ellipticity considerations around~\eqref{EqBgipEll} which only require pointwise control of $h$ itself; correspondingly, at no point do we need to use the smallness of any higher regularity norms of $h$. (See the end of~\cite[\S6.4]{HormanderNonlinearLectures} for a related discussion.)

Next, we deal with a small technical complication stemming from the fact that for $m\neq 0$, the closure of $\{t=0\}$, on which in Theorem~\ref{ThmIBaby} we compare the initial data with those of the Schwarzschild metric in its standard form, inside of ${}^m\ol{\R^4}$ is not a smooth hypersurface when $m\neq 0$, the issue being smoothness at $\pa{}^m\ol{\R^4}$; furthermore, our discussion of linear Cauchy problems used ${}^m\Sigma\neq\overline{\{t=0\}}$ as the Cauchy surface. We resolve this issue by solving the initial value problem for a short amount of time in the radial compactification ${}^0\ol{\R^4}$, with initial surface $\{t=0\}$ (whose closure \emph{is} smooth in ${}^0\ol{\R^4}$), pushing the local solution forward to ${}^m\ol{\R^4}$, and then solving globally from there. Recall the function $t_\bop$ from~\eqref{EqCptATimeB}, and the notation~\eqref{EqCptAExpl}. (Thus, ${}^0 t_\bop$ is a rescaling of $t$, and ${}^0\Sigma=\{{}^0 t_\bop=0\}$.)

\begin{lemma}
\label{LemmaPf0}
  Fix $N$ large, and let $b_0>0$, $\eps>0$. Suppose $\gamma,k\in\CI(\R^3;S^2 T^*\R^3)$ are vacuum initial data on $\R^3$, that is, solutions of the constraint equations~\eqref{EqIConstraints}, such that for some $m\in\R$,
  \begin{equation}
  \label{EqPf0GammaTilde}
    \wt\gamma := \gamma - \chi(r)\bigl((1-\tfrac{2 m}{r})^{-1}dr^2+r^2\slg\bigr) \in \rho_0^{1+b_0}\Hb^\infty(\ol{\R^3};S^2\,\Tsc^*\ol{\R^3})
  \end{equation}
  and $k\in\rho_0^{2+b_0}\Hb^\infty(\ol{\R^3};S^2\,\Tsc^*\ol{\R^3})$ satisfy
  \begin{equation}
  \label{EqPf0Small}
    |m|+\|\wt\gamma\|_{\rho_0^{1+b_0}\Hb^{N+1}}+\|k\|_{\rho_0^{2+b_0}\Hb^N}<\delta,
  \end{equation}
  where $\delta>0$ is a sufficiently small constant; here $\chi=\chi(r)$ is a cutoff, $\chi\equiv 0$ for $r<1$, $\chi\equiv 1$ for $r>2$. Then, identifying $\ol{\R^3}\cong{}^0\Sigma\subset{}^0\!M$ via $\R^3\ni x\mapsto(0,x)\in\R^4$, there exists a solution $g$ of the Einstein vacuum equation $\Ric(g)=0$ in the neighborhood
  \begin{equation}
  \label{EqPf0Nbh}
    U:=\{|{}^0 t_\bop|<\tfrac14\},
  \end{equation}
  attaining the data $(\gamma,k)$ at ${}^0\Sigma$ (that is, \eqref{EqIEinIVP} holds) and satisfying the gauge condition $\Ups(g;g_m)=0$; moreover, $g=g_m+\rho h$, where $h\in\rho_0^{b_0}\Hb^\infty(U;S^2\,\Tsc^*\,{}^0\ol{\R^4})$ has norm $\|h\|_{\rho_0^{b_0}\Hb^{N+1}(U)}<\eps$.
\end{lemma}
\begin{proof}
  Note that the metric $g_m$ is smooth on $U\subset{}^0\ol{\R^4}$, as near $I^0$ it is given by the Schwarzschild metric $g_m^S$, see~\eqref{EqISchw}. Using the product decomposition $\R^4=\R_t\times\R^3_x$, we define a Lorentzian signature metric over the interior $({}^0\Sigma)^\circ=\{t=0\}$ by
  \begin{equation}
  \label{EqPf0G0}
    g_0 := (1-\chi(r)\tfrac{2 m}{r})d t^2 - \gamma \in \CI(({}^0\Sigma)^\circ;S^2 T^*\R^4),
  \end{equation}
  whose pullback to ${}^0\Sigma$ is equal to $-\gamma$. We next find $g_1\in\CI({}^0\Sigma;S^2 T^*\R^4)$ such that $k=\II_{g_0+t g_1}$; denoting by $N=(1-\chi(r)\tfrac{2 m}{r})^{-1/2}\pa_t$ the future unit normal, this is equivalent, by polarization, to
  \[
    g_0((\nabla_X^{g_0+t g_1}-\nabla_X^{g_0})X,N)=k(X,X)\ \ 
    \forall\,X\in T({}^0\Sigma)^\circ;
  \]
  Here, we view $g_0$ as a stationary metric near $t=0$, which due to its symmetry under time reversal $t\mapsto-t$ has vanishing second fundamental form: $g_0(\nabla_X^{g_0}X,N)\equiv 0$. A calculation in normal coordinates for $g_0$ shows that this is uniquely solved by
  \begin{equation}
  \label{EqPf0G1Spatial}
    g_1(X,X)=-2(N t)^{-1} k(X,X) = -2(1-\chi(r)\tfrac{2 m}{r})^{1/2}k(X,X).
  \end{equation}
  It remains to specify $g_1(N,\cdot)$ and $g_1(N,N)$, which involves the gauge condition at $t=0$; that is, for all $V\in T_{\{t=0\}}\R^4$, we require
  \begin{equation}
  \label{EqPf0G1Rest}
  \begin{split}
    -\Ups(g_0;g_m)(V) &= \bigl(\Ups(g_0+t g_1;g_m)-\Ups(g_0;g_m)\bigr)(V) \\
      &= (G_{g_0}g_1)(V,\nabla^{g_0}t) = (1-\chi(r)\tfrac{2 m}{r})^{-1/2} (G_{g_0}g_1)(V,N).
  \end{split}
  \end{equation}
  For $V\in T({}^0\Sigma)^\circ$, this determines $(G_{g_0}g_1)(V,N)=g_1(V,N)$. Lastly, if $E_1,E_2,E_3\in T({}^0\Sigma)^\circ$ completes $N$ to an orthonormal basis, this also determines $(G_{g_0}g_1)(N,N)=\half(g_1(N,N)+\sum_j g_1(E_j,E_j))$ and thus $g_1(N,N)$.

  The assumption on $\gamma$ gives
  \begin{equation}
  \label{EqPf0H0}
    h_0:=\rho_0^{-1}(g_0-g_m)\in\rho_0^{b_0}\Hb^\infty({}^0\Sigma;S^2\,\Tsc^*_{{}^0\Sigma}{}^0\ol{\R^4}).
  \end{equation}
  We claim that likewise
  \begin{equation}
  \label{EqPf0H1}
    h_1:=\rho_0^{-2}g_1\in\rho_0^{b_0}\Hb^\infty({}^0\Sigma;S^2\,\Tsc^*_{{}^0\Sigma}{}^0\ol{\R^4}).
  \end{equation}
  We introduce the extra factor of $\rho_0^{-1}$ since $\rho_0^{-1}\pa_t$ is a smooth b-vector field on ${}^0\ol{\R^4}$ near ${}^0\Sigma$ and transversal to it; that is, in~\eqref{EqBgPartialNu}, we can take
  \[
    \pa_\nu=\rho_0^{-1}\pa_t.
  \]
  Now the restriction of $h_1$ to $S^2\,\Tsc\,{}^0\Sigma$ lies in $\rho_0^{b_0}\Hb^\infty$, as follows from~\eqref{EqPf0G1Spatial}. (Recall that $\Tsc\,{}^0\Sigma$ is spanned by coordinate vector fields on $\R^3$.) To prove~\eqref{EqPf0H1}, it thus suffices to prove that $\Ups(g_0;g_m)(V)\in\rho_0^{2+b_0}\Hb^\infty$ for $V$ equal to $\pa_t$ or a coordinate vector field on $\R^3$; this however follows from~\eqref{EqPf0H0} and the local coordinate expression~\eqref{EqEinUps} of $\Ups$, as such a vector field $V$ is equal to $\rho_0$ times a b-vector field on ${}^0\ol{\R^4}$.

  This construction preserves smallness, i.e.\ we have $\|h_0\|_{\rho_0^{b_0}\Hb^{N+1}}+\|h_1\|_{\rho_0^{b_0}\Hb^N}<C\delta$ for some constant $C$. We can then solve the quasilinear wave equation $P(h)=0$ in the neighborhood $U$ of ${}^0\Sigma$, e.g.\ using Nash--Moser iteration as explained above. (Since we are not solving up to $\scri^+$ where our arguments in~\S\ref{SIt} lose derivatives, one can use a simpler iteration scheme here, see~\cite[\S16.1]{TaylorPDE}.) The constraint equations then imply that $\pa_\nu\Ups(g_m+\rho h;g_m)=0$ at ${}^0\Sigma$, see~\cite[\S2.1]{HintzVasyKdSStability}; since $\Ups$ solves the wave equation~\eqref{EqIEinUpsProp}, we have $\Ups\equiv 0$.
\end{proof}

To extend this to a global solution, we recall from Lemma~\ref{LemmaCptAComp} and the isomorphism~\eqref{EqCptACompHb} that $h$ pushes forward to an element of $\rho_0^{b_0}\Hb^\infty(U')$, $U':=\{|{}^m t_\bop|<\tfrac18\}$, and satisfies a bound $\|h\|_{\rho_0^{b_0}\Hb^{N+1}(U')}<C\eps$, with $C$ a constant depending only on $m$. We can thus use $(h_0,h_1)=(h,\pa_\nu h)|_{{}^m\Sigma}$ as Cauchy data for the equation $P(h)=0$. Note that the gauge condition $\Ups(g)=0$, $g=g_m+\rho h$, holds identically near ${}^m\Sigma$; by uniqueness of solutions of $P(h)=0$ with Cauchy data $(h_0,h_1)$, a global solution $h$ will automatically satisfy $\Ups(g)\equiv 0$, as this holds near ${}^m\Sigma$, and then globally by the argument given around equation~\eqref{EqIEinUpsProp}.

\begin{thm}
\label{ThmPf}
  Fix $N$ large, $b_0>0$, $\eps>0$, and $0<\eta<\min(\half,b_0)$. Then if $m\in\R$ and $h_0,h_1\in\rho_0^{b_0}\Hb^\infty({}^m\Sigma)$ satisfy
  \[
    |m|+\|h_0\|_{\rho_0^{b_0}\Hb^{N+1}}+\|h_1\|_{\rho_0^{b_0}\Hb^N}<\delta,
  \]
  where $\delta>0$ is a small constant, then there exists a global solution $h$ of
  \begin{equation}
  \label{EqPfIVP}
    P(h)=0, \quad (h,\pa_\nu h)|_{{}^m\Sigma}=(h_0,h_1),
  \end{equation}
  that is,
  \[
    \Ric(g)-\tdel^*\Ups(g)=0,\ \ g=g_m+\rho h,
  \]
  which satisfies $h\in\cX^{\infty;b_0,b_I,b'_I,b_+}$ for all weights $b_I<b'_I<\min(1,b_0)$ and $b_+<0$, and so that moreover $\|h\|_{\cX^{6;b_0,\eta,\eta/2,-\eta}}<\eps$. If in addition $\Ups(g_m+\rho h;g_m)=0$, $\pa_\nu\Ups(g_m+\rho h;g_m)=0$ at ${}^m\Sigma$, then $g$ solves
  \[
    \Ric(g)=0
  \]
  in the gauge $\Ups(g)=0$.
\end{thm}

As explained above, data for which the assumption in the second part of the theorem holds arise from an application of Lemma~\ref{LemmaPf0}. This assumption is equivalent to the statement that the Riemannian metric and second fundamental form of ${}^m\Sigma$ induced by a metric $g_m+\rho h$ with $(h,\pa_\nu h)|_{{}^m\Sigma}=(h_0,h_1)$ satisfy the constraint equations, and that the gauge condition $\Ups(h;g_m)=0$ holds pointwise at ${}^m\Sigma$. These are assumptions only involving the data $(h_0,h_1)$; the vanishing of $\pa_\nu\Ups(h)|_{{}^m\Sigma}$ for the solution $h$ of $P(h)=0$ with these data follows as in the proof of Lemma~\ref{LemmaPf0}.

\begin{proof}[Proof of Theorem~\usref{ThmPf}]
  This follows, with $b_I<b'_I<\min(\half,b_0)$ at first, for $N=2 d=26$, from Theorem~\ref{ThmPfNM} applied to the map in~\eqref{EqPfMap}. The constant $\delta>0$ depends in particular on the constants $C_s$ in~\eqref{EqPfNM1} for $s\leq D=3287$; that is, $\delta=\delta(\|h_0\|_{\rho_0^{b_0}\Hb^{D+1}}+\|h_1\|_{\rho_0^{b_0}\Hb^D})$. Repeating the arguments in \S\S\ref{SsIti0Scri} and \ref{SsItipScri} once more shows that one can take $b_I<b'_I<\min(1,b_0)$; see also the proof of Theorem~\ref{ThmPhg} below.
  
  We remark that $h$ is in fact small in $\cX^{3 d}=\cX^{39}$, but if one is interested in the size of up to two derivatives (e.g.\ curvature) of $h$, control of its $\cX^6$ norm is sufficient by Sobolev embedding.
\end{proof}

\begin{rmk}
\label{RmkPfCts}
  In other words, using the notation of the proof and $d\geq 13$, $N=2 d$, $D=16 d^2+43 d+24=3287$, and fixing $m$ and $b_0$, we can solve the initial value problem~\eqref{EqPfIVP} for data in the space $\sD:=\bigcup_C \sD(C)$, where
  \begin{align*}
    \sD(C):=\Bigl\{(h_0,h_1) \colon h_0,\,h_1\in\rho_0^{b_0}\Hb^\infty({}^m\Sigma),\ \ 
      |m|+{}&\|h_0\|_{\rho_0^{b_0}\Hb^{N+1}}+\|h_1\|_{\rho_0^{b_0}\Hb^N} <\delta(C), \\
      &\|h_0\|_{\rho_0^{b_0}\Hb^{D+1}}+\|h_1\|_{\rho_0^{b_0}\Hb^D}<C \Bigr\}.
  \end{align*}
  An inspection of the proof of Theorem~\ref{ThmPfNM} in \cite{SaintRaymondNashMoser} shows that $\lim_{C\to 0}\delta(C)>0$, so $\sD$ in particular contains all conormal data $(h_0,h_1)$ for which $|m|+\|h_0\|_{\rho_0^{b_0}\Hb^{D+1}}+\|h_1\|_{\rho_0^{b_0}\Hb^D}<\delta_0$, where $\delta_0>0$ is a universal constant (i.e.\ depending only on $m$ and $b_0$). Moreover, one also has a continuity statement: for any choice of weights $b_I,b'_I,b_+$ as in Theorem~\ref{ThmPf}, the solution $h\in\cX^{3 d;b_0,b_I,b'_I,b_+}$ of~\eqref{EqPfIVP} depends continuously on $(h_0,h_1)\in\sD$, the latter being equipped with the $\rho_0^{b_0}\Hb^{D+1}\oplus\rho_0^{b_0}\Hb^D$ topology.\footnote{Hamilton~\cite{HamiltonNashMoser} shows that the data-to-solution map is in fact a tame smooth map $D^{\infty;b_0}\ni(h_0,h_1)\mapsto h\in\cX^{\infty;b_0,b_I,b'_I,b_+}$ (defined in the neighborhood $\sD$ of the origin of $D^{\infty;b_0}$).} Indeed, to obtain continuity at the Minkowski solution, note that the map $\phi$ in~\eqref{EqPfMap} depends parametrically on the data $(h_0,h_1)\in\sD$, but the constants appearing in the estimates in~\cite{SaintRaymondNashMoser} can be taken to be \emph{uniform} when $(h_0,h_1)$ varies in $\sD(C)$ with $C$ fixed. Continuity at other solutions is similarly automatic, but the base point of the Nash--Moser iteration (called $u_0$ in \cite[Lemma~1]{SaintRaymondNashMoser}) should then be given by the solution one is perturbing around.
\end{rmk}

The solution $h$ of~\eqref{EqPfIVP} in fact has a leading term at $I^+$, as will follow from the arguments in \S\ref{SPhg}, see the discussion around~\eqref{EqPhgLocip}; this precise information was not needed to close the iteration scheme, hence we did not encode it in the spaces $\cX^s$.

The conclusion in the form given in Theorem~\ref{ThmIBaby} can be obtained by combining Lemma~\ref{LemmaPf0} and Theorem~\ref{ThmPf}: using the coordinate $t_\bop$ on ${}^m\!M'$, the initial surface ${}^0\Sigma$ in Minkowski space is given by $t_\bop=-2 m \rho_0\chi(r)\log(r-2 m)$. A diffeomorphism of ${}^m\ol{\R^4}$ which near ${}^m\Sigma$ is not smooth but rather polyhomogeneous with index set $\cE_{\rm log}$, and which is the identity away from ${}^m\Sigma$, can be used to map $\{t_\bop\geq -2 m\rho_0\chi(r)\log(r-2 m)\}\subset{}^m\!M'$ to ${}^m\!M=\{t_\bop\geq 0\}$; pushing the solution $g$ obtained from Lemma~\ref{LemmaPf0} and Theorem~\ref{ThmPf}, which is defined on $t\geq 0$, forward using this diffeomorphism produces the solution $g$ as in Theorem~\ref{ThmIBaby}. (The gauge condition satisfied by $g$ is the wave map condition with respect to the background metric which is the pushforward of $g_m$.) We omit the proofs of future causal geodesic completeness of $(M,g)$, as one can essentially copy the arguments of Lindblad--Rodnianski \cite[\S16]{LindbladRodnianskiGlobalExistence}.

\begin{rmk}
\label{RmkPfPointwise}
  By Sobolev embedding, $h$ obeys the pointwise bound
  \begin{equation}
  \label{EqPfPointwise}
    |h|\leq C_\eta(1+t+r)^{-1+\eta}(1+(r_*-t)_+)^{b_0}\ \ \forall\,\eta>0
  \end{equation}
  and is small for fixed $\eta>0$ if $\delta=\delta(\eta)>0$ in the theorem is sufficiently small; here, we measure the size of $h$ using any fixed Riemannian inner product on the fibers of $\beta^*S^2$, equivalently, by measuring $\sum_{i j}|h(Z_i,Z_j)|$, where $\{Z_i\}=\{\pa_t,\pa_{x^1},\pa_{x^2},\pa_{x^3}\}$ are coordinate vector fields. The bound~\eqref{EqPfPointwise} also holds for all covariant derivatives of $h$ along b-vector fields on ${}^m\!M$. In particular, by Lemma~\ref{LemmaEinNgmMink}, $|g-\ul g|\leq C_\eta(1+t+r)^{-1+\eta}$, $\eta>0$. The Riemann curvature tensor also decays to $0$ as $t+r\to\infty$, with the decay rate depending on the component: this follows from an inspection of the expressions in \S\ref{SsCoPert}. Note however that the components in the frame~\eqref{EqCptASplProd} have no geometric meaning away from $\scri^+$. Geometric and more precise decay statements were obtained by Klainerman--Nicol\`o \cite{KlainermanNicoloEvolution}.
\end{rmk}

\begin{rmk}
\label{RmkPfLargeMass}
  If the ADM mass $m$ of the initial data is large, there does not exist a metric with the mass $m$ Schwarzschild behavior near $\scri^+$ but Minkowski-like far from $I^0\cup\scri^+$ which is sufficiently close to being Ricci flat for an application of a small data nonlinear iteration scheme like Nash--Moser: this follows from work of Christodoulou~\cite{ChristodoulouTrapped}, Klainerman--Rodnianski and Luk~\cite{KlainermanRodnianskiTrapped,KlainermanLukRodnianskiTrapped}, An--Luk \cite{AnLukTrapped}, and (for the non\-char\-ac\-ter\-is\-tic problem) Li--Yu \cite{LiYuTrapped}. On the other hand, for arbitrary $m$, but without the smallness condition~\eqref{EqIBabySmall} on the data, one does obtain small data by restricting to the complement of a sufficiently large ball. Working on a suitable submanifold of ${}^m\!M$, defined near $I^0\cap\scri^+$ by $\rho_0<\eps+\rho_I^\beta$ for $\beta\in(0,b_0)$ and $\eps>0$ sufficiently small, cf.\ \eqref{EqBgscriDom}, our method of proof then ensures the existence of a vacuum solution on this submanifold; in particular, the solution includes a piece of null infinity.
\end{rmk}

We can also solve towards the past: Lemma~\ref{LemmaPf0} produces a solution $g$ of Einstein's equation in the gauge $\Ups(g;g_m)=0$ in a full neighborhood of $\{t=0\}$, and we can then use the time-reversed analogue of Theorem~\ref{ThmPf} for solving backwards in time, obtaining a global solution $g$ on $\R^4$. Note here that by construction, the background metric $g_m$ is invariant under the time reversal map $\iota\colon t\mapsto -t$ on $\R^4$, hence the gauge conditions of the future and past solutions match. To describe the behavior of $g$ on a compact space, as illustrated in Figure~\ref{FigIBaby}, let us denote by ${}_m\ol{\R^4}$ the compactification defined like ${}^m\ol{\R^4}$ in \S\ref{SsCptA} but with $t$ replaced by $-t$ everywhere. Thus, $\iota$ induces diffeomorphisms ${}^m\ol{\R^4}\cong{}_m\ol{\R^4}$; denote by $S^-$ the image of $S^+$. The identity map on $\R^4$ induces an identification of the interiors of ${}^m\ol{\R^4}$ and ${}_m\ol{\R^4}$ which extends to be polyhomogeneous of class $\cA_\phg^{\cE_{\rm log}}$ on the maximal domain of existence by a simple variant of Lemma~\ref{LemmaCptAComp}. We then define the compact topological space ${}^m_m\ol{\R^4}$ to be the union of ${}^m\ol{\R^4}$ and ${}_m\ol{\R^4}$ quotiented out by this identification; this is thus a manifold of class $\cA_\phg^{\cE_{\rm log}}$, and in fact of class $\CI$ away from $\pa{}^m\ol{\R^4}\cap\pa{}_m\ol{\R^4}$, hence in particular near $S^\pm$ as well as near ${}^m\beta({}^m I^+)$ and its image under $\iota$. Define the blown-up space
\[
  {}^m_m M := [{}^m_m\ol{\R^4};S^+,S^-],
\]
i.e.\ blow up both $S^+$ and $S^-$; these are closed and disjoint submanifolds, hence the order of blow-up does not matter. Then ${}^m_m M$ is a polyhomogeneous manifold, covered by the two smooth manifolds ${}^m\!M'$ and ${}_m\!M':=[{}_m\ol{\R^4};S^-]$, and with interior naturally diffeomorphic to $\R^4_{t,x}$. We denote its boundary hypersurfaces by $\scri^\pm$ and $i^\pm$ in the obvious manner, see Figure~\ref{FigIBaby}, and $I^0$ is the closure of the remaining part of the boundary. In view of the isomorphism~\eqref{EqCptACompHb}, weighted b-Sobolev spaces on ${}^m_m\!M$ are well-defined. For future use, we also note that polyhomogeneity at $I^0$ with index set $\cE_0$ is well-defined provided
\begin{equation}
\label{EqPfGlobalPhg}
  \cE_0+\cE_{\rm log}=\cE_0,
\end{equation}
as follows from~\eqref{EqCptACompPhg}; note that given any index set $\cE_0^0$, the index set $\cE_0:=\cE_0^0+\cE_{\rm log}$ satisfies~\eqref{EqPfGlobalPhg} (and is the smallest such index set which contains $\cE_0^0$) since $\cE_{\rm log}+\cE_{\rm log}=\cE_{\rm log}$.

It is useful to describe ${}^m_m\!M$ as the union of \emph{three} (overlapping) smooth manifolds, namely ${}^m\!M$, ${}_m M:=\iota{}^m\!M$, and the set $U$ defined in~\eqref{EqPf0Nbh}. We can then define the function space
\[
  \cX^{\infty;b_0,b_I,b'_I,b_+}_{\rm global}
\]
to consist of all distributions on $\R^4$ which lie in $\rho_0^{b_0}\Hb^\infty$ on $U$, and such that their restriction as well as the restriction of their pullback by $\iota$ to ${}^m\!M$ lie in $\cX^{\infty;b_0,b_I,b'_I,b_+}$.

\begin{thm}
\label{ThmPfGlobal}
  Given initial data $\gamma,k$ as in Lemma~\usref{LemmaPf0}, there exists a global solution $g$ of the Einstein vacuum equation $\Ric(g)=0$, attaining the data $\gamma,k$ at $\{t=0\}$ and satisfying the gauge condition $\Ups(g)$, which is of the form $g=g_m+\rho h$ with $h\in\cX^{\infty;b_0,b_I,b'_I,b_+}_{\rm global}$ for all $b_I<b'_I<b_0$ and $b_+<0$.
\end{thm}

\section{Polyhomogeneity}
\label{SPhg}

We state and prove a precise version of the polyhomogeneity statement, made in Theorem~\ref{ThmIBaby}, about the solution of the initial value problem which we constructed in \S\ref{SPf}. We use the short hand notations~\eqref{EqCptFPhgBound} and~\eqref{EqCptFPhgShorthand}.

\begin{thm}
\label{ThmPhg}
  Let $b_0>0$, and let $\cE^0_0\subset\C\times\N_0$ be an index set with $\Im\cE^0_0<-b_0$. Suppose $\gamma,k\in\CI(\R^3;S^2 T^*\R^3)$ are initial data such that $m\in\R$, $\wt\gamma$, defined in~\eqref{EqPf0GammaTilde}, and $k$ satisfy the smallness condition~\eqref{EqPf0Small}, for $N$ large and $\delta>0$ small.\footnote{We can take $N=26$ as in (the proof of) Theorem~\ref{ThmPf}.} Assume moreover that the initial data are polyhomogeneous (namely, $\cE_0^0$-smooth):
  \begin{equation}
  \label{EqPhgData}
    \rho_0^{-1}\wt\gamma,\,\rho_0^{-2}k \in \cA_\phg^{\cE_0^0}(\ol{\R^3};S^2\,\Tsc^*\ol{\R^3}).
  \end{equation}
  Let $h$ denote the global solution of $\Ric(g)=0$, $g=g_m+\rho h$, in $M$, satisfying the gauge condition $\Ups(g;g_m)=0$. Then $h$ is polyhomogeneous on $M$. More precisely, $h$ is $\cE$-smooth, $\cE=(\cE_0,\cE_I,\cE_+)$:
  \[
    h\in\cA_\phg^{\cE_0,\cE_I,\cE_+},
  \]
  with the refinements $\pi_{1 1}^c h\in\cA_\phg^{\cE_0,\bar\cE_I,\cE_+}$ and $\pi_0 h \in \cA_\phg^{\cE_0,\cE'_I,\cE_+}$ near $\scri^+$, where the index sets are the smallest ones satisfying\footnote{We shall prove that such index sets indeed exist.}
  \begin{subequations}
  \begin{equation}
  \label{EqPhgIndexi0}
    \cE_0\supset\cE_0^0+\cE'_{\rm log},\ \ \cE_0\supset j(\cE_0-i)+i\ \ \forall\,j\in\N
  \end{equation}
  at $I^0$, with $\cE'_{\rm log}$ defined in~\eqref{EqCptIndexSetLog}, while at $\scri^+$,
  \begin{align}
    \label{EqPhgIndex0}
      \cE'_I &\supset \cE_0\extcup(2\cE_I-i) \\
    \label{EqPhgIndex11c}
      \bar\cE_I &\supset 0\cup\bigl(\cE_0\extcup\bigl((\bar\cE_I+\cE'_I)\cup(2\cE_I-i)\bigr)\bigr), \\
    \label{EqPhgIndex11}
      \cE_I &\supset 0\extcup\cE_0\extcup\bigl((\cE_I+\cE'_I)\cup(2\bar\cE_I)\bigr), \\
    \label{EqPhgIndexNonlin}
      \cE_I &\supset j(\cE_I-i)+i\ \ \forall\,j\in\N,
  \end{align}
  and finally at $I^+$,
  \begin{equation}
    \label{EqPhgIndexip}
      \cE_+ \supset (-i\extcup 0) \cup \bigl((\cE_+-i)\extcup-i\extcup(\cE_I\setminus\{(0,1)\})\bigr).
  \end{equation}
  \end{subequations}
\end{thm}

At $I^0$, we only need to capture the index set arising from nonlinear terms in Einstein's equation since the background metric $g_m$ solves $\rho^{-3}\Ric(g_m)=0$ identically near $I^0$; the addition of the index set $\cE_{\rm log}$ arises when pushing the solution near $\{t=0\}\subset{}^0\ol{\R^4}$ forward to ${}^m\!M$; see~\eqref{EqPfGlobalPhg}. We point out that the index sets we obtain are very likely to be nonoptimal due to our rather coarse analysis of nonlinear interactions.

\begin{example}
\label{ExPhgSchwartz}
  For data which are Schwarzschildean modulo Schwartz functions, i.e.\ $\cE_0^0=\emptyset$, the above gives $\cE_0=\emptyset$ and
  \[
    \cE_I = \bigcup_{j\in\N_0} (-i j,3 j+1),\ \ 
    \bar\cE_I = 0 \cup \cE'_I,\ \ 
    \cE'_I = \bigcup_{j\in\N} (-i j,3 j-1),\ \ 
    \cE_+ = \bigcup_{j\in\N_0} \bigl(-i j,\tfrac32 j(j+3)\bigr).
  \]
  Recalling the notation $\log^{\leq k}$ introduced around~\eqref{EqIDetailLogPowers}, this gives, schematically, leading terms $\pi_{1 1}h\sim\log^{\leq 1}\rho_I + \rho_I\log^{\leq 4}\rho_I$, $\pi_{1 1}^c h\sim 1 + \rho_I\log^{\leq 2}\rho_I$, $\pi_0 h\sim \rho_I\log^{\leq 2}\rho_I$ at $\scri^+$ (near the interior of which one can take $\rho_I=r^{-1}$), and $h\sim 1+\rho_+\log^{\leq 6}\rho_+$ at $I^+$ (near the interior of which one can take $\rho_+=t^{-1}$).
\end{example}
\begin{example}
\label{ExPhgSmooth}
  Consider $\cE_0^0=-i$: this corresponds to initial data which have a full Taylor expansion in $1/r$ at infinity, beginning with $\cO(r^{-2})$ perturbations of the Schwarzschild metric. In this case, we get many additional logarithmic terms from $\cE_0=\cE_0^0+\cE_{\rm log}=\bigcup_{j\in\N}(-i j,j-1)$, namely
  \begin{gather*}
    \cE_I = \bigcup_{j\in\N_0} \bigl(-i j,\half j(3 j+7)+1\bigr),\ \ 
    \bar\cE_I = 0 \cup \bigcup_{j\in\N} \bigl(-i j,\half j(3 j+5)\bigr), \\
    \cE'_I = \bigcup_{j\in\N} \bigl(-i j,\half j(3 j+3)\bigr), \ \
    \cE_+ = \bigcup_{j\in\N_0} \bigl(-i j,\half j(j^2+5 j+10)\bigr),
  \end{gather*}
  so $\pi_{1 1}h\sim\log^{\leq 1}\rho_I+\rho_I\log^{\leq 6}\rho_I$, $\pi_{1 1}^c h\sim 1+\rho_I\log^{\leq 4}\rho_I$, $\pi_0 h\sim\rho_I\log^{\leq 3}\rho_I$ at $\scri^+$, and $h\sim 1 +\rho_+\log^{\leq 8}\rho_+$ at $I^+$.
\end{example}

\begin{rmk}
\label{RmkPhgSmoothNoLog}
  Let us consider the index set $\cE_0^0=-i$ again. As indicated above, the addition of $\cE'_{\rm log}$ in~\eqref{EqPhgIndexi0} is only due to an inconvenient choice of initial surface which produces logarithmic terms when passing from ${}^0\ol{\R^4}$ (which the initial surface in Theorem~\ref{ThmPhg} is a smooth submanifold of) to ${}^m\ol{\R^4}$. If instead one is given the ADM mass $m$ and initial data $(\gamma,k)$ on ${}^m\Sigma$, with $(\gamma,k)$ close to the data induced by $g_m$ on ${}^m\Sigma$ (measured in $\rho_0^{b_0}\Hb^N({}^m\Sigma;S^2\,\Tsc^*\,{}^m\Sigma)$ for suitable $N$), then the index set at $I^0$ can be defined as in~\eqref{EqPhgIndexi0} \emph{but without $\cE'_{\rm log}$}. Correspondingly, the index sets at the other boundary faces have fewer logarithms:
  \begin{gather*}
    \cE_I = \bigcup_{j\in\N_0} \bigl(-i j,5 j+1\bigr),\ \ 
    \bar\cE_I = 0 \cup \bigcup_{j\in\N} \bigl(-i j,5 j-1\bigr), \\
    \cE'_I = \bigcup_{j\in\N} \bigl(-i j,5 j-2\bigr), \ \
    \cE_+ = \bigcup_{j\in\N_0} \bigl(-i j,\half j(5 j+11)\bigr),
  \end{gather*}
  so $\pi_{1 1}h\sim\log^{\leq 1}\rho_I+\rho_I\log^{\leq 6}\rho_I$, $\pi_{1 1}^c h\sim 1+\rho_I\log^{\leq 4}\rho_I$, $\pi_0 h\sim\rho_I\log^{\leq 3}\rho_I$ at $\scri^+$, and $h\sim 1 +\rho_+\log^{\leq 8}\rho_+$ at $I^+$. (The exponents in subsequent terms of the expansion are smaller than in Example~\ref{ExPhgSmooth}.)
\end{rmk}

The proof of Theorem~\ref{ThmPhg} is straightforward but requires some bookkeeping: we will peel off the polyhomogeneous expansion at the various boundary faces iteratively, writing the nonlinear equation $P(h)=0$ as a linear equation plus error terms with better decay, much like in \S\ref{SIt}. As a preparation, we prove a few lemmas for ODEs which were already used in \S\ref{SIt}:
\begin{lemma}
\label{LemmaPhgODE1d}
  Let $X:=[0,\infty)_\rho$, $u\in\rho^{-\infty}\Hb^\infty(X)$, $\supp u\subset[0,1]$, and $f:=\rho D_\rho u$. Then:
  \begin{enumerate}
  \item\label{ItPhgODE1dbLess} $f\in\rho^a\Hb^\infty(X)$, $a<0$ $\Rightarrow$ $u\in\rho^a\Hb^\infty(X)$;
  \item\label{ItPhgODE1dbGtr} $f\in\rho^a\Hb^\infty(X)$, $a>0$ $\Rightarrow$ $u\in\cA_\phg^0(X)+\rho^a\Hb^\infty(X)$;
  \item\label{ItPhgODE1dp} $f\in\cA_\phg^\cE(X)$, $\cE$ any index set $\Rightarrow$ $u\in\cA_\phg^{\cE\extcup 0}(X)$; if $(0,0)\notin\cE$, then $u\in\cA_\phg^{\cE\cup 0}(X)$.
  \end{enumerate}
\end{lemma}
\begin{proof}
  This follows immediately from the characterization of b-Sobolev and polyhomogeneous spaces using the Mellin transform \cite[\S4]{MelroseDiffOnMwc}. Alternatively, one can explicitly construct the unique solution of $\rho D_\rho u=f$ with support in $\rho\leq 1$: part~\eqref{ItPhgODE1dbLess} follows easily from $u=-i\int_\rho^1 f\,\frac{d\rho}{\rho}$, while for part~\eqref{ItPhgODE1dbGtr}, $u=-i\int_0^1 f\,\frac{d\rho}{\rho}+i\int_0^\rho f\,\frac{d\rho}{\rho}$ gives the decomposition into constant and remainder term. The appearance of the extended union in~\eqref{ItPhgODE1dp} is due to the fact that while $\rho D_\rho u=\rho^{i z}(\log\rho)^k$, $k\in\N_0$, is solved to leading order by $u=\tfrac{1}{z}\rho^{i z}(\log\rho)^k$ for $z\neq 0$, we need an extra logarithmic term for $z=0$, as $\rho D_\rho(\tfrac{1}{k+1}(\log\rho)^{k+1} a)=-i(\log\rho)^k a$ plus lower order terms.
\end{proof}

Adding more dimensions is straightforward:

\begin{lemma}
\label{LemmaPhgODE}
  Let $X=[0,\infty)_{\rho_1}\times[0,\infty)_{\rho_2}\times\R^n_\omega$, $U=\{\rho_1<1,\,\rho_2<1\}\subset X$, $\rho=\rho_1\rho_2$, and let $\cE_1,\cE_2$ denote two index sets. Suppose $u\in\rho^{-\infty}\Hb^\infty(X)$ has support in $U$, and let $f:=\rho_1 D_{\rho_1}u$. Then:
  \begin{enumerate}
  \item\label{ItPhgODEbb} $f\in\rho_1^{a_1}\rho_2^{a_2}\Hb^\infty(X)$, $a_1\neq 0$ $\Rightarrow$ $u\in\cA_{\phg,\bop}^{0,a_2}(X)+\rho_1^{a_1}\rho_2^{a_2}\Hb^\infty(X)$;
  \item\label{ItPhgODEbp} $f\in\cA_{\bop,\phg}^{a_1,\cE_2}(X)$, $a_1\neq 0$ $\Rightarrow$ $u\in\cA_\phg^{0,\cE_2}(X)+\cA_{\bop,\phg}^{a_1,\cE_2}(X)$;
  \item\label{ItPhgODEpp} $f\in\cA_\phg^{\cE_1,\cE_2}(X)$ $\Rightarrow$ $u\in\cA_\phg^{\cE_1\extcup 0,\cE_2}(X)$; if $(0,0)\notin\cE_2$, then $u\in\cA_\phg^{\cE_1\cup 0,\cE_2}(X)$.
  \end{enumerate}
\end{lemma}

\begin{lemma}
\label{LemmaPhgODE2}
  In the notation of Lemma~\usref{LemmaPhgODE}, with $u\in\rho^{-\infty}\Hb^\infty(X)$ supported in $\rho_1\leq 1$, let $f:=(\rho_1 D_{\rho_1}-\rho_2 D_{\rho_2})u$. Let $\chi=\chi(\rho_1,\rho_2)\in\CIc([0,1)^2)$ denote a localizer, identically $1$ in a neighborhood of the corner $\rho_1=\rho_2=0$. See Figure~\usref{FigPhgODE2}. Then:
  \begin{enumerate}
  \item\label{ItPhgODE2bb} $f\in\rho_1^{b_1}\rho_2^{b_2}\Hb^\infty(X)$, $b_2>b_1$ $\Rightarrow$ $\chi u\in\rho_1^{b_1}\rho_2^{b_2}\Hb^\infty(X)$;
  \item\label{ItPhgODE2bp} $f\in\cA_{\bop,\phg}^{b_1,\cE_2}(X)$, $\Im z\neq-b_1$ whenever $(z,0)\in\cE_2$ $\Rightarrow$ $\chi u\in\cA_\phg^{\cE_2,\cE_2}(X)+\cA_{\bop,\phg}^{b_1,\cE_2}(X)$;
  \item\label{ItPhgODE2pp} $f\in\cA_\phg^{\cE_1,\cE_2}(X)$ $\Rightarrow$ $\chi u\in\cA_\phg^{\cE_1\extcup\cE_2,\cE_2}(X)$.
  \end{enumerate}
\end{lemma}

\begin{figure}[!ht]
\includegraphics{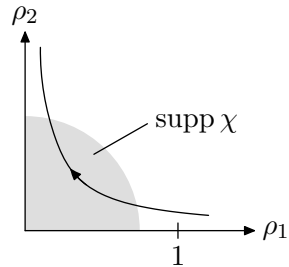}
\caption{Illustration of Lemma~\ref{LemmaPhgODE2} which describes solutions of the transport equation along the vector field $-\rho_1\pa_{\rho_1}+\rho_2\pa_{\rho_2}$; one integral curve of this vector field is shown here.}
\label{FigPhgODE2}
\end{figure}

\begin{proof}
  We drop the $\R^n_\omega$ factor from the notation for brevity. For~\eqref{ItPhgODE2bb}, write $u(\rho_1,\rho_2)=-i\int_{\rho_1}^1 f(t^{-1}\rho_1,t\rho_2)\,\frac{d t}{t}$ and $f=\rho_1^{b_1}\rho_2^{b_2}\wt f$, $\wt f\in\Hb^\infty$, then for $0<\eps<b_2-b_1$
  \begin{align*}
    \|\chi u\|_{\rho_1^{b_1}\rho_2^{b_2}L^2_\bop}^2 &\leq \int_0^1\!\!\int_0^1 \biggl|\int_{\rho_1}^1 t^{b_2-b_1}\wt f(t^{-1}\rho_1, t\rho_2)\,\frac{d t}{t}\biggr|^2 \frac{d\rho_1}{\rho_1}\frac{d\rho_2}{\rho_2} \\
    &\leq \int_0^1 \biggl(\int_{\rho_1}^1 t^{b_2-b_1} \Bigl(\int_0^1 |\wt f(t^{-1}\rho_1,x_2)|\,\frac{d x_2}{x_2}\Bigr)^{1/2}\frac{d t}{t}\biggr)^2 \frac{d\rho_1}{\rho_1} \\
    &\leq \biggl(\int_0^1 t^{2(b_2-b_1-\eps)}\,\frac{d t}{t}\biggr)\cdot \int_0^1\!\!\int_0^1\!\!\int_0^1 t^{2\eps}|\wt f(x_1,x_2)|^2\,\frac{d x_2}{x_2}\frac{d t}{t}\frac{d x_1}{x_1} \\
    &\leq C\|f\|_{\rho_1^{b_1}\rho_2^{b_2}L^2_\bop},
  \end{align*}
  as desired; higher b-regularity follows by commuting $\rho_j D_{\rho_j}$ through the equation for $u$.\footnote{A more conceptual proof, which does not rely on explicit integration of the vector field, uses a positive commutator argument with the commutant $a=\chi_1(\rho_1)\chi_2(\rho_2)\rho_1^{-b_1}\rho_2^{-b_2}$, $\chi_j\in\CIc([0,\infty))$, $\chi_j(\rho)\equiv 1$ near $0$, and $\chi_j'\leq 0$, i.e.\ the evaluation of $2\Im\la(\rho_1 D_{\rho_1}-\rho_2 D_{\rho_2})u,a^2 u\ra_{L^2_\bop}$, in two different ways: once by using the equation satisfied by $u$, and once by integrating by parts and using that $(\rho_1\pa_{\rho_1}-\rho_2\pa_{\rho_2})a$ has a constant sign on $\supp a\cap\supp u$. See~\eqref{EqBgi0UItself} for a similar argument.}

  For the proof of~\eqref{ItPhgODE2bp}, it suffices to consider a single term
  \begin{equation}
  \label{EqPhgODE2Single}
    f_k=\rho_2^{i z}(\log\rho_2)^k a_k(\rho_1),
  \end{equation}
  with $a_k\in\rho_1^{b_1}\Hb^\infty(H_1)$ supported in $\rho_1\leq 1$. Let $u_k=\rho_2^{i z}(\log\rho_2)^k b_k(\rho_1)$, where $b_k=b_k(\rho_1)$ solves
  \begin{equation}
  \label{EqPhgODE2Int}
    (\rho_1 D_{\rho_1}-z)b_k=a_k
  \end{equation}
  and is supported in $\rho_1\leq 1$, then the error term
  \begin{align*}
    f_{k-1} &:= (\rho_1 D_{\rho_1}-\rho_2 D_{\rho_2})u_k-f_k = \bigl((\rho_1 D_{\rho_1}-z)-(\rho_2 D_{\rho_2}-z)\bigr)u_k-f_k \\
      &= \rho_2^{i z}(\log\rho_2)^{k-1}a_{k-1}(\rho_1),\quad a_{k-1}:=i k b_k,
  \end{align*}
  is one power of $\log\rho_2$ better than $f_k$. Rewriting equation~\eqref{EqPhgODE2Int} as $\rho_1 D_{\rho_1}(\rho_1^{-i z}b_k)=\rho_1^{-i z}a_k\in\rho_1^{b_1+\Im z}\Hb^\infty(H_1)$, we can use Lemma~\ref{LemmaPhgODE1d} to obtain $b_k\in \cA_\phg^z(H_1)+\rho_1^{b_1}\Hb^\infty(H_1)$; therefore $u_k\in\cA_\phg^{z,(z,k)}(X)+\cA_{\bop,\phg}^{b_1,(z,k)}(X)$. Proceeding iteratively, we next solve $(\rho_1 D_{\rho_1}-\rho_2 D_{\rho_2})u_{k-1}=f_{k-1}$ to leading order, etc., reducing $k$ by $1$ at each step, and picking up one extra power of $\log\rho_1$ at each stage by Lemma~\ref{LemmaPhgODE1d}\eqref{ItPhgODE1dp} (conjugated by $\rho^{i z}$). We obtain $u=\sum_{j=0}^k u_j\in\cA_\phg^{(z,k),(z,k)}(X)+\cA_{\bop,\phg}^{b_1,(z,k)}(X)$.

  The proof of~\eqref{ItPhgODE2pp} proceeds in the same manner: if $f_k$ is of the form~\eqref{EqPhgODE2Single}, now with $a_k\in\cA_\phg^{\cE_1}(H_1)$, then $b_k\in\cA_\phg^{\cE_1\extcup z}(H_1)$, so $u_k\in\cA_\phg^{\cE_1\extcup z,(z,k)}(X)$ and $f_{k-1}\in\cA_\phg^{\cE_1\extcup z,(z,k-1)}(X)$. Iterating as before gives $u\in\cA_\phg^{\cE_1\extcup(z,k),(z,k)}(X)$.
\end{proof}

\begin{proof}[Proof of Theorem~\usref{ThmPhg}]
  We shall first prove that if the Cauchy data $(h_0,h_1)$ in the notation of Theorem~\ref{ThmPf} are polyhomogeneous at $\pa{}^m\Sigma$,
  \begin{equation}
  \label{EqPhgH01}
    h_0,\,h_1 \in \cA_\phg^{\cE_0}({}^m\Sigma),
  \end{equation}
  then the conclusion of Theorem~\ref{ThmPhg} holds. Now, by Theorem~\ref{ThmPf}, we have $h\in\cX^{\infty;b_0,b_I,b'_I,b_+}$ for all $b_I<b'_I<b_0$ and $b_+<0$. Note that since the gauge condition $\Ups(g)=0$ is satisfied identically, $h$ solves $\Ric(g)-\tdel^*\Ups(g)=0$ for \emph{any} choice of $\tdel^*$; this will be useful as it will allow us to work with simpler normal operator models.
  
  For now, consider $h$ as a solution of $P(h)=0$ for $\gamma>b'_I$ as in Theorem~\ref{ThmBg}. We write
  \begin{equation}
  \label{EqPhgNonlinLin}
    0 = P(h) = p_0 + \int_0^1 L_{t h}(h)\,d t,\ \ p_0:=P(0)\in\cA_\phg^{\emptyset,\emptyset,0}.
  \end{equation}
  (In fact, $\supp p_0\cap(I^0\cup\scri^+)=\emptyset$ since $g_m$ is the Schwarzschild metric near $I^0\cap\scri^+$.)
  
  Let us first work near $I^0$, away from $I^+$. Suppose that for some $c\geq b_0$, we already have $h\in\cA_{\phg,\bop}^{\cE_0,-0}+\rho_0^c\rho_I^{-0}\Hb^\infty$, $\pi_0 h\in\cA_{\phg,\bop}^{\cE_0,b'_I-0}+\rho_0^c\rho_I^{b'_I-0}\Hb^\infty$, with the exponents referring to the behavior at $I^0$ and $\scri^+$, respectively. Then
  \begin{equation}
  \label{EqPhgNonlinLin2}
    L_0 h = -p_0 + \int_0^1(L_0-L_{t h})(h)\,d t;
  \end{equation}
  we have $L_0-L_{t h}\in(\cA_{\phg,\bop}^{\cE_0-i,-1-0}+\rho_0^{c+1}\rho_I^{-1-0}\Hb^\infty)\Diffb^2$ by an inspection of the proof of Lemma~\ref{LemmaEinNscri}, and it respects the improved behavior of $\pi_0 h$, so we find
  \[
    L_0 h \in \cA_{\phg,\bop}^{2\cE_0-i,-1-0}+\rho_0^{c+1}\rho_I^{-1-0}\Hb^\infty,\ \ 
    \pi_0 L_0 h \in \cA_{\phg,\bop}^{2\cE_0-i,-1+b'_I-0}+\rho_0^{c+1}\rho_I^{-1+b'_I-0}\Hb^\infty.
  \]
  Denote by $\cE_1:=\{(z,j)\in\cE_0\colon \Im z\geq-c\}$ the (finite) set of exponents already captured, and let $\cE_2:=\{(z,j)\in\cE_0\colon -c-1\leq\Im z<-c\}$. Let
  \[
    R_j := \prod_{(z,k)\in\cE_j} (\rho_0 D_{\rho_0}-z),\ \ 
    R=R_2\circ R_1.
  \]
  Let $N(L_0)\in\rho_I^{-1}\Diffb^2(M)$ denote the normal operator of $L_0$ at $I^0$, i.e.\ freezing the coefficients of $L_0$ at $\rho_0=0$ for a fixed choice of a collar neighborhood $[0,\eps)_{\rho_0}\times I^0$ of $I^0$; thus $N(L_0)$ commutes with $\rho_0\pa_{\rho_0}$, and $L_0-N(L_0)\in\rho_0\rho_I^{-1}\Diffb^2$. Then $R h\in\rho_0^c\rho_I^{-0}\Hb^\infty$ solves the equation $N(L_0)(R h)=f$, where
  \[
    f := -R(L_0-N(L_0))h + R L_0 h \in \rho_0^{c+1}\rho_I^{-1-0}\Hb^\infty,\ \ 
    \pi_0 f\in\rho_0^{c+1}\rho_I^{-1+b'_I-0}\Hb^\infty,
  \]
  due to $2\cE_0-i\subset\cE_0$; the Cauchy data of $R h$ lie in $\rho_0^{c+1}\Hb^\infty$ due to the polyhomogeneity of $h_0$ and $h_1$. The background estimate near $I^0$ being sharp with regards to the weight at $I^0$, see Propositions~\ref{PropBgi0} and \ref{PropBgscri}, this gives $R h\in\rho_0^{c+1}\rho_I^{-0}\Hb^\infty$, $\pi_0 R h\in\rho_0^{c+1}\rho_I^{b'_I-0}\Hb^\infty$. Thus, $h\in\cA_{\phg,\bop}^{\cE_0,-0}+\rho_0^{c+1}\rho_I^{-0}\Hb^\infty$, $\pi_0 h\in\cA_{\phg,\bop}^{\cE_0,b'_I-0}+\rho_0^{c+1}\rho_I^{b'_I-0}\Hb^\infty$. Iterating this gives
  \begin{equation}
  \label{EqPhgi0}
    h\in\cA_{\phg,\bop}^{\cE_0,-0},\ \ 
    \pi_0 h\in\cA_{\phg,\bop}^{\cE_0,b'_I-0}\ \ \tn{near}\ I^0.
  \end{equation}

  Following the structure of the argument in \S\ref{SIt}, we next prove the polyhomogeneity at $\scri^+\setminus(\scri^+\cap I^+)$ using Lemmas~\ref{LemmaPhgODE} and \ref{LemmaPhgODE2}. We now take $\gamma=0$ in the definition of $P$ and its linearization. Thus, let us work near $I^0\cap\scri^+$, and assume that we already have
  \begin{equation}
  \label{EqPhgLeading}
    \pi_0 h\in\cA_\phg^{\cE_0,\cE'_I} + \cA_{\phg,\bop}^{\cE_0,c'_I-0},\ \ 
    \pi_{1 1}^c h\in\cA_\phg^{\cE_0,\bar\cE_I} + \cA_{\phg,\bop}^{\cE_0,c_I-0},\ \ 
    \pi_{1 1} h\in\cA_\phg^{\cE_0,\cE_I} + \cA_{\phg,\bop}^{\cE_0,c_I-0},\ \ 
  \end{equation}
  for some $0\leq c_I<c'_I\leq c_I+1$. Using~\eqref{EqPhgNonlinLin} and the structure of $L_{t h}=L_{t h}^0+\wt L_{t h}$, we find
  \begin{equation}
  \label{EqPhgLeading11c}
    \pi_{1 1}^c L_0^0\pi_{1 1}^c h = -\pi_{1 1}^c p_0 - \int_0^1 \bigl(\pi_{1 1}^c\wt L_{t h}\pi_{1 1}^c h + \pi_{1 1}^c L_{t h}\pi_0 h + \pi_{1 1}^c L_{t h}\pi_{1 1}h\bigr)\,d t.
  \end{equation}
  The proof of Lemma~\ref{LemmaEinNscri}, condition~\eqref{EqPhgIndexNonlin}, and the fact that $\cE_I\supset\bar\cE_I\supset\cE'_I\supset\cE_I-i$ give
  \begin{gather}
  \label{EqPhgTildeL}
    \wt L_{t h}\in(\CI+\cA_\phg^{\cE_0-i,\cE_I}+\cA_{\phg,\bop}^{\cE_0-i,c_I-0})\Diffb^2, \\
    \pi_{1 1}^c L_{t h}^0\pi_0 \in \rho_I^{-1}(\CI+\cA_\phg^{\cE_0-i,\bar\cE_I}+\cA_{\phg,\bop}^{\cE_0-i,c_I-0})\Diffb^1, \nonumber
  \end{gather}
  and $\pi_{1 1}^c L_{t h}^0\pi_{1 1}=0$. Multiplying~\eqref{EqPhgLeading11c} by $\rho_I$, grouping function spaces in the order of the summands in the integrand above, and simplifying using $2\cE_0-i\subset\cE_0$ and $0\subset\cE_I$, this gives
  \[
    \rho_I\pa_{\rho_I}(\rho_0\pa_{\rho_0}-\rho_I\pa_{\rho_I})\pi_{1 1}^c h \in \cA_\phg^{\cE_0,\cE_I+\bar\cE_I-i} + \cA_\phg^{\cE_0,\bar\cE_I+\cE'_I} + \cA_\phg^{\cE_0,2\cE_I-i} + \cA_{\phg,\bop}^{\cE_0,c'_I-0};
  \]
  the first space is contained in the second. In view of condition~\eqref{EqPhgIndex11c} (note that the index sets in parentheses there lie in $\Im z<0$), we obtain
  \begin{equation}
  \label{EqPhgImproved11c}
    \pi_{1 1}^c h \in \cA_\phg^{\cE_0,\bar\cE_I}+\cA_{\phg,\bop}^{\cE_0,c'_I-0},
  \end{equation}
  which improves on the a priori weight of the remainder term at $\scri^+$. Next,
  \[
    \pi_{1 1}L_0^0\pi_{1 1}h = -\pi_{1 1}p_0 - \int_0^1 \bigl(\pi_{1 1}\wt L_{t h}\pi_{1 1}h + \pi_{1 1}L_{t h}\pi_0 h+\pi_{1 1}L_{t h}\pi_{1 1}^c h\bigr)\,d t.
  \]
  Lemma~\ref{LemmaEinNscri} and the membership~\eqref{EqPhgImproved11c} imply
  \begin{align*}
    \pi_{1 1}L_{t h}^0\pi_0 &\in \rho_I^{-1}(\CI+\cA_\phg^{\cE_0-i,\cE_I}+\cA_{\phg,\bop}^{\cE_0-i,c_I-0})\Diffb^1, \\
    \pi_{1 1}L_{t h}^0\pi_{1 1}^c &\in \rho_I^{-1}(\cA_\phg^{\cE_0-i,\bar\cE_I}+\cA_{\phg,\bop}^{\cE_0-i,c'_I-0})\Diffb^1,
  \end{align*}
  with $\rho_I$ times the latter having a leading order term at $\scri^+$, cf.\ the discussion of~\eqref{EqIti0Scri11}; together with~\eqref{EqPhgTildeL} and \eqref{EqPhgImproved11c}, and using $\bar\cE_I\subset\cE_I$, one finds
  \[
    \rho_I\pa_{\rho_I}(\rho_0\pa_{\rho_0}-\rho_I\pa_{\rho_I})\pi_{1 1}h \in \cA_\phg^{\cE_0,2\cE_I-i} + \cA_\phg^{\cE_0,\cE_I+\cE'_I} + \cA_\phg^{\cE_0,2\bar\cE_I} + \cA_{\phg,\bop}^{\cE_0,c'_I-0},
  \]
  with the first space again contained in the second. Condition~\eqref{EqPhgIndex11} then gives
  \begin{equation}
  \label{EqPhgImproved11}
    \pi_{1 1} h \in \cA_\phg^{\cE_0,\cE_I} + \cA_{\phg,\bop}^{\cE_0,c'_I-0}.
  \end{equation}
  Lastly then, we can improve on the asymptotics of $\pi_0 h$ at $\scri^+$ by writing
  \[
    \pi_0 L_0^0\pi_0 h = -\pi_0 p_0 - \int_0^1\bigl(\pi_0\wt L_{t h}\pi_0 h+\pi_0 L_{t h}\pi_{1 1}^c h + \pi_0 L_{t h}\pi_{1 1}h\bigr)\,d t;
  \]
  now $\pi_0 L_{t h}^0\pi_{1 1}^c=0=\pi_0 L_{t h}^0\pi_{1 1}$ and $\cE'_I\subset\bar\cE_I\subset\cE_I$, so, since $\gamma=0$,
  \[
    \rho_I\pa_{\rho_I}(\rho_0\pa_{\rho_0}-\rho_I\pa_{\rho_I})\pi_0 h \in \cA_\phg^{\cE_0,2\cE_I-i} + \cA_{\phg,\bop}^{\cE_0,c'_I+1-0};
  \]
  but condition~\eqref{EqPhgIndex0} and Lemma~\ref{LemmaPhgODE2} imply
  \[
    \rho_I\pa_{\rho_I}\pi_0 h \in \cA_\phg^{\cE_0,\cE'_I} + \cA_{\phg,\bop}^{\cE_0,c'_I+1-0};
  \]
  an application of Lemma~\ref{LemmaPhgODE} gives the same membership for $\pi_0 h$, since we already know that $\pi_0 h$ has no leading term at $\scri^+$. This establishes~\eqref{EqPhgLeading} for $(c_I,c'_I)$ replaced by $(c'_I,c'_I+1)$, and we can iterate the procedure to establish the full polyhomogeneity away from $I^+$. Near $\scri^+\cap I^+$, the arguments are completely analogous, except we only have conormal regularity $\rho_+^{b_+}\Hb^\infty$ at $I^+$. Thus,
  \[
    \pi_0 h\in\cA_{\phg,\phg,\bop}^{\cE_0,\cE'_I,b_+},\ \ 
    \pi_{1 1}^c h\in\cA_{\phg,\phg,\bop}^{\cE_0,\bar\cE_I,b_+},\ \ 
    \pi_{1 1} h\in\cA_{\phg,\phg,\bop}^{\cE_0,\cE_I,b_+}.
  \]

  Next, we use this information to obtain an expansion at $I^+$, similarly to the arguments around~\eqref{EqItipEqn}. We shall use the linearization $L_0$, still defined using $\gamma=0$, and its normal operator \emph{at $I^+\subset M$}---instead of its normal operator at the boundary of $\ol{\R^4}$, which obviates the need to relate (partially) polyhomogeneous function spaces on $\ol{\R^4}$ and $M$. Namely, fix a collar neighborhood
  \[
    U:=[0,1)_{\rho_+}\times I^+,\ \ I^+=\{Z\in\R^3\colon|Z|\leq 1\},
  \]
  of $I^+$ in $M$, and denote by $\cV_{\bop,-}(U)\subset\cV(U)$ the Lie subalgebra of vector fields tangent to $I^+$ but \emph{with no condition at $\scri^+$}. Then for $\gamma=0$, we have $L_0\in\Diff_{\bop,-}^2(U)$ (the algebra generated by $\cV_{\bop,-}$), acting on sections of $\beta^*S^2|_U$: by Lemma~\ref{LemmaEinNscri}, $\wt L_0\in\Diffb^2(M)\hra\Diff_{\bop,-}^2(M)$ certainly has smooth coefficients, and the same is true for $L_0^0=-2\rho^{-2}\pa_0\pa_1=\pa_{\rho_I}(\rho_I\pa_{\rho_I}-\rho_+\pa_{\rho_+})+\Diffb^2(M)$, $\rho_I=1-|Z|^2$. Furthermore, by Lemmas~\ref{LemmaCptAComp} and \ref{LemmaEinNi0p} as well as equation~\eqref{EqEinNipL0}, the normal operator $N(L_0)$ of $L_0$ at $I^+$ can be identified with $N(\ul L)$, so that in fact $N(L_0)=\Box_{g_\dS}-2$, defined using the expressions~\eqref{EqBgExplDS} and \eqref{EqBgExplStaticdS}, acting component-wise on the fibers of the trivial bundle $\ul\R^{10}$, where we use Lemma~\ref{LemmaCptACompBundle} to identify $\beta^*S^2|_{I^+}\cong{}^0\beta^*(S^2\,\Tsc^*{}^0\,\ol{\R^4})|_{{}^0 I^+}\cong\ul\R^{10}$ by means of coordinate differentials. By \cite[\S4]{VasyMicroKerrdS} and the module regularity proved in~\cite{HaberVasyPropagation},
  \begin{equation}
  \label{EqPhgNormalOp}
    \wh{L_0}(\sigma)^{-1}\colon \Hext^{s-1,k}(I^+)\to\Hext^{s,k}(I^+)
  \end{equation}
  is meromorphic for $\sigma\in\C$ with $s>\half-\Im\sigma$, where the bar refers to extendibility at $\pa I^+=\{|Z|=1\}$, while the parameter $k\in\N_0$ measures the amount of regularity under the $\CI(I^+)$-module $\Diffb^1(I^+)$; that is, $\Hext^{s,k}(I^+)$ consists of $H^s$ functions on $I^+$ which remain in $H^s$ under application of any operator in $\Diffb^k(I^+)$. (This is analogous to Lemma~\ref{LemmaItipFred}, except in the present de~Sitter setting we work on \emph{high} regularity spaces rather than the low regularity spaces in the Minkowski setting, see \cite[\S5]{VasyMicroKerrdS}.) Strictly speaking, the references only apply to the operator obtained from $L_0$ by smooth extension across $\pa I^+$ to an operator on a slightly larger space than $I^+$; but \eqref{EqPhgNormalOp} follows simply by using extension and restriction operators, and the choice of extensions is irrelevant since $\wh{L_0}(\sigma)$ is principally a wave operator beyond $\pa I^+$.
  
  The divisor $\cR$ of $L_0$, see Remark~\ref{RmkItipDivisor}, is then
  \begin{equation}
    \cR=-i;
  \end{equation}
  indeed, using the relation between asymptotics on global de~Sitter space and resonances on static de~Sitter space as in \cite[Appendix~C]{HintzVasyKdSStability}, this follows from \cite[Theorem~1.1]{VasyWaveOndS} for $n=4$, $\lambda=2$, with the logarithmic terms absent: the indicial roots are $1$ and $2$, see \cite[Lemma~4.13]{VasyWaveOndS}, and in the notation of~\eqref{EqBgExplStaticdS}, the difference of $\Box_{g_\dS}$ and its indicial operator $-(\hat\tau\pa_{\hat\tau})^2+3\hat\tau\pa_{\hat\tau}$ is $\hat\tau^2\Delta_{\hat x}$, thus vanishes \emph{quadratically} in $\hat\tau$ as a b-operator on $[0,\infty)_{\hat\tau}\times\R^3_{\hat x}$. Hence, for the formal solution $u=\hat\tau v_-+\hat\tau^2 v_+$ constructed in \cite[Lemma~4.13]{VasyWaveOndS}, the Taylor series of $v_\pm$ only contain \emph{even} powers of $\hat\tau$; $1-2\N_0$ and $2-2\N_0$ being disjoint, there are no integer coincidences which would cause logarithmic terms.
  
  Now, consider again~\eqref{EqPhgNonlinLin2}: if $\chi=\chi(\rho_+)$ denotes a localizer near $I^+$, identically $1$ near $I^+$ and vanishing near $I^0$, we have
  \begin{equation}
  \label{EqPhgLocip}
    L_0(\chi h) = -\chi p_0 + [L_0,\chi]h + \int_0^1 \chi(L_0-L_{t h})(h)\,d t.
  \end{equation}
  We have $\chi p_0\in\cA_\phg^{\emptyset,0}$, with the exponents now referring to the behavior at $\scri^+$ and $I^+$, respectively. Suppose we already have
  \begin{equation}
  \label{EqPhgInductiveip}
    \pi_0 h\in\cA_\phg^{\cE'_I,\cE_+}+\cA_{\phg,\bop}^{\cE'_I,c_+},\ \ 
    \pi_{1 1}^c h\in\cA_\phg^{\bar\cE_I,\cE_+}+\cA_{\phg,\bop}^{\bar\cE_I,c_+},\ \ 
    \pi_{1 1} h\in\cA_\phg^{\cE_I,\cE_+}+\cA_{\phg,\bop}^{\cE_I,c_+}.
  \end{equation}
  Using that $\cE_+-i$ is closed under nonlinear operations, i.e.\ $j(\cE_+-i)+i\subset\cE_+$, $j\in\N$, we find $L_0-L_{t h}\in\cA_\phg^{\cE_+-i}+\rho_+^{c_+ +1}\Hb^\infty$ near $(I^+)^\circ$; see also Lemma~\ref{LemmaEinNi0p}. Using the structure of $L_{t h}$ near $\scri^+\cap I^+$ from Lemma~\ref{LemmaEinNscri} as above, and noting that $\supp[L_0,\chi]h\subset\supp d\chi$ is disjoint from $I^+$, we deduce that
  \[
    L_0(\chi h) \in \cA_\phg^{\emptyset,0} + \cA_\phg^{\wt\cE_I+i,\emptyset} + \cA_\phg^{\wt\cE_I+i,2\cE_+-i} + \cA_{\phg,\bop}^{\wt\cE_I+i,c_+ +1},\ \ \wt\cE_I:=\cE_I\setminus\{(0,1)\},
  \]
  where the weight of the remainder term is as stated since all $(z,k)\in\cE_+$ except for $(0,0)$ have $\Im z<0$. (Here $\wt\cE_I\supset\bar\cE_I+i\supset\cE'_I+i$ allows for a nonlogarithmic leading term at $\scri^+$, capturing the worst component of elements of the space $\cY^\infty$ in Definition~\ref{DefEinFY}, and moreover captures all nonlinear terms of~\eqref{EqPhgLocip}.) Replacing $L_0$ by $N(L_0)$ causes another error term, $(L_0-N(L_0))(\chi h)\in\cA_\phg^{\wt\cE_I+i,\cE_+-i}+\cA_{\phg,\bop}^{\wt\cE_I+i,c_++1}$, so
  \[
    N(L_0)(\chi h) \in \cA_\phg^{\emptyset,0} + \cA_\phg^{\wt\cE_I+i,\cE_+-i} + \cA_{\phg,\bop}^{\wt\cE_I+i,c_++1}.
  \]
  Mellin transforming in $\rho_+$ at $\Im\sigma=-b_+$, inverting $\wh{L_0}(\sigma)$ on $\cA_\phg^{\wt\cE_I+i}(I^+)$ using Lemma~\ref{LemmaPhgNormalPhg} below, taking the inverse Mellin transform, and shifting the contour to $\Im\sigma=-c_+-1$, we obtain
  \[
    \chi h \in \cA_\phg^{0,\cR\extcup 0} + \cA_\phg^{0\extcup\wt\cE_I,(\cR\extcup\wt\cE_I)\extcup(\cE_+-i)} + \cA_{\phg,\bop}^{0\extcup\wt\cE_I,c_+ +1}.
  \]
  The index set at $I^+$ is contained in $\cE_+$ by condition~\eqref{EqPhgIndexip}, so this improves over~\eqref{EqPhgInductiveip} by the weight $1$ in the remainder term; the index sets at $\scri^+$ on the other hand are automatically the ones stated (but now with the improvement at $I^+$), as the presence of a nonzero term in the expansion of $\pi_{1 1} h$, say, at $\scri^+$ corresponding to some element in $(0\extcup\wt\cE_I)\setminus\cE_I$, would contradict our a priori knowledge~\eqref{EqPhgInductiveip}. Iterating this gives the polyhomogeneity at $I^+$, as claimed.

  Next, let us show that the smallest sets satisfying conditions~\eqref{EqPhgIndexi0}--\eqref{EqPhgIndexip} are indeed \emph{index} sets: we need to verify condition~\eqref{EqCptFPhg2}. For $\cE_0$, this is clear since, letting $\wt\cE_0^0:=\cE_0^0+\cE'_{\rm log}$,
  \[
    \cE_0=\wt\cE_0^0\cup\bigcup_{j\in\N}j(\wt\cE_0^0-i)+i
  \]
  and $\Im\wt\cE_0^0<0$; note that this gives $\Im\cE_0<0$. At $\scri^+$, we take $\cE'_I=\bigcup_{k\in\N}\cE'_{I,k}$, likewise for $\bar\cE_I$ and $\cE_I$, where we recursively define $\cE'_{I,0}=\bar\cE_{I,0}=\cE_{I,0}=\emptyset$ and
  \begin{subequations}
  \begin{align}
  \label{EqPhgPfId0}
    \cE'_{I,k+1} &= \cE_0 \extcup (2\cE_{I,k}-i), \\
  \label{EqPhgPfId11c}
    \bar\cE_{I,k+1} &= 0 \cup \bigl(\cE_0\extcup\bigl((\bar\cE_{I,k}+\cE'_{I,k})\cup(2\cE_{I,k}-i)\bigr)\bigr), \\
  \label{EqPhgPfId11}
    \cE_{I,k+1} &= \Bigl(0 \extcup\cE_0\extcup\bigl((\cE_{I,k}+\cE'_{I,k})\cup(2\bar\cE_{I,k})\bigr)\Bigr) \cup \bigcup_{j\in\N}\bigl(j(\cE_{I,k}-i)+i\bigr).
  \end{align}
  \end{subequations}
  It easy to see by induction that
  \[
    \Im\cE'_{I,k},\ \Im\bigl(\bar\cE_{I,k}\setminus(0,0)\bigr),\ \Im\bigl(\cE_{I,k}\setminus(0,1)\bigr)\leq -c,\ \ c := \min(1,-\sup\Im\cE_0)>0,
  \]
  for all $k$. Therefore, to compute the index sets in any fixed half space $\Im z>-N$, it suffices to restrict to $j\leq N+1$ in~\eqref{EqPhgPfId11}, which implies that the truncated sets $\cE'_{I,k;N}:=\cE'_{I,k}\cap\{\Im z>-N\}$ etc.\ are finite for all $k$; we must show that $\cE'_{I,k;N}$ etc.\ are independent of $k$ for sufficiently large $k$ (depending on $N$). Note then:
  \begin{itemize}
  \item $\cE'_{I,k+1;N}$ only depends on $\cE_{I,k;(N-1)/2}$;
  \item $\bar\cE_{I,k+1;N}$ only depends on $\cE_{I,k;(N-1)/2}$, $\bar\cE_{I,k;N-c}$, and $\cE'_{I,k;N}$;
  \item $\cE_{I,k+1;N}$ only depends on $\cE_{I,k;N-c}$, $\cE_{I,k;(N-1)/2}$, $\bar\cE_{I,k;N}$, and $\cE'_{I,k;N}$.
  \end{itemize}
  Combining these, one finds that, a fortiori, $\cE'_{I,k+1;N}$, $\bar\cE_{I,k+1;N}$, and $\cE_{I,k+1;N}$ only depend on the sets $\cE'_{I,k-\ell;N-c}$, $\bar\cE_{I,k-\ell;N-c}$, $\cE_{I,k-\ell;\max(N-c,(N-1)/2)}$, $\ell=0,1,2$. Therefore, for $N>0$, $\cE'_{I,k;N}$ etc.\ are independent of $k$ for $k>3 N/c$, as desired. An analogous argument implies that $\cE_+$ is an index set as well.

  Finally, we show that the polyhomogeneity of the initial data $\gamma$ and $k$ in the sense of~\eqref{EqPhgData} implies that the solution in the neighborhood $U$, see~\eqref{EqPf0Nbh}, of $\ol{\{t=0\}}$ constructed in Lemma~\ref{LemmaPf0} is indeed polyhomogeneous at $I^0\cap U$ with index set $\cE_0$; this however follows from the same arguments used to prove~\eqref{EqPhgi0} (and we can in fact ignore the weight at $\scri^+$). In fact, working on ${}^0\ol{\R^4}$, we have $h\in\cA_\phg^{\cE'_0}(U)$ where $\cE'_0=\bigcup_{j\in\N_0}\bigr(j(\cE_0^0-i)+i\bigr)$ does not include the extra logarithmic terms from $\cE_{\rm log}$; this relies on the observation that the gauged Cauchy data constructed in the proof of Lemma~\ref{LemmaPf0}, see~\eqref{EqPf0H0}--\eqref{EqPf0H1}, lie in $\cA_\phg^{\cE'_0}({}^0\Sigma)$, which follows from an inspection of the proof. Upon pushing the local solution $h$ in $U$ forward to ${}^m\ol{\R^4}$, we incur the logarithmic terms encoded in the index set $\cE_{\rm log}$, see~\eqref{EqPfGlobalPhg}; this proves~\eqref{EqPhgH01}.
\end{proof}

To complete the proof, we need to study the action of $\wh{L_0}(\sigma)^{-1}$ on polyhomogeneous spaces. Let $\cE$ be an index set, and let $c\in\R$ be such that $\Im z<-c$ for all $(z,0)\in\cE$; then $\cA_\phg^{\cE+i}(I^+)\subset\rho_I^{c-1}\Hb^\infty(I^+)\subset \Hext^{-1/2+c-0,\infty}(I^+)$.
\begin{lemma}
\label{LemmaPhgNormalPhg}
  The operator $\wh{L_0}(\sigma)^{-1}$ in~\eqref{EqPhgNormalOp} extends from $\Im\sigma>-c$ as a meromorphic operator family $\wh{L_0}(\sigma)^{-1}\colon\cA_\phg^{\cE+i}(I^+)\to\cA_\phg^{0\extcup\cE}(I^+)$ with divisor contained in $\cR\extcup\cE$.
\end{lemma}
\begin{proof}
  Given $f\in\rho_I^{-1}\cA_\phg^\cE(I^+)$, we shall explicitly construct a formal solution $u_\phg$ of $\wh{L_0}(\sigma)u_\phg=f$ at $\pa I^+$, which we then correct using the inverse~\eqref{EqPhgNormalOp} acting on $\CIdot(I^+)$. The construction uses that
  \begin{equation}
  \label{EqPhgNormalPhgOp}
    \wh L_0(\sigma)=-D_{\rho_I}(\rho_I D_{\rho_I}-\sigma)+\Diffb^2(I^+),
  \end{equation}
  which follows from the form~\eqref{EqBgipDualMetric} of the dual metric of $\rho^{-2}g_m$. Thus, consider $(z,k)\in\cE$, $f_0\in\CI(\pa I^+)=\CI(\Sph^2)$, and suppose $f=\rho_I^{i z-1}(\log\rho_I)^k f_k\in\rho_I^{-1}\cA_\phg^{(z,k)}(I^+)$ near $\rho_I=0$. If $z\neq 0$, we then have
  \begin{align*}
    &\wh{L_0}(\sigma)\bigl(-z^{-1}(z-\sigma)^{-1}\rho_I^{i z}(\log\rho_I)^k f_k\bigr) - f_\phg \\
    &\qquad\qquad = (z-\sigma)^{-1}\rho_I^{i z-1}(\log\rho_I)^{k-1}f_{k-1} + (z-\sigma)^{-1}f'
  \end{align*}
  for some $f_{k-1}\in\CI(\pa I^+)$, and with $f'\in\rho_I^{-1}\cA_\phg^{(z,k)-i}(I^+)$ holomorphic in $\sigma$. We can iteratively solve away the first term, obtaining $u_j\in\CI(\pa I^+)$ such that
  \[
    \wh{L_0}(\sigma)\Biggl(\sum_{j=0}^k (z-\sigma)^{-j-1}\rho_I^{i z}(\log\rho_I)^{k-j}u_j\Biggr) - f = \sum_{j=0}^k(z-\sigma)^{-j-1}f'_j,
  \]
  where $f'_j\in\rho_I^{-1}\cA_\phg^{(z,k-j)-i}(I^+)$ is holomorphic in $\sigma$ and has improved asymptotics at $\pa I^+$. If on the other hand $z=0$, $f=\rho_I^{-1}(\log\rho_I)^k f_k\in\rho_I^{-1}\cA_\phg^{(0,k)}(I^+)$, we need an extra $\log\rho_I$ term: there exist $u_j\in\CI(\pa I^+)$ such that
  \[
    \wh{L_0}(\sigma)\Biggl(\sum_{j=0}^k\sigma^{-j-1}(\log\rho_I)^{k+1-j}u_j\Biggr) - f = \sum_{j=0}^k \sigma^{-j-1}f'_j,\ \ f'_j\in\rho_I^{-1}\cA_\phg^{(0,k+1-j)-i}(I^+).
  \]
  (Note that there is no term on the left with $(\log\rho_I)^0$.) In general, given $f\in\rho_I^{-1}\cA_\phg^{\cE}(I^+)$, we can use these arguments and asymptotic summation to construct, locally in $\sigma$, a family $u_\phg\in\cA_\phg^{0\extcup\cE}(I^+)$, depending meromorphically on $\sigma$ with divisor contained in $\cE$, such that
  \[
    \wh{L_0}(\sigma)u_\phg - f =: f' \in \cA_\phg^\emptyset(I^+) = \CIdot(I^+)
  \]
  is meromorphic with divisor contained in $\cE$; applying $\wh{L_0}(\sigma)^{-1}$ to this gives an element of $\CI(I^+)=\cA_\phg^0(I^+)$, and
  \[
    u:=\wh{L_0}(\sigma)^{-1}f = u_\phg - \wh{L_0}(\sigma)^{-1}f'
  \]
  solves $\wh{L_0}(\sigma)u=f$, with divisor contained in $\cR\extcup\cE$ due to the second term.
\end{proof}

The global solution $g=g_m+\rho h$ constructed on the space ${}^m_m M$ in Theorem~\ref{ThmPfGlobal} is polyhomogeneous as well; the only place where this is not immediate is $I^0$, where however polyhomogeneity is well-defined under the assumption~\eqref{EqPfGlobalPhg} on the index set $\cE_0$, which is already satisfied for the set $\cE_0$ constructed in Theorem~\ref{ThmPhg}. Thus, the index sets of $h$ at $I^-$, $\scri^-$, $I^0$, $\scri^+$, and $I^+$ are $\cE_+$, $\cE_I$, $\cE_0$, $\cE_I$, and $\cE_+$, respectively, likewise for the refined asymptotics of $\pi_{1 1}^c h$ and $\pi_0 h$ near $\scri^\pm$.

\section{Bondi mass and the mass loss formula}
\label{SL}

We shall first use a different characterization of the Bondi mass than the one outlined in~\S\ref{SsIBondi}: the Bondi mass can be calculated from the leading lower order terms of the metric $g$ in a so-called \emph{Bondi--Sachs coordinate system} in \S\ref{SsLBS}; in order to define these coordinates, we first need to study a special class of null-geodesics in \S\ref{SsLR}, namely those which asymptotically look like outgoing radial null-geodesics in the Schwarzschild spacetime. For simplicity, we work with the infinite regularity solutions of Theorem~\ref{ThmIDetail}, and we only control the Bondi--Sachs coordinates in a small neighborhood of $(\scri^+)^\circ$, as this is all that is needed for deriving the mass loss formula. More precise estimates, including up to $\scri^+\cap I^+$, of this coordinate system, and a precise description of future-directed null-geodesics and other aspects of the geometry near (null) infinity will be discussed elsewhere.

\subsection{Asymptotically radial null-geodesics}
\label{SsLR}

Suppose $g=g_m+\rho h$, $h\in\cX^{\infty;b_0,b_I,b'_I,b_+}$, solves $\Ric(g)=0$ in the gauge $\Ups(g;g_m)=0$, where the weights are as in Definition~\ref{DefEinF}; by an inspection of the expressions in~\S\ref{SsCoPert}, the gauge condition implies improved decay of certain (sums and derivatives of) components of the metric perturbation $h$, for instance, $\Ups(g)_0=0$ implies
\begin{equation}
\label{EqLRGamma000}
  \Gamma_0^{0 0}\in m r^{-2}+\Hb^{\infty;2+b_0,2+b_I,2+b_+}.
\end{equation}
We wish to study null-geodesics near $(\scri^+)^\circ$. Introducing coordinates $v^\mu$ on $T\R^4$ by writing tangent vectors as $v^\mu\pa_{x^\mu}$, the geodesic vector field $H\in\cV(T\R^4)$ takes the form
\[
  H = v^\mu\pa_{x^\mu} + \Gamma^\mu_{\kappa\lambda}v^\kappa v^\lambda \pa_{v^\mu}.
\]
As usual, we will use $x^0=t+r_*$, $x^1=t-r_*$, and local coordinates $x^2,x^3$ on $\Sph^2$. Consider first the case that $h=0$, so $g$ is the Schwarzschild spacetime near $\scri^+$. Radial null-geodesics then have constant $x^1$ and $x^b$, $b=2,3$, while $v^0(s)=\dot x^0(s)$ satisfies the ODE $\dot v^0=-m r(s)^{-2}(v^0)^2$, so $\ddot x^0=-m r^{-2}(\dot x^0)^2$. We then use:

\begin{lemma}
\label{LemmaLRrstar}
  We have $r=r_*-2 m\log r_* + \cO(r_*^{-1}\log r_*)$, and $r_*=\half(x^0-x^1)$.
\end{lemma}
\begin{proof}
  Let $r_0(r_*)\equiv r_*$ and
  \[
    r_{k+1}(r_*)=r_*-2 m\log(r_k(r_*)-2 m)=r_*-2 m\log(r_k)-2 m\log(1-2 m r_k^{-1}),
  \]
  then $|r_{k+1}-r_k|\leq C r_*^{-1}|r_k-r_{k-1}|$, $k\geq 1$, and the fact that $|r_1-r_0|=\cO(\log r_*)$ show that $r-r_1=\cO(r_*^{-1}\log r_*)$, hence evaluation of $r_1$ gives the result.
\end{proof}

Often, we will only need the consequence that
\begin{equation}
\label{EqLRrstar}
  r=\half x^0+\cO(\log x^0)
\end{equation}
for bounded $x^1$, suggesting the approximation $\ddot x^0=-4 m(x^0)^{-2}(\dot x^0)^2$ for the geodesic equation. Solving this by Picard iteration with initial guess $x^0_0(s)\equiv s$ gives
\[
  x^0_1(s)=s+4 m\log s,\ \ 
  \dot x^0_1(s)=1+4 m s^{-1},
\]
and subsequent iterations give $\cO(s^{-1}\log s)$, resp.\ $\cO(s^{-2}\log s)$, corrections to $x^0_1(s)$, resp.\ $\dot x^0_1(s)$. Let us generalize such radial null-geodesics:
\begin{prop}
\label{PropLR}
  Fix a point $p\in(\scri^+)^\circ$ with coordinates $x^i(p)=:\bar x^i$. Then there exists a future-directed null-geodesic $\gamma\colon[0,\infty)\to M$, $\gamma(s)=(x^\mu(s))$ such that $\gamma(s)\to p$ in $M$ and $x^a(s)-\bar x^a=o(s^{-1})$ as $s\to\infty$.
\end{prop}
\begin{proof}
  We will normalize $\gamma$ by requiring that $x^0(s)\sim s+4 m\log s$, and we shall seek $\gamma\colon[s_0,\infty)\to M$ for $s_0>0$ large. For weights $\alpha_0,\alpha_1,\slalpha>0$, to be specified in~\eqref{EqLRPfWeights} below, we will solve the geodesic equation on the level of the velocity $v^\mu=\dot x^\mu$ using a suitable Picard iteration scheme on the Banach space
  \begin{equation}
  \label{EqLRPfBanach}
    X := \bigl\{ v=(v^\mu)\colon[s_0,\infty)\to\R^4 \colon \wt v^0 \in s^{-1-\alpha_0}\cC^0,\ v^1 \in s^{-1-\alpha_1}\cC^0,\ v^a \in s^{-1-\slalpha}\cC^0 \bigr\},
  \end{equation}
  where we use the notation
  \[
    \wt v^0(s):=v^0(s)-(1+4 m s^{-1}),
  \]
  and where $\cC^0\equiv\cC^0([s_0,\infty))$ is equipped with the $\sup$ norm; as the norm on $X$, we then take the maximum of the weighted $\cC^0$ norms of $\wt v^0$ and $v^i$, $i=1,2,3$. For $v\in X$, we define its integral $x=I(v)$, $\dot x^\mu(s)=v^\mu(s)$, by
  \begin{equation}
  \label{EqLRPfx}
  \begin{split}
    x^0(s) &:= s+4 m\log s - \int_s^\infty \wt v^0(u)\,d u, \\
    x^i(s) &:= \bar x^i - \int_s^\infty v^i(u)\,d u,\ \ i=1,2,3.
  \end{split}
  \end{equation}
  As the first iterate, we take
  \[
    \wt v_0^0(s),\, v_0^i(s)\equiv 0,\ \ 
    x_0:=I(v_0);
  \]
  note that $\|v_0\|_X=0$. For $k\geq 0$, $v_k\in X$, $\|v_k\|_X\leq 1$, and $x_k=I(v_k)$, let then
  \begin{equation}
  \label{EqLRPfIt}
    v^\mu_{k+1}(s) := v_k^\mu(\infty) + \int_s^\infty \Gamma^\mu_{\kappa\lambda}|_{x_k(u)} v_k^\kappa(u)v_k^\lambda(u)\,d u,\ \ 
    x_{k+1} := I(v_{k+1}).
  \end{equation}
  Note that for some fixed constant $C>0$,
  \begin{equation}
  \label{EqLRPfxEst}
    |x_k^0(s)-s-4 m\log s|\leq C s^{-\alpha_0},\ \ 
    |x_k^1(s)-\bar x^1|\leq C s^{-\alpha_1},\ \ 
    |x_k^i(s)-\bar x^i|\leq C s^{-\slalpha},
  \end{equation}
  which in particular allows us to estimate the Christoffel symbols appearing in~\eqref{EqLRPfIt}. For $\mu=0$, writing $r_k(s)=r(x_k(s))$, and using the improved decay of various Christoffel symbols due to the gauge condition $\Ups(g)=0$, we have
  \begin{equation}
  \label{EqLRPfV0}
  \begin{split}
    \wt v_{k+1}^0(s) &= -4 m s^{-1} + \int_s^\infty m r_k(u)^{-2}\,d u + \int_s^\infty \cO_{s_0}(u^{-2-b_I})\,d u \\
      &\qquad + \int_s^\infty \cO_{s_0}(u^{-2}\log u\cdot 1\cdot u^{-1-\alpha_1}) + \cO_{s_0}(u^{-1}\cdot 1\cdot u^{-1-\slalpha}) \\
      &\qquad\qquad + \cO_{s_0}(u^{-1}\log u\cdot u^{-2-\alpha_1}) + \cO_{s_0}(u^{-1}\log u\cdot u^{-1-\alpha_1}\cdot u^{-1-\slalpha}) \\
      &\qquad\qquad + \cO_{s_0}(u\cdot u^{-2-2\slalpha})\,d u,
  \end{split}
  \end{equation}
  with the integrals on the first line coming from terms with $(\kappa,\lambda)=(0,0)$ and using~\eqref{EqLRGamma000}, while the remaining terms come from $(\kappa,\lambda)=(0,1)$, $(0,b)$, $(1,1)$, $(1,b)$, $(a,b)$, in this order, using that $v_k^0=\cO(1)$, $v_k^1=\cO(s^{-1-\alpha_1})$, and $v_k^a=\cO(s^{-1-\slalpha})$. As for the notation, the constants implicit in the $\cO_{s_0}$ notation depend only on $s_0$ and are nonincreasing with $s_0$, as they come from the size of the Christoffel symbols along $x_k(s)$, which satisfies~\eqref{EqLRPfxEst}. By~\eqref{EqLRrstar} and~\eqref{EqLRPfxEst}, we have
  \[
    \int_s^\infty m r_k(u)^{-2}\,d u=\int_s^\infty 4 m(u^{-2}+\cO(u^{-3}\log u))\,d u = 4 m s^{-1} + \cO(s^{-2}\log s).
  \]
  Therefore, we have
  \[
    |\wt v^0_{k+1}(s)| \lesssim_{s_0} s^{-1-b_I} + s^{-2-\alpha_1}\log s + s^{-1-\slalpha} + s^{-2\slalpha},
  \]
  which, for fixed $\alpha_0<b_I$, is bounded by $\frac{1}{10}s^{-1-\alpha_0}$ for large $s_0$, provided $\alpha_0<\min(\slalpha,1+\alpha_1,2\slalpha-1)$; in particular, this requires $\slalpha>\half$.
  
  We obtain estimates on $v_{k+1}^i(s)$, $i=1,2,3$, in a similar manner. Namely,
  \begin{equation}
  \label{EqLRPfV1}
  \begin{split}
    v_{k+1}^1(s) &= \int_s^\infty \cO_{s_0}(u^{-2-b'_I}\cdot 1^2) + \cO_{s_0}(u^{-2}\cdot 1\cdot u^{-1-\alpha_1}) + \cO_{s_0}(u^{-1-b'_I}\cdot 1\cdot u^{-1-\slalpha}) \\
      &\qquad\qquad + \cO_{s_0}(u^{-1}\cdot u^{-2-2\alpha_1}) + \cO_{s_0}(u^{-1}\cdot u^{-1-\alpha_1}\cdot u^{-1-\slalpha}) \\
      &\qquad\qquad + \cO_{s_0}(u\cdot u^{-2-2\slalpha})\,d u
  \end{split}
  \end{equation}
  satisfies $|v_{k+1}^1(s)|\lesssim_{s_0}s^{-1-b'_I}+s^{-2\slalpha}$, hence $|v_{k+1}^1(s)|<\frac{1}{10}s^{-1-\alpha_1}$ provided the weights satisfy $\alpha_1<\min(b'_I,2\slalpha-1)$, and provided we increase $s_0$, if necessary.
  
  Lastly, using the precise form of the leading term of $\Gamma_{0 b}^c$,
  \begin{equation}
  \label{EqLRPfVA}
  \begin{split}
    v^a_{k+1}(s) &= \int_s^\infty \cO_{s_0}(u^{-3-b'_I}) + \cO_{s_0}(u^{-3}\cdot 1\cdot u^{-1-\alpha_1}) \\
      &\qquad\qquad + \bigl( u^{-1}\cdot 1\cdot u^{-1-\slalpha} + \cO_{s_0}(u^{-2}\cdot 1\cdot u^{-1-\slalpha}) \bigr) \\
      &\qquad\qquad + \cO_{s_0}(u^{-2}\cdot u^{-2-2\alpha_1}) + \cO_{s_0}(u^{-1}\cdot u^{-1-\alpha_1}\cdot u^{-1-\slalpha}) \\
      &\qquad\qquad + \cO_{s_0}(1\cdot u^{-2-2\slalpha})\,d u.
  \end{split}
  \end{equation}
  Integrating the first term in the second line gives a term bounded from above by
  \[
    \bigl|{-\tfrac{1}{1+\slalpha}}s^{-1-\slalpha}\bigr|<\tfrac23 s^{-1-\slalpha}\ \ (\slalpha>\half),
  \]
  so we get $|v^a_{k+1}(s)|<(\frac{2}{3}+\frac{1}{10})s^{-1-\slalpha}$ provided $\slalpha<1+b'_I$ (which is consistent with $\slalpha>\half$). Thus, the iteration~\eqref{EqLRPfIt} maps the unit ball in $X$ into itself, provided we fix weights
  \begin{equation}
  \label{EqLRPfWeights}
    \alpha_0\in(0,b_I),\ \ 
    \alpha_1\in(0,b'_I),\ \ 
    \slalpha\in(\half,1+b'_I),
  \end{equation}
  and choose $s_0$ large; recall here that $0<b_I<b'_I<1$. Moreover, taking $s_0$ larger if necessary, $v_k\mapsto v_{k+1}$ is a contraction; such an estimate is only nonobvious for the difference of quadratic terms in~\eqref{EqLRPfIt} involving the component $v^0$; however, the corresponding terms come with a \emph{small} prefactor due to the smallness of the relevant Christoffel symbols.

  Let now $v:=\lim_{k\to\infty} v_k\in X$ denote the limiting curve in $T\R^4$, and integrate it by setting $\gamma:=I(v)$. Then $v$ satisfies the integral equation~\eqref{EqLRPfIt} with $v_k$ and $v_{k+1}$ replaced by $v$, so $v$ is $\cC^1$, hence $\gamma$ is a $\cC^2$ geodesic. In particular, $|v(s)|_{g(s)}^2$ is constant, hence equal to its limit as $s\to\infty$, which is
  \begin{align*}
    &\cO(s^{-1-b'_I}\cdot 1^2) + \cO(1\cdot 1\cdot s^{-1-\alpha_1}) + \cO(s^{-b'_I}\cdot 1\cdot s^{-1-\slalpha}) \\
    &\quad + \cO(s^{-1}\log s\cdot s^{-2-2\alpha_1}) + \cO(1\cdot s^{-1-\alpha_1}\cdot s^{-1-\slalpha}) + \cO(s^2\cdot s^{-2-2\slalpha}) = o(1),\ \ s\to\infty.
  \end{align*}
  This proves that $\gamma$ is a null-geodesic with the desired properties.
\end{proof}

Note that $\gamma$ is the \emph{unique} null-geodesic, up to translation of the affine parameter, tending to $p$ and such that $\dot\gamma\in X$. (Indeed, for any such $\gamma$, the velocity $\dot\gamma$ has \emph{small} norm in a space defined like $X$ but with weights decreased by a small amount and for $s_0$ large enough. The uniqueness then follows from the fixed point theorem.)

\begin{definition}
\label{DefLR}
  For $p\in(\scri^+)^\circ$, denote by $\gamma_p(s)$ the maximal null-geodesic such that $v=\dot\gamma_p$ and $x=\gamma_p$ satisfy equation~\eqref{EqLRPfx} and $v\in X$, with $X$ given in~\eqref{EqLRPfBanach}. We call $\gamma_p$ a \emph{radial null-geodesic}.
\end{definition}

We record the following stronger regularity property of the geodesics $\gamma_p$:
\begin{lemma}
\label{LemmaLRSymb}
  In the notation of Proposition~\usref{PropLR}, let $\gamma_p(s)=(x^\mu(s))$ denote a radial null-geodesic; then we have
  \[
    \wt x^0(s) \in S^{-\alpha_0}([s_0,\infty)),\ \ 
    \wt x^1(s) \in S^{-\alpha_1}([s_0,\infty)),\ \ 
    \wt x^a(s) \in S^{-\slalpha}([s_0,\infty)),
  \]
  for all weights $\alpha_0<b_I$, $\alpha_1<b'_I$, $\slalpha<1+b'_I$, where $\wt x^0(s):=x^0(s)-(s+4 m\log s)$, $\wt x^i(s):=x^i(s)-\bar x^i$, and where $S^m([s_0,\infty))$ denotes symbols of order $m$, i.e.\ functions $u\in\CI([s_0,\infty))$ such that for any $k\in\N_0$, $|u^{(k)}(s)|\leq C_k\la s\ra^{m-k}$.
\end{lemma}
\begin{proof}
  Certainly $x^\mu(s)$ is smooth as a geodesic in a spacetime with smooth metric tensor. The symbolic estimates for $\pa_s^k\wt x^\mu(s)$ for $k=0,1$ follow immediately from the construction of $\gamma_p$ in the proof of Proposition~\ref{PropLR}; for $k=2$, they follow from the proof as well, specifically, from the decay of the integrands in~\eqref{EqLRPfV0}--\eqref{EqLRPfVA}. Assuming that for some $k\geq 1$ we have $|\pa_s^j\wt x^0(s)|\lesssim\la s\ra^{\alpha_0-j}$, $0\leq j\leq k+1$, with $\alpha_0$ as in~\eqref{EqLRPfWeights}, likewise for $\wt x^i$, $i=1,2,3$, we have
  \[
    \pa_s^k(\pa_s^2\wt x^0) = \pa_s^k\ddot x^0 - \pa_s^{k+2}(s+4 m\log s) = \pa_s^k\ddot x^0 + \pa_s^k(4 m s^{-2}),
  \]
  and $\pa_s^k\ddot x^0=-\pa_s^k(\Gamma^0_{\mu\nu}\dot x^\mu\dot x^\nu)$. Note that $x^0(s)=\cO(s)$, $\pa_s x^0(s)=\cO(1)$, and $\pa_s^j x^0(s)=\cO(s^{-1-j})$ for $2\leq j\leq k+1$. Expanding the derivatives using the Leibniz and chain rules thus gives the following types of terms: for $(\mu,\nu)=(0,0)$ and all derivatives falling on the Christoffel symbol,
  \begin{align*}
    (\pa_s^k\Gamma^0_{0 0})(\dot x^0)^2&=\pa_s^k(4 m s^{-2}+\cO(s^{-2-b_I}))(1+\cO(s^{-1}\log s)) \\
      &=\pa_s^k(4 m s^{-2}) + \cO(s^{-k-2-b_I})
  \end{align*}
  by the inductive hypothesis and the b-regularity of the remainder term in $\Gamma^0_{0 0}$; the remaining $(\mu,\nu)=(0,0)$ terms are, with $\ell_1+\ell_2+\ell_3=k$ and $\ell_2>0$,
  \[
    (\pa_s^{\ell_1}\Gamma_{0 0}^0)(\pa_s^{\ell_2}\dot x^0)(\pa_s^{\ell_3}\dot x^0) = \cO(s^{-2-\ell_1}\cdot s^{-1-\ell_2}\cdot s^{-\ell_3}) = \cO(s^{-k-3}).
  \]
  Estimating the terms with $(\mu,\nu)\neq(0,0)$ does not require special care: derivatives falling on $\dot x^\mu$ are estimated using the inductive hypothesis (thus every derivative gives an extra power of decay in $s$); a derivative falling on $\Gamma_{\mu\nu}^0$ on the other hand either produces $(\pa_0\Gamma_{\mu\nu}^0)\dot x^0$, which gains an order of decay due to the Christoffel symbol (recall that $\pa_0$ is a b-derivative which \emph{vanishes} at $\scri^+$), or $(\pa_i\Gamma_{\mu\nu}^0)\dot x^i$, which gains an order of decay due to $\dot x^i=\cO(s^{-1})$. Thus, the bound $\pa_s^k(\pa_s^2\wt x^0)=\cO(s^{-k-2-\alpha_0})$ follows from the same arithmetic of weights as used after~\eqref{EqLRPfV0}.

  The arguments for the other components $\wt x^i$ are completely analogous, and in fact simpler as no terms need to be handled separately. This finishes the inductive step, and thus the proof of the lemma.
\end{proof}

We further note that for any compact subset $K\Subset(\scri^+)^\circ$, there exists a uniform value $s_0\in\R$ such that the null-geodesics $\gamma_p$, $p\in K$, are defined on $[s_0,\infty)$; since moreover $\gamma_p$ arises, via $\gamma_p=I(\dot\gamma_p)$ as in~\eqref{EqLRPfx}, from the Banach fixed point theorem for a smooth (in $p$) contraction, Lemma~\ref{LemmaLRSymb} holds \emph{smoothly} in the parameter $p$, that is, making the dependence on $p$ explicit as a subscript, we have $\wt x^0_p(s)\in \CI(K;S^{-\alpha_0}([s_0,\infty)))$ etc.

Consider now the union of radial null-geodesics tending to the points of particular $\Sph^2$ sections of $\scri^+$. Concretely, for fixed $\bar x^1\in\R$, denote
\begin{equation}
\label{EqLRSpheres}
  S(\bar x^1):=\{p\in\scri^+\colon x^1(p)=\bar x^1\},\ \ 
  C_{\bar x^1}:=\bigcup_{p\in S(\bar x^1)} \gamma_p((s_0,\infty)),
\end{equation}
where $s_0$ is chosen sufficiently large, \emph{which will always be assumed from now on}. See Figure~\ref{FigLRCu}. Thus, on the Schwarzschild spacetime, $C_{\bar x^1}$ is the part of the null hypersurface $x^1=\bar x^1$ on which $x^0\gtrsim s_0$.

\begin{figure}[!ht]
\includegraphics{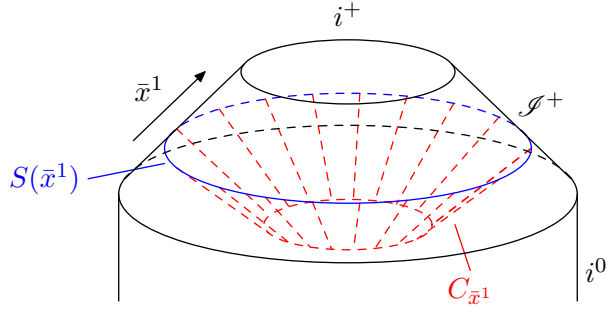}
\caption{The outgoing light cone $C_{\bar x^1}$ limiting to the sphere $S(\bar x^1)\subset(\scri^+)^\circ$. Also shown are a number of radial null-geodesics.}
\label{FigLRCu}
\end{figure}

\begin{lemma}
\label{LemmaLRCongr}
  For $\bar x^1\in\R$, the set $C_{\bar x^1}$ is a smooth null hypersurface near $\scri^+$. Moreover, if $I^1\Subset\R$ is a precompact open interval, then there exists a function $u$ such that
  \begin{equation}
  \label{EqLRCongr}
    u-x^1=:\wt u\in\rho_I^{b'_I-0}\Hb^\infty(M); \quad
    C_{\bar x^1} = \{ u=\bar x^1 \},\ \ \bar x^1\in I^1.
  \end{equation}
\end{lemma}
\begin{proof}
  With coordinates $x^a$, $a=2,3$, on $\Sph^2$, write $\gamma(\bar x^1;s,\bar x^2,\bar x^3) := \gamma_{(\bar x^1,\bar x^2,\bar x^3)}(s)$. First, we shall prove that there exists a coordinate change of $\R_{x^0}\times\R^2_{x^2,x^3}$,
  \begin{equation}
  \label{EqLRCongrPhi}
    \Phi(\bar x^1;x^0,x^2,x^3)=(x^0-4 m\log x^0+\wt\Phi^0,x^2+\wt\Phi^2,x^3+\wt\Phi^3)=:(\Phi^0,\Phi^2,\Phi^3),
  \end{equation}
  depending parametrically on $\bar x^1\in I^1$, and with $\wt\Phi^0\in S^{-\alpha_0}$, $\wt\Phi^a\in S^{-\slalpha}$ for weights as in~\eqref{EqLRPfWeights} (with the symbolic behavior in $x^0$), such that the map
  \[
    \delta(x^0,\bar x^1,x^2,x^3) := \gamma(\bar x^1;\Phi(\bar x^1;x^0,x^2,x^3))
  \]
  satisfies $x^i\circ\delta=x^i$, $i=0,2,3$. To do this, recall that, putting $\gamma^\mu:=x^\mu\circ\gamma$, we have $\gamma^0-(s+4 m\log s)=:\wt\gamma^0\in S^{-\alpha_0}$, $\gamma^1-\bar x^1=:\wt\gamma^1\in S^{-\alpha_1}$, and $\gamma^a-\bar x^a=:\wt\gamma^a\in S^{-\slalpha}$, so after some simplifications, our task becomes choosing $\wt\Phi^i$ such that
  \begin{equation}
  \label{EqLRCongrPhiIt}
    \wt\Phi^0 = 4 m\log\bigl(1-4 m(x^0)^{-1}(\log x^0+\wt\Phi^0)\bigr)-\wt\gamma^0(\bar x^1;\Phi), \quad
    \wt\Phi^a = -\wt\gamma^a(\bar x^1;\Phi);
  \end{equation}
  this can be solved, first with $\wt\Phi^0\in(x^0)^{-\alpha_0}\cC^0$ etc.\ using the fixed point theorem, and then in symbol spaces using the smoothness of $\wt\Phi^0$ (which follows from the implicit function theorem) and an iterative argument.
  
  Let us drop $x^0,x^2,x^3$ from the notation. The desired function $u$ is then defined implicitly by $u\circ\delta=\bar x^1$. Writing $x^1(\delta(\bar x^1))=:\bar x^1+f$, where $f\in S^{-\alpha_1}$ by Lemma~\ref{LemmaLRSymb}, we see that $\delta$ is one to one for large $x^0$, as $\bar x^1+f(\bar x^1)=\bar y^1+f(\bar y^1)$ implies $0\geq|\bar x^1-\bar y^1|-C(x^0)^{-\alpha_1}|\bar x^1-\bar y^1|$, so $\bar x^1=\bar y^1$ if $x^0$ is large. Writing $u=\bar x^1+\wt u$, we thus need to solve
  \[
    (\bar x^1+\wt u)+f(\bar x^1+\wt u)=\bar x^1\ \ \Llra\ \ 
    \wt u = -f(\bar x^1+\wt u),
  \]
  which by another application of the fixed point theorem has a solution $\wt u\in S^{-\alpha_1}$. Lastly, note that the vector fields $\pa_{x^i}$, $i=2,3,4$, and $x^0\pa_{x^0}$ span $\Vb(M)$ near $(\scri^+)^\circ$ in view of $\rho_I=1/x^0$, hence $S^{-\alpha_1}\subset\rho_I^{\alpha_1-0}\Hb^\infty$ near $(\scri^+)^\circ$. Since we can take $\alpha_1$ arbitrarily close to $b'_I$ by~\eqref{EqLRPfWeights}, the existence of $u$ and smoothness of $C_{\bar x^1}$ follows.

  It remains to prove that $C_{\bar x^1}$ is a \emph{null} hypersurface. To this end, we sketch a different way of constructing $C_{\bar x^1}$: let $\bar x^0>0$, and consider the 2-sphere $S_{\bar x^1\bar x^0}=\{x^0=\bar x^0,\ x^1=\bar x^1\}$. For sufficiently large $\bar x^0$, $S_{\bar x^1\bar x^0}$ is spacelike; hence, for any $p\in S_{\bar x^1\bar x^0}$, there are precisely $4$ rays of lightlike directions in $(T_p S_{\bar x^1\bar x^0})^\perp$, and there exists a unique $v(p)\in(T_p S_{\bar x^1\bar x^0})^\perp$ which is future lightlike and outgoing (i.e.\ $d r(v(p))>0$), and for which $v(p)^0=1+\frac{2 m}{r(p)}$. By writing out the condition $g(v(p),\pa_a)=0$ using the form~\eqref{EqEinFMetric} of $g$, one obtains an expression for $v(p)^a$ in terms of a small multiple of $v(p)^1$ and certain metric coefficients, while using $|v(p)|_g^2=0$ (and using the nonvanishing of $g_{0 1}$) gives an expression for $v(p)^1$ in terms of a small multiple of $v(p)^a$, plus certain metric coefficients. Solving this simple system, one finds that the components of $v(p)$ satisfy $v(p)^1=\cO(r^{-1-b_I'})$ and $v(p)^a=\cO(r^{-2-b_I'})$; they are thus \emph{small} when measured in the norm of $X$ (restricted to a single point) in~\eqref{EqLRPfBanach}, cf.\ the upper bounds on the weights in~\eqref{EqLRPfWeights}.
  
  A small modification of the fixed point argument in the proof of Proposition~\ref{PropLR} shows that we can solve the geodesic equation with initial data $v(p)$ in the backwards direction up to a \emph{fixed} value of $x^0$, say $x^0=C\gg 1$; denote the union of these null-geodesic segments emanating from points on $S_{\bar x^1\bar x^0}$ by $C_{\bar x^1\bar x^0}$. Letting $\bar x^0\to\infty$, it then follows that $C_{\bar x^1\bar x^0}$ converges over every compact subset of $\R^4\cap\{x^0>C\}$ to $C_{\bar x^1}$ in the $\cC^1$ topology. By construction, every $C_{\bar x^1\bar x^0}$ is a null hypersurface; thus, its $\cC^1$ limit $C_{\bar x^1}$ is a null hypersurface as well.
\end{proof}

The function $u$ is uniquely defined by~\eqref{EqLRCongr}; thus, Lemma~\ref{LemmaLRCongr} shows the existence of a neighborhood
\begin{equation}
\label{EqLRNbh}
  (\scri^+)^\circ\subset U^+\subset M
\end{equation}
and a function $u\in x^1+\rho_I^{b'_I-0}\Hbloc^\infty(U^+)$ such that $C_{\bar x^1}\cap U^+=\{u=\bar x^1\}$ for all $\bar x^1\in\R$.

\begin{rmk}
  The weight in~\eqref{EqLRCongr} is consistent with the choice of the domain~\eqref{EqBgscriDom} whose boundary component $U_\eps^\pa$ is \emph{spacelike}, see~\eqref{EqBgscriFinalDefFn}.
\end{rmk}

Since $|\nabla u|^2\equiv 0$ by construction, the vector field $\nabla u$ consists of null-generators of its level sets $C_u$; more precisely, we have $\nabla_{\nabla u}\nabla u=0$, so restricted to the image of a radial null-geodesic $\gamma_p\subset C_u$, we have $(\nabla u)|_{\gamma_p(s)}=c_p\dot\gamma_p(s)$ for some \emph{constant} $c_p$. Taking the inner product with $\pa_1$ and using the form~\eqref{EqEinFMetric} of $g$ yields $1+\cO(s^{-b'_I+0})=c_p(\half+\cO(s^{-1}))$, so letting $s\to\infty$ gives $c_p=2$ and thus
\[
  (\nabla u)|_{\gamma_p(s)} = 2\dot\gamma_p(s).
\]
We can then extract more information using $r=\half s+\cO(\log s)$ and $g_{0 1}=\half+2 s^{-1}(h_{0 1}-m)+\cO(s^{-2}\log s)$: Lemma~\ref{LemmaLRSymb} then gives $2\la\dot\gamma_p(s),\pa_1\ra=1+4 s^{-1}h_{0 1}+\cO(s^{-1-\alpha_0})$, so
\begin{equation}
\label{EqLRuDer}
  \pa_1\wt u-2 r^{-1}h_{0 1}\in\rho_I^{1+b_I-0}\Hb^\infty.
\end{equation}

\subsection{Bondi--Sachs coordinates; proof of the mass loss formula}
\label{SsLBS}

The function $u$ has nonvanishing differential everywhere on $C_{\bar x^1}$ when $x^0$ is large; we will use it one coordinate of a \emph{Bondi--Sachs} coordinate system $(u,\mathring r,\mathring x^2,\mathring x^3)$, where the coordinates $\mathring r$ and $\mathring x^a$, $a=2,3$, are geometrically defined and constructed below; with respect to such a coordinate system, the metric takes the form
\[
  g = g_{u u}\,d u^2 + 2 g_{u\mathring r}\,d u\,d\mathring r - \mathring r^2 q_{a b}(d\mathring x^a-\wt U^a\,d u)(d\mathring x^b-\wt U^b\,d u)
\]
for some $g_{u u}$, $g_{u \mathring r}$, $q_{a b}$, and $\wt U^a$, and quantities of geometric or physical interest such as the Bondi mass and the gravitational energy flux can be calculated in terms of certain lower order terms of these metric coefficients \cite{BondivdBMetznerGravity,MaedlerWinicourBondi}. We begin by defining $\mathring r$. Introduce a projection $\pi\colon U^+\to\Sph^2$ by
\[
  \pi(\gamma_{(\bar x^1,\theta)}(s)) := \theta,\ \ \theta\in\Sph^2,
\]
which is well-defined due to Lemma~\ref{LemmaLRCongr}; in fact, in the notation of its proof, using local coordinates $x^a$, $a=2,3$, on $\Sph^2$, we have
\begin{equation}
\label{EqLBSProjExpl}
  \pi(x^0,x^1,x^2,x^3)=(\Phi^a(x^1+\wt u;x^0,x^2,x^3))_{a=2,3},
\end{equation}
which in particular gives
\begin{equation}
\label{EqLBSProj}
  \pi(x^0,x^1,x^2,x^3)-(x^2,x^3)\in S^{-\slalpha}.
\end{equation}
The map $\pi$ defines a fibration of every $C_u$; these fibrations have natural sections, as we proceed to explain invariantly. Let $N:=\ker\pi_*$ denote the subbundle (smooth in $M^\circ$) consisting of vectors tangent to the fibers of $\pi$: this is the bundle of null generators of the null hypersurfaces $C_u$, and therefore $N \perp T C_u$. This implies that the spacetime metric $g$ restricts to an element
\[
  [g]\in S^2(T C_u/N)^*.
\]
On the other hand, the pull-back $\pi^*\slg$ induces a Riemannian metric $[\pi^*\slg]$ on $T C_u/N$, i.e.\ an isomorphism $T C_u/N\to(T C_u/N)^*$, hence $[\pi^*\slg]^{-1}[g]\in\End(T C_u/N)$ is well-defined. We then define the \emph{area radius} $\mathring r$ by the formula
\[
  \mathring r^4 := \det\bigl([\pi^*\slg]^{-1}[g]\bigr),\ \ \mathring r>0.
\]
\begin{lemma}
\label{LemmaLBSAreaR}
  We have $\mathring r-r\in\rho_I^{b'_I-0}\Hb^\infty$ and $\pa_0\mathring r=\half-m r^{-1}+\rho_I^{1+b'_I-0}\Hb^\infty$ near $(\scri^+)^\circ$.
\end{lemma}
\begin{proof}
  It suffices to prove the first claim. We start by finding representatives in $T C_u$ of a basis of $T C_u/N$ by considering the vector fields
  \begin{equation}
  \label{EqLBSAreaRVec}
    V_a=f_a\pa_1+\pa_a,\ \ a=2,3,
  \end{equation}
  with $f_a$ to be determined. Working over the image of a fixed geodesic $\gamma_p\colon[s_0,\infty)\to M$, we use $\dot\gamma_p=(1+\cO(s^{-1}))\pa_0+\cO(s^{-1-\alpha_1})\pa_1+\sum_c\cO(s^{-1-\slalpha})\pa_c$ and the form of $g$ to calculate
  \[
    g(\dot\gamma_p,V_a) = (\half+\cO(s^{-1}))(1+\cO(s^{-1}))f_a + \cO(s^{1-\slalpha});
  \]
  demanding this to vanish determines $f_a=\cO(s^{1-\slalpha})$. Since $\slalpha<1+b'_I$ is arbitrary, we conclude that
  \begin{equation}
  \label{EqLBSAreaRg}
    g(V_a,V_b) = -r^2\slg_{a b}+r h_{\bar a\bar b} + \cO(r^{-b'_I+0}),
  \end{equation}
  while the observation~\eqref{EqLBSProj} implies that $\pi_*(V_a)\in\pa_a+C_a^b\pa_b$, $C_a^b=\cO(s^{-\slalpha})$, hence
  \begin{equation}
  \label{EqLBSAreaRslg}
    (\pi^*\slg)(V_a,V_b) = \slg_{a b} + \cO(s^{-1-b'_I+0}).
  \end{equation}
  Therefore,
  \[
    \mathring r^4 = r^4 \det\bigl(1-r^{-1}(\slg^{b c}h_{\bar a\bar b})_{a,c=2,3}+\cO(s^{-1-b'_I+0})\bigr) = r^4(1-r^{-1}\sltr h+\cO(s^{-1-b'_I+0})),
  \]
  which is equal to $r^4(1+\cO(s^{-1-b'_I+0}))$ due to the decay of $\sltr h$ at $\scri^+$ coming from the membership $h\in\cX^{\infty;b_0,b_I,b'_I,b_+}$, i.e.\ ultimately from the gauge condition. Taking fourth roots, carrying symbolic behavior in $s$ through the argument, and noting that these calculations depend smoothly on the parameter $p\in(\scri^+)^\circ$ completes the proof.
\end{proof}

\begin{cor}
\label{CorLBSCoord}
  Define the punctured neighborhood $\dot U^+:=U^+\setminus(\scri^+)^\circ$ of $(\scri^+)^\circ$, see~\eqref{EqLRNbh}. Then if $U^+$ is a sufficiently small neighborhood, $(u,\mathring r,\pi)\colon\dot U^+\to\R\times\R\times\Sph^2$ is a coordinate system on $\dot U^+$.
\end{cor}
\begin{proof}
  This follows from Lemma~\ref{LemmaLBSAreaR} and the asymptotics of $u$ and $\pi$ in~\eqref{EqLRCongr} and~\eqref{EqLBSProj}.
\end{proof}

Choosing local coordinates $x^a$ on $\Sph^2$ and letting $\mathring x^a:=x^a\circ\pi=x^a+\rho_I^{1+b'_I-0}\Hb^\infty$, we can introduce the Bondi--Sachs coordinates
\begin{equation}
\label{EqLBSCoord}
  (u,\mathring r,\mathring x^2,\mathring x^3)
\end{equation}
on $U$; the metric $g$ and its dual $G=g^{-1}$ simplify in this coordinate system since, by construction,
\begin{equation}
\label{EqLBSg1}
  G(d u,d u)\equiv 0, \quad
  G(d u,d\mathring x^a)=(\nabla u)(\mathring x^a)\equiv 0.
\end{equation}
Furthermore, using~\eqref{EqLRuDer} and Lemma~\ref{LemmaLBSAreaR},
\begin{equation}
\label{EqLBSg2}
\begin{split}
  G(d u,d\mathring r) &= 1 + \rho_I^{1+b'_I-0}\Hb^\infty, \\
  G(d\mathring x^a,d\mathring x^b) &= -\mathring r^{-2}\slg^{a b}-\mathring r^{-3}h^{\bar a\bar b} + \rho_I^{3+b'_I-0}\Hb^\infty,
\end{split}
\end{equation}
where the leading term in the first expression comes from $g^{0 1}(\pa_1 u)(\pa_0\mathring r)$. In order to calculate $G(d\mathring r,d\mathring r)$ to the same level of precision, we need to sharpen Lemma~\ref{LemmaLBSAreaR}.

\begin{lemma}
\label{LemmaLBSAreaRder}
  Near $(\scri^+)^\circ$, we have
  \[
    \pa_1\mathring r=-\half+\bigl(m+\half(h_{1 1}-2 h_{0 1})+r\pa_0 h_{1 1}-\tfrac14\slnabla_a\slnabla_b h^{\bar a\bar b}\bigr)r^{-1}+\rho_I^{1+b_I-0}\Hb^\infty.
  \]
\end{lemma}

Note that in~\eqref{EqLBSAreaRg}, we already control $g(V_a,V_b)$ modulo terms more than two orders beyond the leading term, which suffices for present purposes. On the other hand, the remainder term in~\eqref{EqLBSAreaRslg} is not precise enough. 

\begin{proof}[Proof of Lemma~\usref{LemmaLBSAreaRder}]
  Put $A:=[\pi^*\slg]^{-1}[r^{-2}g]\in\End(T C_u/N)$, so $(\mathring r/r)^4=\det A$, and Lemma~\ref{LemmaLBSAreaR} gives $A_a^b=\delta_a^b-r^{-1}h_{\bar a}{}^{\bar b}+\rho_I^{1+b'_I-0}\Hb^\infty$ and $(\det A)-1\in\rho_I^{1+b'_I-0}\Hb^\infty$. Suppose now that
  \begin{equation}
  \label{EqLBSAreaRderPaDet}
    \pa_1(\det A)=r^{-2}\mu+o(r^{-2}),
  \end{equation}
  then $\pa_1((\mathring r-r)/r)=\tfrac14(\det A)^{-3/4}\pa_1(\det A)=\tfrac14 r^{-2}\mu+o(r^{-2})$, so expanding the left hand side as $r^{-1}(\pa_1\mathring r+\half-m r^{-1})+o(r^{-2})$ implies that
  \begin{equation}
  \label{EqLBSAreaRderRes}
    \pa_1\mathring r = -\half + r^{-1}(m+\tfrac14 \mu) + o(r^{-1})
  \end{equation}
  Our calculations will imply that the $o(r^{-1})$ remainder is of size $\cO(r^{-1-b_I+0})$, but we shall stick to $o(r^{-1})$ etc.\ for brevity. Trivializing $T C_u/N$ locally using the frame $\{V_a\colon a=2,3\}$, with $V_a$ defined in~\eqref{EqLBSAreaRVec}, $A$ becomes a $2\times 2$ matrix-valued function. We can thus use the formula $\pa_1(\det A)=(\det A)\tr(A^{-1}\pa_1 A)$, so it suffices to determine the function $\mu$ in $\tr(A^{-1}\pa_1 A)=r^{-2}\mu+o(r^{-2})$. One contribution comes from differentiating $[r^{-2}g]$, which by~\eqref{EqLBSAreaRg} and $\Ups(g)_1=0$ yields
  \begin{equation}
  \label{EqLBSAreaRderTrg}
  \begin{split}
    \tr([r^{-2}g]^{-1}\pa_1[r^{-2}g]) &= (-\slg^{a b}-r^{-1}h^{\bar a\bar b}+\cO(r^{-2+0}))(\pa_1(r^{-1}h_{\bar a\bar b})+\cO(r^{-2-b'_I+0})) \\
      &= -r^{-1}\pa_1\sltr h - r^{-2}h^{\bar a\bar b}\pa_1 h_{\bar a\bar b} + o(r^{-2}) \\
      &= 2 r^{-2}(h_{1 1}-2 h_{0 1}) - 2 r^{-2}\slnabla_d h_1{}^{\bar d} + 4 r^{-1}\pa_0 h_{1 1} + o(r^{-2}).
  \end{split}
  \end{equation}
  The remaining contribution to $\tr(A^{-1}\pa_1 A)$ is $-\!\tr([\pi^*\slg]^{-1}\pa_1[\pi^*\slg])$ (using the cyclicity of the trace). Let us work near a point $z_0\in\R^4$, and suppose $x^2,x^3$ are normal coordinates on $\Sph^2$ centered at the point $\pi(z_0)$. Then
  \begin{align*}
    \bigl(\pa_1(\pi^*\slg)(V_a,V_b)\bigr)|_{z_0} &= \pa_1\bigl((\slg_{c d}\circ\pi)(\pi_*V_a)^c(\pi_*V_b)^d\bigr)|_{z_0} \\
      &= \slg_{c d}|_{\pi(z_0)}(\pa_1(\pi_*V_a)^c)(\pi_*V_b)^d + \slg_{c d}|_{\pi(z_0)}(\pi_*V_a)^c(\pa_1(\pi_*V_b)^d).
  \end{align*}
  Now $(\pi_*V_a)^c=\delta_a^c+\cO(r^{-1-b'_I+0})$, whose derivative along $\pa_1$ is of size $\cO(r^{-1-b'_I+0})$, so
  \begin{equation}
  \label{EqLBSAreaRderTrslg}
    \pa_1(\pi^*\slg)(V_a,V_b) = \slg_{b c}\pa_1(\pi_*V_a)^c + \slg_{a c}\pa_1(\pi_*V_b)^c + o(r^{-2})\quad\text{at}\ z_0.
  \end{equation}
  Let us first calculate the contribution to this coming from the term $\pa_a$ in $V_a$. By~\eqref{EqLBSProjExpl} and recalling the form of the map $\Phi$ from~\eqref{EqLRCongrPhi} as well as its defining relation~\eqref{EqLRCongrPhiIt}, we have
  \begin{equation}
  \label{EqLBSAreaRderPushfwd}
  \begin{split}
    \pa_1(\pi_*\pa_a)^b &= \pa_1\pa_a\wt\Phi^b(x^1+\wt u; x^0,x^2,x^3) \\
      &= -\pa_1\pa_a\wt\gamma^b(x^1+\wt u;x^0-4 m\log x^0+\wt\Phi^0,x^2+\wt\Phi^2,x^3+\wt\Phi^3);
  \end{split}
  \end{equation}
  now $\wt\gamma^b$, its $x^c$-derivatives ($c=2,3$), and $\wt\Phi^b$ are of size $\cO((x^0)^{-1-b'_I+0})$, so dropping $\wt\Phi^2$ and $\wt\Phi^3$ gives an $o(r^{-2})$ error; likewise, $\pa_{x^0}\wt\gamma^b=\cO(r^{-2-b'_I+0})$, so replacing the second argument by $x^0$ gives another $o(r^{-2})$ error.
  
  To analyze this further, we need to digress: consider the 1-parameter family $w(s;\eps):=\gamma_{(x^1+\eps,x^2,x^3)}(s)$ of null-geodesics, with $x^2,x^3$ fixed, and let
  \[
    Y(s):=\pa_\eps w(s;0)\equiv\pa_1\gamma_{(x^1,x^2,x^3)}(s)
  \]
  denote the Jacobi field along $\gamma(s):=w(s;0)$. The asymptotics proved in Proposition~\ref{PropLR} give the a priori information
  \begin{equation}
  \label{EqLBSAreaRderYapriori}
  \begin{split}
    Y(s) &= \cO(s^{-b_I+0})\pa_0 + (1+\cO(s^{-b'_I+0}))\pa_1 + \sum_c \cO(s^{-1-b'_I+0})\pa_c, \\
    \pa_s Y(s) &= \cO(s^{-1-b_I+0})\pa_0 + \cO(s^{-1-b'_I+0})\pa_1 + \sum_c \cO(s^{-2-b'_I+0})\pa_c.
  \end{split}
  \end{equation}
  We shall determine the component $Y(s)^b$ by solving the Jacobi equation
  \begin{equation}
  \label{EqLBSAreaRderJacobi}
    \bigl(\nabla_{\dot\gamma}\nabla_{\dot\gamma}Y(s) + R(Y,\dot\gamma)\dot\gamma\bigr)^b = 0.
  \end{equation}
  Heuristically, it suffices to calculate this modulo $o(s^{-4})$ errors, as the second integral of such error terms (integrating from infinity) is $o(s^{-2})$; we will verify this heuristic in the course of our calculations. Using $\dot\gamma^0=1+\cO(s^{-1})$, $\dot\gamma^1=\cO(s^{-1-b'_I+0})$, $\dot\gamma^c=\cO(s^{-2-b'_I+0})$, the a priori information~\eqref{EqLBSAreaRderYapriori}, and the expressions for the curvature tensor in~\eqref{EqCoPertRiem}, one finds
  \[
    (R(Y,\dot\gamma)\dot\gamma)^b = R^b{}_{\lambda\mu\nu}\dot\gamma^\lambda Y^\mu\dot\gamma^\nu = -R^b{}_{0 0 1}(\dot\gamma^0)^2 Y^1 - R^b{}_{0 0 a}(\dot\gamma^0)^2 Y^a + o(s^{-4}).
  \]
  Now, using the gauge condition $\Ups(g)_0=0$ and the expressions for Christoffel symbols given in~\eqref{EqCoPertGamma}, one finds that in fact $R^b{}_{0 0 a}=o(s^{-3})$, rendering the second term size $o(s^{-4})$. Let us calculate $R^b{}_{0 0 1}=\pa_0\Gamma^b_{0 1}-\pa_1\Gamma_{0 0}^b+\Gamma_{0 1}^\mu\Gamma_{0 \mu}^b-\Gamma_{0 0}^\mu\Gamma_{\mu 1}^b$ more accurately than in~\eqref{EqCoPertRiem}. In the third term, the only contribution which is not $o(r^{-4})$ comes from $\mu=2,3$, giving $-\tfrac14 r^{-3}\pa_1 h_0{}^{\bar b}+\tfrac14 r^{-4}\slnabla^b h_{0 1}$; the fourth term is $o(r^{-4})$. For the second term, we use
  \[
    \Gamma^b_{0 0} = g^{0 b}\Gamma_{0 0 0}+g^{1 b}\Gamma_{1 0 0}+g^{a b}\Gamma_{a 0 0} = o(s^{-4}) + o(s^{-4}) - (r^{-2}\pa_0 h_0{}^{\bar b}-\half r^{-3}\slnabla^b h_{0 0}),
  \]
  exploiting $\Ups(g)_0=0$. In view of the leading order vanishing of $h_0{}^{\bar b}$ and $h_{0 0}$ at $\scri^+$, we have $\pa_1\Gamma^b_{0 0}=-r^{-2}\pa_0(\pa_1 h_0{}^{\bar b})+\half r^{-3}\slnabla^b\pa_1 h_{0 0}+o(s^{-4})$; now $\pa_1 h_0{}^{\bar b}$ can be rewritten, using $\Ups(g)_b=0$, in terms of $h_{0 1}$, $h_{\bar b\bar c}$, and $h_{1\bar b}$; since these have (size $1$) leading terms at $\scri^+$, subsequent differentiation along $\pa_0$ only produces nontrivial terms (i.e.\ not of size $o(r^{-4})$) when acting on the $r$-weights. On the other hand, $\pa_1 h_{0 0}=-r^{-1}h_{0 1}+o(r^{-1})$ from $\Ups(g)_0=0$. Arguing similarly for the computation of $\pa_0\Gamma^b_{0 1}$, one ultimately finds that all nontrivial terms cancel, so
  \[
    R^b{}_{0 0 1} = o(r^{-4}).
  \]
  Thus, the curvature term of the Jacobi equation~\eqref{EqLBSAreaRderJacobi} is of size $o(s^{-4})$ simply. Regarding the first term of~\eqref{EqLBSAreaRderJacobi}, the information~\eqref{EqLBSAreaRderYapriori} and a brief calculation give $(\nabla_{\dot\gamma}Y)^0=\cO(s^{-1-b_I+0})$, $(\nabla_{\dot\gamma}Y)^1 = \cO(s^{-1-b'_I+0})$, and, using $r^{-1}=2 s^{-1}+\cO(s^{-2}\log s)$,
  \begin{align*}
    (\nabla_{\dot\gamma}Y)^b &= \pa_s Y^b + \Gamma^b_{\mu\lambda}\dot\gamma^\mu Y^\lambda \\
      &= \pa_s Y^b + s^{-1}Y^b - 2 s^{-3}\slnabla_d h^{\bar b\bar d} + 4 s^{-3}h_1{}^{\bar b} + o(s^{-3}),
  \end{align*}
  with nontrivial contributions only from $(\mu,\lambda)=(0,1)$, $(0,c)$. In particular, $\nabla_{\dot\gamma}Y$ satisfies the same rough asymptotics as $\pa_s Y$ in~\eqref{EqLBSAreaRderYapriori}. Since differentiation of $h^{\bar b\bar d}$ and $h_1{}^{\bar b}$ along $\dot\gamma$ gains a weight $s^{1+b_I}$ due to these components having a leading term, this and~\eqref{EqLBSAreaRderJacobi} imply
  \begin{align*}
    o(s^{-4}) = (\nabla_{\dot\gamma}\nabla_{\dot\gamma}Y)^b &= \pa_s(\nabla_{\dot\gamma}Y)^b + s^{-1}(\nabla_{\dot\gamma}Y)^b + o(s^{-4}) \\
      &= \pa_s^2 Y^b + 2 s^{-1}\pa_s Y^b + 4 s^{-4}\slnabla_d h^{\bar b\bar d}-8 s^{-4}h_1{}^{\bar b} + o(s^{-4}) \\
      &= s^{-2}\bigl(\pa_s(s^2 \pa_s Y^b) - 2\wt\mu s^{-2} + o(s^{-2})\bigr),
  \end{align*}
  where
  \[
    \wt\mu=\lim_{s\to\infty}(4 h_1{}^{\bar b}-2\slnabla_d h^{\bar b\bar d})
  \]
  is the value of this combination of metric coefficients at $\gamma(\infty)\in\scri^+$. Since $\lim_{s\to\infty}s^2\pa_s Y^b=0$ due to~\eqref{EqLBSAreaRderYapriori}, we find $\pa_s Y^b=-2\wt\mu s^{-3}+o(s^{-3})$ and thus
  \begin{equation}
  \label{EqLBSAreaRderDeviation}
    Y^b=\wt\mu s^{-2}+o(s^{-2})
  \end{equation}
  since $\lim_{s\to\infty} Y^b=0$.

  Returning to the expression~\eqref{EqLBSAreaRderPushfwd}, dropping $\wt u$ gives an $\cO(r^{-2-b'_I+0})$ error term by Lemma~\ref{LemmaLRCongr}; we thus conclude that
  \begin{equation}
  \label{EqLBSAreaRderPushfwdRes}
    \pa_1(\pi_*\pa_a)^b = -\pa_1\pa_a\gamma^b(x^1;x_0,x^2,x^3) + o(r^{-2}) = \bigl(-\slnabla_a h_1{}^{\bar b}+\half\slnabla_a\slnabla_d h^{\bar b\bar d}\bigr)r^{-2} + o(r^{-2}).
  \end{equation}
  We have another term in~\eqref{EqLBSAreaRderTrslg} coming from the term $f_a\pa_1$ in $V_a$; but $f_a$ and its derivative along $x^1$ being of size $\cO(r^{-b'_I+0})$ (see the proof of Lemma~\ref{LemmaLBSAreaR}), it suffices to show that $(\pi_*\pa_1)^c=\cO(r^{-2})$ in order to conclude that $\pa_1(\pi_*(f_a\pa_1))^c=o(r^{-2})$ is a lower order term. But we can simplify $(\pi_*\pa_1)^c|_{(x^0,x^1,x^2,x^3)}=\pa_1\Phi^c=-\pa_1\wt\gamma^c(x^1;x^0,x^2,x^3)+o(r^{-2})=\cO(r^{-2})$ (using~\eqref{EqLBSAreaRderDeviation}) in the same manner as we simplified~\eqref{EqLBSAreaRderPushfwd}.
  
  Finally then, plugging~\eqref{EqLBSAreaRderPushfwdRes} into~\eqref{EqLBSAreaRderTrslg}, and adding the result to~\eqref{EqLBSAreaRderTrg} yields~\eqref{EqLBSAreaRderPaDet} for
  \[
    \mu = 2(h_{1 1}-2 h_{0 1}) + 4 r\pa_0 h_{1 1} - \slnabla_a\slnabla_b h^{\bar a\bar b},
  \]
  which by~\eqref{EqLBSAreaRderRes} proves the lemma.
\end{proof}

We can also compute $\pa_1\mathring x^b=\pa_1\pi^b$ modulo $o(r^{-2})$, as this is given by the component $Y^b$ of the Jacobi vector field of the proof of Lemma~\ref{LemmaLBSAreaRder}, so $\pa_1\mathring x^b=(h_1{}^{\bar b}-\half\slnabla_d h^{\bar b\bar d})r^{-2}+o(r^{-2})$. In summary, we have shown that
\begin{equation}
\label{EqLBSDiff}
\begin{split}
  d u &= o(r^{-1})d x^0 + \bigl(1+2 r^{-1}h_{0 1}+o(r^{-1})\bigr)d x^1 + \sum\nolimits_c o(1)d x^c, \\
  d\mathring r &= \bigl(\half-m r^{-1}+o(r^{-1})\bigr)d x^0 \\
    &\qquad + \bigl(-\half+(m+\half(h_{1 1}-2 h_{0 1})+r\pa_0 h_{1 1}-\tfrac14 \slnabla_a\slnabla_b h^{\bar a\bar b})r^{-1}+o(r^{-1})\bigr)d x^1 \\
    &\qquad + \sum\nolimits_c o(1)d x^c, \\
  d\mathring x^a &= o(r^{-2})d x^0 + \bigl((h_1{}^{\bar a}-\half\slnabla_d h^{\bar a\bar d})r^{-2}+o(r^{-2})\bigr)d x^1 + d x^a + \sum\nolimits_c o(r^{-1})d x^c,
\end{split}
\end{equation}
where the remainders are in fact more precise: $o(r^{-k})$ can be replaced by $\rho_I^{k+b_I-0}\Hb^\infty$ near $(\scri^+)^\circ$, so a fortiori by $\cO(r^{-k-b_I+0})$. We can now supplement~\eqref{EqLBSg1}--\eqref{EqLBSg2} by
\begin{equation}
\label{EqLBSg3}
\begin{split}
  G(d\mathring r,d\mathring r) &= -1 + 2 m r^{-1} + 2\pa_0 h_{1 1} - \half r^{-1}\slnabla_a\slnabla_b h^{\bar a\bar b} + \rho_I^{1+b_I-0}\Hb^\infty, \\
  G(d\mathring r,d\mathring x^b) &= (h_1{}^{\bar b}-\half\slnabla_d h^{\bar b\bar d})r^{-2} + \rho_I^{2+b'_I-0}\Hb^\infty.
\end{split}
\end{equation}
(Note that in the first line, the logarithmically divergent terms $h_{1 1}$ from $g^{0 0}(\pa_0\mathring r)^2$ and $g^{1 1}(\pa_1\mathring r)^2$ cancel.) Let us summarize the calculations~\eqref{EqLBSg1}--\eqref{EqLBSg2} and \eqref{EqLBSg3}:
\begin{prop}
\label{PropLBS}
  In the Bondi--Sachs coordinates~\eqref{EqLBSCoord}, the dual metric $G=g^{-1}$ is
  \begin{align*}
    G &= 2(1+o(\mathring r^{-1}))\pa_u\pa_{\mathring r} - \bigl(1-2 m\mathring r^{-1}-2\pa_0 h_{1 1}+\half\mathring r^{-1}\slnabla_a\slnabla_b h^{\bar a\bar b}+o(\mathring r^{-1})\bigr)\pa_{\mathring r}^2 \\
      &\quad - \mathring r^{-2}(\slg^{a b}+\mathring r^{-1}h^{\bar a\bar b}+o(\mathring r^{-1}))\bigl(\pa_{\mathring x^a}+(U_a\mathring r^{-2}+o(\mathring r^{-2}))\pa_{\mathring r}\bigr)\bigl(\pa_{\mathring x^b}+(U_b\mathring r^{-2}+o(\mathring r^{-2}))\pa_{\mathring r}\bigr),
  \end{align*}
  where $U_a=-\half h_{1\bar a}+\tfrac14 \slnabla^c h_{\bar a\bar c}$. The metric $g$ itself takes the form
  \begin{align*}
    g &= \bigl(1-2 m\mathring r^{-1}-2\pa_0 h_{1 1}+\half\mathring r^{-1}\slnabla_a\slnabla_b h^{\bar a\bar b}+o(\mathring r^{-1})\bigr)d u^2 + 2(1+o(\mathring r^{-1}))d u\,d\mathring r \\
      &\quad - \mathring r^2(\slg_{a b}-\mathring r^{-1}h_{\bar a\bar b}+o(\mathring r^{-1}))\bigl(d\mathring x^a-(U^a\mathring r^{-2}+o(\mathring r^{-2}))d u\bigr)\bigl(d\mathring x^b-(U^b\mathring r^{-2}+o(\mathring r^{-2}))d u\bigr).
  \end{align*}
  The $o(\mathring r^{-k})$ remainders can be replaced by $\rho_I^{k+b_I-0}\Hb^\infty=\cO(r^{-k-b_I+0})$ near $(\scri^+)^\circ$. Furthermore, the coordinate vector fields satisfy
  \begin{equation}
  \label{EqLBSVF}
  \begin{split}
    \pa_u &= \bigl(1-(h_{1 1}+2 r\pa_0 h_{1 1}-\half\slnabla_a\slnabla_b h^{\bar a\bar b})r^{-1}+o(r^{-1})\bigr)\pa_0 \\
      &\qquad + (1-2 h_{0 1}r^{-1}+o(r^{-1}))\pa_1 + \bigl((-h_1{}^{\bar a}+\half\slnabla_b h^{\bar a\bar b})r^{-2}+o(r^{-2})\bigr)\pa_a, \\
    \pa_{\mathring r} &= (2+4 m r^{-1}+o(r^{-1}))\pa_0 + o(r^{-1})\pa_1 + \sum\nolimits_c o(r^{-2})\pa_c, \\
    \pa_{\mathring x^a} &= o(1)\pa_0 + o(1)\pa_1 + \pa_a + \sum\nolimits_c o(r^{-1})\pa_c.
  \end{split}
  \end{equation}
\end{prop}
\begin{proof}
  The statement~\eqref{EqLBSVF} on the dual basis of~\eqref{EqLBSDiff} follows by matrix inversion.
\end{proof}

\begin{rmk}
  For comparison, the Bondi--Sachs coordinates on Schwarzschild are simply $u=x^1$, $\mathring r=r$, and spherical coordinates $\mathring x^a=x^a$, and the metric takes the form
  \[
    (g_m^S)^{-1} = 2\pa_u\pa_{\mathring r} - (1-2 m\mathring r^{-1})\pa_{\mathring r}^2 - \mathring r^{-2}\slG,\quad
    g_m^S = (1-2 m\mathring r^{-1})d u^2+2 d u\,d\mathring r-\mathring r^2\slg.
  \]
\end{rmk}

\begin{rmk}
\label{RmkLBSReg}
  Near $(\scri^+)^\circ$ and relative to the smooth structure on $M$, the conformally rescaled metric $r^{-2}g$ is singular as an incomplete metric at $\scri^+$: indeed, $r^2\pa_0$ is a nonzero multiple of $\pa_{\rho_I}$ by~\eqref{EqCptASplNullExpl}, and $r^2 g(r^2\pa_0,r^2\pa_0)=r h_{0 0}=\cO(\rho_I^{-1+b'_I})$. On the other hand, changing the smooth structure of $M$ near $(\scri^+)^\circ$ by declaring $(\mathring r^{-1},u,\mathring x^2,\mathring x^3)$ to be a smooth coordinate system, so $\mathring\rho_I:=\mathring r^{-1}$ is a defining function of $\scri^+$, we have $\mathring r^{-2} g\in\cC^{1,b_I-0}$. Indeed, $\pa_{\mathring\rho_I}=-\mathring r^2\pa_{\mathring r}$ is null, while $(\mathring r^{-2}g)(\pa_{\mathring\rho_I},\pa_u)=1+\cO(\mathring\rho_I^{1+b_I-0})$ is $\cC^{1,b_I-0}$, and the remaining metric coefficients have at least this amount of regularity. Since by Theorem~\ref{ThmPf} one can take $b_I$ arbitrarily close to $\min(b_0,1)$, this gives
  \begin{equation}
  \label{EqLBSReg}
    \mathring r^{-2}g \in \cC^{1,\alpha}\quad \forall\,\alpha<\min(b_0,1),
  \end{equation}
  relative to the new smooth structure. As mentioned in \S\ref{SsIBondi}, smoothness properties of conformal compactifications have been widely discussed, in particular from the point of view of asymptotic simplicity \cite{PenroseAsymptotics} and the decay properties of the curvature tensor \cite{KlainermanNicoloPeeling,ChristodoulouNoPeeling}; see also~\cite{FriedrichSmoothScriReview} for further references. Whether or not there exists a compactification with \emph{smooth} (or at least highly regular) $\scri^+$, meaning that the conformally rescaled metric extends smoothly and nondegenerately across $\scri^+$, is a delicate issue as it depends very sensitively on the precise choice of the conformal factor and the smooth structure near $\scri^+$ and requires the identification of at least two `incommensurable' geometric quantities.\footnote{An example would be given by two metric components which have nonzero leading terms of size $\rho_I$ and $\rho_I\log\rho_I$, respectively, though we reiterate that this depends on the choice of $\rho_I$, i.e.\ of the smooth structure.} The observation~\eqref{EqLBSReg} shows that this cannot happen prior to the next-to-leading order terms in the expansion of $g$ at $\scri^+$. Work by Christodoulou~\cite{ChristodoulouNoPeeling} on the other hand (see also \cite[\S1.5.3]{DafermosChristodoulouExpose}) strongly suggests that the conformal compactification is generically at most of class $\cC^{1,\alpha}$.
\end{rmk}

Therefore, the \emph{mass aspect}, see~\cite[Equation~(37)]{MaedlerWinicourBondi}, is $-\half$ times the $\mathring r^{-1}$ coefficient of the $d u^2$ component,
\begin{equation}
\label{EqLBSMassAsp}
  M_A(p) = m + (r\pa_0 h_{1 1} - \tfrac14\slnabla_a\slnabla_b h^{\bar a\bar b})|_p,\ \ p\in(\scri^+)^\circ,
\end{equation}
and the \emph{Bondi mass} $M_{\rm B}(u):=\frac{1}{4\pi}\int_{S(u)}M_A\,d\slg$ is
\begin{equation}
\label{EqLBSMass}
  M_{\rm B}(u) = m + \frac{1}{4\pi}\int_{S(u)} r\pa_0 h_{1 1}\,d\slg,\ \ u\in\R,
\end{equation}
where we exploited that the divergence in the expression~\eqref{EqLBSMassAsp} integrates to zero.

\begin{rmk}
\label{RmkLBSLog}
  Recall that near $(\scri^+)^\circ$, $h_{1 1}$ can be written as $h_{1 1}^{(1)}\log\rho_I+h_{1 1}^{(0)}+\rho_I^{b_I}\Hb^\infty$, with $h_{1 1}^{(j)}\in\CI((\scri^+)^\circ)$, $j=0,1$, so $r\pa_0 h_{1 1}|_{\scri^+}=-\half h_{1 1}^{(1)}$ picks out the logarithmic term.
\end{rmk}

\begin{thm}
\label{ThmLBSLoss}
  The Bondi mass~\eqref{EqLBSMass} satisfies the \emph{mass loss formula}
  \begin{equation}
  \label{EqLBSLoss}
    \frac{d}{d u}M_{\rm B}(u) = -\frac{1}{32\pi}\int_{S(u)} |N|_{\slg}^2 \,d\slg,\quad N_{a b}:=\pa_u h_{\bar a\bar b}|_{\scri^+}.
  \end{equation}
  Moreover, $M_{\rm B}(-\infty)=m$ is the ADM mass of the initial data, while $M_{\rm B}(+\infty)=0$.
\end{thm}
\begin{proof}
  The formula~\eqref{EqLBSLoss} is an immediate consequence of Lemma~\ref{LemmaEinP}, and $M_{\rm B}(-\infty)=m$ follows from the fact that $r\pa_0 h_{1 1}\in\rho_0^{b_0}\rho_+^{b_+}\Hb^\infty(\scri^+)$ decays to $0$ as $\rho_0\to 0$.
  
  Let us fix the boundary defining function $\rho$ to be equal to $r^{-1}$ near $\scri^+$, and fix $\rho_I$ and $\rho_+$ near $I^+$ so that $\rho_I\rho_+=\rho$. In order to prove $M_{\rm B}(+\infty)=0$, we analyze the equation satisfied by $h^+:=h|_{I^+}$. The existence of this leading term was proved in~\S\ref{SPhg} starting with equation~\eqref{EqPhgLocip} (in which we do not use constraint damping); that is, restricting that equation to $I^+$ and using the Mellin-transformed normal operators $\wh{L_0}(0)=\wh{\ul L}(0)\in\rho_I^{-1}\Diffb^2(I^+)$ at frequency $0$ (so this is the action of $L_0$ on 2-tensors smooth down to $I^+$ followed by restriction to $I^+$), we have
  \begin{equation}
  \label{EqLBSLossEq}
    \wh{\ul L}(0) h^+ = -P(0)|_{I^+} = -\rho^{-3}\Ric(g_m)|_{I^+}.
  \end{equation}
  Moreover, $h^+_{1 1}$ has a logarithmic leading order term $h^+_\ell\log\rho_I$,
  \begin{equation}
  \label{EqLBSLossLog}
    h^+-h^+_\ell\log\rho_I\,(d x^1)^2 \in\CI(I^+)+\rho_I^{b_I}\Hb^\infty(I^+)\subset\Hext^{1/2+b_I-0}(I^+),
  \end{equation}
  where $h^+_\ell=(\rho_I\pa_{\rho_I}h_{1 1})|_{\pa I^+}=(-2 r\pa_0 h_{1 1})_{\pa I^+}$, so by Lemma~\ref{LemmaEinP}
  \[
    h_\ell^+(\theta) = \frac14 \int_{\beta^{-1}(\theta)} |N|^2\,d x^1,\ \ \theta\in\pa I^+.
  \]
  Since $\wh{\ul L}(0)$ is injective on $\Hext^{1/2+0}(I^+)$, the tensor $h^+$ on $I^+$ is uniquely determined by equation~\eqref{EqLBSLossEq} and the `boundary condition'~\eqref{EqLBSLossLog}. The strategy is to evaluate $h^+_{0 0}|_{\pa I^+}$ in two ways: one the one hand, this quantity vanishes identically by construction of the metric $h$ in our DeTurck gauge; on the other hand, we will show that solving~\eqref{EqLBSLossEq} directly yields the relation
  \begin{equation}
  \label{EqLBSLossAim}
    \frac{1}{4\pi}\int_{\pa I^+} h^+_{0 0}|_{\pa I^+}\,d\slg = \half m - \tfrac14 c, \quad c:=\frac{1}{4\pi}\int_{\pa I^+} h^+_\ell\,d\slg,
  \end{equation}
  which thus gives the desired conclusion. For the proof of~\eqref{EqLBSLossAim}, let us split $h^+=h'+h''$, where
  \begin{equation}
  \label{EqLBSLossMinkBC}
    h'\in\CI(I^+;S^2\,\Tsc^*_{I^+}\ol{\R^4}),\ \ 
    h''\in h^+_\ell\log\rho_I\,(d x^1)^2 + \Hext^{1/2+0}(I^+;S^2\,\Tsc^*_{I^+}\ol{\R^4})
  \end{equation}
  are the unique solutions with these properties solving the equations
  \begin{align}
  \label{EqLBSLossSchw}
    \wh{\ul L}(0)h' &= -P(0)|_{I^+}, \\
  \label{EqLBSLossMink}
    \wh{\ul L}(0)h'' &= 0;
  \end{align}
  the first equation is uniquely solvable in this regularity class due to $P(0)\in\CIc((I^+)^\circ)$. We first solve~\eqref{EqLBSLossMink} with the boundary condition~\eqref{EqLBSLossMinkBC}, to the extent that we can determine $h''_{0 0}$. This can be viewed as a calculation of (a part of) the `scattering matrix' of the operator $\wh{\ul L}(0)$ on $I^+$,\footnote{Trivializing the 2-tensor bundle using coordinate differentials on $\R^4$, a conjugated version of $\wh{\ul L}(0)$ acts component-wise as the Laplacian of exact hyperbolic space with spectral parameter at the bottom of the spectrum; see Equations~(4.1), (6.11), and~(6.13) in~\cite{HintzZworskiHypObs}.} which can be done explicitly: writing points in $I^+$ using spherical coordinates as $Z=R\omega\in\R^3$, $R=r/t\in[0,1]$, $\omega\in\Sph^2$, we have
  \[
    -2\wh{\ul L}(0) = R^{-2}D_R R^2(1-R^2)D_R + R^{-2}\slDelta_\omega + 2,
  \]
  acting component-wise on the coordinate trivialization of $\Tsc^*\ol{\R^4}$; see~\eqref{EqBgExplDS} and \eqref{EqBgExplStaticdS}. Since $\wh{\ul L}(0)$ is $SO(3)$-invariant, it suffices to calculate $u_{0 0}|_{\pa I^+}$ for the solution of $\wh{\ul L}(0)u=0$ for which $u-c\log\rho_I(d x^1)^2\in\Hext^{1/2+0}(I^+)$; recall that $c$ was defined in~\eqref{EqLBSLossAim}. Now at $I^+$,
  \begin{equation}
  \label{EqLBSLossdx1sq}
    (d x^1)^2 = d t^2 - 2\tfrac{Z^i}{|Z|}\,d t\,d x_i + \tfrac{Z^i Z^j}{|Z|^2}\,d x_i\,d x_j,
  \end{equation}
  where we write $x_i$ for the Euclidean coordinates on $\R^3$; observe then that if $Y_\ell\in\CI(\Sph^2)$, $\slDelta Y_\ell=\ell(\ell+1)Y_\ell$, denotes a spherical harmonic, then $\wh{\ul L}(0)\bigl(u_\ell(R) Y_\ell(\omega)\bigr) = 0$ holds for
  \begin{equation}
  \label{EqLBSLossODE}
    u_0=R^{-1}\log\bigl(\tfrac{1-R}{1+R}\bigr),\ \ 
    u_1=R^{-2}\log\bigl(\tfrac{1-R}{1+R}\bigr)+2 R^{-1},\ \ 
    u_2=\tfrac{3-R^2}{2 R^3}\log\bigl(\tfrac{1-R}{1+R}\bigr)+3 R^{-2};
  \end{equation}
  Taylor expanding at $R=0$, one sees that $R^{-\ell}u_\ell$ is a smooth function of $R^2$, hence $u_\ell Y_\ell$ is smooth there; moreover, $u_\ell$ satisfies the boundary condition $u_\ell-\log\rho_I=\cO(1)$, $\rho_I=1-R$, at $R=1$. (In fact, $u_\ell$ is the unique solution with these two properties.) Using~\eqref{EqLBSLossdx1sq}, we find $h''=c\cdot\bigl(u_0\,d t^2-2 u_1\,d t\,d r+u_2\,d r^2\bigr)$, so writing $d t=(d x^0+d x^1)/2$, $d r=\frac{Z^i}{|Z|}d x_i$, and $r=(d x^0-d x^1)/2$ near $\pa I^+$ within $I^+$, this gives
  \begin{equation}
  \label{EqLBSLossHdbl}
    h''_{0 0}|_{\pa I^+} = c\cdot\bigl(\tfrac14 u_0 - \tfrac12 u_1 + \tfrac14 u_2\bigr)\big|_{R=1} = -\tfrac14 c.
  \end{equation}

  In order to solve~\eqref{EqLBSLossSchw}, note first that the map $\ul h\in\CI(I^+)\mapsto\rho^{-3}\Ric(\ul g+\rho\ul h)|_{I^+}$ is \emph{linear} in $\ul h$,\footnote{This reflects the fact that the normal operator of the linearization of the Einstein equation around a metric of the form $\ul g+\rho\ul h$ only depends on the leading order part of the metric at $I^+$, i.e.\ on $\ul g$; see also Lemma~\ref{LemmaEinNi0p}.} hence writing $g_m=:\ul g+\rho\ul h$, we have
  \[
    P(0)|_{I^+}=\rho^{-3}(\Ric(\ul g+\rho\ul h)-\Ric(\ul g))|_{I^+} = \wh{\ul L}(0)\ul h - \rho^{-3}\delta^*_{\ul g}\delta_{\ul g}G_{\ul g}\rho\ul h;
  \]
  for later use, we note that in a neighborhood of $\pa I^+$ in $I^+$,
  \begin{equation}
  \label{EqLBSLossulh}
    \ul h=-2 m \rho^{-1}r^{-1}(d t^2+d r^2)=-m((d x^0)^2+(d x^1)^2).
  \end{equation}
  This suggests writing $\rho h'$ as the sum of $-\rho\ul h$ (to solve away the first term) and a pure gauge term, so we make the ansatz
  \begin{equation}
  \label{EqLBSLossHprimeAnsatz}
    h'=-\ul h+\rho^{-1}\delta^*_{\ul g}\omega+\wt h',
  \end{equation}
  where $\omega\in\CI((I^+)^\circ;\Tsc^*_{I^+}\ol{\R^4})$ solves\footnote{We abuse notation by using the same expression for a b-operator on $\ol{\R^4}$ and its Mellin-transformed normal operator at $0$ frequency. Note that for a b-differential operator $A$, the operator $\wh{A}(0)$ is independent of the choice of boundary defining function (unlike $\wh{A}(\sigma)$ for $\sigma\neq 0$); see also \cite[p.~762]{VasyPropagationCorners}.}
  \begin{equation}
  \label{EqLBSLossOmega}
    \rho^{-2}\delta_{\ul g}G_{\ul g}\delta_{\ul g}^*\omega = \vartheta := \rho^{-2}\delta_{\ul g}G_{\ul g}(\rho\ul h) \in \CI(I^+;\Tsc_{I^+}^*\ol{\R^4}),
  \end{equation}
  and $\wt h'$ is a solution of $\wh{\ul L}(0)\wt h'=0$ which we will use to solve away any singular terms. We compute $\vartheta$ to leading order at $\pa I^+$ by using $r^2\delta_{\ul g}r^{-1} d t^2=0$ and $r^2\delta_{\ul g}r^{-1} d r^2=d r$, so
  \[
    \vartheta = -2 m\,d r = -m\,d x^0+m\,d x^1 + \rho_I\,\CI(I^+).
  \]
  Write $\rho^{-2}\delta_{\ul g}G_{\ul g}\delta_{\ul g}^*=\rho\ul D\rho^{-1}$, where $\ul D=\rho^{-3}\delta_{\ul g}G_{\ul g}\delta_{\ul g}^*\rho$ is $\half$ times the wave operator on 1-forms on Minkowski space, re-weighted to a b-operator as usual; then equation~\eqref{EqLBSLossOmega} becomes
  \begin{equation}
  \label{EqLBSLossOmegaEq}
    \wh{\ul D}(i)(\rho_I^{-1}\omega) = \rho_I^{-1}\vartheta.
  \end{equation}
  Now $\rho_I^{-1}\vartheta\in\Hext^{-1/2-0,\infty}(I^+)$, while $\wh{\ul D}(i)^{-1}\colon\Hext^{s-1,\infty}(I^+)\to\Hext^{s,\infty}(I^+)$ for $s>-\half$, cf.\ \eqref{EqPhgNormalOp}. Therefore, the solution satisfies $\omega\in\rho_I\Hext^{1/2-0,\infty}(I^+)\subset\rho_I^{1-0}\Hb^\infty(I^+)$ (by Sobolev embedding for functions of the single variable $\rho_I$), which using the expression~\eqref{EqCoSchwDelta} implies that $\omega$ does not contribute to $h'_{0 0}|_{\pa I^+}$, namely $(\rho^{-1}\delta_{\ul g}^*\omega)_{0 0}|_{\pa I^+} = (\rho^{-1}\pa_0\omega_0)|_{\pa I^+}=0$, where we used that $\rho^{-1}\pa_0$ is a multiple of the b-vector field $\rho_I\pa_{\rho_I}$ at $\pa I^+$.

  A careful inspection of the solution of~\eqref{EqLBSLossOmegaEq} shows that $\rho^{-1}\delta_{\ul g}^*\omega$ is not smooth. Indeed, in the bundle splitting~\eqref{EqCptASpl}, we have $\ul D\in-2\rho^{-2}\pa_0\pa_1+\Diffb^2({}^0\!M)$, as follows from the same calculations as~\eqref{EqEinNscriBoxForm}, so using the expression~\eqref{EqPhgNormalPhgOp} for $\sigma=i$, we have $\wh{\ul D}(i)\in\pa_{\rho_I}(\rho_I\pa_{\rho_I}+1)+\Diffb^2(I^+)$, which implies that\footnote{Using the arguments employed in the proof of Lemma~\ref{LemmaPhgNormalPhg}, we in fact have $\rho_I^{-1}\omega\in\log\rho_I\,\CI(I^+)+\CI(I^+)$, as follows by constructing a formal solution at $\rho_I=0$, starting with the stated leading order term, and solving away the remaining smooth error using $\wh{\ul D}(i)^{-1}$.} $\omega=\rho_I\log\rho_I\,\vartheta|_{\pa I^+}+\Hext^{5/2-0,\infty}(I^+)$; therefore
  \[
    (\rho^{-1}\delta_{\ul g}^*\omega)|_{I^+} = \bigl(-d x^0\,d x^1+(d x^1)^2\bigr)m \log\rho_I + \Hext^{3/2-0,\infty}(I^+).
  \]
  Therefore, while we do have $\wh{\ul L}(0)(-\ul h+\rho^{-1}\delta_{\ul g}^*\omega)=-P(0)$, we need to correct the 2-tensor on the left by adding the unique solution $\wt h'$ of
  \[
    \wh{\ul L}(0)\wt h'=0,\ \ \wt h' \in \bigl(d x^0\,d x^1-(d x^1)^2\bigr)m\log\rho_I + \Hext^{1/2+0}(I^+)
  \]
  in order for $h'$ in~\eqref{EqLBSLossHprimeAnsatz} to have regularity \emph{above} $\Hext^{1/2}(I^+)$, which, as remarked before, implies that it is the unique \emph{smooth} solution of \eqref{EqLBSLossSchw}, as desired. Arguing similarly as around~\eqref{EqLBSLossdx1sq}--\eqref{EqLBSLossODE} and noting that $d x^0\,d x^1=d t^2-d r^2=d t^2-\frac{Z^i Z^j}{|Z|^2}d x_i\,d x_j$, the solution is given by $\wt h'=m(u_0\,d t^2-u_2\,d r^2)-m(u_0\,d t^2-2 u_1\,d t\,d r+u_2\,d r^2)$. This gives
  \[
    \wt h'_{0 0}|_{\pa I^+} = \tfrac14 m (u_0-u_2)|_{\pa I^+} - m\cdot(-\tfrac14) = -\half m.
  \]
  In view of~\eqref{EqLBSLossulh}, we conclude that
  \[
    h'_{0 0}|_{\pa I^+} = -\ul h_{0 0}|_{\pa I^+} + \wt h'_{0 0}|_{\pa I^+} = \half m.
  \]
  Adding this to~\eqref{EqLBSLossHdbl} establishes the relation~\eqref{EqLBSLossAim}, and proves $M_{\rm B}(+\infty)=0$.
\end{proof}

\begin{rmk}
  The construction of Bondi--Sachs coordinates is local near $(\scri^+)^\circ$ and as such did not rely on $h$ being small. (The proof of Proposition~\ref{PropLR} used the smallness of certain Christoffel symbols in a weighted $\cC^0$ space, but this is automatic for any fixed $h\in\cX^\infty$ if one relaxes the weights at $\scri^+$ by a little and works in a sufficiently small neighborhood of $\scri^+$.) Likewise, the proof of Theorem~\ref{ThmLBSLoss} did not require $h$ to be small. Therefore, we in fact conclude that \emph{any (large) solution} of the Einstein vacuum equation of the form $g=g_m+\rho h$ (with $m$ possibly large), $h\in\cX^\infty$---which requires it to decay to the Minkowski solution at $I^+$---satisfies the conclusions of Theorem~\ref{ThmLBSLoss}.
\end{rmk}

Let us connect this to the alternative definition of the Bondi mass and the mass loss formula used in \S\ref{SsIBondi}, which has a more geometric flavor \cite{ChristodoulouNonlinear}. To describe this, consider an outgoing null cone $C_u$ and let
\[
  S_{u,\mathring r}:=C_u\cap\{r=\mathring r\}
\]
denote the $2$-sphere of constant area radius (which is a particular choice of transversal of $C_u$). Let $L\in(T C_u)^\perp$ be a future-directed null normal vector field, i.e.\ a smooth positive multiple of $\nabla u$; then the \emph{null second fundamental form} is
\[
  \chi_L(X,Y) := g(\nabla_X L,Y),\ \ X,Y\in T S_{u,\mathring r}.
\]
Note that $\chi_{a L}=a\chi_L$ for any function $a$. There exists a unique future-directed null vector field
\begin{equation}
\label{EqLBSLbar}
  \ul L\in( T S_{u,\mathring r})^\perp\ \ \text{such that}\ \ g(L,\ul L)=2.
\end{equation}
Define $T\ul C_u:=T S_{u,\mathring r}\oplus\la\ul L\ra$, which is the tangent space (at $S_{u,\mathring r}$) of a null hypersurface $\ul C_u$ which is the congruence of null-geodesics with initial condition on $S_{u,\mathring r}$ and initial velocity $\ul L$. ($L$ and $C_u$, resp.\ $\ul L$ and $\ul C_u$, are often called `outgoing' and `ingoing,' respectively.) The \emph{conjugate null second fundamental form} is then
\[
  \ul\chi_{\ul L}(X,Y) := g(\nabla_X\ul L,Y)=-g(\nabla_X Y,\ul L),\ \ X,Y\in T S_{u,\mathring r},
\]
with the second expression showing that this depends only on $\ul L$ at $S_{u,\mathring r}$. Letting $\mathring g:=g|_{S_{u,\mathring r}}$ denote the induced metric, the trace-free parts of $\chi$ and $\ul\chi$ are
\[
  \hat\chi_L := \chi_L - \half\mathring g\tr_{\mathring g}(\chi_L), \quad
  \hat{\ul\chi}_{\ul L} := \ul\chi_{\ul L}-\half\mathring g\tr_{\mathring g}(\ul\chi_{\ul L}).
\]
Rescaling $L$ to $a L$, we must rescale $\ul L$ to $a^{-1}\ul L$, so the product $\tr\chi_L\tr\chi_{\ul L}$ is well-defined, and we may drop the subscripts on $\chi$ and $\ul\chi$. The \emph{Hawking mass} of $S_{u,\mathring r}$ is defined as
\begin{equation}
\label{EqLBSHawking}
  M_{\rm H}(u,\mathring r) := \frac{\mathring r}{2}\biggl(1+\frac{1}{16\pi}\int_{S(u,\mathring r)} \tr\chi \tr\ul\chi\,d S\biggr),
\end{equation}
where $d S$ is the induced surface measure. For a 1-form, let us write its components $\omega$ in Bondi--Sachs coordinates as $\omega_u$, $\omega_{\mathring r}$, $\omega_{\mathring a}$, $a=2,3$, similarly for higher rank tensors.

\begin{lemma}
\label{LemmaLBSHawking}
  We have $|M_{\rm H}(u,\mathring r)-M_{\rm B}(u)|\lesssim\mathring r^{-b_I+0}$, hence
  \[
    \lim_{\mathring r\to\infty}M_{\rm H}(u,\mathring r)=M_{\rm B}(u).
  \]
\end{lemma}
\begin{proof}
  We work in Bondi--Sachs coordinates, so $T S_{u,\mathring r}=\la\pa_{\mathring x^2},\pa_{\mathring x^3}\ra$, and
  \[
    \mathring g_{\mathring a\mathring b}=-\mathring r^2\slg_{a b}+\mathring r h_{\bar a\bar b}+o(\mathring r), \quad
    (\mathring g^{-1})^{\mathring a\mathring b}=-\mathring r^{-2}\slg^{a b}-\mathring r^{-3}h^{\bar a\bar b}+o(\mathring r^{-3}).
  \]
  Let us take $L=\pa_{\mathring r}$ and write $\chi\equiv\chi_L$, then $\chi_{\mathring a\mathring b}$ is the Christoffel symbol of the first kind, $\Gamma_{\mathring b\mathring a\mathring r}=g(\nabla_{\pa_{\mathring x^a}}\pa_{\mathring r},\pa_{\mathring x^b})$. By Proposition~\ref{PropLBS}, $g(\pa_{\mathring x^a},\pa_{\mathring r^a})\equiv 0$, therefore
  \begin{equation}
  \label{EqLBSHawkingChi}
    \chi_{\mathring a\mathring b} = \half\pa_{\mathring r}g_{\mathring a\mathring b} = -\mathring r \slg_{a b}+\half h_{\bar a\bar b} + o(1),
  \end{equation}
  which due to $\sltr h=o(1)$ gives
  \begin{equation}
  \label{EqLBSHawkingChiHat}
    \tr\chi = 2\mathring r^{-1} + o(\mathring r^{-2}), \quad
    \hat\chi_{\mathring a\mathring b} = -\half h_{\bar a\bar b} + o(1).
  \end{equation}
  Next, a simple calculation shows that the unique future-directed null vector field $\ul L$ defined in~\eqref{EqLBSLbar} is given by
  \begin{align*}
    \ul L &= (2+o(\mathring r^{-1}))\pa_u - \bigl(1-2 m\mathring r^{-1}-2\pa_0 h_{1 1}+\half\mathring r^{-1}\slnabla_a\slnabla_b h^{\bar a\bar b}+o(\mathring r^{-1})\bigr)\pa_{\mathring r} \\
      &\qquad + \bigl((-h_1{}^{\bar a}+\half\slnabla_b h^{\bar a\bar b})\mathring r^{-2}+o(\mathring r^{-2})\bigr)\pa_{\mathring a}.
  \end{align*}
  (The spherical component is determined by $g(\ul L,\pa_{\mathring c})=0$, $c=2,3$, the $\pa_u$ component by $g(\ul L,L)=2$, and the $\pa_{\mathring r}$ component by $g(\ul L,\ul L)=0$.) Working in normal coordinates on $\Sph^2$, using $\Gamma_{u\mathring a\mathring b} = -\half\mathring r\pa_u h_{\bar a\bar b} - \tfrac14(\slnabla_a h_{1\bar b}+\slnabla_b h_{1\bar a})+\tfrac18(\slnabla_a\slnabla_c h_{\bar b}{}^{\bar c}+\slnabla_b\slnabla_c h_{\bar a}{}^{\bar c}) + o(1)$, $\Gamma_{\mathring r\mathring a\mathring b} = \mathring r\slg_{a b}-\half h_{\bar a\bar b}+o(1)$, and $\Gamma_{\mathring c\mathring a\mathring b} = o(\mathring r^2)$, the components of $\ul\chi:=\ul\chi_{\ul L}$ are
  \begin{equation}
  \label{EqLBSHawkingChibar}
  \begin{split}
    \ul\chi_{\mathring a\mathring b} = -\Gamma_{\mu\mathring a\mathring b}\ul L^\mu &= (\mathring r-2 m-2\mathring r\pa_0 h_{1 1}+\half\slnabla_c\slnabla_d h^{\bar c\bar d})\slg_{a b} + \mathring r\pa_u h_{\bar a\bar b} \\
      &\qquad -\half h_{\bar a\bar b}+\half(\slnabla_a h_{1\bar b}+\slnabla_b h_{1\bar a})-\tfrac14(\slnabla_a\slnabla_c h_{\bar b}{}^{\bar c}+\slnabla_b\slnabla_c h_{\bar a}{}^{\bar c}) + o(1),
  \end{split}
  \end{equation}
  which gives
  \begin{equation}
  \label{EqLBSHawkingChibarHat}
  \begin{split}
    \tr\ul\chi &= -2\mathring r^{-1}+(4 m+4\mathring r\pa_0 h_{1 1}-\half\slnabla_a\slnabla_b h^{\bar a\bar b}-\slnabla_a h_1{}^{\bar a})\mathring r^{-2} + o(\mathring r^{-2}), \\
    \hat{\ul\chi}_{\mathring a\mathring b} &= \mathring r\pa_u h_{\bar a\bar b} + (\tfrac14 \slnabla_c\slnabla_d h^{\bar c\bar d}-\half\slnabla_c h_1{}^{\bar c})\slg_{a b}+\half h_{\bar a\bar b} \\
      &\qquad +\half(\slnabla_a h_{1\bar b}+\slnabla_b h_{1\bar a})-\tfrac14(\slnabla_a\slnabla_c h_{\bar b}{}^{\bar c}+\slnabla_b\slnabla_c h_{\bar a}{}^{\bar c}) + o(1).
  \end{split}
  \end{equation}
  Finally, the surface measure on $S_{u,\mathring r}$ is $|\det\mathring g|^{1/2}|d\mathring x^a\,d\mathring x^b|=(\mathring r^2\slg_{a b}+o(\mathring r))|d\mathring x^a\,d\mathring x^b|$, hence the Hawking mass is $M_{\rm H}(u,\mathring r)=m+\tfrac{1}{4\pi}\int_{S(u)}\mathring r\pa_0 h_{1 1}\,d\slg+o(1)=M_{\rm B}(u)+o(1)$. (As usual, the $o(1)$ remainder is really symbolic as $\mathring r\to 0$, namely of class $S^{-b_I+0}$.)
\end{proof}

With $L$ and $\ul L$ defined as in the proof of the lemma, consider the conjugate null vectors $a L$ and $a^{-1}\ul L$. By~\eqref{EqLBSHawkingChiHat} and \eqref{EqLBSHawkingChibarHat}, there exists a unique $a=1+\cO(\mathring r^{-1})$ such that
\begin{equation}
\label{EqLBSNormalize}
  \tr\chi_{a L}+\tr\ul\chi_{a^{-1}\ul L}=a^{-1}(a^2\tr\chi+\tr\ul\chi)=0;
\end{equation}
thus $\hat{\ul\chi}_{a^{-1}\ul L}=\mathring r\pa_u h_{\bar a\bar b}+\cO(1)=\hat{\ul\chi}+\cO(1)$, hence to leading order, the normalization~\eqref{EqLBSNormalize} does not change $\hat{\ul\chi}$. We can now calculate the \emph{outgoing energy flux} through $S_{u,\mathring r}$,
\[
  \ul E(u,\mathring r) = \frac{1}{32\pi}\int_{S(u,\mathring r)} |\hat{\ul\chi}|^2\,d S = \frac{1}{32\pi}\int_{S(u)} |N|_{\slg}^2\,d S + o(1),
\]
with $N_{a b}=\pa_u h_{\bar a\bar b}$ is as in Theorem~\ref{ThmLBSLoss}.\footnote{Using~\eqref{EqLBSHawkingChibarHat}, we could compute $a$ as well as $\ul E(u,\mathring r)$ to one more order, exhibiting a $\mathring r^{-1}$ term plus a $o(\mathring r^{-1})$ remainder for both.} Clearly, $\ul E$ has a limit $E(u) = \lim_{\mathring r\to\infty} \ul E(u,\mathring r)$ at null infinity, and the Bondi mass loss formula~\eqref{EqLBSLoss} then takes the equivalent form
\[
  \frac{d}{d u}M_{\rm B}(u) = -E(u).
\]

\appendix

\section{Connection coefficients, curvature components, and natural operators}
\label{SCo}

We list the results of calculations used in the main body of the paper: geometric quantities and relevant differential operators for the exact Schwarzschild metric in \S\ref{SsCoSchw}, its perturbations (as considered in \S\ref{SsEinF}) near null infinity in \S\ref{SsCoPert}, and near the temporal face of the Minkowski metric in \S\ref{SsCoMink}.

\subsection{Schwarzschild}
\label{SsCoSchw}

In the notation of \S\ref{SsCptA}, in particular around~\eqref{EqCptASplProd}, the Schwarzschild metric
\[
  g\equiv g_m^S=(1-\tfrac{2 m}{r})dq\,ds-r^2\slg
\]
and the dual metric $g^{-1}$ have components
\[
  \setarraystretch
  \begin{array}{llllll}
    g_{0 0}=0, & g_{0 1}=\half(1-\tfrac{2 m}{r}), & g_{0 b}=0, & g_{1 1}=0, & g_{1 b}=0, & g_{a b}=-r^2\slg_{a b}, \\
    g^{0 0}=0, & g^{0 1}=\tfrac{2 r}{r-2 m}, & g^{0 b}=0, & g^{1 1}=0, & g^{1 b}=0, & g^{a b}=-r^{-2}\slg^{a b}.
  \end{array}
\]
The only nonzero Christoffel symbols in this frame are $\Gamma_{c a b}=-r^2\slGamma_{c a b}$, $\Gamma^c_{a b}=\slGamma^c_{a b}$, and
\[
  \setarraystretch
  \begin{array}{lll}
    \Gamma_{1 0 0}=\half m r^{-3}(r-2 m), & \Gamma_{c 0 b}=-\half(r-2 m)\slg_{b c}, & \Gamma_{0 1 1}=-\half m r^{-3}(r-2 m), \\
    \Gamma_{c 1 b}=\half(r-2 m)\slg_{b c}, & \Gamma_{0 a b}=\half(r-2 m)\slg_{a b}, & \Gamma_{1 a b}=-\half(r-2 m)\slg_{a b}, \\
    \Gamma^0_{0 0}=m r^{-2}, & \Gamma^c_{0 b}=\half r^{-1}(1-\tfrac{2 m}{r})\delta_b^c, & \Gamma^1_{1 1}=-m r^{-2}, \\
    \Gamma^c_{1 b}=-\half r^{-1}(1-\tfrac{2 m}{r})\delta_b^c, & \Gamma^0_{a b}=-r\slg_{a b}, & \Gamma^1_{a b}=r\slg_{a b}.
  \end{array}
\]
The only nonzero components of the Riemann curvature tensor (up to reordering the last two indices) are $R^a{}_{b c d}=2 m r^{-1}(\delta^a_c\slg_{b d}-\delta^a_d\slg_{b c})$ and
\[
  \setarraystretch
  \begin{array}{lll}
    R^0{}_{0 0 1}=-m r^{-3}(1-\tfrac{2 m}{r}), & R^0{}_{b 0 d}=-m r^{-1}\slg_{b d}, & R^1{}_{1 0 1}=m r^{-3}(1-\tfrac{2 m}{r}), \\
    R^1{}_{b 1 d}=-m r^{-1}\slg_{b d}, & R^a{}_{0 1 d}=-\half m r^{-3}(1-\tfrac{2 m}{r})\delta_d^a, & R^a{}_{1 0 d}=-\half m r^{-3}(1-\tfrac{2 m}{r})\delta_d^a.
  \end{array}
\]
With respect to the rescaled bundle splittings \eqref{EqCptASpl} and \eqref{EqCptASpl2}, we have
\[
  g_m^S=(0,\half(1-\tfrac{2 m}{r}),0,0,0,-\slg)^T,\ \ \tr_{g_m^S}=(0,\tfrac{4 r}{r-2 m},0,0,0,-\sltr),
\]
further
\begin{equation}
\label{EqCoSchwDelta}
\def\arraystretch{1.1}
  G_{g_m^S}=
  \openbigpmatrix{2pt}
    1 & 0 & 0 & 0 & 0 & 0 \\
    0 & 0 & 0 & 0 & 0 & (\tfrac14-\tfrac{m}{2 r})\sltr \\
    0 & 0 & 1 & 0 & 0 & 0 \\
    0 & 0 & 0 & 1 & 0 & 0 \\
    0 & 0 & 0 & 0 & 1 & 0 \\
    0 & \tfrac{2 r\slg}{r-2 m} & 0 & 0 & 0 & G_\slg
  \closebigpmatrix\!,
  \ 
  \delta_{g_m^S}^*=
  \openbigpmatrix{2pt}
    \pa_0-\tfrac{m}{r^2} & 0 & 0 \\
    \half\pa_1 & \half\pa_0 & 0 \\
    \half r^{-1}\sld & 0 & \half\pa_0 - r^{-1}(\tfrac14-\tfrac{m}{2 r}) \\
    0 & \pa_1+\tfrac{m}{r^2} & 0 \\
    0 & \half r^{-1}\sld & \half\pa_1 + r^{-1}(\tfrac14-\tfrac{m}{2 r}) \\
    r^{-1}\slg & -r^{-1}\slg & r^{-1}\sldelta^*
  \closebigpmatrix.
\end{equation}

We also record $dt=(\half,\half,0)^T$, $\nabla^{g_m^S} t=\tfrac{r}{r-2 m}(1,1,0)$, and, paralleling the definition of $\tdel^*$ from \eqref{EqEinTdel}, we have, near $S_+$,
\begin{equation}
\label{EqCoSchwTdeldiff}
\begin{split}
  &-2\gamma\frac{d(t^{-1})}{t^{-1}}\otimes_s(\cdot) + \gamma(\iota_{t\nabla^{g_m}(t^{-1})}(\cdot))g_m \\
  &\qquad = 
    \gamma t^{-1}
    \begin{pmatrix}
      1 & 0 & 0 \\
      \half & \half & 0 \\
      0 & 0 & \half \\
      0 & 1 & 0 \\
      0 & 0 & \half \\
      0 & 0 & 0
    \end{pmatrix}
   -\gamma t^{-1}
    \begin{pmatrix}
      0 & 0 & 0 \\
      \half & \half & 0 \\
      0 & 0 & 0 \\
      0 & 0 & 0 \\
      0 & 0 & 0 \\
      -\tfrac{r}{r-2 m}\slg & -\tfrac{r}{r-2 m}\slg & 0
    \end{pmatrix}.
\end{split}
\end{equation}

\subsection{Perturbations of Schwarzschild near the light cone}
\label{SsCoPert}

We consider a metric $g=g_m+\rho h=g_m^S+r^{-1} h$, with the perturbation $h\in\cX^{\infty;b_0,b_I,b'_I,b_+}$ lying in the function space of Definition~\ref{DefEinF}, and continue using the splittings \eqref{EqCptASplProd} of $(S^2)T^*\R^4$; however, we express the components of $h$ using the rescaled splitting \eqref{EqCptASpl2} as in~\eqref{EqCptASplSphWeight}, since for $h\in\cX^\infty$ all components $h_{\bar\mu\bar\nu}$ lie in the same space $\Hb^{\infty;b_0,-\eps,b_+}$; more precisely, they satisfy \eqref{EqEinFGood}--\eqref{EqEinFRest}. The components $g_{\mu\nu}$ and $g^{\mu\nu}$ were already computed, see~\eqref{EqEinFMetric} and \eqref{EqEinFinverse}. Recall also the observation~\eqref{EqEinFLeading} and the memberships \eqref{EqCptASplNull}. We shall write $b-0$ for weights which can be taken to be $b-\eps$ for any $\eps>0$; any two choices of $\eps$ are equivalent due to the assumption that all components of $h$ have leading terms (possibly with a factor $\log\rho_I$ for $h_{1 1}$) at $\scri^+$. The only part of the analysis that relies on the precise structure of the gauge-fixed Einstein equation is the analysis at $\scri^+$, so in the calculations below, the weight at $\scri^+$ is the most important one. We compute:
\begin{align*}
  \Gamma_{0 0 0} &\in \Hb^{\infty;2+b_0,2+b'_I,2+b_+}, \\
  \Gamma_{1 0 0} &\in \half r^{-2}(m-h_{0 1}) - \half r^{-1}\pa_1 h_{0 0} + \Hb^{\infty;2+b_0,2+b_I,2+b_+}, \\
  \Gamma_{c 0 0} &\in \Hb^{\infty;1+b_0,1+b'_I,1+b_+}, \\
  \Gamma_{0 0 1} &\in \half r^{-1}\pa_1 h_{0 0} + \Hb^{\infty;2+b_0,2+b'_I,2+b_+}, \\
  \Gamma_{1 0 1} &\in \half r^{-1}\pa_0 h_{1 1} - \tfrac14 r^{-2}h_{1 1} + \Hb^{\infty;3+b_0,3-0,3+b_+}, \\
  \Gamma_{c 0 1} &\in \half\pa_1 h_{0\bar c} - \half r^{-1}\pa_c h_{0 1} + \Hb^{\infty;1+b_0,1+b_I,1+b_+}, \\
  \Gamma_{0 0 b} &\in \Hb^{\infty;1+b_0,1+b'_I,1+b_+}, \\
  \Gamma_{1 0 b} &\in \half r^{-1}\pa_b h_{0 1}-\half\pa_1 h_{0\bar b} + \Hb^{\infty;1+b_0,1+b_I,1+b_+}, \\
  \Gamma_{c 0 b} &\in -\half(r-2 m)\slg_{b c} + \tfrac14 h_{\bar b\bar c} + \Hb^{\infty;b_0,b_I,b_+}, \\
  \Gamma_{0 1 1} &\in \half r^{-2}(h_{0 1}-m) + r^{-1}\pa_1 h_{0 1} - \half r^{-1}\pa_0 h_{1 1} + \tfrac14 r^{-2}h_{1 1} + \Hb^{\infty;3-0,3-0,3+b_+}, \\
  \Gamma_{1 1 1} &\in \half r^{-1}\pa_1 h_{1 1} + \tfrac14 r^{-2}h_{1 1} + \Hb^{\infty;3+b_0,3-0,3+b_+}, \\
  \Gamma_{c 1 1} &= \pa_1 h_{1\bar c} - \half r^{-1}\pa_c h_{1 1}, \\
  \Gamma_{0 1 b} &\in \half\pa_1 h_{0\bar b}+\half r^{-1}\pa_b h_{0 1} + \Hb^{\infty;1+b_0,1+b_I,1+b_+}, \\
  \Gamma_{1 1 b} &= \half r^{-1}\pa_b h_{1 1}, \\
  \Gamma_{c 1 b} &\in \half(r-2 m)\slg_{b c}+\half r\pa_1 h_{\bar b\bar c} - \tfrac14 h_{\bar b\bar c} + \half(\pa_b h_{1\bar c}-\pa_c h_{1\bar b}) + \Hb^{\infty;1+b_0,1-0,1+b_+}, \\
  \Gamma_{0 a b} &\in \half(r-2 m)\slg_{a b}-\tfrac14 h_{\bar a\bar b} + \Hb^{\infty;b_0,b_I,b_+}, \\
  \Gamma_{1 a b} &\in -\half(r-2 m)\slg_{a b}-\half r\pa_1 h_{\bar a\bar b} + \half(\pa_a h_{1\bar b}+\pa_b h_{1\bar a}) + \tfrac14 h_{\bar a\bar b} + \Hb^{\infty;1+b_0,1-0,1+b_+}, \\
  \Gamma_{c a b} &= -r^2\slGamma_{c a b} + \half r(\pa_a h_{\bar b\bar c}+\pa_b h_{\bar a\bar c}-\pa_c h_{\bar a\bar b}).
\end{align*}
The Christoffel symbols of the second kind are therefore
\begin{align}
\label{EqCoPertGamma}
  \Gamma^0_{0 0} &\in r^{-2}(m-h_{0 1}) - r^{-1}\pa_1 h_{0 0} + \Hb^{\infty;2+b_0,2+b_I,2+b_+}, \\
  \Gamma^1_{0 0} &\in \Hb^{\infty;2+b_0,2+b'_I,2+b_+}, \nonumber\\
  \Gamma^c_{0 0} &\in \Hb^{\infty;3+b_0,3+b'_I,3+b_+}, \nonumber\\
  \Gamma^0_{0 1} &\in r^{-1}\pa_0 h_{1 1}-\half r^{-2}h_{1 1}+\Hb^{\infty;3+b_0,2+b'_I-0,3+2 b_+}, \nonumber\\
  \Gamma^1_{0 1} &\in r^{-1}\pa_1 h_{0 0} + \Hb^{\infty;2+b_0,2+b'_I,2+b_+}, \nonumber\\
  \Gamma^c_{0 1} &\in -\half r^{-2}\pa_1 h_0{}^{\bar c}+\half r^{-3}\slnabla^c h_{0 1} + \Hb^{\infty;3+b_0,3+b_I,3+b_+}, \nonumber\\
  \Gamma^0_{0 b} &\in -\pa_1 h_{0\bar b}+r^{-1}\pa_b h_{0 1}-r^{-1}h_{1\bar b} + \Hb^{\infty;1+b_0,1+b_I,1+b_+}, \nonumber\\
  \Gamma^1_{0 b} &\in \Hb^{\infty;1+b_0,1+b'_I,1+b_+}, \nonumber\\
  \Gamma^c_{0 b} &\in \half r^{-1}(1-\tfrac{2 m}{r})\delta_b^c + \tfrac14 r^{-2}h_{\bar b}{}^{\bar c} + \Hb^{\infty;2+b_0,2+b_I,2+b_+}, \nonumber\\
  \Gamma^0_{1 1} &\in r^{-1}\pa_1 h_{1 1}+\half r^{-2}h_{1 1} + 2 r^{-2}(m-h_{0 1})\pa_1 h_{1 1} \nonumber\\
    &\qquad - 4 r^{-2}h_{1 1}\pa_1 h_{0 1} + 2 r^{-2}h_1{}^{\bar d}\pa_1 h_{1\bar d} + \Hb^{\infty;3+b_0,3-0,3+2 b_+}, \nonumber\\
  \Gamma^1_{1 1} &\in r^{-2}(h_{0 1}-m)+2 r^{-1}\pa_1 h_{0 1}-r^{-1}\pa_0 h_{1 1} \nonumber\\
    &\qquad +\half r^{-2}h_{1 1}+4 r^{-2}(m-h_{0 1})\pa_1 h_{0 1} + \Hb^{\infty;3-0,2+b'_I-0,3+2 b_+}, \nonumber\\
  \Gamma^c_{1 1} &\in -r^{-2}\pa_1 h_1{}^{\bar c}+\half r^{-3}\slnabla^c h_{1 1}+2 r^{-3}h_1{}^{\bar c}\pa_1 h_{0 1} \nonumber\\
    &\qquad -r^{-3}h^{\bar c\bar d}\pa_1 h_{1\bar d} + \Hb^{\infty;4+b_0,3+b'_I-0,4+2 b_+}, \nonumber\\
  \Gamma^0_{1 b} &\in r^{-1}\pa_b h_{1 1}+r^{-1}h_{1\bar b}+r^{-1}h_1{}^{\bar d}\pa_1 h_{\bar b\bar d} + \Hb^{\infty;2+b_0,1+b'_I-0,2+2 b_+}, \nonumber\\
  \Gamma^1_{1 b} &\in \pa_1 h_{0\bar b}+r^{-1}\pa_b h_{0 1} + \Hb^{\infty;1+b_0,1+b_I,1+b_+}, \nonumber\\
  \Gamma^c_{1 b} &\in -\half r^{-1}(1-\tfrac{2 m}{r})\delta_b^c-\half r^{-1}\pa_1 h_{\bar b}{}^{\bar c}-\tfrac14 r^{-2}h_{\bar b}{}^{\bar c} \nonumber\\
    &\qquad + \half r^{-2}\slg^{c d}(\pa_ d h_{1\bar b}-\pa_b h_{1\bar d}) - \half r^{-2}h^{\bar c\bar d}\pa_1 h_{\bar b\bar d} + \Hb^{\infty;3+b_0,2+b'_I,3+2 b_+}, \nonumber\\
  \Gamma^0_{a b} &\in (-r+2 h_{0 1}-2 h_{1 1})\slg_{a b}-(r+2 m-2 h_{0 1})\pa_1 h_{\bar a\bar b} \nonumber\\
    &\qquad + (\slnabla_a h_{1\bar b}+\slnabla_b h_{1\bar a})+\half h_{\bar a\bar b} + \Hb^{\infty;1-0,1-0,1+2 b_+}, \nonumber\\
  \Gamma^1_{a b} &\in (r-2 h_{0 1})\slg_{a b} - \half h_{\bar a\bar b} + \Hb^{\infty;b_0,b_I,b_+}, \nonumber\\
  \Gamma^c_{a b} &\in \slGamma^c_{a b} + r^{-1}h_1{}^{\bar c}\slg_{a b} - \half r^{-1}(\slnabla_a h_{\bar b}{}^{\bar c}+\slnabla_b h_{\bar a}{}^{\bar c}-\slnabla^c h_{\bar a\bar b}) + \Hb^{\infty;1+b_0,1+b'_I,1+b_+}. \nonumber
\end{align}
We can then calculate $\Ups(g)^\nu=g^{\kappa\lambda}(\Gamma(g)_{\kappa\lambda}^\nu-\Gamma(g_m)_{\kappa\lambda}^\nu)$, see \eqref{EqEinUps}, to wit
\begin{align}
\label{EqCoPertUpsUpper}
  \Ups(g)^0 &\in r^{-1}\pa_1\sltr h + 2 r^{-2}(h_{1 1}-2 h_{0 1})-2 r^{-2}\slnabla_d h_1{}^{\bar d} \\
    &\qquad + 4 r^{-1}\pa_0 h_{1 1} + r^{-2}h^{\bar d\bar e}\pa_1 h_{\bar d\bar e} + \Hb^{\infty;2+b_0,2+b'_I-0,2+b_+}, \nonumber\\
  \Ups(g)^1 &\in 4 r^{-1}\pa_1 h_{0 0} + 4 r^{-2}h_{0 1} + \Hb^{\infty;2+b_0,2+b_I,2+b_+}, \nonumber\\
  \Ups(g)^c &\in -2 r^{-2}\pa_1 h_0{}^{\bar c} + 2 r^{-3}\slnabla^c h_{0 1}+r^{-3}\slnabla_d h^{\bar c\bar d}-2 r^{-3}h_1{}^{\bar c} + \Hb^{\infty;3+b_0,3+b_I,3+b_+}, \nonumber
\end{align}
and therefore
\begin{align}
\label{EqCoPertUpsLower}
  \Ups(g)_0 &\in 2 r^{-1}\pa_1 h_{0 0}+2 r^{-2}h_{0 1}+\Hb^{\infty;2+b_0,2+b_I,2+b_+}, \\
  \Ups(g)_1 &\in \half r^{-1}\pa_1\sltr h+r^{-2}(h_{1 1}-2 h_{0 1})-r^{-2}\slnabla_d h_1{}^{\bar d} \nonumber\\
    &\qquad +2 r^{-1}\pa_0 h_{1 1}+\half r^{-2}h^{\bar d\bar e}\pa_1 h_{\bar d\bar e} + \Hb^{\infty;2+b_0,2+b'_I-0,2+b_+}, \nonumber\\
  \Ups(g)_c &\in 2\pa_1 h_{0\bar c}-2 r^{-1}\pa_c h_{0 1}-r^{-1}\slnabla^d h_{\bar c\bar d}+2 r^{-1}h_{1\bar c} + \Hb^{\infty;1+b_0,1+b_I,1+b_+}. \nonumber
\end{align}
Using \eqref{EqCoSchwDelta}, this gives
\begin{align}
\label{EqCoPertDelUps}
  (\delta_{g_m}^*\Ups(g))_{0 0} &\in \Hb^{\infty;3+b_0,2+b'_I,3+b_+}, \\
  (\delta_{g_m}^*\Ups(g))_{0 1} &\in r^{-1}\pa_1\pa_1 h_{0 0}+r^{-2}\pa_1 h_{0 1}+\Hb^{\infty;3+b_0,2+b'_I,3+b_+}, \nonumber\\
  (\delta_{g_m}^*\Ups(g))_{0 \bar b} &\in \Hb^{\infty;3+b_0,2+b'_I,3+b_+}, \nonumber\\
  (\delta_{g_m}^*\Ups(g))_{1 1} &\in \half r^{-1}\pa_1\pa_1\sltr h + 2 r^{-1}\pa_1\pa_0 h_{1 1} - r^{-2}\pa_1\slnabla_d h_1{}^{\bar d} + \half r^{-2}h^{\bar d\bar e}\pa_1\pa_1 h_{\bar d\bar e} \nonumber\\
    &\qquad + r^{-2}(\pa_1 h_{1 1}-2\pa_1 h_{0 1}) + \half r^{-2}\pa_1 h^{\bar d\bar e}\pa_1 h_{\bar d\bar e} + \Hb^{\infty;3+b_0,2+b'_I-0,3+b_+}, \nonumber\\
  (\delta_{g_m}^*\Ups(g))_{1 \bar b} &\in r^{-1}\pa_1\pa_1 h_{0\bar b}-r^{-2}\pa_b\pa_1 h_{0 1}-\half r^{-2}\slnabla^d\pa_1 h_{\bar b\bar d} + r^{-2}\pa_1 h_{1\bar b} + \Hb^{\infty;3+b_0,2+b'_I,3+b_+}, \nonumber\\
  (\delta_{g_m}^*\Ups(g))_{\bar a \bar b} &\in \Hb^{\infty;3+b_0,2+b'_I,3+b_+}. \nonumber
\end{align}

Next, we calculate the curvature components; as explained in \S\ref{SIt}, we shall need to know the components $\Ric_{\bar\mu\bar\nu}$ modulo terms decaying faster than $\rho_0^{3+b_0}$, $\rho_I^{2+b_I}$, and $\rho_+^{3+b_+}$ at $I^0$, $\scri^+$, and $I^+$, respectively, in order to control each step in our iteration scheme. At $I^0$, the leading contribution to the curvature components will come from the Schwarzschild part of $g$; cf.\ the calculations in \S\ref{SsCoSchw}. Thus, we compute
\begin{align}
\label{EqCoPertRiem}
  R^0{}_{0 0 1} &\in -m r^{-3} + r^{-1}\pa_1\pa_1 h_{0 0}+r^{-2}\pa_1 h_{0 1} + \Hb^{\infty;3+b_0,2+b_I,3+b_+}, \\
  R^0{}_{0 0 d} &\in \Hb^{\infty;2+b_0,1+b'_I,2+b_+}, \nonumber\\
  R^0{}_{0 1 d} &\in -\pa_1\pa_1 h_{0\bar d}+r^{-1}\pa_d\pa_1 h_{0 1}-r^{-1}\pa_1 h_{1\bar d} + \Hb^{\infty;2+b_0,1+b_I,2+b_+}, \nonumber\\
  R^0{}_{0 c d} &\in \Hb^{\infty;1+b_0,b'_I,1+b_+}, \nonumber\\
  R^0{}_{1 0 1} &\in \Hb^{\infty;3+b_0,2+b_I,3+b_+}, \nonumber\\
  R^0{}_{1 0 d} &\in \Hb^{\infty;2+b_0,1+b'_I,2+b_+}, \nonumber\\
  R^0{}_{1 1 d} &\in r^{-1}h_1{}^{\bar e}\pa_1\pa_1 h_{\bar d\bar e} + \Hb^{\infty;2+b_0,1+b'_I,2+b_+}, \nonumber\\
  R^0{}_{1 c d} &\in \Hb^{\infty;1+b_0,1-0,1+b_+}, \nonumber\\
  R^0{}_{b 0 1} &\in \pa_1\pa_1 h_{0\bar b}-r^{-1}\pa_b\pa_1 h_{0 1}+r^{-1}\pa_1 h_{1\bar b} + \Hb^{\infty;2+b_0,1+b_I,2+b_+}, \nonumber\\
  R^0{}_{b 0 d} &\in -m r^{-1}\slg_{b d} + \Hb^{\infty;1+b_0,b'_I,1+b_+}, \nonumber\\
  R^0{}_{b 1 d} &\in -(r+2 m-2 h_{0 1})\pa_1\pa_1 h_{\bar b\bar d} + (2\pa_1 h_{0 1}-\pa_1 h_{1 1})\slg_{b d} \nonumber\\
    &\qquad + \pa_1(\pa_b h_{1\bar d}+\pa_d h_{1\bar b}) + 2\pa_1 h_{0 1}\pa_1 h_{\bar b\bar d} - \half\pa_1 h_{\bar b\bar e}\pa_1 h_{\bar d}{}^{\bar e} + \Hb^{\infty;1+b_0,1-0,1+b_+}, \nonumber\\
  R^0{}_{b c d} &\in r\pa_1(\slnabla_d h_{\bar b\bar c}-\slnabla_c h_{\bar b\bar d}) + \Hb^{\infty;b_0,-1+b'_I,b_+}, \nonumber\\
  R^1{}_{0 0 1} &\in \Hb^{\infty;3+b_0,2+b'_I,3+b_+}, \nonumber\\
  R^1{}_{0 0 d} &\in \Hb^{\infty;2+b_0,2+b'_I,2+b_+}, \nonumber\\
  R^1{}_{0 1 d} &\in \Hb^{\infty;2+b_0,1+b'_I,2+b_+}, \nonumber\\
  R^1{}_{0 c d} &\in \Hb^{\infty;1+b_0,1+b'_I,1+b_+}, \nonumber\\
  R^1{}_{1 0 1} &\in m r^{-3}-r^{-1}\pa_1\pa_1 h_{0 0}-r^{-2}\pa_1 h_{0 1} + \Hb^{\infty;3+b_0,2+b'_I,3+b_+}, \nonumber\\
  R^1{}_{1 0 d} &\in \Hb^{\infty;2+b_0,1+b'_I,2+b_+}, \nonumber\\
  R^1{}_{1 1 d} &\in \pa_1\pa_1 h_{0\bar d}-r^{-1}\pa_1\pa_d h_{0 1}+r^{-1}\pa_1 h_{1\bar d} + \Hb^{\infty;2+b_0,1+b_I,2+b_+}, \nonumber\\
  R^1{}_{1 c d} &\in \Hb^{\infty;1+b_0,b'_I,1+b_+}, \nonumber\\
  R^1{}_{b 0 1} &\in \Hb^{\infty;2+b_0,1+b'_I,2+b_+}, \nonumber\\
  R^1{}_{b 0 d} &\in \Hb^{\infty;1+b_0,b'_I,1+b_+}, \nonumber\\
  R^1{}_{b 1 d} &\in -m r^{-1}\slg_{b d} + \Hb^{\infty;1+b_0,b'_I,1+b_+}, \nonumber\\
  R^1{}_{b c d} &\in \Hb^{\infty;b_0,-1+b'_I,b_+}, \nonumber\\
  R^a{}_{0 0 1} &\in \Hb^{\infty;4+b_0,3+b'_I,4+b_+}, \nonumber\\
  R^a{}_{0 0 d} &\in \Hb^{\infty;3+b_0,2+b'_I,3+b_+}, \nonumber\\
  R^a{}_{0 1 d} &\in -\half m r^{-3}\delta_d^a + \Hb^{\infty;3+b_0,2+b'_I,3+b_+}, \nonumber\\
  R^a{}_{0 c d} &\in \Hb^{\infty;2+b_0,1+b'_I,2+b_+}, \nonumber\\
  R^a{}_{1 0 1} &\in \half r^{-2}\pa_1\pa_1 h_0{}^{\bar a}-\half r^{-3}\slnabla^a\pa_1 h_{0 1} + \half r^{-3}\pa_1 h_1{}^{\bar a} + \Hb^{\infty;4+b_0,3+b_I,4+b_+}, \nonumber\\
  R^a{}_{1 0 d} &\in -\half m r^{-3}\delta_d^a + \Hb^{\infty;3+b_0,2+b_I,3+b_+}, \nonumber\\
  R^a{}_{1 1 d} &\in -\half r^{-1}\pa_1\pa_1 h_{\bar d}{}^{\bar a} + \half r^{-2}\pa_1(\slnabla^a h_{1\bar d}+\slnabla_d h_1{}^{\bar a}) - \half r^{-2}h^{\bar a\bar e}\pa_1\pa_1 h_{\bar d\bar e} \nonumber\\
    &\qquad +r^{-2}(\pa_1 h_{0 1}-\half\pa_1 h_{1 1})\delta_d^a + r^{-2}\pa_1 h_{0 1}\pa_1 h_{\bar d}{}^{\bar a} \nonumber\\
    &\qquad -\tfrac14 r^{-2}\pa_1 h^{\bar a\bar e}\pa_1 h_{\bar d\bar e} + \Hb^{\infty;3+b_0,2+b'_I,3+b_+}, \nonumber\\
  R^a{}_{1 c d} &\in \half r^{-1}\pa_1(\slnabla_d h_{\bar c}{}^{\bar a}-\slnabla_c h_{\bar d}{}^{\bar a}) + \Hb^{\infty;2+b_0,1+b'_I,2+b_+}, \nonumber\\
  R^a{}_{b 0 1} &\in \Hb^{\infty;3+b_0,2+b_I,3+b_+}, \nonumber\\
  R^a{}_{b 0 d} &\in \Hb^{\infty;2+b_0,1+b'_I,2+b_+}, \nonumber\\
  R^a{}_{b 1 d} &\in \half r^{-1}\pa_1(\slnabla^a h_{\bar b\bar d}-\slnabla_b h_{\bar d}{}^{\bar a}) + \Hb^{\infty;2+b_0,1+b'_I,2+b_+}, \nonumber\\
  R^a{}_{b c d} &\in 2 m r^{-1}(\delta_k^a\slg_{b d}-\delta_d^a\slg_{b c}) \nonumber\\
    &\qquad + \half(\pa_1 h_{\bar b\bar c}\delta_d^a-\pa_1 h_{\bar b\bar d}\delta_k^a + \pa_1 h_{\bar d}{}^{\bar a}\slg_{b c}-\pa_1 h_{\bar c}{}^{\bar a}\slg_{b d}) + \Hb^{\infty;1+b_0,1-0,1+b_+}, \nonumber
\end{align}
and the Ricci tensor
\begin{align}
  \Ric(g)_{0 0} &\in \Hb^{\infty;3+b_0,2+b'_I,3+b_+}, \nonumber\\
  \Ric(g)_{0 1} &\in r^{-1}\pa_1\pa_1 h_{0 0}+r^{-2}\pa_1 h_{0 1} + \Hb^{\infty;3+b_0,2+b_I,3+b_+}, \nonumber\\
  \Ric(g)_{0\bar b} &\in \Hb^{\infty;3+b_0,2+b'_I,3+b_+}, \nonumber\\
\label{EqCoPertRic}
  \Ric(g)_{1 1} &\in \half r^{-1}\pa_1\pa_1\sltr h -r^{-2}\pa_1\slnabla^d h_{1\bar d}+\half r^{-2}h^{\bar d\bar e}\pa_1\pa_1 h_{\bar d\bar e} \\
    &\qquad + r^{-2}(\pa_1 h_{1 1}-2\pa_1 h_{0 1}) + \tfrac14 r^{-2}\pa_1 h^{\bar d\bar e}\pa_1 h_{\bar d\bar e} + \Hb^{\infty;3+b_0,2+b_I,3+b_+}, \nonumber\\
  \Ric(g)_{1\bar b} &\in r^{-1}\pa_1\pa_1 h_{0\bar b}-r^{-2}\pa_1\pa_b h_{0 1} - \half r^{-2}\pa_1\slnabla_d h_{\bar b}{}^{\bar d} + r^{-2}\pa_1 h_{1\bar b} + \Hb^{\infty;3+b_0,2+b_I,3+b_+}, \nonumber\\
  \Ric(g)_{\bar a\bar b} &\in \Hb^{\infty;3+b_0,2+b'_I,3+b_+}. \nonumber
\end{align}

\subsection{Perturbations of Minkowski space near the temporal face}
\label{SsCoMink}

We work on $\R^4=\R_t\times\R^3_x$, equipped with the Minkowski metric $\ul g=dt^2-dx^2$, and consider the linearization of $\ul P_0(g):=\Ric(g)-\ul\tdel^*\ul\Ups(g)$,
\[
  (\ul\tdel^*-\delta_{\ul g}^*)u := 2\gamma t^{-1}\,dt\otimes_s u - \gamma t^{-1}(\iota_{\nabla^{\ul g}t}u)\ul g,\ \ 
  \ul\Ups(g) = g \ul g^{-1}\delta_g G_g \ul g,
\]
around $g=\ul g$; concretely, let $\ul L:=\ul\rho^{-3}D_{\ul g}\ul P_0\ul\rho$, where $\ul\rho:=t^{-1}$ is a boundary defining function of ${}^0\ol{\R^4}$ in $t>\eps r$, $\eps>0$. We have
\[
  \ul L = t^3\bigl(\half\Box_{\ul g} + (\ul\tdel^*-\delta_{\ul g}^*)\delta_{\ul g} G_{\ul g}\bigr)t^{-1}.
\]
Splitting
\begin{equation}
\label{EqCoMinkBundle}
  T^*\R^4 = \la dt\ra \oplus T^*\R^3,\ \ 
  S^2 T^*\R^4 = \la dt^2\ra \oplus (2 dt\otimes_s T^*\R^3) \oplus S^2 T^*\R^3,
\end{equation}
and writing $e=dx^2$ for the Euclidean metric on $\R^3$, we have $\ul g=(1,0,-e)^T$, $\tr_{\ul g}=(1,0,-\tr_e)$,
\[
  \ul\tdel^*-\delta_{\ul g}^*=t^{-1}
  \begin{pmatrix}
    \gamma & 0 \\
    0 & \gamma \\
    \gamma e & 0
  \end{pmatrix},\ \ 
  G_{\ul g}=
  \begin{pmatrix}
    \half & 0 & \half\tr_e \\
    0 & 1 & 0 \\
    \half e & 0 & 1-\half e\tr_e
  \end{pmatrix},\ \ 
  \delta_{\ul g}=
  \begin{pmatrix}
    -\pa_t & -\delta_e & 0 \\
    0 & -\pa_t & -\delta_e
  \end{pmatrix}.
\]
Moreover, $\Box_{\ul g}$ is diagonal using the standard trivialization of $T^*\R^3$, and the scalar wave operator is $t^3\Box_{\ul g} t^{-1}=-(t\pa_t-\tfrac32)^2-t^2\Delta_x+\tfrac14$, hence
\begin{align*}
  \ul L &= \half\bigl(-(t\pa_t-\tfrac32)^2-t^2\Delta_x+\tfrac14\bigr) \\
   &\qquad -\begin{pmatrix}
     \half\gamma(t\pa_t-1) & \gamma t\delta_e & \half\gamma(t\pa_t-1)\tr_e \\
     -\half\gamma t d_x & \gamma(t\pa_t-1) & \gamma t(\delta_e+\half d_x\tr_e) \\
     \half\gamma(t\pa_t-1)e & \gamma t e\delta_e & \half\gamma(t\pa_t-1)e\tr_e
   \end{pmatrix}.
\end{align*}

\section{Proofs of Lemmas~\ref{LemmaEinNScalarBox} and \ref{LemmaEinNscri}}
\label{SAppEinN}

We perform the necessary calculations using the results in \S\ref{SsCoPert}. 

\begin{proof}[Proof of Lemma~\usref{LemmaEinNScalarBox}]
  We use the invariance properties of the conformal wave operator (i.e.\ the conformal Laplacian in Lorentzian signature),
  \[
    A:=\rho^{-3}(\Box_g-\tfrac16 R_g)\rho = \Box_{g_\bop}-\tfrac16 R_{g_\bop},\quad g_\bop=\rho^2 g.
  \]
  Here, the scalar curvature satisfies $\rho^{-2}R_g\in\Hb^{\infty;1+b_0,-1+b'_I,1+b_+}$; indeed, in view of~\eqref{EqCoPertRic}, and using in addition the memberships~\eqref{EqCptASplNull} of the operators $\pa_0,\pa_1$ (and spherical vector fields, which lie in $\Vb(M)$) as well as the memberships of the metric coefficients of $h$ as encoded in Definition~\ref{DefEinF}, one concludes that $\rho^{-2}\Ric(g)\in\Hb^{\infty;1+b_0,-1+b'_I,1+b_+}$; since the metric coefficients of $g^{-1}$ are bounded and conormal, the rescaled scalar curvature $\rho^{-2}R_g=\tr_g(\rho^{-2}\Ric(g))$ lies in the same space.
  
  We next write the wave operator as
  \[
    \Box_g u = -r^{-s(\mu,\nu)}g^{\bar\mu\bar\nu}\pa_\mu\pa_\nu u + r^{-s(\kappa)}g^{\bar\mu\bar\nu}\Gamma_{\bar\mu\bar\nu}^{\bar\kappa}\pa_\kappa.
  \]
  In the first term, when $\mu=0$, the terms with $\nu\neq 1$ contribute $\Hb^{\infty;3+b_0,3-0,3+b_+}\Diffb^2$, while $\nu=1$ gives $-4\pa_0\pa_1+(\rho^2+\Hb^{\infty;3+b_0,2-0,3+b_+})\Diffb^2$. For $\mu=1$, $\nu=1$ produces a term in $\Hb^{\infty;3+b_0,1+b'_I,3+b_+}\cM^2$ due to the decay of $h_{0 0}$ at $\scri^+$, while $\nu$ spherical gives an element of $\Hb^{\infty;3+b_0,2-0,3+b_+}\Diffb^2$. Lastly, $\mu$ and $\nu$ both spherical give $(\rho^2\CI+\Hb^{\infty;3+b_0,3-0,3+b_+})\Diffb^2$. For the second summand, recall \eqref{EqEinNscriContractedGamma}, while for $\kappa\neq 1$, $g^{\bar\mu\bar\nu}\Gamma_{\bar\mu\bar\nu}^{\bar\kappa}\in\rho\CI+\Hb^{\infty;2+b_0,1-0,2+b_+}$ by~\eqref{EqEinNscriGamma}. Thus, $\Box_g\equiv -4\pa_0\pa_1-2 r^{-1}\pa_1$ modulo a term lying in $\rho^2$ times the error space in \eqref{EqEinNScalarBox}. Since $\Box_{g_\bop}=A-A(1)$, the claim follows.
\end{proof}

\begin{proof}[Proof of Lemma~\usref{LemmaEinNscri}]
  We consider each of the terms in \eqref{EqEinNOp} separately. The contribution from
  \[
    (\rho^{-3}\sR_g\rho u)_{\bar\mu\bar\nu} = \rho^{-2}(R_g)^{\bar\kappa}{}_{\bar\mu\bar\nu\bar\lambda}g^{\bar\lambda\bar\sigma}u_{\bar\kappa\bar\sigma} + \half \rho^{-2}g^{\bar\lambda\bar\sigma}(\Ric(g)_{\bar\mu\bar\lambda}u_{\bar\sigma\bar\nu} + \Ric(g)_{\bar\nu\bar\lambda}u_{\bar\sigma\bar\mu})
  \]
  to terms of size at least $\rho_I^{-1}$ at $\scri^+$ comes from those components of $R_g$ and $\Ric(g)$ of size at least $\rho_I$. The only such components of $R_g$ are
  \[
    R^0{}_{\bar b 1\bar d} \in -r^{-1}\pa_1\pa_1 h_{\bar b\bar d} + \Hb^{\infty;3+b_0,2-0,3+b_+},\ \ 
    R^{\bar a}{}_{1 1\bar d} \in -\half r^{-1}\pa_1\pa_1 h_{\bar d}{}^{\bar a} + \Hb^{\infty;3+b_0,2-0,3+b_+},
  \]
  while all other components lie in $\Hb^{\infty;3-0,1+b'_I,3+b_+}$; the decay order at $I^0$ is due to the contributions from the asymptotic Schwarzschild metric, as e.g.\ in $R^0{}_{0 0 1}$. On the other hand, \eqref{EqCoPertRic} shows that $\Ric(g)\in\Hb^{\infty;3+b_0,1+b'_I,3+b_+}$. Using the form \eqref{EqEinFinverse} of $g^{-1}$, this gives
  \begin{equation}
  \label{EqEinNscriRg}
  \begin{split}
    \rho^{-2}\sR_g &\in
    \openbigpmatrix{1pt}
      0 & 0 & 0 & 0 & 0 & 0 & 0 \\
      0 & 0 & 0 & 0 & 0 & 0 & 0 \\
      0 & 0 & 0 & 0 & 0 & 0 & 0 \\
      0 & 0 & 0 & 0 & 0 & 0 & \half\rho^{-1}\pa_1\pa_1 h^{\bar a\bar b} \\
      0 & 0 & \rho^{-1}\pa_1\pa_1 h^{\bar a}{}_{\bar b} & 0 & 0 & 0 & 0 \\
      0 & 0 & 0 & 0 & 0 & 0 & 0 \\
      2\rho^{-1}\pa_1\pa_1 h_{\bar a\bar b} & 0 & 0 & 0 & 0 & 0 & 0
    \closebigpmatrix \\
    &\qquad+ \rho\,\CI
    + \Hb^{\infty;1+b_0,-1+b'_I,1+b_+}.
  \end{split}
  \end{equation}

  Next, we have $(\sY_g u)_{\bar\kappa}=\Ups(g)^{\bar\lambda}u_{\bar\kappa\bar\lambda}$, with $\Ups(g)^{\bar\lambda}\in\Hb^{\infty;2+b_0,1+b'_I,2+b_+}$ by \eqref{EqCoPertUpsUpper}. Now, equation~\eqref{EqEinTdel} implies
  \begin{equation}
  \label{EqEinNscriTdeldiff}
    \tdel^*-\delta_{g_m}^* \in \rho\,\CI(M;\Hom(\beta^*\,\Tsc^*\ol{\R^4},\beta^*S^2)),
  \end{equation}
  so the expression for $\tdel^*$ obtained from \eqref{EqCoSchwDelta} and the inclusions \eqref{EqCptScriNull} show that
  \begin{equation}
  \label{EqEinNscriTdel}
    \tdel^*\in \rho_0\rho_+\cM_{\beta^*\,\Tsc^*\ol{\R^4},\beta^*S^2} + \rho\,\Diffb^1(M;\beta^*\,\Tsc^*\ol{\R^4},\beta^*S^2),
  \end{equation}
  and therefore
  \begin{equation}
  \label{EqEinNscrisY}
    \rho^{-3}\tdel^*\sY_g\rho \in \Hb^{\infty;1+b_0,-1+b'_I,1+b_+}\cM + \Hb^{\infty;1+b_0,b'_I,1+b_+}\Diffb^1.
  \end{equation}

  Next, the only parts of $\sC_g$ which will contribute leading terms to~\eqref{EqEinNOp} come from those components $C^{\bar\lambda}_{\bar\mu\bar\nu}$ which are of size at least $\rho_I$ at $\scri^+$; these are, modulo $\Hb^{\infty;2+b_0,2-0,2+b_+}$,
  \begin{equation}
  \label{EqEinNscriCTensor}
  \begin{gathered}
    C^0_{1 1} \equiv r^{-1}\pa_1 h_{1 1},\quad
    C^1_{1 1} \equiv 2 r^{-1}\pa_1 h_{0 1},\quad
    C^{\bar c}_{1 1} \equiv -r^{-1}\pa_1 h_1{}^{\bar c}, \\
    C^{\bar c}_{1\bar b} \equiv -\half r^{-1}\pa_1 h_{\bar b}{}^{\bar c}, \quad
    C^0_{\bar a\bar b} \equiv -r^{-1}\pa_1 h_{\bar a\bar b},
  \end{gathered}
  \end{equation}
  while all other components of $C^{\bar\lambda}_{\bar\mu\bar\nu}$ lie in $\Hb^{\infty;2+b_0,1+b'_I,2+b_+}$. Therefore, writing sections of $\beta^*\,\Tsc^*\ol{\R^4}$ in terms of the splitting \eqref{EqCptASpl}, we have
  \[
    \sC_g \in
      \begin{pmatrix}
        4 r^{-1}\pa_1 h_{0 1} & 0 & 0 & 0 & 0 & 0 & 0 \\
        2 r^{-1}\pa_1 h_{1 1} & 0 & 0 & 0 & 0 & 0 & -\half r^{-1}\pa_1 h^{\bar a\bar b} \\
        4 r^{-1}\pa_1 h_{1\bar b} & 0 & -2 r^{-1}\pa_1 h_{\bar b}{}^{\bar a} & 0 & 0 & 0 & 0
      \end{pmatrix}
      + \Hb^{\infty;2+b_0,1+b'_I,2+b_+},
  \]
  and then \eqref{EqCoSchwDelta}, \eqref{EqEinNscriTdeldiff}, and \eqref{EqEinNscriTdel} give
  \begin{align}
  \label{EqEinNscrisC}
    \rho^{-3}\tdel^*\sC_g\rho &\in
      \rho^{-1}\pa_1\circ
      \begin{pmatrix}
        0 & 0 & 0 & 0 & 0 & 0 & 0 \\
        2\pa_1 h_{0 1} & 0 & 0 & 0 & 0 & 0 & 0 \\
        0 & 0 & 0 & 0 & 0 & 0 & 0 \\
        2\pa_1 h_{1 1} & 0 & 0 & 0 & 0 & 0 & -\half\pa_1 h^{\bar a\bar b} \\
        2\pa_1 h_{1\bar b} & 0 & -\pa_1 h_{\bar b}{}^{\bar a} & 0 & 0 & 0 & 0 \\
        0 & 0 & 0 & 0 & 0 & 0 & 0 \\
        0 & 0 & 0 & 0 & 0 & 0 & 0
      \end{pmatrix} \\
      &\qquad + \Hb^{\infty;1+b_0,-1+b'_I,1+b_+}\cM + \Hb^{\infty;1+b_0,-0,1+b_+}\Diffb^1; \nonumber
  \end{align}
  the only terms of $\tdel^*$ which contribute leading terms to this operator are the $\pa_1$ derivatives in $\delta_{g_m}^*$.

  For the second summand in \eqref{EqEinNOp}, we note that $G_{g_m}\in\CI(\ol{\R^4},\End(S^2\,\Tsc^*\ol{\R^4}))$, while equation~\eqref{EqEinFMetricRough} gives $G_g\in G_{g_m}+\Hb^{\infty;1+b_0,1-0,1+b_+}(\beta^*S^2)$. Further, using the notation~\eqref{EqCptASplSphWeight} and setting $\Gamma^{\bar\kappa}_{\bar\mu\bar\nu}:=r^{s(\kappa)-s(\mu,\nu)}\Gamma^\kappa_{\mu\nu}$, we have
  \begin{equation}
  \label{EqEinNscriDelta}
    (\delta_g u)_{\bar\mu} = -r^{-s(\mu,\nu,\lambda)}g^{\bar\nu\bar\lambda}\pa_\lambda(r^{s(\mu,\nu)}u_{\bar\mu\bar\nu}) + g^{\bar\nu\bar\lambda}(\Gamma^{\bar\kappa}_{\bar\mu\bar\lambda}u_{\bar\kappa\bar\nu} + \Gamma^{\bar\kappa}_{\bar\nu\bar\lambda}u_{\bar\mu\bar\kappa});
  \end{equation}
  now $r^{-s(\lambda)}\pa_\lambda\in\rho\,\Vb(M)$ unless $\lambda=1$, and moreover
  \begin{equation}
  \label{EqEinNscriGamma}
    \Gamma_{\bar\mu\bar\nu}^{\bar\kappa}\in \rho\,\CI + \Hb^{\infty;2+b_0,1-0,2+b_+}
  \end{equation}
  for all indices, and $g^{0 1}-2\in\rho\,\CI+\Hb^{\infty;1+b_0,1-0,1+b_+}$, hence only the terms with $g^{0 1}\pa_1$ survive to leading order:
  \[
    \delta_g \in
    \begin{pmatrix}
      -2\pa_1 & 0 & 0 & 0 & 0 & 0 & 0 \\
      0 & -2\pa_1 & 0 & 0 & 0 & 0 & 0 \\
      0 & 0 & -2\pa_1 & 0 & 0 & 0 & 0
    \end{pmatrix}
    + (\rho\,\CI + \Hb^{\infty;2+b_0,1-0,2+b_+})\Diffb^1.
  \]
  Now $((\delta_g^*-\delta_{g_m}^*)u)_{\bar\mu\bar\nu}=-C^{\bar\kappa}{}_{\bar\mu\bar\nu}u_{\bar\kappa}$ can be calculated using~\eqref{EqEinNscriCTensor}; hence, we can now use the expressions \eqref{EqCoSchwDelta} and \eqref{EqCoSchwTdeldiff} for $G_{g_m}$ and $\tdel^*-\delta_{g_m}^*$ to evaluate $\tdel^*-\delta_g^*=(\tdel^*-\delta_{g_m}^*)-(\delta_g^*-\delta_{g_m}^*)$ and thus obtain
  \begin{equation}
  \label{EqEinNscriCD}
  \begin{split}
    \rho^{-3}(\tdel^*-\delta_g^*)\delta_g G_g\rho &\in
     -\rho^{-1}\pa_1
      \openbigpmatrix{2pt}
        2\gamma & 0 & 0 & 0 & 0 & 0 & 0 \\
        0 & 0 & 0 & 0 & 0 & 0 & 0 \\
        0 & 0 & \gamma & 0 & 0 & 0 & 0 \\
        2\pa_1 h_{1 1} & 0 & -2\pa_1 h_1{}^{\bar a} & 0 & 0 & \gamma+2\pa_1 h_{0 1} & 0 \\
        0 & 0 & \gamma-\pa_1 h_{\bar a}{}^{\bar b} & 0 & 0 & 0 & 0 \\
        2\gamma & 0 & 0 & 0 & 0 & \gamma & 0 \\
        -2\pa_1 h_{\bar a\bar b} & 0 & 0 & 0 & 0 & 0 & 0
      \closebigpmatrix \\
      &\qquad + (\CI+\Hb^{1+b_0,-0,1+b_+})\Diffb^1.
  \end{split}
  \end{equation}
  Finally, we determine the leading terms of
  \begin{equation}
  \label{EqEinNscriBox}
  \begin{split}
    (\Box_g u)_{\bar\mu\bar\nu} &= -r^{-s(\mu,\nu,\kappa,\lambda)}g^{\bar\kappa\bar\lambda}\pa_\lambda(r^{s(\mu,\nu,\kappa)}u_{\bar\mu\bar\nu;\bar\kappa}) \\
      &\qquad + g^{\bar\kappa\bar\lambda}(\Gamma_{\bar\mu\bar\lambda}^{\bar\sigma}u_{\bar\sigma\bar\nu;\bar\kappa}+\Gamma_{\bar\nu\bar\lambda}^{\bar\sigma}u_{\bar\mu\bar\sigma;\bar\kappa}+\Gamma_{\bar\kappa\bar\lambda}^{\bar\sigma}u_{\bar\mu\bar\nu;\bar\sigma}).
  \end{split}
  \end{equation}
  Consider $u_{\bar\mu\bar\nu;\bar\kappa}=r^{-s(\mu,\nu,\kappa)}\pa_\kappa(r^{s(\mu,\nu)}u_{\bar\mu\bar\nu}) - \Gamma_{\bar\mu\bar\kappa}^{\bar\lambda}u_{\bar\lambda\bar\nu} - \Gamma_{\bar\kappa\bar\nu}^{\bar\lambda}u_{\bar\mu\bar\lambda}$. For $\kappa=0$, all Christoffel symbols except those with $\mu,\lambda$ both spherical (second summand) or $\nu,\lambda$ both spherical (third summand) lie in $\rho^2\,\CI+\Hb^{\infty;2+b_0,1+b'_I,2+b_+}$, while $\Gamma^{\bar c}_{0\bar b}\in\half r^{-1}\delta_b^c+\rho^2\,\CI+\Hb^{\infty;2+b_0,2-0,2+b_+}$; the contributions of the latter cancel the leading part of the term coming from differentiating the weight $r^{-s(\mu,\nu)}\pa_0(r^{s(\mu,\nu)})=\half s(\mu,\nu)r^{-1}+r^{-2}\CI$. For $\kappa\neq 0$, we use the rough estimate \eqref{EqEinNscriGamma}, and obtain
  \begin{equation}
  \label{EqEinNscriD}
  \begin{split}
    u_{\bar\mu\bar\nu;0} &\in \pa_0 u_{\bar\mu\bar\nu} + (\rho^2\,\CI+\Hb^{\infty;2+b_0,1+b'_I,2+b_+})u \\
      &\qquad\qquad\qquad \subset (\rho\,\CI+\Hb^{\infty;2+b_0,1+b'_I,2+b_+})\Diffb^1 u, \\
    u_{\bar\mu\bar\nu;1} &\in \pa_1 u_{\bar\mu\bar\nu}+(\rho\,\CI+\Hb^{\infty;2+b_0,1-0,2+b_+})u, \\
    u_{\bar\mu\bar\nu;\bar c} &\in (\rho\,\CI+\Hb^{\infty;2+b_0,1-0,2+b_+})\Diffb^1\,u.
  \end{split}
  \end{equation}
  In the second line of \eqref{EqEinNscriBox} then, the only relevant terms (namely, with coefficients not decaying faster than $\rho_I$) are those with $u$ differentiated along $\pa_1$ and the corresponding prefactor being of size at least $\rho_I$; using
  \begin{equation}
  \label{EqEinNscriContractedGamma}
    g^{\bar\kappa\bar\lambda}\Gamma^1_{\bar\kappa\bar\lambda}\in -2 r^{-1}+\rho^2\,\CI+\Hb^{\infty;2+b_0,1+b'_I,2+b_+},
  \end{equation}
  this leaves us with
  \begin{align*}
    &g^{1 0}\Gamma_{\bar\mu 0}^{\bar\sigma}u_{\bar\sigma\bar\nu;1} + g^{1 0}\Gamma_{\bar\nu 0}^{\bar\sigma}u_{\bar\mu\bar\sigma;1} + g^{\bar\kappa\bar\lambda}\Gamma^1_{\bar\kappa\bar\lambda}u_{\bar\mu\bar\nu;1} + (\rho^2\,\CI+\Hb^{\infty;3+b_0,2-0,3+b_+})\Diffb^1 u \\
    &\quad\subset (s(\mu,\nu)-2)r^{-1}\pa_1 u_{\bar\mu\bar\nu} + \Hb^{3+b_0,1+b'_I,3+b_+}\cM u + (\rho^2\,\CI+\Hb^{\infty;3+b_0,2-0,3+b_+})\Diffb^1 u.
  \end{align*}
  Turning to the first line of \eqref{EqEinNscriBox}, for $\lambda=0$, indices $\kappa\neq 1$ contribute terms of the form $\Hb^{\infty;3+b_0,3-0,3+b_+}\Diffb^2 u$ due to \eqref{EqEinNscriD} and the decay of $g^{\bar\kappa\bar\lambda}$, while $\kappa=1$ gives a term $-2 r^{-s(\mu,\nu)}\pa_0 r^{s(\mu,\nu)}\pa_1 u_{\bar\mu\bar\nu}+(\rho^2\,\CI+\Hb^{\infty;3+b_0,2-0,3+b_+})\Diffb^2 u$. For $\lambda=1$, the term with $\kappa=0$ is equal to $-2\pa_1\pa_0 u_{\bar\mu\bar\nu}+(\rho^2\,\CI+\Hb^{\infty;3+b_0,2-0,3+b_+})\Diffb^2 u$; $\kappa=1$ produces (due to the decay of the long range component $h_{0 0}$)
  \[
    -r^{-s(\mu,\nu)}g^{1 1}\pa_1(r^{s(\mu,\nu)}u_{\bar\mu\bar\nu;1}) \in \Hb^{\infty;3+b_0,1+b'_I,3+b_+}\cM^2 u
  \]
  and spherical $\kappa$ give $\Hb^{\infty;3+b_0,2-0,3+b_+}\Diffb^2 u$. Lastly, if $\lambda$ is a spherical index and $\kappa=0,1$, we get a term in $\Hb^{\infty;3+b_0,2-0,3+b_+}\Diffb^2 u$, while for spherical $\kappa$, we use \eqref{EqEinFMetricRough} to deduce that the nontrivial spherical components of $g^{-1}$ give a term in $(\rho^2\CI+\Hb^{\infty;3+b_0,2-0,3+b_+})\Diffb^2 u$. Putting everything together, and conjugating by weights, we obtain
  \begin{equation}
  \label{EqEinNscriBoxForm}
    \rho^{-3}\Box_g \rho \in -4\rho^{-2}\pa_0\pa_1 + \Hb^{1+b_0,-1+b'_I,1+b_+}\cM^2 + (\CI+\Hb^{\infty;1+b_0,-0,1+b_+})\Diffb^2.
  \end{equation}
  (Note that due to the discussion after \eqref{EqCptScriNull}, the first term here is well-defined modulo $\Diffb^1(M;\beta^*S^2)$.) Together with the expressions \eqref{EqEinNscriRg}, \eqref{EqEinNscrisY}, \eqref{EqEinNscrisC}, and \eqref{EqEinNscriCD}, this proves the lemma.
\end{proof}

\bibliographystyle{alpha}

\end{document}